\documentclass[11pt, a4paper, reqno]{amsart}
\DeclareMathAlphabet{\mathscrbf}{OMS}{mdugm}{b}{n}

\usepackage[utf8]{inputenc}
\usepackage[slovene, english]{babel}
\usepackage{polski}
\usepackage{amsmath,amsfonts,amssymb,mathrsfs,amsthm,amscd, ulem, stmaryrd }
\usepackage[pdftex]{graphicx}
\usepackage{tikz-cd}
\usepackage{soul} 
\usepackage{MnSymbol}
\usepackage[all,cmtip]{xy}
\usepackage{enumitem}
\usepackage{manfnt}

\usepackage{slashed}

\usepackage{pgfplots}

\usepackage[mathcal]{euscript}  

\usepackage{url}

\usetikzlibrary{arrows}

 \usepackage{relsize}

\tikzset{
  no line/.style={draw=none,
    commutative diagrams/every label/.append style={/tikz/auto=false}},
  from/.style args={#1 to #2}{to path={(#1)--(#2)\tikztonodes}}}

\usetikzlibrary{circuits.ee.IEC}

\title[Even periodization of spectral stacks]{
Even periodization of spectral stacks}
\author{Rok Gregoric}
\thanks{Johns Hopkins University}
\date{\today}
\address{Johns Hopkins University, Baltimore, MD 21218, USA}
\email{rgrego12@jhu.edu}
\newtheorem{theoremi}{Theorem}
\newtheorem{theorem}{Theorem}[subsection]

\newtheorem{corollary}[theorem]{Corollary}
\newtheorem{lemma}[theorem]{Lemma}
\newtheorem{prop}[theorem]{Proposition}

\theoremstyle{definition}
\newtheorem{definition}[theorem]{Definition}
\newtheorem{definitioni}[theoremi]{Definition}
\newtheorem{variant}[theorem]{Variant}

\newtheorem{cons}[theorem]{Construction}

\newtheorem{exun}[theorem]{Example}
\newtheorem{remark}[theorem]{Remark}

\usepackage{tikz, calc}
\usetikzlibrary{matrix,arrows}

\newcommand*{\CAlg}{{\operatorname{CAlg}}}

\newcommand*{\Ani}{{\operatorname{Ani}}}

\newcommand*{\Aff}{{\mathrm{Aff}}}

\newcommand*{\mC}{\mathcal C}

\newcommand*{\mO}{\mathcal O}

\newcommand*{\mG}{\mathcal G}

\newcommand*{\mM}{\mathcal M}

\newcommand*{\mL}{\mathcal L}
\newcommand*{\mP}{\mathcal P}
\newcommand*{\mX}{\mathcal X}
\newcommand*{\mY}{\mathcal Y}

\newcommand*{\mS}{\mathcal S}

\newcommand*{\sO}{\mathcal O}
\newcommand*{\sF}{\mathscr F}

\newcommand*{\E}{\mathbb E_\infty}

\newcommand*{\heart}{\heartsuit}

\newcommand*{\sheafhom}{\mathscr{H}\kern -.5pt om}
\DeclareMathOperator{\Novak}{\mathscr{N}\text{\kern -3pt {\calligra\large ovak}}\,\,}

\usepackage{amsmath,calligra,mathrsfs}
\DeclareMathOperator{\fHom}{\mathscr{H}\text{\kern -3pt {\calligra\large om}}\,}

\DeclareMathOperator{\Hom}{\operatorname{Hom}}

\DeclareMathOperator{\Sp}{\operatorname{Sp}}

\DeclareMathOperator{\Fun}{\operatorname{Fun}}

\DeclareMathOperator{\Spec}{\operatorname{Spec}}
\DeclareMathOperator{\Spf}{\operatorname{Spf}}

\DeclareMathOperator{\Proj}{\operatorname{Proj}}

\DeclareMathOperator{\Map}{\operatorname{Map}}
\DeclareMathOperator{\QCoh}{\operatorname{QCoh}}

\DeclareMathOperator{\Rep}{\operatorname{Rep}}

\DeclareMathOperator{\Sym}{\operatorname{Sym}}

\DeclareMathOperator{\MU}{\mathrm{MU}}
\DeclareMathOperator{\MUP}{\mathrm{MUP}}
\DeclareMathOperator{\M}{\mathcal M^\mathrm{or}_\mathrm{FG}}

\DeclareMathOperator{\GL}{\operatorname{GL}}
\DeclareMathOperator{\G}{\mathbf G}
\DeclareMathOperator{\Z}{\mathbf Z}
\DeclareMathOperator{\Q}{\mathbf Q}

\DeclareMathOperator{\LMod}{\operatorname{LMod}}

\DeclareMathOperator{\Mod}{\operatorname{Mod}}

\renewcommand{\i}{\infty}

\newcommand{\T}{\mathbf T}

\newcommand{\w}{\widehat}

\renewcommand{\i}{\infty}

\makeatletter
\def\@tocline#1#2#3#4#5#6#7{\relax
  \ifnum #1>\c@tocdepth 
  \else
    \par \addpenalty\@secpenalty\addvspace{#2}%
    \begingroup \hyphenpenalty\@M
    \@ifempty{#4}{%
      \@tempdima\csname r@tocindent\number#1\endcsname\relax
    }{%
      \@tempdima#4\relax
    }%
    \parindent\z@ \leftskip#3\relax \advance\leftskip\@tempdima\relax
    \rightskip\@pnumwidth plus4em \parfillskip-\@pnumwidth
    #5\leavevmode\hskip-\@tempdima
      \ifcase #1
       \or\or \hskip 1em \or \hskip 2em \else \hskip 3em \fi%
      #6\nobreak\relax
    \dotfill\hbox to\@pnumwidth{\@tocpagenum{#7}}\par
    \nobreak
    \endgroup
  \fi}
\makeatother

\usepackage{color}   
\usepackage[hypertexnames=false]{hyperref}
\hypersetup{
    linktoc=all,     
    linkcolor=black,  
}

\renewcommand{\i}{\infty}
\renewcommand{\o}{\otimes}

\usepackage{epigraph}

\newcommand{\stackspace}{3.5}
\newcommand{\stack}[2][1cm]{\;\tikz[baseline, yshift=.65ex]%
    {\foreach \k [evaluate=\k as \r using (.5*#2+.5-\k)*\stackspace] in {1,...,#2}{%
    \ifodd\k{\draw[<-](0,\r pt)--(#1,\r pt);}%
    \else{\draw[->](0,\r pt)--(#1,\r pt);}\fi
    }}\;}

\usepackage{relsize}
\usepackage[bbgreekl]{mathbbol}
\usepackage{amsfonts}

\DeclareSymbolFontAlphabet{\mathbb}{AMSb}
\DeclareSymbolFontAlphabet{\mathbbl}{bbold}
\newcommand{\Prism}{{\mathlarger{\mathbbl{\Delta}}}}

\newcommand{\shear}{{\mathbin{\mkern-4mu\fatslash}}}

\begin{document}

\begin{abstract}
We introduce the operation of \textit{even periodization} on nonconnective spectral stacks. We show how to recover from it the even filtration of Hahn-Raksit-Wilson, and (a Nygaard-completion  of) the filtered prismatization stack of Bhatt-Lurie and Drinfeld. We prove that the moduli stack of oriented elliptic curves of Goerss-Hopkins-Miller and Lurie is the even periodization of the spectrum of topological modular forms. Likewise, the chromatic moduli stack, which we previously studied under the name of the moduli stack of oriented formal groups, arises via  even periodization from the sphere spectrum. Over the complex bordism spectrum, we relate even periodization with the symmetric monoidal shearing, in the sense of Devalapurkar, of free gradings.
\end{abstract}

\maketitle

\setcounter{tocdepth}{2}
\tableofcontents

\newpage

\section{Introduction}

The introduction of \textit{prismatic cohomology} \cite{BS22} has spurred on a flurry of activity on both sides of the fence between homotopy theory and 
$p$-adic geometry. This is in part because the original approach to prismatic cohomology by Bhatt-Morrow-Scholze \cite{BMS2} relied on topological Hochschild homology.
\begin{itemize}
\item On the $p$-adic side, the developments have included a site-theoretic approach to prismatic cohomology by Bhatt-Scholze \cite{BS22}, circumventing the use of THH, and relying instead on a novel notion of prisms. Another subsequent development was the introduction
of the \textit{prismatization stacks}
 by Bhatt-Lurie and Drinfeld \cite{BLb}, \cite{Bhatt F-gauges}, \cite{Drinfeld}. These ``geometrize" prismatic cohomology and related structures, in analogy with the stacky geometrizations of de Rham cohomology via the de Rham space of Simpson in characteristic zero  \cite{Simpson}, and via Witt vectors on the perfection by Drinfeld in positive characteristic \cite{Drinfeld stack}.

\item On the homotopy-theoretic side, a notable development has been the introduction of the \textit{even filtration} of Hahn-Raksit-Wilson \cite{HRW}. It specializes to the motivic filtrations on $\mathrm{THH}$, $\mathrm{TC}^-$ and $\mathrm{TP}$ coming from prismatic cohomology, as well as to various other filtrations of interest such as the HKR filtration. Yet it is defined in the astoundingly simple  way of restricting to ring spectra with odd homotopy groups. In particular, unlike the approach to the motivic filtration in \cite{BMS2}, the even periodization does not require the use of any p-adic or number theoretic ideas -- such as qrsp rings or quasi-syntomic descent -- nor the notion of prisms.
\end{itemize}
In this paper, we introduce a joint generalization of the even filtration and prismatization, called \textit{even periodization}. This is an operation on spectral stacks, which  approximates them as closely as possible with affines corresponding to even periodic ring spectra.
When applied to affines, even periodization recovers the even filtration of Hahn-Raksit-Wilson as the Postnikov towers for quasi-coherent sheaves.

Even periodization is explicitly computable in a number of cases,  recovering several spectral stacks that are known to be of interest to chromatic homotopy theory. From the spectrum of the sphere spectrum $\mathrm{Spec}(\mathbf S)$, it produces the \textit{chromatic base stack} $\mM$ (which had been called the \textit{moduli stack of oriented formal groups} and denoted $\M$ in \cite{ChromaticCartoon}), and which was shown to geometrize many aspects of chromatic homotopy theory in \cite{ChromaticFiltration} and be related to synthetic spectra in \cite{Synthetic}. From the ring spectrum of topological modular forms $\mathrm{Spec}(\mathrm{TMF})$,  it reconstructs the \textit{moduli stack of oriented elliptic curves} $\mM^\mathrm{or}_\mathrm{Ell}$ of Goerss-Hopkins-Miller and Lurie, which are usually used to define it. The latter result is connected to the chromatic affineness theorem of Mathew-Meier \cite{Affineness in chromatic}, and gives a precise way in which the spectral stack  $\mM^\mathrm{or}_\mathrm{Ell}$  is ``even periodically affine," in analogy with To\"en's theory of affine stacks \cite{Champs affines}.

Furthering the analogy with the theory of affine stacks, we show that  even periodization  is \textit{fully faithful} on the  class of ``even-convergent" spectral affines. This amounts to fully faithfully encoding 
a large class of ring spectra which are not themselves even periodic -- including, according to upcoming work of Burklund-Krause announced in \cite{Piotr's even}, all connective commutative ring spectra -- in terms of periodic spectral stacks.

Even periodization gives rise to a \textit{natural  spectral enhancement of} (a Nygaard-completed version of) the \textit{filtered prismatization} stack of Bhatt. It is obtained as the coarse quotient of the even periodization of the spectral free loop space under its circle action.
This is a globalization of the approach to prismatic cohomology via topological cyclic homology from \cite{BMS2}.
Notably, this construction is also applicable to objects outside $p$-adic formal geometry. For instance, it identifies the
 universal 1-dimensional smooth formal group as the
 Nygaard-complete filtered prismatization of the sphere spectrum. This  is consistent with the work in progress of Devalapurkar-Hahn-Raksit-Yuan, announced in \cite{Arpon's talk}.

\subsection{Summary of results}

By the term \textit{spectral stack}, we will mean\footnote{In the main body of the text, we usually say \textit{nonconnective spectral stack}, to emphasize that we are not imposing any connectivity assumptions on $\E$-rings. But here in the introduction, we will forego the \textit{nonconnective} adjective.} an accessible functor $\mX:\CAlg\to \mathrm{Ani}$ which satisfies descent for the fpqc topology  in the sense of Lurie \cite{SAG}, i.e.\ the $\i$-category of spectral stacks is
$$
\mathrm{SpStk} \, :=\,  \mathrm{Shv}^\mathrm{acc}_\mathrm{fpqc}(\CAlg).
$$
For example, any $\E$-ring $A$, connective or otherwise, gives rise to an \textit{affine} $\Spec(A)$, given by the corepresentable functor $B\mapsto \Map_\CAlg(A, B)$. For any spectral stack $\mX$, we define its \textit{connective cover} $\mX^\mathrm{cn}$ and  \textit{underlying classical stack} $\mX^\heart$  by left Kan extension from affines, for which we set $\Spec(A)^\mathrm{cn}:=\Spec(\tau_{\ge 0}(A))$ and $\Spec(A)^\heart:=\Spec(\pi_0(A))$.

\begin{definitioni}\label{Defi of evp}
For any spectral stack $\mX$, we define its \textit{even periodization} as the colimit
$$
\mX^\mathrm{evp}\,\,\simeq \varinjlim_{\underset{A \text{ even periodic $\E$ ring}}{\Spec(A)\to \mX}}\Spec(A),
$$
where an $\E$-ring $A$ is \textit{even periodic} if $\pi_{2*+1}(A)\simeq 0$ and 
$\pi_{*+2}(A)\simeq \pi_*(A)\o_{\pi_0(A)}\pi_2(A)$.
\end{definitioni}
If one wishes, the weak $2$-periodicity requirement above can be substituted with the stronger requirement $\pi_{*+2}(A)\simeq \pi_*(A)$  without affecting the result of even periodization.

Any spectral stack $\mX$ admits a canonical map
$
\mX^\mathrm{evp}\to \mX
$
that is the closest approximation to it via an even periodic spectral stack.
Indeed, the even periodization functor $\mX\mapsto \mX^\mathrm{evp}$ is a colimit-preserving colocalization of $\mathrm{SpStk}$, with the subcategory of even periodic spectral stacks $\mathrm{SpStk}^\mathrm{evp}\subseteq\mathrm{SpStk}$ as its essential image. The latter also admits a description as the $\i$-category of accessible fpqc sheaves on even periodic $\E$-rings
$$
\mathrm{SpStk}^\mathrm{evp}\,\simeq\, \mathrm{Shv}^\mathrm{acc}_\mathrm{fpqc}(\CAlg^\mathrm{evp}).
$$

In many cases, the even periodization is explicitly computable:

\begin{theoremi}\label{Theoremi I.1}
Even periodization in the following cases recovers other previously known even periodic spectral stacks:
\begin{eqnarray}
\Spec(\mathbf S)^\mathrm{evp}&\simeq & \mM \label{I.1.1}\\
\Spec(L^f_n\mathbf S)^\mathrm{evp}\simeq \Spec(L_n\mathbf S)^\mathrm{evp} &\simeq & \mM^{\le n}\\
\Spec(\mathrm{TMF})^\mathrm{evp} &\simeq & \mM{}^\mathrm{or}_\mathrm{Ell}\\
\Spec(\mathrm{Tmf})^\mathrm{evp} &\simeq & \overline{\mM}{}^\mathrm{or}_\mathrm{Ell}\\
\Spec(\mathrm{KO})^\mathrm{evp} & \simeq & \Spec(\mathrm{KU})/\mathrm C_2\\
\Spec(\mathrm{MU})^\mathrm{evp} & \simeq & \Spec(\mathrm{MUP})/\mathbf G_m \label{I.1.6}\\
\Spec(\mathbf Z)^\mathrm{evp} &\simeq & \Spec(\mathbf Z[\beta^{\pm 1}])/\mathbf G_m\label{I.1.7}
\end{eqnarray}
where respectively:
\begin{enumerate}
\item $\mM$ is the chromatic base stack (which we have previously studied under the name of the moduli stack of oriented formal groups and denoted $\M$ in \cite{ChromaticCartoon}).

\item $\mM^{\le n}$ is the chromatic open substack of height $\le n$ oriented formal groups (previously denoted $\mM{}^{\mathrm{or}, \le n}_\mathrm{FG}$ in \cite{ChromaticFiltration}).

\item $\mM{}^\mathrm{or}_\mathrm{Ell}$ is the  moduli stack of oriented elliptic curves, constructed by obstruction theory by Goerss-Hopkins-Miller, and as a spectral moduli stack by  Lurie in \cite{Elliptic 2}.

\item $\overline{\mM}{}^\mathrm{or}_\mathrm{Ell}$ is its Deligne-Mumford compactification, also due to Goerss-Hopkins-Miller.

\item The complex topological K-theory spectrum $\mathrm{KU}$ is equipped with the  $\mathrm C_2$-action, corresponding to complex conjugation of $\mathbf C$-vector spaces.

\item $\mathrm{MUP}$ is the periodic complex bordism spectrum, with its Thom $\E$-ring structure. The $\mathbf G_m$-action on $\Spec(\mathrm{MUP})$ corresponds to the grading $\mathrm{MUP}\simeq \bigoplus_{n\in \mathbf Z}\Sigma^{2n}(\mathrm{MU})$.

\item the generator $\beta$ is in graded degree $1$ and homotopical degree $2$.

\end{enumerate}
\end{theoremi}

Points \eqref{I.1.6} and \eqref{I.1.7} above can be extended by expressing the even periodization of any even $\E$-algebra $A$ over the complex bordism spectrum $\mathrm{MU}$ as the shearing $\Spec(A)^\shear$, in terms of the symmetric monoidal shearing of gradings for $\mathrm{MU}$-modules from \cite{Sanath}.

In all the points \eqref{I.1.1} -- \eqref{I.1.7} of Theorem \ref{Theoremi I.1}, even periodization is applied to affines $\Spec(A)$, and the $\E$-ring $A$ can be recovered from the spectral stack $\mX=\Spec(A)^\mathrm{evp}$ as its $\E$-ring of global functions $\sO(\mX)\simeq A$. These are instances of a general phenomenon:

\begin{theoremi}\label{Theoremi ffaith}
Let $\CAlg^\mathrm{evc}\subseteq\CAlg$ denote the full subcategory of $\E$-rings $A$ whose even filtration $\mathrm{fil}_\mathrm{ev}^{\ge *}(A)$ converges to $A$. Then even periodization $A\mapsto\Spec(A)^\mathrm{evp}$ gives a fully faithful embedding
$$
(\CAlg^\mathrm{evc})^\mathrm{op}\hookrightarrow\mathrm{SpStk}^\mathrm{evp},
$$
with a left inverse given by the $\E$-ring of global sections functor, i.e.\
$$
\sO(\Spec(A)^\mathrm{evp})\,\simeq \, A
$$
for all $A\in\CAlg^\mathrm{evc}$.
\end{theoremi}

Let us give three different points one can draw from Theorem \ref{Theoremi ffaith}:
\begin{itemize}
\item It seems  reminiscent of the theory of \textit{affine stacks} from \cite{Champs affines} (called \textit{coaffine stacks} in \cite{DAGVIII}), where sending a coconnective $\E$-ring $A$ over a field $k$ of characteristic $0$ to the restriction $\Spec(A)\vert{}_{\CAlg_k^\heartsuit}$ of the \textit{nonconnective} affine $\Spec(A):\CAlg_k\to \mathrm{Ani}$ to the subcategory $\CAlg_k^\heart\subseteq\CAlg_k$ defines a fully faithful embedding. Conversely, in Theorem \ref{Theoremi ffaith} we send an $\E$-ring $A\in \CAlg{}^{\mathrm{evc}}$ to the restriction $\Spec(A)\vert_{\CAlg^\mathrm{evp}}$  of the affine $\Spec(A)$ to the subcategory $\CAlg^\mathrm{evp}\subseteq\CAlg$. It would be interesting to know if the essential image of this fully faithful embedding inside $\mathrm{SpStk}$ also admits a simple explicit description, as affine stacks of To\"en do.

 \item Any even $\E$-ring $A$ obviously belongs to $\CAlg^\mathrm{evc}$, and may thus by Theorem \ref{Theoremi ffaith} be faithfully encoded by the even periodic stack $\Spec(A)^\mathrm{evp}$. The underlying classical stack of the latter can be identified as the quotient stack
 $$
(\Spec(A)^\mathrm{evp})^\heart\simeq \Spec_{\mathrm B\mathbf G_m}(\pi_{2*}(A)),
$$
recording all the homotopy groups of $A$ and their grading. The canonical map of underlying classical stacks $(\Spec(A)^\mathrm{evp})^\heart\to \mM^\heart\simeq \mathcal M^\heart_\mathrm{FG}$ therefore classifies a graded formal group over $\pi_{2*}(A)$, which is precisely the classical Quillen formal group $\Spf(A^*(\mathbf{CP}^\infty))$. This \textit{recovers} the classical graded Quillen formal groups from the \textit{periodic} Quillen formal groups $\Spf(B^0(\mathbf{CP}^\i))$ that we use instead for even periodic $\E$-rings $B$. More broadly, the even periodization gives a way to counter, at least in the even case, the criticism of spectral algebraic geometry that its underlying classical objects only see the information of $\pi_0$ as opposed to $\pi_*$.

\item 
According to upcoming work of Achim Krause and Robert Burklund, announced in
\cite[Remark 8.5]{Piotr's even}, all connective $\E$-rings belong to $\CAlg^\mathrm{evc}$. Theorem \ref{Theoremi ffaith} therefore restricts to a fully faithful embedding of all connective affines into even periodic spectral stacks. Somewhat counterintuitively, this suggests that there is ``sufficient room" inside even periodic geometry to faithfully model all of connective spectral algebraic geometry.

\end{itemize}

As mentioned  in the first part of the introduction,
one major motivation for the even periodization was in trying to globalize (in the sense of extending to non-affine spectral stacks) the even filtration of Hahn-Raksit-Wilson \cite{HRW}. To explain how the two are connected,  note that any quasi-coherent sheaf $\sF\in\QCoh(\mX)$ gives rise to a sheaf-level Postnikov filtration  $\tau_{\ge *}(\sF)\in \mathrm{Fil}(\mathrm{QCoh}(\mX^\mathrm{cn}))$. We showed in \cite{ChromaticCartoon} that the chromatic base stack $\mM$ \textit{geometrizes} the Adams-Novikov filtration in the sense that there is a canonical isomorphism of filtered $\E$-rings
$$
\mathrm{fil}_\mathrm{AN}^{\ge *}(\mathbf S)\,\simeq \, \Gamma( \mM^\mathrm{cn}; \, \tau_{\ge 2*}(\sO_{\mM})).
$$
Since the Adams-Novikov filtration agrees with the even filtration for the sphere spectrum, this is generalized by the even periodization geometrizing the even filtration:

\begin{theoremi}\label{Theoremi evfiltr compare}
For any $\E$-ring $A$, its even filtration may be expressed via the Postnikov filtration of $\sO_{\Spec(A)^\mathrm{evp}}$ as
$$
\mathrm{fil}_\mathrm{ev}^{\ge *}(A)\,\simeq \, \Gamma((\Spec(A)^\mathrm{evp})^\mathrm{cn};\,\tau_{\ge 2*}(\sO_{\Spec(A)^\mathrm{evp}}))
$$
\end{theoremi}

In fact, the functor $\mX\mapsto \mX\times \mM$ can be viewed analogously to Definition \ref{Defi of evp} as \textit{complex periodization}, and expressed via an analogous colimit formula
$$
\mX\times \mM\,\,\simeq \varinjlim_{\underset{A \text{ complex periodic $\E$ ring}}{\Spec(A)\to \mX}}\Spec(A),
$$
where following \cite{Elliptic 2} an $\E$-ring is \textit{complex periodic} if it is both complex orientable and weakly $2$-periodic. There is a canonical map $\mX^\mathrm{evp}\to \mX\times \mM$ for any spectral stack $\mX$, which geometrizes the relationship between the even filtration and the Adams-Novikov filtration. For instance, the \textit{Adams-Novikov descent} of \cite{HRW} becomes the observation that $\mX^\mathrm{evp}\simeq (\mX\times \mM)^\mathrm{evp}$ for any spectral stack $\mX$.

Next, we wish to describe analogously to Theorem \ref{Theoremi evfiltr compare}, and in that sense geometrize, the circle-equivariant version of the even filtration $\mathrm{fil}_{\mathrm{ev}, \mathrm h\mathbf T}(A)$ for an $\E$-ring with a $\T$-action $A$ from \cite{HRW}. For this purpose, we are led to desire a notion of a \textit{coarse quotient} $\mX/\!/\mathbf T$ of a spectral stack $\mX$ along a $\mathbf T$-action, which would be defined  in the affine case by $\Spec(A)/\!/\mathbf T\,\approx\, \Spec(A^{\mathrm h\mathbf T})$ and extended by descent. 
The issue (fundamentally the same as underlying Mumford's geometric invariant theory \cite{GIT}) is that the passage to homotopy fixed-points need not preserve fpqc covers, and so there is no \textit{\`a priori} way to make this well defined.
Our solution is instead to use the following slight strengthening of a unipotence result from \cite{MNN}:

\begin{theoremi}\label{Theoremi H}
The homotopy fixed-point functor induces an equivalence of symmetric monoidal $\i$-categories
$$
\Mod_A(\Sp^{\mathrm B\mathbf T})\simeq \Mod_{A^{\mathrm h\mathbf T}}^\mathrm{cplt}
$$
for any complex periodic $\E$-ring with a $\T$-action $A$.
\end{theoremi}

Using this result, we can show the functor $A\mapsto \Spf(A^{\mathrm h\mathbf T})$ does satisfy $\mathbf T$-equivariant fpqc descent when for even periodic $\E$-rings $A$, allowing us to define a coarse quotient
$$\mX/\!/\mathbf T \,\,:= \varinjlim_{\underset{\text{$\T$-equivariant}}{\Spec(A)\to \mX}}\Spf(A^{\mathrm h\mathbf T})$$
as a \textit{formal} spectral stack, for any even periodic spectral stack with a $\T$-action $\mX$. The passage to the coarse quotient is not a very destructive process: there is a canonical equivalence $\mathfrak{QC}\mathrm{oh}(\mX/\!/\mathbf T)\simeq \QCoh(\mX/\mathbf T)$ between  the complete quasi-coherent sheaves on the coarse quotient and the quasi-coherent sheaves on the $\i$-categorical quotient.

\begin{theoremi}
Let $A$ be an $\E$-ring with a $\T$-action and let $\mX\,:=\,\Spec(A)^{\mathrm{evp}}/\!/\T$ be the coarse quotient of its even periodization. There is an equivalence of filtered $\E$-rings
$$
\mathrm{fil}^{\ge *}_{\mathrm{ev}, \mathrm h\T}(A)\, \simeq \, \Gamma(\mX^\mathrm{cn}
;\,
\tau_{\ge 2*}(\sO_{\mX})).
$$\end{theoremi}

Finally, we use the even periodization and the coarse quotient to give a canonical spectral enhancement of  version of the \textit{filtered prismatization} $X^{\mathcal N}$ of Bhatt-Lurie-Drinfeld \cite{Bhatt F-gauges} -- or at least a Nygaard-complete version $X^{\widehat{\mathcal N}}$ thereof -- where $X$ is a $p$-adic formal scheme.

\begin{theoremi}\label{Theoremi Nygaard}
Let $X$ be a quasi-syntomic $p$-adic formal scheme, and let $\mathfrak X \,:=\, (\mathscr LX)^\mathrm{evp}/\!/\T$ denote the coarse quotient of the even periodization of the free loop space $\mathscr LX$ in $\mathrm{SpStk}$. Its underlying classical formal stack may be identified as
$
\mathfrak X^\heart\, \simeq \, X^{\widehat{\mathcal N}}
$
with the Nygaard-complete filtered prismatization of $X$.
\end{theoremi}

This may be seen as a globalization of the original approach to prismatization using topological Hochschild homology in \cite{BMS2}, namely of the identification of graded rings $\pi_{2*}(\mathrm{TC}^-(S)) \simeq  \mathrm{Fil}^{\ge *}_{\mathcal N}(\widehat{\Prism}_S)$ for any quasiregular semiperfectoid ring $S$.
In analogy with Atiyah-Segal completion-type  results, we suspect (but leave for future work) that ``Nygaard decompleting" Theorem \ref{Theoremi Nygaard}, i.e.\ finding a canonical spectral enhancement of the filtered prismatization $X^\mathcal N$, would require taking into account the genuine $\mathbf T$-equivariance on the free loop space $\mathscr LX$.

By virtue of Theorem \ref{Theoremi Nygaard}, we may take
$$
\mX^{\widehat{\mathcal N}}:=((\mathscr L\mX)^\mathrm{evp})/\!/\mathbf T)^\heart
$$
as \textit{the definion} of the Nygaard-complete filtered prismatization for any spectral stack $\mX$. For lci schemes over $\mathbf Q$, this recovers the Hodge-filtered de Rham stack of Simpson. In this sense, we  also find the Nygaard-complete filtered prismatization of the sphere spectrum
$$
\Spec(\mathbf S)^{\widehat{\mathcal N}}\simeq \widehat{\mathbf{G}}^\mathrm{univ}
$$
to be the universal (1-dimensional smooth) formal group.

\subsection{Related work}
Related to this paper's approach of using spectral algebraic geometry to  ``geometrize" the even filtration of \cite{HRW} is the work on the synthetic circle in \cite{Alice and Tasos} and \cite{Ben and Noah}.  Similarly to our use of Theorem \ref{Theoremi H} in defining the coarse quotient, the latter's approach to circle-equivariance also makes crucial use of a version \cite[Lemma 3.11 (Synthetic unipotence)]{Ben and Noah} of the unipotence results from \cite{MNN}, suggesting that perhaps their theory of even synthetic spectra might also admit a geometric interpretation along the lines of this paper.
Also closely related to this paper is \cite{Piotr's even}, which gives a site theoretic reinterpretation of (and extension to $\mathbb E_1$-rings of) the even filtration. It is closely connected with the \textit{synthetic spectra} of \cite{Piotr's Synthetic} (the idea of which can be further traced back to \cite{C-mmf}), and which we related to the chromatic moduli base $\mM$ in \cite{Synthetic}.

Perhaps the most substantial overlap that the present paper has is with the upcoming work of  Devalapurkar, Hahn, Raksit, Senger, Yuan, some of which has been announced in \cite{Arpon's talk} and \cite{Jeremy's talk}.
The author has been vaguely aware of such work in progress, but not of the specific details; to the best of our knowledge, the two projects arose independently.
The work in question concerns stacky versions of the even filtration, as well as spectral enhancements of prismatization. As opposed to the present paper, they  manage to give spectral analogues and enhancements  of
 \textit{non-Nygaard-complete} versions of prismatization as well.
It this very fact which leads the author to suspect that their constructions will likely be more sophisticated, as opposed to the fundamentally elementary nature of even periodization.

Finally, allow us to highlight that none of the above discussions  of the even filtration and its variants consider a \textit{periodic} version. The focus on even \textit{periodicity} --  taking the deceptively-simple idea of restricting to even $\E$-rings, which is the impetus of the even filtration, and restricting even further to even periodic $\E$-rings -- and the advantages, as well as the ostensive lack of substantial disadvantages, that this brings, is perhaps the key innovation of the present paper.

\subsection{Notation and terminology}
We freely use the language of $\i$-categories, and mostly follow the conventions from \cite{HA}, \cite{SAG}, \cite{Elliptic 2}, etc. One prominent exception is that we conform to the rising trend of using the term \textit{anima} and notation $\mathrm{Ani}$ for what Lurie, along with a majority of others working in the field, call \textit{spaces} and denote their $\i$-category by $\mS$. Some other  names for objects of this $\i$-category include $\i$-groupoids, $(\infty, 0)$-categories, $\i$-categorical presheaves on a point, animated sets, and homotopy types.

Unless said otherwise, our notions are interpreted $\i$-categorically. That is to say, we write $\CAlg_R$ and $\Mod_R$ for the $\i$-categories of $\E$-algebras over $R$ and $R$-module spectra respectively. When working over the initial $\E$-ring, i.e.\ the sphere spectrum $R=\mathbf S$, we often omit the subscripts.
We will use $\otimes_R$ to denote the relative smash product over an $\E$-ring $R$. When $R$ is a classical ring, this is the derived tensor product, often rather denoted by $\otimes_R^\mathrm L$. We likewise use $\QCoh$ to mean the $\i$-categorical quasi-coherent sheaves, and indicate with a superscript $\heart$ when we mean the classical notion (as it is conveniently the heart of a $t$-structure on the $\i$-categorical variant).

 \subsection{Acknowledgments}
This project owes its inception to a visit to the University of Chicago in Fall 2021, where Akhil Mathew introduced the author to prismatization, and suggested a relationship with the chromatic base stack. The author is grateful to Tomer Schlank and Jack Morava for displaying interest in the project at several key stages, and thus preventing its abandonment.
Thanks also to Jeremy Hahn, who has kindly explained to the author the connections between this paper and his upcoming work with Devalapurkar, Raksit, Senger, and Yuan, as well as to Ben Antieau for several corrections and useful remarks on an earlier draft.
Most of all, we are exceptionally grateful to David Ben-Zvi, Anish Chevalavada, David Gepner, and Deven Manam, for many useful ideas and inspiring conversations regarding this project. 
\newpage

\section{The even filtration}\label{Section 2}

\subsection{Even and even periodic filtration}
Let us begin by making a list of certain classes of $\E$-rings which we will be especially concerned with in this paper. 

\begin{definition}\label{List of rings}
Let $A$ be an $\E$-ring spectrum. We say that:
\begin{itemize}
\item $A$ is \textit{even} if $\pi_\mathrm{odd}(A)=0$.
\item $A$ is \textit{weakly $2$-periodic} if $\pi_2(A)$ is a flat $\pi_0(A)$-module and the multiplication map $\pi_n(A)\otimes_{\pi_0(A)}\pi_2(A)\to \pi_{n+2}(A)$ is an isomorphism of abelian groups for all $n\in \mathbf Z$.
\item $A$ is \textit{complex periodic} if it is complex orientable and weakly $2$-periodic.
\item $A$ is \textit{even periodic} if it is even and weakly $2$-periodic.
\item $A$ is \textit{even strongly periodic} if it is even periodic and $\pi_2(A)$ is a free $\pi_0(A)$-module.
\end{itemize}
Let $\CAlg^\mathrm{ev}, \CAlg^\mathrm{evp},  \CAlg^{\mathrm{ev}\mathfrak P},\CAlg^{\mathbf C\mathrm {p}},
\CAlg^{\mathbf C\mathrm {or}}
\subseteq\CAlg$ denote the full subcategories of $\E$-ring, spanned respectively by even, even periodic, even strongly  periodic, complex periodic, and complex orientable $\E$-rings. They are interrelated by the inclusions
$$
\begin{tikzcd}
 & & \CAlg^\mathrm{ev} \arrow[hookrightarrow]{dr} & & \\
\CAlg^{\mathrm{ev}\mathfrak P}\arrow[hookrightarrow]{r} &\CAlg^\mathrm{evp} \arrow[hookrightarrow]{ur} \arrow[hookrightarrow]{dr} & &\CAlg^{\mathbf C\mathrm{or}} \arrow[hookrightarrow]{r} & \CAlg\\
 & & \CAlg^{\mathbf C\mathrm{p}} \arrow[hookrightarrow]{ur} & &
\end{tikzcd}
$$
We will also sometimes use $\CAlg^?_A$ to denote the pullback $\i$-category $\CAlg_A\times_{\CAlg}\CAlg^?$, where $?\in \{\mathrm{ev}, \mathrm{evp}, \mathrm{ev}\mathfrak P, \mathbf C\mathrm p, \mathbf C\mathrm{or} \}$.
\end{definition}

\begin{remark}
Our conventions disagree slightly from the standard convention in stable homotopy theory, where ``even periodic" is usually used for what we call ``even strongly periodic", and our notion of even periodicity has no name.
We will justify this shift in emphasis
in Section \ref{Section: variant strict}, see in particular the discussion above Proposition \ref{EVP affines}.
This convention is also not without precedents, e.g. \cite[Definition 2.3]{Affineness in chromatic}.
\end{remark}

\begin{remark}\label{Remark weak vs strong 2-periodicity remark}
Note that weak $2$-periodicity of an $\E$-ring $A$ is equivalent to demanding that the $A$-module $\Sigma^2(A)$ is invertible. In contrast, strong $2$-periodicity -- asking that $\pi_{n+2}(A)\simeq \pi_n(A)$ for all $n\in \mathbf Z$ -- would instead demand that there exists an $A$-module equivalence $\Sigma^2(A)\simeq A$.
\end{remark}

In some sense, most of this paper concerns the interactions between the proper inclusions $\CAlg^\mathrm{evp}\subset\CAlg^{\mathbf C\mathrm{p}}\subset \CAlg$. But, following Hahn-Raskit-Wilson, we begin instead with the inclusion $\CAlg^\mathrm{ev}\subset\CAlg$.

\begin{definition}[{\cite[Construction 2.1.3]{HRW}}]\label{Def 1 - list}
The \textit{even filtration} is the right Kan extension $\mathrm{fil}_\mathrm{ev}^{\ge *} : \CAlg\to\mathrm{FilCAlg}$ of the double-speed Postnikov filtration $\tau_{\ge 2*}$ along the inclusion $\CAlg^\mathrm{ev}\subseteq\CAlg$. That is to say, for an $\E$-ring $A$ it is given by
$$
\mathrm{fil}^{\ge *}_\mathrm{ev}(A)\,\,\simeq\varprojlim_{B\in\CAlg_A^\mathrm{ev}}\tau_{\ge 2*}(B).
$$
\end{definition}

The key property of the even filtration, and which makes it computable in a range of cases, is that is satisfies descent with respect to a variant of the flat topology on even $\E$-rings, which we now recall.

\begin{definition}[{\cite[Definition 2.2.15]{HRW}}]\label{Def eff}
A map of $\E$-rings $A\to B$ is said to be \textit{eff} if it is the case that for any $C\in\CAlg^\mathrm{ev}_A$, the base-change $B\o_AC$ is even and the ring map $\pi_*(B)\to \pi_*(B\o_A C)$ is faithfully flat. 
\end{definition}

\begin{theorem}[eff descent, {\cite[Corollary 2.2.17]{HRW}}]\label{Eff descent citation result}
Let $A\to B$ be an eff morphism of $\E$-rings.
Then the canonical map of filtered spectra
$$
\mathrm{fil}^{\ge *}_\mathrm{ev}(A)\to \varprojlim_{\mathbf\Delta}\mathrm{fil}^{\ge *}_\mathrm{ev}(B^\bullet)
$$
is an equivalence,
where $B^\bullet = B^{\otimes_A (\bullet+1)}$ is the cosimplicial cobar resolution.\end{theorem}

Using the above eff descent result, we show that the even filtration could have been obtained from even periodic $\E$-rings, in the sense of Definition \ref{Def 1 - list}, a much smaller class than that of all even $\E$-rings.

\begin{prop}\label{Even filtration is even periodic filtration}
In the definition of the even filtration, the subcategory $\CAlg^\mathrm{ev}\subseteq\CAlg$ may at no loss be replaced with $\CAlg^\mathrm{evp}\subseteq\CAlg$.
\end{prop}

\begin{proof}
Choose an $\E$-ring structure on the periodic complex bordism spectrum $\mathrm{MUP}$.
We claim that the canonical map $A\to A\otimes_\mathbf S\mathrm{MUP}
$ is eff for any $\E$-ring $A$. Since eff maps are closed under base-change, it suffices to show that the unit map $\mathbf S\to \mathrm{MUP}$ is eff. The latter map can be factored as as $\mathbf S\to \MU\to\mathrm{MUP}$, whose first map being eff is the result of a standard complex-oriented cohomology computation (or see \cite[Proposition 2.2.20]{HRW}). It thus remains to show that the periodization map $\MU\to \mathrm{MUP}$ is eff. It is obtained by smash-product along $\mathbf S\to \MU$ from a $\mathbb E_2$-ring map $\mathbf S\to \bigoplus_{n\in \mathbf Z}\Sigma^{2n}(\mathbf S)$. So we are reduced to showing that for any even $\E$-ring $B$, the ring map
$$
\pi_*(B)\to\pi_*\left(\bigoplus_{n\in\mathbf Z}\Sigma^{2n}(B)\right)= \bigoplus_{n\in\mathbf Z}\pi_{*+2n}(B)\simeq \pi_*(B)[t^{\pm 1}]
$$
is faithfully flat. To verify that, note that this map is base-changed to $\pi_*(B)$ from the faithfully flat map $\mathbf G_m\to\mathrm{Spec}(\mathbf Z)$.

Having established that $A\to A\o_\mathbf S\MUP$ is eff for an arbitrary fixed $\E$-ring $A$, it follows from eff-descent that we may express then even filtration for $A$ as
$$
\mathrm{fil}^{\ge *}_\mathrm{ev}(A)\simeq \varprojlim_{\mathbf\Delta}\mathrm{fil}^{\ge *}_\mathrm{ev}(A\o_\mathbf S \MUP^\bullet).
$$
The $\E$-rings $A\o_\mathbf S\MUP^\bullet$ are all complex periodic, thus any even $\E$-ring $B$ with a map $A\o_\mathbf S\MUP^\bullet\to B$ has to be be weakly 2-periodic as  well (see for instance Lemma \ref{Key property}). Since it is also even,  it must in fact be even periodic.
If we let $\CAlg^\mathrm{evp}_R\subseteq\CAlg_R$ denote the full subcategory of those $\E$-$R$-algebras which are even, the subcategory inclusion therefore gives the equivalence of $\i$-categories
 $$
\CAlg^\mathrm{evp}_{A\o_\mathbf S\MUP^\bullet} \simeq\CAlg^\mathrm{ev}_{A\o_\mathbf S\MUP^\bullet}.
$$
It follows that
the canonical map
\begin{equation}\label{tie between ev and evp}
\mathrm{fil}^{\ge *}_\mathrm{ev}(A) \to\varprojlim_{\mathbf\Delta}\Big( \varprojlim_{B\in \CAlg^\mathrm{evp}_{A\o_\mathbf S \MUP^\bullet } }\tau_{\ge 2*}(B)\Big)
\end{equation}
is an equivalence of filtered $\E$-rings.

On the other hand, we can imitated the basic setup of the even filtration from \cite{HRW}, but replacing everywhere the subcategory $\CAlg^\mathrm{ev}\subseteq\CAlg$ with $\CAlg^\mathrm{evp}\subseteq\CAlg$. That is to say, we provisionally define the even periodic filtration by
$$
\mathrm{fil}^{\ge *}_\mathrm{evp}(A):=\varprojlim_{B\in\CAlg^\mathrm{evp}}\tau_{\ge 2*}(B),
$$
(where for instance this limit being well-defined is due to accessibility of the subcategory $\CAlg^\mathrm{evp}\subseteq\CAlg$, see Lemma \ref{Lemma cover preservation and lifting} and Remark \ref{Remark Lemma cover preservation and lifting}).
The inclusion $\CAlg^\mathrm{evp}\subseteq\CAlg^\mathrm{ev}$ defines a canonical map of filtered $\E$-rings
\begin{equation}\label{filtration comparison ev and evp}
\mathrm{fil}^{\ge *}_\mathrm{ev}(A)\to \mathrm{fil}^{\ge *}_\mathrm{evp}(A).
\end{equation}
The arguments for \cite[Corollary 2.2.17]{HRW} go through unchanged in the even periodic setting, showing that $\mathrm{fil}^{\ge *}_\mathrm{evp}$ satisfies eff descent. Arguing as at the start of this proof, we see that the $\E$-ring map $A\to A\o_S\MUP$ being eff implies that the canonical map of filtered $\E$-rings
$$
\mathrm{fil}^{\ge *}_\mathrm{evp}(A) \to\varprojlim_{\mathbf\Delta}\Big( \varprojlim_{B\in \CAlg^\mathrm{evp}_{A\o_S \MUP^\bullet } }\tau_{\ge 2*}(B)\Big)
$$
is an equivalence. Since the canonical map \eqref{tie between ev and evp} clearly factors through this map, and both are equivalences, the filtration comparison map \eqref{filtration comparison ev and evp} is an equivalence as well.
\end{proof}

\begin{remark}
In light of the preceding result, we can reduce without loss to the even periodic setting. This seems like a needless restriction, but it has some advantages. For instance, a map of even periodic $\E$-rings $A\to B$ is eff if and only if it is \textit{faithfully flat} in the usual sense (e.g. as in \cite{SAG}) that $\pi_*(A)\otimes_{\pi_0(A)}\pi_0(B)\to\pi_*(B)$ is an isomorphism of graded rings and $\pi_0(A)\to\pi_0(B)$ is faithfully flat. One substantial advantage is that modules satisfy flat satisfy descent in the latter sense, but not always in the broader eff sense, see \cite[Remark 2.3.7]{HRW}. Also, working in the even periodic setting eases the comparison to the usual foundation of spectral algebraic geometry in the sense of \cite{SAG}.
\end{remark}

\subsection{The Adams-Novikov filtration}\label{ANsection}
One way to obtain the Adams-Novikov spectral sequence is as (the d\'ecalage of) the spectral sequence associated to a filtered spectrum.

\begin{definition}\label{The party line on ANf}
The \textit{Adams-Novikov filtration} on an $\E$-ring $A$ is defined to be the filtered $\E$-ring
$$
\mathrm{fil}_\mathrm{AN}^{\ge *}(A) \,:=\, \varprojlim_{\mathbf{\Delta}} \tau_{\ge 2*}(A\o_{\mathbf S}\mathrm{MU}^\bullet),
$$
where $\mathrm{MU}^\bullet \simeq \mathrm{MU}^{\otimes_{\mathbf S}[\bullet]}$ is the cosimplicial cobar construction for the unit map $\mathbf S\to\mathrm{MU}$.
\end{definition}

\begin{remark}
The term ``Adams-Novikov filtration" tends to be reserved for a slightly different filtration, of which our is the d\'ecalage. But since the former will play no role in our discussion, we have taken the liberty of omitting omnipresent ``the d\'ecalage of" qualifier, and adopted the above slightly non-standard terminology.
\end{remark}

The construction of the Adams-Novikov spectral sequence is evidently functorial, since it may be written as the composite
$$
\CAlg\xrightarrow{-\otimes_{\mathbf S}\mathrm{MU}^\bullet } \Fun(\mathbf{\Delta}, \CAlg)\xrightarrow{(\tau_{\ge 2*})_*} \Fun(\mathbf{\Delta}, \mathrm{FilCAlg})\xrightarrow{\varprojlim_{\mathbf\Delta}} \mathrm{FilCAlg}.
$$
We wish to express it in another functorial way, more reminiscent of the even filtration. For that, let us temporarily introduce a couple of variants on the theme of the even filtration, which will turn out to be alternative descriptions of the Adams-Novikov filtration.

\begin{definition}
The \textit{complex orientable filtration} and the \textit{complex periodic filtration} are the functors
$$
\mathrm{fil}^{\ge *}_{\mathbf C\mathrm {or}}, \mathrm{fil}^{\ge *}_{\mathbf C\mathrm{p}} : \CAlg\to\mathrm{FilCAlg}
$$
given by the right Kan extension of the double-speed Postnikov filtration $\tau_{\ge 2*}$ along the inclusions $\CAlg^{\mathbf C\mathrm{or}}\subseteq\CAlg$ and $\CAlg^{\mathbf C\mathrm{p}}\subseteq\CAlg$ respectively (well-defined because these are both accessible subcategories; see Lemma \ref{Lemma 4.7.2} for a proof for the latter case). That is to say, for an $\E$-ring $A$ they are given by
$$
\mathrm{fil}_{\mathbf C\mathrm{or}}^{\ge *}(A) \,\simeq \varprojlim_{B\in\CAlg_A^{\mathbf C\mathrm {or}}}\tau_{\ge 2*}(B),\quad \qquad \mathrm{fil}_{\mathbf C\mathrm{p}}^{\ge *}(A) \,\simeq \varprojlim_{B\in\CAlg_A^{\mathbf C\mathrm {p}}}\tau_{\ge 2*}(B).
$$
\end{definition}

\begin{remark}\label{Filtrations in the trivial case remark}
There is a canonical map of filtered $\E$-rings
$
\mathrm{fil}^{\ge *}_{\mathbf C\mathrm{or}}(A)\to \tau_{\ge *}(A)
$
for any $\E$-ring $A$. It $A$ is complex orientable, it is an equivalence. If $A$ is complex periodic, the canonical map
$
\mathrm{fil}^{\ge *}_{\mathbf C\mathrm{p}}(A)\to \tau_{\ge *}(A)
$
is an equivalence of filtered $\E$-rings as well.
\end{remark}

In imitation of the even filtration, we define an appropriate complex periodic version of the eff topology, and show that $\mathrm{fil}^{\ge *}_{\mathbf C\mathrm{or}}$ and $\mathrm{fil}^{\ge *}_{\mathbf C\mathrm p}$ satisfy descent with respect to it.

\begin{definition}
A map of $\E$-rings $A\to B$ is \textit{$\mathbf C$ff} if is the case that for any $C\in\CAlg^{\mathbf C\mathrm{or}}$, the ring map $\pi_*(C)\to \pi_*( C\o_A B)$ is faithfully flat.
\end{definition}

\begin{lemma}[$\mathbf C$ff descent]\label{Cff descent}
Let $A\to B$ be an $\mathbf C$ff morphism of $\E$-rings.
Then the canonical map of filtered spectra
$$
\mathrm{fil}^{\ge *}_\mathrm{\mathbf C\mathrm{or}}(A)\to \varprojlim_{\mathbf\Delta}\mathrm{fil}^{\ge *}_\mathrm{\mathbf C\mathrm{or}}(B^\bullet), \qquad 
\mathrm{fil}^{\ge *}_\mathrm{\mathbf C\mathrm{p}}(A)\to \varprojlim_{\mathbf\Delta}\mathrm{fil}^{\ge *}_\mathrm{\mathbf C\mathrm p}(B^\bullet)
$$
is an equivalence,
where $B^\bullet = B^{\otimes_A (\bullet+1)}$ is the cosimplicial cobar resolution.\end{lemma}

\begin{proof}
The proof works just like the analogous eff descent, recalled here as Theorem \ref{Eff descent citation result}, is proved for the even filtration in \cite[Corollary 2.2.17]{HRW}.
\end{proof}

\begin{cons}\label{Const filtration comparison for AN}
The $\E$-ring $A\o_{\mathbf S}\mathrm{MU}^{\o_{\mathbf S} n}$ is complex orientable for any $\E$-ring $A$ because the complex bordism spectrum $\mathrm{MU}$ is, giving rise to a natural comparison map of filtered $\E$-rings
$$
\mathrm{fil}_{\mathbf C\mathrm{or}}^{\ge *}(A)\to \mathrm{fil}_\mathrm{AN}^{\ge *}(A).
$$
Likewise does the inclusion $\CAlg^{\mathbf C\mathrm p}\subseteq\CAlg^{\mathbf C\mathrm{or}}$ induce a comparison between the complex orientable and complex periodic filtration
$$
\mathrm{fil}^{\ge *}_{\mathbf C\mathrm{or}}(A)\to \mathrm{fil}^{\ge *}_{\mathbf C\mathrm p}(A)
$$
\end{cons}

\begin{prop}
The natural comparison maps $\mathrm{fil}^{\ge *}_{\mathbf C\mathrm{or}}\to \mathrm{fil}^{\ge *}_{\mathrm{AN}}$ and  $\mathrm{fil}^{\ge *}_{\mathbf C\mathrm{or}}\to \mathrm{fil}^{\ge *}_{\mathbf C\mathrm p}$ of Construction \ref{Const filtration comparison for AN} are equivalences of functors $\CAlg\to \mathrm{FilCAlg}$.
\end{prop}

\begin{proof}
Consider the unit map of $\E$-rings $\mathbf S\to\mathrm{MU}$. 
The  map $\pi_*(B)\to \pi_*(B\o_{\mathbf S}\mathrm{MU})$, that it induces upon base-change to any complex orientable ring spectrum $B$ and passing to homotopy groups, is faithfully flat as a direct consequence of
Quillen's Theorem on complex bordism. The map $\mathbf S\to \mathrm{MU}$ is therefore $\mathbf C$ff. By the $\mathbf C$ff descent  result of Proposition \ref{Cff descent}, the first map in the following factorization of the canonical comparison between the complex oriented filtration and the Adams-Novikov filtration for an arbitrary $\E$-ring $A$
$$
\mathrm{fil}^{\ge *}_\mathrm{\mathbf C\mathrm{or}}(A)\to \varprojlim_{\mathbf\Delta}\mathrm{fil}^{\ge *}_\mathrm{\mathbf C\mathrm{or}}(A\o_{\mathbf S}\mathrm{MU}^\bullet)\to \varinjlim_{\mathbf\Delta} \tau_{\ge 2*}(A\o_{\mathbf S}\mathrm{MU}^\bullet)\,\simeq\, \mathrm{fil}^{\ge *}_\mathrm{AN}(A)
$$
is an equivalence of filtered $\E$-rings. The second arrow in the composite in question is also an equivalence by Remark \ref{Filtrations in the trivial case remark}, since the filtered $\E$-rings $A\o_{\mathbf S}\mathrm{MU}^\bullet$ is complex orientable on each filtered piece. This verifies the claim for the first comparison map.

For the second, we may insted use the $\E$-ring map $\mathbf S\to\mathrm{MUP}$, for any fixed choice of $\E$-ring structure on the periodic complex bordism spectrum $\mathrm{MUP}$, which is also a $\mathbf C$ff map, and argue much like that in the proof of
Proposition \ref{Even filtration is even periodic filtration}.
\end{proof}

\begin{corollary}\label{AN via MUP}
The Adams-Novikov filtration on an $\E$-ring $A$ may be expressed as
$$
\mathrm{fil}_\mathrm{AN}^{\ge *}(A)\,\simeq \, \varprojlim_{\mathbf\Delta} \tau_{\ge 2*}(A\o_{\mathbf S}\mathrm{MUP}^\bullet),
$$
in terms of the cosimplicial cobar construction $\mathrm{MUP}^\bullet\simeq \mathrm{MUP}^{\o_{\mathbf S}[\bullet]}$ of the unit $\E$-ring map $\mathbf S\to\mathrm{MUP}$ for any choice of $\E$-ring structure on $\mathrm{MUP}$.
\end{corollary}

\begin{remark}
By passing to the spectral sequences from the equivalences of filtered spectra of Corollary \ref{AN via MUP} in the case of the sphere spectrum $A=\mathbf S$, we find that the Adams-Novikov spectral sequence may be equivalently described either as the $\mathrm{MU}$-based Adams spectral sequence, or as the $\mathrm{MUP}$-based Adams spectral sequence. This observation is not new; it may also be obtained as a special case of Morava's Change of Rings Theorem, see e.g. \cite[Example 7.4]{Goerss ANSS}.
\end{remark}

The subcategory inclusion $\CAlg^\mathrm{ev}\subseteq\CAlg^{\mathbf C\mathrm{or}}$, of even $\E$-rings into complex orientable ones, gives rise to a canonical comparison natural comparison map
$$
\mathrm{fil}^{\ge *}_{\mathrm{AN}}(A)\,\simeq\,\mathrm{fil}^{\ge *}_{\mathbf C\mathrm{or}}(A)\to \mathrm{fil}^{\ge *}_\mathrm{ev}(A)
$$
between the Adams-Novikov filtration and the even filtration on an $\E$-ring $A$.

\subsection{$I$-adic even filtration}
Parallel to the integral version of the even filtration, \cite{HRW} also develops a $p$-complete version, relevant for the applications to prismatization and to $p$-adic geometry more broadly. In this section we consider a slight generalization, and show that it too is insensitive to passage to the even periodic setting.

Fix for the remainder of this subsection an $\E$-ring $R$ and a regular locally principal ideal $I\subseteq\pi_0(R)$, i.e.\ an effective Cartier divisor $\Spec(\pi_0(R)/I)\subseteq\Spec(\pi_0(R))$, and assume that $R$ is $I$-complete.
The example to keep in mind is $R=\mathbf S^\wedge_p$ and $I=(p)$. In that setting, the following notion reduces to that of bounded $p$-torsion from \cite[Definition 2.1.7]{HRW}:

\begin{definition}\label{Def of I-bounded torsion}
An $R$-module $M$ has \textit{bounded $I$-torsion} if the sequence of $I$-power torsion $\pi_0(R)$-submodules
$$
\pi_n(M)[I]\subseteq \pi_n(M)[I^2]\subseteq\pi_n(M)[I^3]\subseteq\cdots \subseteq\pi_n(M)
$$
eventually stabilizes for all $n\in \mathbf Z$. That is to say, for each $n\in \mathbf Z$ there must be some $N_n>\!>0$ such that $\pi_n(M)[I^{N_n}] = \pi_n(M)[I^\infty]$.
\end{definition}

\begin{definition}\label{Def of I-adic ev and evp rings}
Let
$$
\CAlg^\mathrm{evp}_I
\subseteq \CAlg^\mathrm{ev}_I
\subseteq \CAlg_R^{\mathrm{cplt}}
$$
denote the respective full subcategories spanned by all the $I$-complete $\E$-algebras  over $R$ which are
even periodic and have bounded $I$-torsion,
or even and have bounded $I$-torsion.
\end{definition}

\begin{lemma}\label{I-adic accessibility Lemma}
All the $\i$-categories and inclusions between them featuring in Definition \ref{Def of I-adic ev and evp rings} are respectively accessible and accessible functors.
\end{lemma}

\begin{proof}
This follows from  \cite[Proposition 2.1.10]{HRW}, stated there in the $p$-adic setting, but which carries over without change to the  $I$-adic context. Indeed, established there is the accessibility of $\CAlg_I^\mathrm{ev}\subseteq\CAlg_R^\mathrm{cplt}$. 
Accessibility of the subcategory $\CAlg_I^\mathrm{evp}\subseteq\CAlg_R^\mathrm{cplt}$ finally follows from that of $\CAlg_I^\mathrm{ev}\subseteq\CAlg_R^\mathrm{cplt}$ (or can alternatively be argued directly with an argument analogous to the one laid out in Remark \ref{Remark Lemma cover preservation and lifting}.)
\end{proof}

Thanks to the accessibility of the subcategories in question established by Lemma \ref{I-adic accessibility Lemma}, we may now imitate \cite[Construction 2.1.3]{HRW} to define two \textit{\`a priori} distinct  $I$-adic versions of the even filtration.

\begin{definition}\label{Def of I-complete evf}
The \textit{$I$-adic even filtration}
and the  \textit{$I$-adic even periodic filtration} are the functors
$$
\mathrm{fil}_{\mathrm{ev}, I}^{\ge *}, 
\, \mathrm{fil}_{\mathrm{evp}, I}^{\ge *}   : \CAlg_R^\mathrm{cplt}\to\mathrm{FilCAlg}
$$
given by the right Kan extensions of the double-speed Postnikov filtration $\tau_{\ge 2*}$ along the inclusions $\CAlg^\mathrm{ev}_I\subseteq\CAlg_R^\mathrm{cplt}$ 
and $\CAlg^\mathrm{evp}_I\subseteq\CAlg_R^\mathrm{cplt}$ respectively. That is to say, for an $I$-complete $\E$-algebra $A$ over $R$, they are given by
$$\mathrm{fil}^{\ge *}_{\mathrm{ev}, I}(A)\,\,\simeq 
\varprojlim_{B\in(\CAlg_I^\mathrm{ev})_{A/}}\tau_{\ge 2*}(B),
\qquad\quad
\mathrm{fil}^{\ge *}_{\mathrm{evp}, I}(A)\,\,\simeq 
\varprojlim_{B\in(\CAlg_I^\mathrm{evp})_{A/}}\tau_{\ge 2*}(B).
$$
\end{definition}

As in the integral case, the key property of the $I$-complete even filtration is that 
both variants of it satisfy a form of flat descent. To define it, we use the following  synthesis of \cite[Definitions 2.2.1, 4, 6, \& 15]{HRW}, transplanted from the $p$-adic to the $I$-adic setting.

\begin{definition}\label{Def p-eff}
A map of $I$-complete $\E$-algebras $A\to B$ over $R$ is  \textit{$I$-completely eff} if it is the case that for any $C\in\CAlg^\mathrm{ev}_I$, the completed base-change $B\,\widehat{\o}_A\,C$ is even, and the natural $\pi_0(R)$-algebra map
$$
\big(\bigoplus_{n\in \mathbf Z}\pi_n(B)\big)\o_{\pi_0(R)}\pi_0(R)/I\to \big(\bigoplus_{n\in \mathbf Z}\pi_n(B\,\widehat{\o}_A\, C)\big)\o_{\pi_0(R)}\pi_0(R)/I
$$
is faithfully flat. 
\end{definition}

\begin{remark}\label{Remark to L or not to L}
Keep in mind that our tensor products are always formed $\i$-categorically in spectra. For classical rings, that means these are the derived tensor products, so that the condition in the statement of Definition \ref{Def p-eff} might be more classically written with the base change  $-\o_{\pi_0(R)}\pi_0(R)/I$ denoted by
$-\o^{\mathrm L}_{\pi_0(R)}\pi_0(R)/I$.
\end{remark}

The desired descent result is proved by repeating in the $I$-adic context \textit{mutatis mutandis} the arguments from the proof of \cite[Corollary 2.2.17]{HRW} in the $p$-adic context.

\begin{lemma}[$I$-complete eff descent]\label{Eff descent citation result, p-adic}
Let $A\to B$ be an eff morphism of $I$-complete $\E$-algebras over $R$.
Then the canonical map of filtered spectra
$$
\mathrm{fil}^{\ge *}_{\mathrm{ev}, I}(A)\to \varprojlim_{\mathbf\Delta}\mathrm{fil}^{\ge *}_{\mathrm{ev}, I}(B^\bullet), \qquad \mathrm{fil}^{\ge *}_{\mathrm{evp}, I}(A)\to \varprojlim_{\mathbf\Delta}\mathrm{fil}^{\ge *}_{\mathrm{evp}, I}(B^\bullet)
$$
is an equivalence,
where $B^\bullet = B^{\widehat{\otimes}_A (\bullet+1)}$ is the $I$-completed cosimplicial cobar resolution.
\end{lemma}

The subcategory inclusion from Definition \ref{Def of I-complete evf} gives rise to  a canonical and natural comparison map
$$
\mathrm{fil}^{\ge *}_{\mathrm{ev}, I}(A)\to 
\mathrm{fil}^{\ge *}_{\mathrm{evp}, I}(A)
$$
for any $I$-complete $\E$-algebra $A$ over $R$. Knowing that the $I$-complete eff filtration satisfies descent with respect to $I$-completely eff maps, we can repeat the proof of Proposition \ref{Even filtration is even periodic filtration} to show that this comparison map is an equivalence, identifying the two versions of the $I$-adic even filtration from Definition \ref{Def of I-complete evf}.

\begin{prop}\label{Complete evp = ev filtr}
The natural comparison map $\mathrm{fil}^{\ge *}_{\mathrm{ev}, I}\to \mathrm{fil}^{\ge *}_{\mathrm{evp}, I}$ is an equivalence of functors $\CAlg_R^\mathrm{cplt}\to \mathrm{FilCAlg}$.
\end{prop}

\begin{proof}
Pick any $\E$-ring structure on the periodic complex bordism spectrum $\mathrm{MUP}$.
We claim that the canonical map of $I$-complete $\E$-algebras $R\to (R\otimes_{\mathbf S}\mathrm{MUP})^\wedge_I$ over $R$ is $I$-completely eff. To see that, it suffices to show that it is $I$-completely evenly free, in the sense that for any $A\in\CAlg_I^\mathrm{ev}$ the $I$-completed pushout
$$
A\,\widehat{\o}_R\,(R\otimes_{\mathbf S}\mathrm{MUP})^\wedge_I\simeq (A\o_{\mathbf S}\mathrm{MUP})^\wedge_I
$$
is equivalent as an $A$-module to an $I$-completion of a sum of even shifts of $A$. Since the map $A\to A\otimes_{\mathbf S}\mathrm{MU}$ is evenly free in this sense as a consequence of Quillen's Theorem (i.e.\ the computation of the $\mathrm{MU}$-cohomology of even ring spectra), and we have by the definition of periodic complex bordism that $A\otimes_{\mathbf S}\mathrm{MUP}\simeq \bigoplus_{n\in \mathbf Z}\Sigma^{2n}(A\otimes_{\mathbf S}\mathrm{MU})$, this is indeed clear.

From $R\to (R\otimes_{\mathbf S}\mathrm{MUP})^\wedge_I$ being $I$-completely eff, Lemma \ref{Eff descent citation result, p-adic} gives rise to a canonical equivalences of filtered $\E$-rings
$$
\mathrm{fil}_{\mathrm{ev}, I}^{\ge *}(A)\,\simeq \, \varprojlim_{\mathbf\Delta}\varprojlim_{B\in(\CAlg_I^\mathrm{ev})_{A/}}\tau_{\ge 2*}(B)\,\simeq \, \varprojlim_{\mathbf\Delta}\varprojlim_{B\in(\CAlg_I^\mathrm{evp})_{A/}}\tau_{\ge 2*}(B)\,\simeq\,\mathrm{fil}_{\mathrm{evp}, I}^{\ge *}(A)
$$
for any $I$-complete $\E$-algebra $A$ over $R$. By unpacking the definitions, we find that this equivalence is indeed precisely the natural comparison map $\mathrm{fil}^{\ge *}_{\mathrm{ev}, I}\to \mathrm{fil}^{\ge *}_{\mathrm{evp}, I}$ as desired.
\end{proof}

\newpage

\section{Spectral algebraic geometry via the functor of points}\label{Section great 2}

Let us review the basic language of nonoconnective spectral algebraic geometry. This is the version of $\i$-categorical algebraic geometry in which we take as our affines $\Spec(A)$ for nonconnective -- i.e.\ \textit{not necessarily} connective -- $\E$-rings $A$. To this end, we adopt a \textit{functor of points} approach, as opposed to a ringed space (or ringed $\i$-topos) approach.

\subsection{The setup of nonconnective spectral stacks}

Recall from \cite[Subsection B.6.1]{SAG} that the notion of faithfully flat maps induces a Grothendieck topology on the $\i$-category $\mathrm{Aff}:=\CAlg^\mathrm{op}$, which is the \textit{fpqc topology} of spectral algebraic geometry.

\begin{definition}\label{Def 3.1 of spectral stack}
A \textit{nonconnective spectral stack} is an accessible functor $\mX:\CAlg\to \mathrm{Ani}$ which satisfies fpqc descent. We let
$$
\mathrm{SpStk}\, :=\,\mathrm{Shv}^\mathrm{acc}_\mathrm{fpqc}(\mathrm{Aff})
$$ denote the $\i$-category of nonconnective spectral stacks.
\end{definition}

In general, the term ``spectral stack" with no further qualifiers should in this paper  always be understood to be nonconnective. Nonetheless, we will try  to limit  omissions of the adjective ``nonconnective" very often, so as to emphasize how crucial it is to our discussion. Indeed, we aim to discuss periodic $\E$-rings and geometry built upon them, so restricting the foundations of spectral algebraic geometry to connective $\E$-rings would inconvenient at best and infeasible at worst.

\begin{exun}
For any $\E$-ring $A$, the corepresentable functor $B\to \Map_{\CAlg}(A, B)$ is a spectral stack, which we call $\Spec(A)$.The resulting functor $\Spec :\CAlg^\mathrm{op}\to\mathrm{SpStk}$ is fully faithful and exhibits the (definitional) equivalence $\CAlg^\mathrm{op}\simeq \mathrm{Aff}$. Its objects are called \textit{affine nonconnective spectral schemes} or \textit{affines} for short.
\end{exun}

The other approach to algebraic geometry, than the functor of points approach we are pursuing here, is the perhaps more familiar ringed space approach. It works well to define schemes and (if ringed spaces are relaxed to ringed topoi) Deligne-Mumford stacks, but not stacks with more ``stackiness", i.e.\ bigger inertia groups. But the functor of points approach subsumes both:

\begin{exun}
Let $\mX$ be a nonconnective spectral Deligne-Mumford space, defined as in \cite[Definition 1.4.4.2]{SAG} to be a certain kind of spectrally-ringed $\i$-topos. We may identify it with its functor of points, i.e.\ the eponymous functor $\mX: \CAlg\to \mathrm{Ani}$ given by
$$
\mX(A)\,:=\, \Map_{\mathrm{SpDM}^\mathrm{nc}}(\mathrm{Sp\acute{e}t}(A), \mX).
$$
This gives by \cite[Proposition 1.6.4.2]{SAG} a fully faithful embedding $\mathrm{SpStk}^\mathrm{nc}\hookrightarrow \mathrm{SpStk}$, under which $\mathrm{Sp\acute{e}t}(A)\mapsto \Spec(A)$ for all $\E$-rings $A$. As indicated by the fact that we are not notationally distinguishing between $\mX$ and its functor of points, we will always identify nonconnective spectral Deligne-Mumford stacks with their corresponding nonconnective spectral stacks.
\end{exun}

Restricting in the expected way to the full subcategories $\CAlg\supseteq\CAlg^\mathrm{cn}\supseteq\CAlg^\heart$ of connective and classical (i.e.\ connected and $0$-truncated) $\E$-rings gives the respective fpqc topologies on $\mathrm{Aff}^\mathrm{cn}:=(\CAlg^\mathrm{cn})^\mathrm{op}$ and $\mathrm{Aff}^\heart := (\CAlg^\heart)^\mathrm{op}$.

\begin{variant}\label{Def of conn and ord}
A \textit{connective spectral stack} (resp.\,\textit{classical stack}) is an accessible functor $\mX:\CAlg^\mathrm{cn}\to \mathrm{Ani}$ (resp.\ $\mX:\CAlg^\heart\to \mathrm{Ani}$) which satisfies fpqc descent. We let $$\mathrm{SpStk}^\mathrm{cn} :=\mathrm{Shv}^\mathrm{acc}_\mathrm{fpqc}(\mathrm{Aff}^\mathrm{cn})\qquad \big(\text{resp.\,\,\,} \mathrm{Stk}^\heart :=\mathrm{Shv}^\mathrm{acc}_\mathrm{fpqc}(\mathrm{Aff}^\heart)\,\,\big)$$ denote the $\i$-category of connective spectral stacks (resp.\,classical stacks).
\end{variant}

\begin{definition}\label{geometric stack}
A nonconnective spectral stack $\mX$ be \textit{geometric} if it satisfies both of the following conditions:
\begin{enumerate}[label = (\roman*)]
\item There exists an fpqc cover $\Spec(A^0)\to \mX$ for some $\E$-ring $A^0$.
\item For any pair of $\E$-rings $A$, $B$, and maps $\Spec(A)\to \mX$ and $\Spec(B)\to \mX$, the pullback $\Spec(A)\times_{\mX}\Spec(B)$ in $\mathrm{SpStk}$ is an affine nonconnective spectral stack. (Equivalently: the diagonal $\Delta :\mX\to\mX\times\mX$ is an affine morphism.) 
\end{enumerate}
\end{definition}

This is equivalent to admitting a presentation
$$
\mX\,\simeq \varinjlim_{\,\,\,\,\mathbf\Delta^\mathrm{op}}\Spec(A^\bullet)
$$
for a cosimplicial $\E$-ring $A^\bullet$, all of whose face maps are faithfully flat. The relationship between the two descriptions is given by by
$$
A^\bullet \,\simeq\, \sO(\Spec(A^0)\times_{\mX}\cdots\times_{\mX}\Spec(A^0)),
$$
which is to say, by the simplical affine spectral scheme $\Spec(A^\bullet)$ is the \v{C}ech nerve of the cover $\Spec(A^0)\to \mX$.

\begin{remark}[Accessible presheaves -- what?]\label{Remark access why}
Recall that a functor $\mX:\CAlg\to\mathrm{Ani}$ is accessible if and only if it commutes with $\kappa$-filtered colimits for any sufficiently-large small regular cardinal $\kappa$. Since preservation of filtered (i.e.\ $\aleph_0$-filtered) colimits for a functor of points is equivalent to the underlying algebro-geometric object (scheme, stack, etc., spectral or otherwise) to be locally of finite presentation, see e.g. \cite[Definition 17.4.1.1]{SAG}, we may view our restriction to accesible functors as demanding that our stacks must always be \textit{locally of small presentation}. By \cite[Proposition A.2]{Dirac II}, a functor $\mX:\CAlg\to \mathrm{Ani}$ being accessible has some equivalent alternative characterizations, which we will sometimes make (explicit or implicit) use of in this paper:
\begin{enumerate}[label = (\roman*)]
\item It is left Kan extended from its restriction to the full subcategory $\CAlg^\kappa\subseteq\CAlg$ of $\kappa$-presentable $\E$-rings for some -- and hence for any -- sufficiently-large small cardinal $\kappa$.
\item \label{la characterisation deux d'accessibilite} It is a colimit in the $\i$-category $\Fun(\CAlg, \mathrm{Ani})$ of a small diagram of affines, i.e.\ it is of the form $X\simeq\varinjlim_{i\in \mathcal I}\Spec(A_i)$ for some small indexing $\i$-category $\mathcal I$.
\end{enumerate}
Characterization \ref{la characterisation deux d'accessibilite} might at first glance seem in conflict with the standard universal property of the presheaf $\i$-category $\mP(\mC)\simeq\Fun(\mC^\mathrm{op}, \mathrm{Ani})$ as the cocompletion of any small $\i$-category $\mC$. Of course, this is because the $\i$-category $\mathrm{Aff}\simeq\CAlg^\mathrm{op}$ is not small, and so the presheaf $\i$-category does not possess that universal property. In fact, such a universal property is in this case satisfied precisely by the subcategory $\mP^\mathrm{acc}(\mathrm{Aff})\subseteq\mP(\mathrm{Aff})$ of accessible presheaves, see \cite[Remark A.4]{Dirac II}.
\end{remark}

\begin{remark}[Accessible presheaves -- why?]\label{sheafification issues}
In short, our main reason for imposing the accessibility assumption in Definition \ref{Def 3.1 of spectral stack} is to deal with certain size-theoretical issues related to the fpqc topology. We are following the treatment in \cite{Dirac II}, where the reader may find a careful discussion in the closely-related supergeometric setting referred to in \textit{loc.~cit.} as Dirac geometry. But let us explain it in a little more detail here.

The fully faithful embedding $\mathrm{SpStk}\to \Fun(\CAlg, \mathrm{Ani})$, which sends each nonconnective spectral stack to its functor of points -- or more generally the inclusion of fpqc sheaves into presheaves $\mathrm{Shv}_\mathrm{fpqc}(\mathrm{Aff})\hookrightarrow\mP(\mathrm{Aff})$ -- does not admit a left adjoint. This is one incarnation of the often-stated motto that ``fpqc sheafification does not exist." The reason for this issue, and what sets apart the fpqc topology from the other standard Grorthendieck topologies on affines such as the Zariski, \'etale, Nisnevich, fppf, etc.\,is that the collection of all the fpqc covers of a fixed affine do not always admit a small cofinal subset. Such a passage to a cofinal small subset of covers is used to establish the existance of the sheafification with respect to these other topologies. On the other hand, it was shown by \cite{Waterhouse} that there exist presheaves $\CAlg^\heart\to \mathrm{Set}$ which genuinely \textit{do not admit} an fpqc sheaficiation.

One way to circumvent this issue is by the usual ``universe enlargement" trick; that is to say, to work with the big $\i$-category $\widehat{\mP}(\mathrm{Aff}):=\Fun(\CAlg, \widehat{\mathrm{Ani}})$ where $\widehat{\mathrm{Ani}}$ is the big $\i$-category of not-necessarily-small anima. With respect to it, the fpqc site $\mathrm{Aff}$ is small, and so the $\i$-category of big fpqc-sheaves $\widehat{\mathrm{Shv}}_\mathrm{fpqc}(\mathrm{Aff})\subseteq \widehat{\mP}(\mathrm{Aff})$ is a (big) $\i$-topos. It in particular admits a left adjoint $\widehat L:\widehat{\mP}(\mathrm{Aff})\to \widehat{\mathrm{Shv}}_\mathrm{fpqc}(\mathrm{Aff})$, the desired fpqc sheafification functor. 
By the universal property of accessible presheaves, discussed in Remark \ref{Remark access why}, there exists a unique small colimit preserving functor $\mP^\mathrm{acc}(\mathrm{Aff})\to\widehat{\mP}(\mathrm{Aff})$ whcih is the identify functor on representables, i.e.\ sends $\Spec(A)\mapsto\Spec(A)$ for any $\E$-ring $A$.
The composite
$$
L:\mP^\mathrm{acc}(\mathrm{Aff})\to\widehat{\mP}(\mathrm{Aff})\xrightarrow{\widehat L}\widehat{\mathrm{Shv}}_\mathrm{fpqc}(\mathrm{Aff}),
$$
thus also preserves small colimits, and its essential image consists of those fpqc sheaves $\CAlg\to\widehat{\mathrm{Ani}}$ which are generated from affines via small colimits. In particular, any such functor takes values in the subcategory of small anima $\mathrm{Ani}\subseteq\widehat{\mathrm{Ani}}$, and is, as a small colimit of representables, an accessible functor. The essential image of the functor $L$ may therefore be identified with the $\i$-category of accesible fpqc sheaves on $\mathrm{Aff}$, which is to say, with the $\i$-category of nonconnective spectral stacks $\mathrm{SpStk}$.  Indeed, this functor $L$ provides a left adjoint to the inclusion $\mathrm{SpStk}\subseteq\mP(\mathrm{Aff})$, which is the vaunted fpqc sheafification.
\end{remark}

\subsection{The connective localization}\label{Section connective localization}
In this section, we discuss in some detail how to best approximate a nonconnective spectral stack by a connective one. Our discussion is  close to that of \cite[Section 1.2]{ChromaticCartoon}.

\begin{cons}\label{Cons connective localization}
The connective cover functor $\tau:=\tau_{\ge 0}$ participates in the adjunction with the subcategory inclusion $\iota$ of connective $\E$-rings into all $\E$-rings
$$\iota:\CAlg^\mathrm{cn}\rightleftarrows \CAlg :\tau,
$$
which induces on anima-valued presheaves a double adjunction 
$$
\xymatrix{
 \mP(\mathrm{Aff}) \ar@<1.2ex>[rr]^{\tau_!} \ar@<-1.2ex>[rr]_{\iota^*}&  &\mP(\mathrm{Aff}^\mathrm{cn}).\ar[ll]|-{\tau^*\simeq\,\iota_!}
}
$$
Since $\iota$ and $\tau_{\ge 0}$ are both accessible functors, and both preserve fpqc covers, this induces a double adjunction between the $\i$-categories of spectral stacks
$$
\xymatrix{
 \mathrm{SpStk} \ar@<1.2ex>[rr]^{\tau_!} \ar@<-1.2ex>[rr]_{\iota^*}&  &\mathrm{SpStk}^\mathrm{cn}.\ar[ll]|-{\tau^*\simeq\,\iota_!}
}
$$
The sheaf-level functors $\iota^*$ and $\tau^*$ agrees with its eponymous presheaf analogues, and so send $\mX\mapsto\mX\vert_{\CAlg^\mathrm{cn}}$
and $\mX\mapsto \mX\circ \tau_{\ge 0}$ respectively. On the other hand, the sheaf-level left adjoint $\tau_!$ is obtained by first applying the presheaf-level $\tau_!$, i.e.\ left Kan extension, followed by fpqc sheafification.
\end{cons}

\begin{definition}
The \textit{connective localization} of $\mX\in\mathrm{SpStk}$ is given by
$
\mX^\mathrm{cn}:= \tau^*\tau_!\mX.
$
\end{definition}

\begin{remark}
In previous work, e.g. \cite{ChromaticCartoon}, \cite{ChromaticFiltration}, etc., we had referred to connective localization as the \textit{connective cover} and denoted it by $\tau_{\ge 0}(\mX)$. This is unfortunate for at least two reasons. Firstly, it might be misinterpreted as the functor $A\mapsto \tau_{\ge 0}(\mX(A))$. But even if this danger is avoided, it is still unfortunate because the canonical map goes in the direction $\mX\to\mX^\mathrm{cn}$, and as such $\mX^\mathrm{cn}$ in no way a \textit{cover} of $\mX$.
\end{remark}

It is immediate from the definition and Construction \ref{Cons connective localization} that connective localization is functorial, and that there is a canonical map $\mX\to \mX^\mathrm{cn}$ for all $\mX\in \mathrm{SpStk}$.

\begin{prop}\label{Connective stacks inside stacks}
The following statements hold:
\begin{enumerate}[label = (\alph*)]
\item The connective localization functor $\mX\mapsto \mX^\mathrm{cn}$ preserves all small colimits in $\mathrm{SpStk}$.\label{Prop 13, a}
\item The functor $\tau^* : \mathrm{SpStk}^\mathrm{cn}\to \mathrm{SpStk}$ is fully faithful, and its essential image is generated under small colimits by the affines $\Spec(A)$ for $A\in\CAlg^\mathrm{cn}$.\label{Prop 13, b}
\item A nonconnective spectral stack $\mX$ belongs to the essential image of $\mathrm{SpStk}^\mathrm{cn}$ inside $\mathrm{SpStk}$, as explained in
\ref{Prop 13, b}, if and only if $\mX\to\mX^\mathrm{cn}$ is an equivalence.\label{Prop 13, c}
\item The canonical map $\mX^\mathrm{cn}\to (\mX^\mathrm{cn})^\mathrm{cn}$ is an equivalence for any $\mX\in\mathrm{SpStk}$.\label{Prop 13, d}
\end{enumerate}
\end{prop}

\begin{proof}
For \ref{Prop 13, a}, observe that connective localization may be written by Construction \ref{Cons connective localization} as the composite of two left adjoint functors $\tau^*$ and $\tau_!$.

Claims \ref{Prop 13, c} and \ref{Prop 13, d} follow easily from \ref{Prop 13, b}. To prove the latter, we must show that the counit $\tau_!\tau^*\mX\to \mX$ is an equivalence for every $\mX\in\mathrm{SpStk}^\mathrm{cn}$. Noting that $\tau^*\simeq \iota_!$, this we have $\tau_!\tau^*\simeq \tau_!\iota_!$, which can be described as the left Kan extension along $\iota$, followed by the left Kan extension along $\tau$, and finally post-composed with fpqc sheafification. Equivalently therefore, this is the fpqc sheafification of the left Kan extension along the composite $\tau\circ \iota\simeq \mathrm{id}_{\CAlg^\mathrm{cn}}$, and is therefore the identify endofunctor.
\end{proof}

In light of Proposition \ref{Connective stacks inside stacks}, we can identify the $\i$-category $\mathrm{SpStk}^\mathrm{cn}$ with its essential image inside $\mathrm{SpStk}$, and therefore treat connective spectral stacks as a special class of nonconnective spectral stacks. From this perspective, we can express the universal property of the connective localizations as asking that for any $\mX\in \mathrm{SpStk}$ and $\mY\in \mathrm{SpStk}^\mathrm{cn}$,  the map
$$
\Map_{\mathrm{SpStk}}(\mX^\mathrm{cn}, \mY)\to \Map_{\mathrm{SpStk}}(\mX, \mY),
$$
induced by the canonical map $\mX\to\mX^\mathrm{cn}$, is a homotopy equivalence.

\begin{exun}\label{connective localization for geometric stacks}
Let $\mX$ be a geometric nonconnective spectral stack in the sense of Definition \ref{geometric stack}. By choosing an affine atlas $\Spec(A^0)\to \mX$, we can (non-canonically) present it in the form $\mX\simeq \varinjlim_{\mathbf\Delta^\mathrm{op}}\Spec(A^\bullet)$ for a faithfully flat cosimplicial $\E$-ring $A^\bullet$. The connective localization is therefore presented as
$$
\mX^\mathrm{cn}\simeq \varinjlim_{\,\,\,\,\mathbf{\Delta}^\mathrm{op}}\Spec(\tau_{\ge 0}(A^\bullet)),
$$
showing it to be a geometric (connective) spectral stack as well. Since the face maps $A^m\to A^n$ are faithully flat, we have $\tau_{\ge 0}(A^m)\o_{A^m}A^n\simeq \tau_{\ge 0}(A^n)$, showing the cosimplicial object $\tau_{\ge 0}(A^\bullet)$ is base-changed from $A^0\to \tau_{\ge 0}(A^0)$. Equivalently, $\mX^\mathrm{cn}$ is fully determined from $\Spec(A)\to\Spec(A)^\mathrm{cn}\simeq\Spec(\tau_{\ge 0}(A))$ -- the connective localization can be computed fpqc-locally.
\end{exun}

\subsection{Underlying classical stacks}\label{Section underlying classical stack}
As nonconnective spectral stacks may be approximated by connective spectral stacks, so can the latter similarly be approximated by classical stacks.

\begin{variant}\label{Variant connective cover}
A similar analysis to the one we have carried out above in Section \ref{Cons connective localization} for the inclusion $\CAlg^\mathrm{cn}\subseteq\CAlg$ can also be applied to the adjunction
$$
\pi:=\pi_0 :\CAlg^\mathrm{cn}\rightleftarrows \CAlg^\heart :\lambda.
$$
It commutes with filtered colimits (because $S^0$ is a compact anima) and preserves fpqc covers, thus giving rise to a double adjunction
$$
\xymatrix{
 \mathrm{Stk}^\heartsuit \ar@<1.2ex>[rr]^{\lambda_!\,\,} \ar@<-1.2ex>[rr]_{\pi^*\,\,}&  &\mathrm{SpStk}^\mathrm{cn}.\ar[ll]|-{\lambda^*\simeq\,\pi_!\,}
}
$$
Since the unit of the upper adjunction is given by $\lambda^*\lambda_!\simeq \pi_!\lambda_!\simeq (\pi\circ \lambda)_!\simeq \mathrm{id}$,
it is this time the leftmost adjoint $\lambda_!$ which is fully faithful,
and  preserves all small colimits. We use this fully faithful embedding $\mathrm{Stk}^\heart\hookrightarrow\mathrm{SpStk}^\mathrm{cn}$  to identify classical stacks with a special class of connective spectral stacks.
The \textit{underlying classical stack} functor $\mX\mapsto \mX^\heartsuit$ is defined for any connective spectral stack $\mX$ as the composite
$$
\mX^\heartsuit:=\pi_!\pi^*(\mX\vert)\simeq\pi_!(\mX\vert_{\CAlg^\heart}).
$$
The unit of the adjunction gives a canonical map of spectral stacks $\mX^\heart\to\mX$.
\end{variant}

\begin{remark}
Unlike the fully faithful embedding $\mathrm{SpStk}^\mathrm{cn}\hookrightarrow \mathrm{SpStk}$ of Proposition \ref{Connective stacks inside stacks}, which is given by the left-and-right-adjoint functor $\tau^*\simeq \iota_!$, the fully faithful embedding $\mathrm{Stk}^\heart\hookrightarrow \mathrm{SpStk}^\mathrm{cn}$ of Variant \ref{Variant connective cover} is given by the solely left adjoint functor $\lambda_!$. Indeed, the latter fully faithful embedding generally does not preserve limits -- not even products. Indeed, the affine scheme $\Spec(\mathbf Z)$ is the final object in $\mathrm{Stk}^\heart$, and is as such equal to the product with itself. But the final object in $\mathrm{SpStk}^\mathrm{cn}$ is given by $\Spec(\mathbf S)$, and we have
 $\mathrm{Spec}(\mathbf Z)\times\Spec(\mathbf Z)\simeq \Spec(\mathbf Z\otimes_\mathbf S \mathbf Z)$. The spectrum $\mathbf Z\otimes_\mathbf S \mathbf Z$, an integral version of the dual Steenrod algebra,  has non-zero homotopy groups in arbitrarily high degrees so that we in particular find that $\Spec(\mathbf Z\otimes_\mathbf S \mathbf Z)\notin \mathrm{Stk}^\heart$.
\end{remark}

The passage to the underlying classical stack is still possible for a nonconnective spectral stack, but it is a two-step process, requiring first the passage to the connective localization.

\begin{definition}
The \textit{underlying classical stack} of a nonconnective spectral stack $\mX$ is the classical stack $\mX^\heart:=(\mX^\mathrm{cn})^\heart$.
\end{definition}

Note that for a nonconnective spectral stack $\mX$, its underlying classical stack $\mX^\heart$ is no longer
given by the restriction $\mX^\heart:=(\mX^\mathrm{cn})^\heart$. There is also no longer a direct map from $\mX^\heart$ and $\mX,$ being instead replaced by a canonical cospan
$$
\mX^\heart\to \mX^\mathrm{cn} \leftarrow \mX.
$$
The underlying classical stack functor $(-)^\heart:\mathrm{SpStk}\to \mathrm{Stk}^\heart$ is explicitly given by the composite $\lambda_!\,\tau_!$, with $\lambda$ and $\tau$ from Construction \ref{Cons connective localization} and Variant \ref{Variant connective cover} respectively. As such, it may itself be described as the sheafified left Kan extension along the subcategory inclusion $\CAlg^\heart\subseteq\CAlg^\mathrm{cn}\subseteq\CAlg$; that is to say, as the colimit
$$
\mX^\heart\,\simeq\varinjlim_{\Spec(A)\in \mathrm{Aff}_{/\mX}}\Spec(\pi_0(A)),
$$
indexed over the $\i$-category $\mathrm{Aff}_{/\mX}\simeq \mathrm{Aff}\times_{\mathrm{SpStk}}\mathrm{SpStk}_{/\mX}$  of affines (always always understood to mean nonconnective spectral affines, unless explicitly stated otherwise,) over $\mX$.

\begin{exun}
Let $\mX$ be a geometric nonconnective spectral stack, presented like in Example \ref{connective localization for geometric stacks} in terms of an affine atlas as
$\mX\simeq \varinjlim_{\mathbf\Delta^\mathrm{op}}\Spec(A^\bullet)$.
Its underlying classical stack can then be presented as
$$
\mX^\heart\simeq \varinjlim_{\,\,\,\,\mathbf{\Delta}^\mathrm{op}}\Spec(\pi_0(A^\bullet)).
$$
\end{exun}

\begin{remark}\label{Remark issues of the heart}
The passage to the underlying spectral scheme $\mX\mapsto \mX^\heart$ for \textit{connective} spectral stacks, i.e.\ as a functor $\mathrm{SpStk}^\mathrm{cn}\to \mathrm{Stk}^\heart$, is given in the language of Variant \ref{Variant connective cover} as the left adjoint $\lambda^*$. In particular, it commutes with all small limits. This is false for its nonconnective version $(-)^\heart:\mathrm{SpStk}\to \mathrm{Stk}^\heart$, even for finite limits. On the level of affines, this amounts to the functor $\pi_0:\CAlg\to\CAlg^\heart$ not preserving pushouts. For an explicit example, note that
$\mathrm{KU}\otimes_\mathrm{ku}\Z\simeq 0$ since the image of the Bott element $\beta$ in it would need to be both zero and invertible, whereas of course
$$\pi_0(\mathrm{KU})\otimes_{\pi_0(\mathrm{ku})}\pi_0(\mathbf Z)\simeq \Z\otimes_{\Z}\Z\simeq \Z.$$
Only with a flatness hypothesis on at least one of the factors of such a relative tensor product would we have a guarantee that $\pi_0$ preserved that tensor product.
\end{remark}

\begin{remark}\label{Remark compression}
The fully faithful leftmost adjoint $\lambda_!$ in the double adjunction of Variant \ref{Variant connective cover} is the standard way in which we view classical stacks as special cases of connective spectral stacks. But in fact, the right-most adjoint $\pi^*$ is also fully faithful, since the counit of the lower adjunction is given by $\pi_!\pi^*\simeq \lambda^*\pi^*\simeq (\pi\circ\lambda)^*\simeq \mathrm{id}$. It thus offers another fully faithful embedding $\mathrm{Stk}^\heart\to \mathrm{SpStk}$, given explicitly by sending a classical stack $\mX$ to the composite $\CAlg\xrightarrow{\pi_0}\CAlg^\heart\xrightarrow{\mX}\mathrm{Ani}$. This embedding, which we may suggestively denote $\mX\mapsto \mX\circ \pi_0$, may appear to be more exotic. But it does in fact show up in constructions such as completions and open complements. See \cite[Section 1.1]{ChromaticFiltration}, where the endofunctor of spectral stacks $\mX\mapsto \mX^\heart\circ \pi_0$ is is discussed under the name of \textit{compression}. It arises via the adjunction induced on accessible fpqc sheaves from the functor $\pi_0:\CAlg\to \CAlg^\heart$, and so in particular, the unit of said adjunction gives rise to a canonical \textit{decompression map} $\mX\to \mX^\heart\circ \pi_0$ which is sometimes useful (albeit often implicitly, e.g.\ in discussing open immersions and formal completions).
\end{remark}

\subsection{Quasi-coherent sheaves on a nonconnective spectral stack}\label{Section 3.4}
We may introduce quasi-coherent sheaves on nonconnective spectral sheaves via descent from affines. For details of such an approach\footnote{The discussion in \textit{loc.cit.}\,works in a slighly different setting to ours. Namely, its starting point are functors $\CAlg^\mathrm{cn}\to\widehat{\mathrm{Ani}}$ from \textit{connective} $\E$-rings into \textit{big} anima. In contrast, we are starting with functors $\CAlg\to\mathrm{Ani}$ from \textit{all} $\E$-rings into \textit{small} anima, but we additionally require accessibility. As indicated in {\cite[Remark 6.2.2.2.]{SAG}}, dropping the connectivity assumption causes no essential changes and all the formal arguments remain valid unchanged. The rest amounts to a difference in the approach of dealing with set-theoretical considerations. Our trade-off for some generality is that presentability is automatic -- see {\cite[Proposition 6.2.3.4]{SAG}}. The difference between the two approaches is essentially the same as the one between the pyknotic and vs the ``heavy condensed set" approach to condensed mathematics.
}, see \cite[Section 6.2]{SAG}.

\begin{prop}\label{Def of QCoh}
There exists an essentially unique functor
$$
\QCoh :\mathrm{SpStk}^\mathrm{op}\to \mathrm{CAlg}(\mathrm{Pr^L})
$$
which satisfies the follwoing requirements:
\begin{enumerate}[label = (\roman*)]
\item It commutes with all small limits.
\item The composite
$$
\CAlg\overset{\Spec}\simeq \mathrm{Aff}^\mathrm{op}\hookrightarrow\mathrm{SpStk}^\mathrm{op}\xrightarrow{\QCoh}\CAlg(\mathrm{Pr^L})
$$
recovers the functor $A\mapsto\Mod_A$.
\end{enumerate}
\noindent We call $\QCoh(\mX)$ the \textit{$\i$-category of quasi-coherent sheaves} on the nonconnective spectral stack $\mX$.
\end{prop}

\begin{proof}
The functor $\QCoh$ is the right Kan extension of the functor $A\mapsto\Mod_A$ along the fully faithful embedding $\Spec:\CAlg\to \mathrm{SpStk}^\mathrm{op}$. Among other things,  the $\i$-categories in question failing to be small, makes the
 \textit{\`a priori} existence of such Kan extension unclear. However, the the results of \cite[Section 5.2.1]{SAG} establish their existence in contexts such as ours, and ultimately boils down  to the unstraightening equivalence of \cite{HTT}.
In the case at hand, we apply \cite[Remark 6.2.1.11]{SAG} to the 
coCartesian fibration $\Mod\to\CAlg$. This gives at first rise to a Kan extension
$$
\QCoh :\mathcal P(\mathrm{Aff})^\mathrm{op}\to \widehat{\mathrm{Cat}}_\infty
$$
with codomain consisting of big $\i$-categories. But since its values on representables, i.e.\ affines, restrict to the presentable $\i$-categories $\Mod_A$, and all the functors induced are clearly left adjoint, it follows that the restriction of $\QCoh$ to the full subcategory spanned under small colimits by representables, i.e.\ to $\mP^\mathrm{acc}(\mathrm{Aff})\subseteq\mP(\mathrm{Aff})$, defines a functor
$$
\QCoh :\mP^\mathrm{acc}(\mathrm{Aff})^\mathrm{op}\to\mathrm{Pr^L}.
$$
Further observations concerning the compatibility of the symmetric monoidal structure on $\Mod_A$ with pullback then shows that this functor factors through the forgetful functor $\CAlg(\mathrm{Pr^L})\to \mathrm{Pr^L}$. Finally, the fact that the functor $A\mapsto\Mod_A$ satisfies fpqc descent \cite[Corollary D.6.3.3]{SAG} implies that the functor constructed so far
$$
\QCoh :\mP^\mathrm{acc}(\mathrm{Aff})^\mathrm{op}\to\CAlg(\mathrm{Pr^L}),
$$
which commutes with all small limits by design, factors in a limit-preserving fashion through the localization functor $L:\mP^\mathrm{acc}(\mathrm{Aff})\to \mathrm{SpStk}$ given by fpqc sheafification.
This is the desired functor $\QCoh$.
\end{proof}

The $\i$-category of quasi-coherent sheaves on any $\mX\in \mathrm{SpStk}$ may be explicitly written as the limit
\begin{equation}\label{QCoh as limit}
\QCoh(\mX)\,\,\simeq \varprojlim_{\Spec(A)\in \mathrm{Aff}_{/\mX}}\Mod_A
\end{equation}
where $\mathrm{Aff}_{/\mX}\simeq \mathrm{Aff}\times_{\mathrm{SpStk}}\mathrm{SpStk}_{/\mX}$ is the $\i$-category of affines over $\mX$.

\begin{remark}
The notation $\QCoh(\mX)$ is consistent with that of \cite{SAG}. Note however that, when applied to a classical stack $\mX$, it \textit{does not} recover the eponymous $1$-category of usual quasi-coherent sheaves. Althout that is a discrepancy which classical language, we believe it very seldom leads to confusion. On the contrary, it serves to emphasize that $\QCoh(\mX)$ is more than merely some kind of ``derived" $\i$-category of quasi-coherent sheaves on $\mX$, and is rather the \textit{inherent} notion of quasi-coherent sheaves inside the world of spectral algebraic geometry. It is for this reason that we resist the growing-popularity (in the author's humble opinion a step backwards, at least in the context of derived and spectral algebraic geometry) of denoting $\QCoh(\mX)$ by $\mathcal D(X)$.
\end{remark}

Any morphism $f:\mX\to \mY$ in $\mathrm{SpStk}$ induces an adjunction on quasi-coherent sheaves
$$
f^*:\QCoh(\mX)\rightleftarrows \QCoh(\mY) :f_*,
$$
whose constituent functors $f^*$ and $f_*$ are called the \textit{pullback along $f$} and \textit{pushforward along $f$} respectively.

\begin{cons}
Let $p:\mX\to\Spec(\mathbf S)$ be the final morphism in $\mathrm{SpStk}$. Under the canonical identification $\QCoh(\Spec(\mathbf S))\simeq \Sp$, the pullback $\sO_{\mX}:=p^*(\mathbf S)$ gives the \textit{structure sheaf of $\mX$}. For any quasi-coherent sheaf $\sF\in \QCoh(\mX)$, the pushforward $$\Gamma(\mX; \sF):=p_*(\sF)$$ gives the \textit{global sections of $\sF$}. From the adjunction between the functors $p^*$ and $p_*$, we see that global sections may also be expressed as
$$
\Gamma(\mX; \sF)\,\simeq\, \underline{\Map}_{\Sp}(\mathbf S, p_*(\sF))\,\simeq \,\underline{\Map}_{\QCoh(\mX)}(\sO_{\mX}, \sF),
$$
where $\underline{\Map}$ denotes the mapping spectrum with respect to the usual enrichment of stable $\i$-categories in $\Sp$.
\end{cons}

\begin{exun}\label{Global functions on a stack}
We reserve the special notation $\sO(\mX):=\Gamma(\mX; \sO_{\mX})$ for the  global sections of the structure sheaf $\sO_{\mX}$. By unpacking the definitions, we find that this $\E$-ring of functions on $\mX$ is given by
$$
\sO(\mX)\,\,\simeq\varprojlim_{\Spec(A)\in \mathrm{Aff}_{/\mX}}A.
$$
Comparing this with the analogous expression \eqref{QCoh as limit} for $\QCoh$ suggests that the functor $\mX\mapsto \sO(\mX)$ might also admit an axiomatic definition in analogy with Proposition \ref{Def of QCoh}, as a ``decategorification" of $\mX\mapsto\QCoh(\mX)$. Indeed, it is the unique functor $\mathrm{SpStk}^\mathrm{op}\to \CAlg$ which commutes with all small limits and has a natural equivalence $\sO(\Spec(A))\simeq A$ for every $\E$-ring $A$.
\end{exun}

\begin{cons}\label{Cons 3.22 Postnikov filtration on QCoh}
The functors $\tau_{\ge n}:\Mod_A\to\Mod_{\tau_{\ge 0}(A)}$ for $A\in\CAlg$ induce by descent a functor
$$
\tau_{\ge n}:\QCoh(\mX)\to\QCoh(\mX^\mathrm{cn})
$$
on quasi-coherent sheaves over any $\mX\in \mathrm{SpStk}$. This  extends in the obvious way to the Postnikov filtration  functor
$$
\tau_{\ge *} : \QCoh(\mX)\to\mathrm{Fil}(\QCoh(\mX^\mathrm{cn})).
$$
When applied to the connective localization $\mX^\mathrm{cn}$, the canonical equivalence $\mX^\mathrm{cn}\simeq(\mX^\mathrm{cn})^\mathrm{cn}$ from Proposition \ref{Connective stacks inside stacks}, \ref{Prop 13, d}
allows us to view  $\tau_{\ge n}$ as endofunctors of $\QCoh(\mX^\mathrm{cn})$. They equip the stable $\i$-category $\QCoh(\mX^\mathrm{cn})$ with a canonical $t$-structure. We denote by
$$
\QCoh(\mX^\mathrm{cn})^\mathrm{cn}:=\QCoh(\mX^\mathrm{cn})_{\ge 0}
$$
 the positive part of this $t$-structure, and call it the $\i$-category of \textit{connective quasi-coherent sheaves on $\mX^\mathrm{cn}$}.
\end{cons}

\begin{remark}
For a connective spectral stack $\mX\in \mathrm{SpStk}^\mathrm{cn}$ and a quasi-coherent sheaf $\sF\in\QCoh(\mX)$, we have $\sF\in\QCoh(\mX)^\mathrm{cn}$ if any only if $f^*(\sF)\in \Mod_A^\mathrm{cn}$ for all maps from affines $f:\Spec(A)\to \mX$ for connective $\E$-rings $A$. But the analogous statemenet fails for the negative part of the $t$-structure $\QCoh(\mX^\mathrm{cn})_{\le 0}$, see \cite[Warning 6.2.5.9]{SAG}.
\end{remark}

\begin{remark}\label{Postnikov filtration on functions}
Consider the functor
$$
\mathrm{SpStk}\ni\mX\mapsto \Gamma(\mX^\mathrm{cn}; \tau_{\ge *}(\sO_{\mX}))\in \mathrm{FilCAlg},
$$
sending a nonconnective spectral stack $\mX$ to the global sections of the Posntikov tower of its structure sheaf. It follows Example \ref{Global functions on a stack}, and the fact that limits in $\mathrm{FilCAlg}$ are computed object-wise, that this functor may be expressed as
\begin{equation}\label{Postnikov filtration on functions formula}
\Gamma(\mX^\mathrm{cn}; \tau_{\ge *}(\sO_{\mX}))\simeq \varprojlim_{\Spec(A)\in \mathrm{Aff}_{/\mX}}\tau_{\ge *}(A).
\end{equation}
From this, we can deduce using an analogous argument to \cite[Remark 2.3.1]{HRW},  that the underlying object of this filtration is the $\E$-ring of functions $\sO(\mX)$.
\end{remark}

The underlying classical stack $\mX^\heart\simeq \mX^\mathrm{cn}\vert_{\CAlg^\heart}$ of $\mX$ can be expressed in terms of this $t$-structure as the relative spectrum
$$
\mX^\heart \simeq \Spec_{\mX^\mathrm{cn}}(\pi_0(\sO_{\mX})).
$$
It follows that the pushforward along the canonical map $\mX^\heart\to\mX^\mathrm{cn}$ induces an equivalence of $\i$-categories on quasi-coherent sheaves
$$
\QCoh(\mX^\heart)\simeq \Mod_{\pi_0(\sO_{\mX})}(\QCoh(\mX^\mathrm{cn})).
$$
In particular, the homotopy groups of quasi-coherent sheaves $\sF   \in\QCoh(\mX)$ can naturally be interpreted as belonging to the abelian $1$-category $\pi_n(\sF)\in \QCoh(\mX^\heart)^\heart$.

\subsection{Affine morphisms of nonconnective spectral stacks}\label{Section affine morphisms}

Since we will eventually require some of their basic properties, especially in Section \ref{Section TMF}, let us briefly sketch the theory of affine morphisms in $\mathrm{SpStk}$. 

\begin{definition}
A map of nonconnective spectral stacks $f:\mX\to \mY$ is an \textit{affine morphism} if for all maps from affines $\Spec(A)\to \mY$, the pullback $\Spec(A)\times_{\mY}\mX$ is also an affine.
\end{definition}

All the usual aspects of the theory of affine morphisms work in this setting as well. For instance, the class of affine morphisms is closed under pullback and composition, and affineness of a morphism can be checked fpqc locally. The following standard consequence, proved just like its classical variant \cite[\href{https://stacks.math.columbia.edu/tag/08GB}{Lemma 08GB}]{stacks-project},  will be used in the proof of Theorem \ref{Theorem Evp for TMF}.

\begin{lemma}\label{Lemma affine 3-1}
Let $f:\mX\to \mY$ be a morphism in $\mathrm{SpStk}_{/\mathcal S}$ such that $\mX\to \mS$ and $\Delta :\mY\to \mY\times_{\mS}\mY$ are both affine morphisms.
 Then $f$ is an affine morphism.
\end{lemma}

\begin{proof}
This follows from the stability of affine morphisms under pullback and composition by factoring $f :\mX\to \mY$ as the composite of the two top horizontal arrows in the following pair of pullback squares
$$
\begin{tikzcd}
\mX\arrow{r}{(\mathrm{id}, f)} \arrow{d}{f}  & \mX\times_{\mS}\mY\arrow{d}{(f, \mathrm{id})}
 & 
\mX\times_\mS\mY\arrow{r}{\mathrm{pr}_2} \arrow{d}{\mathrm{pr}_1} & \mY\arrow{d}{}
\\
\mY\arrow{r}{\Delta} & \mY\times_{\mS}\mY
& 
\mX\arrow{r}{}& \mS
\end{tikzcd}
$$
in which  the bottom horizontal arrows are both affine morphisms by assumption.
\end{proof}

We will also need to know how affine morphisms interact with quasi-coherent sheaves. This is quite simple in light of Proposition \ref{Def of QCoh}, but it does first require a base-change formula. The following variant will suffice for our purposes:

\begin{prop}\label{Prop push-pull}
Let $f:\mX\to \mY$ be a map of nonconnective spectral stacks. Assume that either one of the following two conditions holds:
\begin{enumerate}
\item The map $f$ is an affine morphism.
\item Both $\mX$ and $\mY$ are geometric spectral stacks, and the structure sheaf $\sO_{\mX}$ is a compact object in $\QCoh(\mX)$.
\end{enumerate}
Then for every pullback square
$$
\begin{tikzcd}
\mX'\arrow{r}{g'} \arrow{d}{f'} & \mX\arrow{d}{f}\\
\mY'\arrow{r}{g} & \mY
\end{tikzcd}
$$
in $\mathrm{SpStk}$, the associated diagram of $\i$-categories
$$
\begin{tikzcd}
\QCoh(\mY)\arrow{r}{f^*} \arrow{d}{g^*} & \QCoh(\mX)\arrow{d}{g'^*}\\
\QCoh(\mY')\arrow{r}{f'^*} & \QCoh(\mX')
\end{tikzcd}
$$
is right adjointable. That is to say, the Beck-Chevalley transformation $g^*f_*\to f'_*g'^*$ is an equivalence of functors $\QCoh(\mX)\to\QCoh(\mY')$.
\end{prop}

\begin{proof}
The proofs  of the corresponding result for connective spectral (pre)stacks, namely \cite[Proposition 6.3.4.1]{SAG} and \cite[Proposition 9.1.5.7, Proposition 9.1.5.8]{SAG}, do not make any use of the ambient connectivity assumption on the $\E$-rings. They therefore carry over unchanged to the present nonconnective setting. The proof strategy is ultimately the same in both cases either way: it is to reduce, via writing spectral stacks as colimits of affines, to the abstract adjointability result \cite[Corollary 4.7.4.18]{HA}.
\end{proof}

\begin{prop}
The pushforward along
 an affine morphism $f:\mY\to \mX$ induces an equivalence of symmetric monoidal $\i$-categories
$$
\QCoh(\mY)\simeq \Mod_{f_*(\sO_{\mY})}(\QCoh(\mX)).
$$
\end{prop}

\begin{proof} This is also a direct paraphrasing of \cite[Proposition 6.3.4.6]{SAG} in the nonconnective context, but let us this time spell it out.
Write $\mX\simeq \varinjlim_i \Spec(A_i)$ for some small diagram of $\E$-rings $A_i$. Then we have $\mY\simeq \varinjlim_i \Spec(A_i)\times_{\mX}\mY$ and since  $f$ is affine, we have
$$
\Spec(A_i)\times_{\mS}\mY\simeq \Spec(B_i).
$$
Applying the push-pull formula of Proposition \ref{Prop push-pull} to the pullback square
$$
\begin{tikzcd}
\Spec(B_i)\arrow{r} \arrow{d} & \mY\arrow{d}\\
\Spec(A_i) \arrow{r} & \mX
\end{tikzcd}
$$
shows that the restriction of $f_*(\sO_{\mY})\in \QCoh(\mX)$ along the map $\Spec(A_i)\to \mX$ is given by $B_i\in \Mod_{A_i}$.
Due to the functor of quasi-coherent sheaves being defined by taking colimits of nonconnective spectral stacks to limits of $\i$-categories, and the global sections giving an equivalence $\QCoh(\Spec(A))\simeq \Mod_A$, the functor
$$
\QCoh(\mY)\to \Mod_{f_*(\sO_{\mY})}(\QCoh(\mX))
$$
is obtained by passing to the limit over $i$ from the functors
$$
\Mod_{B_i}\to \Mod_{B_i}(\Mod_{A_i}).
$$
Since the latter are of course equivalences of $\i$-categories, so is the former.
\end{proof}

As expected, this has the corollary that the functor $\mathrm{SpStk}_{/\mX}\to \CAlg(\QCoh(\mX))^\mathrm{op}$,
given by $(f:\mY\mapsto \mX)\mapsto f_*(\sO_{\mY})$, restrict to an antiequivalence 
\begin{equation}\label{Relative spec def}
\mathrm{SpStk}_{/\mX}^\mathrm{aff}\,\simeq \, \CAlg(\QCoh(\mX))^\mathrm{op}
\end{equation}
between the $\i$-categories between affine morphisms over $\mX$ -- sometimes called the \textit{relative affines over $\mX$}, though that has the potential to be confused with objects of $\mathrm{Aff}_{/\mX}$ -- and quasi-coherent sheaves of $\E$-algebras over $\mX$, for any $\mX\in \mathrm{SpStk}$. The inverse functor is denoted $\mathcal A\mapsto \Spec_{\mX}(\mathcal A)$ and is called the \textit{relative spectrum over $\mX$}.

\newpage

\section{Even periodization}\label{Chapter EVP}

In this section, we introduce the even periodization, a geometric counterpart to the even filtration from Section \ref{Section 2}, and study some of its properties.

    \subsection{Even periodic spectral stacks and even periodization}\label{Section EVP}
We begin by discussing the meaning of even periodicity in the setting of nonconnective spectral stacks.

\begin{lemma}\label{Lemma cover preservation and lifting}
The inclusion $\CAlg^\mathrm{evp}\subseteq\CAlg$ is accessible, and both preserves fpqc covers, as well as satisfying the cover lifting property for fpqc covers.
\end{lemma}

\begin{proof}
The accessibility claim reduces to accessibility of the subcategory $\CAlg^\mathrm{ev}\subseteq\CAlg$, which is guaranteed by \cite[Proposition 2.1.2]{HRW}.
Since the fpqc topology on $\CAlg^\mathrm{evp}$ is defined by restriction along this subcategory inclusion from the fpqc topology on $\CAlg$, the cover preservation claim is obvious. For the cover lifting property, we need only to note that if $A$ is an even periodic $\E$-ring and $A\to B$ is an faithfully flat map of $\E$-rings, then the $\E$-ring $B$ is also necessarily even.
\end{proof}

\begin{remark}\label{Remark Lemma cover preservation and lifting}
Let us unpack the standard accessibility argument, alluded to above in the proof of Lemma \ref{Lemma cover preservation and lifting}. It is a direct imitation of the proof of \cite[Proposition 2.1.2]{HRW}, and we do not unpack it because there would be any surprises, but rather because we will alude to similar arguments several times in this text, and we believe carefully unpacking at least one might be useful to the reader.
So let
$\mathrm{GrCAlg}^\heart$ denote the 1-category of graded commutative rings. Consider the full subcategory
$$
(\mathrm{GrCAlg}^\heart)^\mathrm{evp}\subseteq\mathrm{GrCAlg}^\heart,
$$
spanned by all those graded commutative rings $A_*$ which are even periodic, in the sense that $A_{2n+1}=0$ and the multiplication map $A_n\otimes_{A_0}A_2\to A_{n+2}$ is an isomorphism, both for all $n\in \mathbf Z$. Both $\mathrm{GrCAlg}^\heart$ and $(\mathrm{GrCAlg}^\heart)^\mathrm{evp}$ are clearly accessible: the former is freely generated under filtered colimits by the polynomial rings $\mathbf Z[t_1, \ldots, t_n]$ with the generators $t_i$ places in arbitrary grading degree, and the latter is likewise generated by the graded rings $\mathbf Z[t_1^{\pm 1}, \ldots, t_n^{\pm 1}]$, where the variables $t_i$ are this time  allowed to lie in arbitrary even graded degree. According to Definition \ref{List of rings}, the inclusion of even periodic $\E$-rings into all $\E$-rings fits into the pullback square of $\i$-categories
$$
\begin{tikzcd}
\CAlg^\mathrm{evp} \arrow{d}{\pi_*} \arrow{r}{} & \CAlg \arrow{d}{\pi_*} \\
(\mathrm{GrCAlg}^\heart)^\mathrm{evp} \arrow{r}{} & \mathrm{GrCAlg}^\heart.
\end{tikzcd}
$$
We already know that all but the upper left vertex are accessible $\i$-categories, and the lower horizontal inclusion of categories, as well as the functor $\pi_*$, obviously commute with filtered colimits and are as such accessible functors. The desired accessibility of upper horizontal arrow $\CAlg^\mathrm{evp}\to\CAlg$ therefore follows from \cite[Proposition 5.4.6.6]{HTT}.
\end{remark}

It follows that the fpqc topology for $\E$-rings restricts to equip $\mathrm{Aff}^\mathrm{evp} :=(\mathrm{CAlg}^\mathrm{evp})^\mathrm{op}$ with a Grothendieck topology.

\begin{definition}\label{Def of even stacks}
An accessible functor $\mX:\CAlg^\mathrm{evp}\to \mathrm{Ani}$ which satisfies fpqc descent is called an \textit{even periodic spectral stack}. We let $$\mathrm{SpStk}^\mathrm{evp} :=\mathrm{Shv}^\mathrm{acc}_\mathrm{fpqc}(\mathrm{Aff}^\mathrm{evp})$$ denote the $\i$-category of even periodic spectral stacks.
\end{definition}

\begin{cons}\label{Functoriality of ev}
Let $\varepsilon:\CAlg^\mathrm{evp}\to\CAlg$ denote the subcategory inclusion. The double adjunction on the level of anima-valued presheaf $\i$-categories 
$$
\xymatrix{
 \mP(\mathrm{Aff}^\mathrm{evp}) \ar@<1.2ex>[r]^{\,\,\varepsilon_!} \ar@<-1.2ex>[r]_{\,\,\varepsilon_*}&  \mP(\mathrm{Aff})\ar[l]|-{\varepsilon^*}
}
$$
induces by Lemma \ref{Lemma cover preservation and lifting} a double-adjunction on the $\i$-categories of accessible fpqc sheaves
$$
\xymatrix{
 \mathrm{SpStk}^\mathrm{evp} \ar@<1.2ex>[r]^{\,\,\varepsilon_!} \ar@<-1.2ex>[r]_{\,\,\varepsilon_*}&  
 \mathrm{SpStk}.\ar[l]|-{\varepsilon^*}
}
$$
The sheaf-level functor $\varepsilon^*$ agrees with its eponymous presheaf analogue, and as such sends $\mX\mapsto\mX\vert_{\CAlg^\mathrm{evp}}$. On the other hand, the sheaf-level left adjoint $\varepsilon_!$ is computed by first applying the presheaf-level functor $\varepsilon_!$, i.e.\ the left Kan extension along $\varepsilon$, followed by fpqc sheafification (well-defined by virtue of being applied to an accessible presheaf).
\end{cons}

\begin{definition}\label{Def of evp localization}
The \textit{even periodization} of $\mX\in\mathrm{SpStk}$ is given by
$
\mX^\mathrm{evp}:= \varepsilon_!\varepsilon^*\mX.
$
\end{definition}

It is immediate from the definition and Construction \ref{Functoriality of ev} that even periodization is functorial, and that there is a canonical map $\mX^\mathrm{evp}\to\mX$ for all $\mX\in \mathrm{SpStk}$.

\begin{prop}\label{EVP as a colimit}
The even periodization of a nonconnective spectral stack $\mX$ may be expressed as
$$
\mX^\mathrm{evp} \,\,\simeq \varinjlim_{\Spec(A)\in \mathrm{Aff}^\mathrm{evp}_{/\mX}}\Spec(A),
$$
with the colimit taken in $\mathrm{SpStk}$, and where  $\mathrm{Aff}^\mathrm{evp}_{/\mX}:= \mathrm{Aff}^\mathrm{evp}\times_{\mathrm{SpStk}}\mathrm{SpStk}_{/\mX}.$
\end{prop}

\begin{proof}
This follows from the $\i$-category $\mathrm{SpStk}^\mathrm{evp}$ being generated under colimits by its subcategory $\mathrm{Aff}^\mathrm{evp}$, so that the canonical map
$$
\varinjlim_{\Spec(A)\in \mathrm{Aff}_{/\mX}^\mathrm{evp}}\Spec(A)\vert_{\CAlg^\mathrm{evp}}\to\mX\vert_{\CAlg^\mathrm{evp}}
$$
is an equivalence, and the functor $\varepsilon_!:\mathrm{SpStk}^\mathrm{evp}\to\mathrm{SpStk}$ preserving colimits.
\end{proof}

\begin{exun}\label{Example for proof}
Let $A$ be an even periodic $\E$-ring. Consider the corresponding affine spectral stack $\Spec(A)$, i.e.\,the corepresentable functor $\CAlg\ni B\mapsto \Map_{\CAlg}(A, B)\in\mathrm{Ani}$. In that case, the $\i$-category $\mathrm{Aff}^\mathrm{evp}_{/\Spec(A)}\simeq (\CAlg^\mathrm{evp}_A)^\mathrm{op}$ has $\Spec(A)$ as its final object, and so  its even periodization is $\Spec(A)^\mathrm{evp}\simeq \Spec(A)$.
\end{exun}

Let us explain how to obtain the even filtration from the even periodization of an affine nonconnective spectral scheme.

\begin{prop}\label{Voila the even filtration}
For an $\E$-ring $A$, there is a canonical natural equivalence of filtered $\E$-rings
$$
\Gamma\big((\Spec(A))^\mathrm{evp})^\mathrm{cn};\, \tau_{\ge 2*}(\sO_{\Spec(A)^\mathrm{evp}})\big)\,\simeq\, \mathrm{fil}^{\ge *}_\mathrm{ev}(A).
$$
\end{prop}

\begin{proof}
Specializing Proposition \ref{EVP as a colimit} to the affine case, we get the even periodization of $\Spec(A)$ expressed as the colimit
\begin{equation}\label{EVP of an affine equation}
\Spec(A)^\mathrm{evp}\,\,\simeq \varinjlim_{B\in\CAlg_A^\mathrm{evp}}\Spec(B).
\end{equation}
Using the identification \eqref{Postnikov filtration on functions formula} from Remark \ref{Postnikov filtration on functions}, we find the filtered $\E$-ring in question to be
$$
\Gamma\big((\Spec(A)^\mathrm{evp})^\mathrm{cn};\,   \tau_{\ge 2*}(\sO_{\Spec(A)^\mathrm{evp}}))\,\,\simeq \varprojlim_{B\in\CAlg_A^\mathrm{evp}}\tau_{\ge 2*}(B),
$$
which agrees with the even filtration on $A$ by
Proposition \ref{Even filtration is even periodic filtration}.
\end{proof}

We can further extract the underlying $\E$-ring of the even filtration as the $\E$-ring of global functions of the even periodization of the corresponding affine:

\begin{corollary}\label{Ring of functions computation}
For any $\E$-ring $A$, there is a canonical equivalence of $\E$-rings
$$
\sO(\Spec(A)^\mathrm{evp})\simeq \varinjlim_n\mathrm{fil}_\mathrm{ev}^{\ge n}(A).
$$
\end{corollary}

\begin{proof}
This follows by passing to the underlying objects of the filtrations of Proposition \ref{Voila the even filtration}. Or more explicitly, we see using
 a truncatedness argument as in \cite[Remark 2.3.1]{HRW} that the second map of
$$
\varinjlim_n\mathrm{fil}_\mathrm{ev}^{\ge n}(A)
\simeq \varinjlim_n \varprojlim_{B\in\CAlg^\mathrm{evp}_A}\tau_{\ge 2n}(B)
\simeq
\varprojlim_{B\in\CAlg^\mathrm{evp}_A}B.
$$
is an equivalence.  The right-most side may be identified with the $\E$-ring of functions on $\Spec(A)^\mathrm{evp}$ due to \eqref{EVP of an affine equation} and the fact that the functor $\sO :\mathrm{SpStk}^\mathrm{op}\to \CAlg$ preserves limits.
\end{proof}

Now that we know that it indeed geometrizes the even filtration, let us consider further properties of the even periodization.

\begin{prop}\label{Even stacks inside stacks}
The following statements hold:
\begin{enumerate}[label = (\alph*)]
\item The even periodization functor $\mX\mapsto \mX^\mathrm{evp}$ preserves all small colimits in $\mathrm{SpStk}$.\label{Corollary 10, a}
\item The functor $\varepsilon_! : \mathrm{SpStk}^\mathrm{evp}\to \mathrm{SpStk}$ is fully faithful. Its essential image is generated under small colimits by  $\Spec(A)$ where $A$ ranges over even periodic $\E$-rings.\label{Corollary 10, b}
\item A nonconnective spectral stack $\mX$ belongs to the essential image of $\mathrm{SpStk}^\mathrm{evp}$ inside $\mathrm{SpStk}$ as explained in
\ref{Corollary 10, b} if and only if $\mX^\mathrm{evp}\to\mX$ is an equivalence.\label{Corollary 10, c}
\item The canonical map $(\mX^\mathrm{evp})^\mathrm{evp}\to \mX^\mathrm{evp}$ is an equivalence for any $\mX\in\mathrm{SpStk}$.\label{Corollary 10, d}
\end{enumerate}
\end{prop}

\begin{proof}
For \ref{Corollary 10, a}, observe that even periodization may be written via Construction \ref{Functoriality of ev} as the composite of two left adjoint functors $\varepsilon_!$and $\varepsilon^*$.

Claims \ref{Corollary 10, c} and \ref{Corollary 10, d} follow easily from \ref{Corollary 10, b},
To establish the latter,
we must show that the unit $\mX\to\varepsilon^*\varepsilon_!\mX$ is an equivalence for any $\mX\in \mathrm{SpStk}^\mathrm{evp}$. Since both sides commute with colimits in $\mX$, we may reduce to the corepresentable case, i.e.\ $\mX =\Spec(A)\vert_{\CAlg^\mathrm{evp}}$ for an arbitrary even periodic $\E$-ring $A$. But the claim is clear case since we have $\varepsilon_!(\Spec(A)\vert_{\CAlg^\mathrm{evp}})\simeq \Spec(A)$ in that case, see Example \ref{Example for proof}.
\end{proof}

In light of Proposition \ref{Even stacks inside stacks}, we will often identify even periodic spectral stacks, in the sense of Definition \ref{Def of even stacks} with their essential image inside $\mathrm{SpStk}$, and therefore treat them as a special class of nonconnective spectral stacks.

\begin{exun}
Let $\mM_\mathrm{Ell}^\mathrm{or}$ and $\overline{\mM}{}^\mathrm{or}_\mathrm{Ell}$ denote the \textit{moduli stack of oriented elliptic curves} in the sense of \cite{Elliptic 2} and its Deligne-Mumford-compactified version,
which are related to topological modular forms by the canonical equivalences of $\E$-rings
 $\sO(\mM_\mathrm{Ell}^\mathrm{or})\simeq \mathrm{TMF}$ and $\sO(\overline{\mM}{}_\mathrm{Ell}^\mathrm{or})\simeq \mathrm{Tmf}$. Since they admit \'etale covers by even periodic $\E$-rings: elliptic spectra obtained from Landweber's exact functor theorem, it follows that $\mM_\mathrm{Ell}^\mathrm{or}$ and $\overline{\mM}{}^\mathrm{or}_\mathrm{Ell}$ are even periodic spectral stacks. We will return to these in Section \ref{Section TMF}.
 \end{exun}

From the perspective of always implicitly embedding $\mathrm{SpStk}^\mathrm{evp}\hookrightarrow\mathrm{SpStk}$, we can express the universal property of the even periodizations as asking that for any $\mX\in \mathrm{SpStk}$ and $\mY\in \mathrm{SpStk}^\mathrm{evp}$,  the map
\begin{equation}\label{Univ prop of evp as colocalization}
\Map_{\mathrm{SpStk}}(\mY, \mX^\mathrm{evp})\to \Map_{\mathrm{SpStk}}(\mY, \mX),
\end{equation}
induced by the canonical map $\mX^\mathrm{evp}\to\mX$, is a homotopy equivalence.

\begin{remark}\label{Evp and limits}
It is immediate
from the universal property \eqref{Univ prop of evp as colocalization}  that  even periodization $\mX\mapsto \mX^\mathrm{evp}$, when viewed as a functor $\mathrm{SpStk}\to\mathrm{SpStk}^\mathrm{evp}$, preserves all small limits. This might seem confusing, since we saw in Proposition \ref{Even stacks inside stacks} that it preserves all colimits. In fact, it is due to the fully faithful embedding $\varepsilon_!:\mathrm{SpStk}^\mathrm{evp}\hookrightarrow\mathrm{SpStk}$ preserving colimits, but not necessarily limits. Indeed, if we view even periodic spectral stacks as a full subcategory of $\mathrm{SpStk}$, i.e.\ if the application of the functor $\varepsilon_!$ is left implicit, then the even periodization is by Definition \ref{Def of evp localization} given by the functor $\varepsilon^*$. That is a left adjoint and so clearly preserves limits -- so long as they are  computing in $\mathrm{SpStk}^\mathrm{evp}$ as opposed to in $\mathrm{SpStk}$.
\end{remark}

\begin{remark}[Even periodization as blowup along the classical locus]
 Ben Antieau suggested to the author that even periodization is akin to a blowup construction.  We attempt to expound upon this analogy here.
Recall from Subsection \ref{Section underlying classical stack} that the connective localization and the underlying classical stack of a spectral stack $\mX$ are all related by a canonical cospan
$$
\mX\xrightarrow{j}\mX^\mathrm{cn}\xleftarrow{i}\mX^\heart.
$$
If the spectral stack $\mX$ is even periodic, the weak $2$-periodicity requirement implies that the  quasi-coherent sheaf $\mL:=\Sigma^{-2}(\tau_{\ge 2}(\sO_{\mX}))$ is a \textit{line bundle}\footnote{Unlike in classical algebraic geometry, the notions of invertible sheaves and line bundles decouple in spectral algebraic geometry. That is because the former may locally over an affine $\Spec(A)$ be of the form $\Sigma^n(A)$ for any $n\in \mathbf Z$, whereas the latter must locally agree with the rank $1$ free $A$-module $A$.}, i.e.\ a locally free of rank 1 over $\mX^\mathrm{cn}$, and the evenness requirement implies there is a canonical cofiber sequence
$$
\tau_{\ge 2}(\sO_{\mX})\to \tau_{\ge 0}(\sO_{\mX})\to \pi_0(\sO_{\mX})
$$
in the stable $\i$-category $\QCoh(\mX^\mathrm{cn})$. We may rewrite this cofiber sequence as
$$
\Sigma^2(\mL)\xrightarrow{\beta}\sO_{\mX^\mathrm{cn}}\to i_*(\sO_{\mX^\heart}).
$$
Since its ``ideal sheaf" is therefore a $2$-shifted line bundle, this exhibits the closed immersion $i:\mX^\heart\to \mX$ as a \textit{$2$-shifted effective Cartier divisor} (see also \cite[Sections 2.2 \& 2.3]{Synthetic} for a related discussion). The effect of passing to the even periodization for an arbitrary spectral stack $\mX$ is therefore to replace the closed immersion $\mX^\heart\to \mX^\mathrm{cn}$ with a  (2-shifted) effective Cartier divisor,  not unlike how the universal property of a blowup in classical algebraic geometry
is to replace a chosen closed immersion with an effective Cartier divisor.
\end{remark}

\subsection{Even periodization of quasi-coherent sheaves}

 A variant of the even filtration $\mathrm{fil}^*_{\mathrm{ev}/A}(M)$ is introduced in \cite[Appendix A]{HRW} that applies not to $\E$-rings themselves, but instead to module spectra $M$ over a fixed $\E$-ring $A$. This construction is particularly natural from the perspective of the even periodization: it is merely the Postnikov filtration of quasi-coherent sheaves on $\Spec(A)^\mathrm{evp}$.

\begin{definition}[{\cite[Construction A.1.3, Remark A.1.4]{HRW}}]\label{Def of EVF for modules}
The \textit{even filtration for modules} is the functor $\Mod\to \mathrm{FilSp}$, denoted $(A, M)\mapsto \mathrm{fil}_{\mathrm{ev}/A}^{\ge *}(M)$,  given by the right Kan extension  of the double-speed Postnikov filtration $\tau_{\ge 2*}$ along the projection onto the first factor $\Mod\times_\CAlg \CAlg^\mathrm{ev}\to\Mod$. That is to say, for an $\E$-ring $A$ and $A$-module $M$, it is explicitly given by
$$
\mathrm{fil}^{\ge *}_{\mathrm{ev}/A}(M)\,\,\simeq\varprojlim_{B\in\CAlg_A^\mathrm{ev}}\tau_{\ge 2*}(M\o_A B).
$$
\end{definition}

\begin{prop}\label{Voila EVF for modules}
For an $\E$-ring $A$, let $\ell:\Spec(A)^\mathrm{evp}\to\Spec(A)$ be the canonical map. For any $A$-module $M$, there is a canonical and natural equivalence of filtered $\E$-rings
$$
\Gamma\big( (\Spec(A)^\mathrm{evp})^\mathrm{cn};\, \tau_{\ge 2*}(\ell^*\widetilde M)\big)\simeq \mathrm{fil}^{\ge *}_{\mathrm{ev}/A}(M),
$$
where $\widetilde M$ is the quasi-coherent sheaf on $\Spec(A)$ corresponding to the $A$-module $M$ under the equivalence of $\i$-categories $\QCoh(\Spec(A))\simeq \Mod_A$.
\end{prop}

\begin{proof}
From the limit definition of quasi-coherent sheaves \eqref{QCoh as limit} and Proposition \ref{EVP as a colimit}, we  find that the quasi-coherent pullback $f^*:\Mod_B\to\Mod_A$ along the  maps of spectral stacks $\Spec(B)\to\Spec(A)$ from even periodic affines $\Spec(B)\in \mathrm{Aff}^\mathrm{evp}$ induce via the equivalence of $\i$-categories
$$
\QCoh(\Spec(A)^\mathrm{evp})\,\,\simeq \varprojlim_{B\in\CAlg_A^\mathrm{evp}}\Mod_B
$$
the pullback $\ell^*:\QCoh(\Spec(A))\to\QCoh(\Spec(A)^\mathrm{evp})$. That is to say, the quasi-coherent sheaf $\ell^*(\widetilde M)$ is determined by the system of modules $M\o_A B\in\Mod_B$ for all $B\in\CAlg_B^\mathrm{evp}$.
Since the Postnikov filtration on quasi-coherent sheaves is defined via descent as well, we find that the canonical map of filtered spectra
$$
\Gamma\big( (\Spec(A)^\mathrm{evp})^\mathrm{cn});\, \tau_{\ge 2*}(\ell^*\widetilde M)\big)\to\varprojlim_{B\in\CAlg_A^\mathrm{evp}}\tau_{\ge 2*}(M\o_A B).
$$
is an equivalence. Comparing with Definition \ref{Def of EVF for modules} of the even filtration for modules, it remains to show that the limit in question remains unaffected by restricting it to the subcategory $\CAlg_A^\mathrm{evp}\subseteq\CAlg_A$. This follows just like how we argued the corresponding fact  for the even filtration of $\E$-rings in
 Proposition \ref{Even filtration is even periodic filtration}.
\end{proof}

\begin{remark}
We can recover
Proposition \ref{Voila the even filtration} from Proposition \ref{Voila EVF for modules} by applying it to the unit $A$-module $A$. Indeed, it is clear from Definition \ref{Voila EVF for modules} that $\mathrm{fil}_{\mathrm{ev}/A}^{\ge *}(A)\simeq \mathrm{fil}_\mathrm{ev}^{\ge *}(A)$,  while on the other hand  $\ell^*\widetilde A\simeq\ell^* \sO_{\Spec(A)}\simeq \sO_{\Spec(A)^\mathrm{evp}}.$
\end{remark}

In light of Proposition \ref{Voila EVF for modules}, we may think of the functor
$$
\ell^*:\QCoh(\mX)\to \QCoh(\mX^\mathrm{evp})
$$
as the \textit{even periodization of quasi-coherent sheaves} on any $\mX\in \mathrm{SpStk}$.

\subsection{First examples of even periodization}
Were we to compare the  discussion of  even periodization in Section \ref{Section EVP} with that of the connective localization from Section \ref{Section connective localization}, we might expect we are due an explicit description of even periodization for geometric stacks in the spirit of Example \ref{connective localization for geometric stacks}.
However, this is complicated by an essential difference between even periodization and connective localization: while the connective spectral stack $\mathrm{Spec}(A)^\mathrm{cn}\simeq\Spec(\tau_{\ge 0}(A))$ is affine for any $\E$-ring $A$,  we will see in Proposition \ref{EVP affines} that the even periodic spectral stack $\Spec(A)^\mathrm{evp}$ is only affine if $A\in \CAlg^\mathrm{evp}$.

We can nevertheless give an explicit description, albeit in a much more limited context. For this purpose, we give a slight generalization of Definition \ref{Def eff}:

\begin{definition}\label{Definition periodically eff}
A map of nonconnective spectral stacks $\Spec(B)\to \mX$ is \textit{periodically eff} if for every $\mathcal Y\in \mathrm{Aff}^\mathrm{evp}_{/\mX}$, the map $\mathcal Y\times_{\mX}\Spec(B)\to\Spec(B)$ is an fpqc cover in $\mathrm{Aff}^\mathrm{evp}$.
\end{definition}

\begin{exun}\label{Example 1 of evp}
In the affine case $\mX\simeq\Spec(A)$, a periodically eff map $\Spec(B)\to\mX$ is equivalent to an eff map of $\E$-rings $A\to B$ in the sense of Definition \ref{Def eff}, together with the additional requirement that the $\E$-ring $B$ must be even periodic. Since we saw in the proof of Proposition \ref{Even filtration is even periodic filtration} that $\mathbf S\to \mathrm{MUP}$ is eff for any $\E$-ring structure on the periodic complex bordism spectrum, the corresponding map of affines $\Spec(\mathrm{MUP})\to\Spec(\mathbf S)$ is an example of a periodically eff map.
\end{exun}

\begin{prop}\label{Computing the evening}
Let $\Spec(B)\to \mX$ be a periodically eff map of nonconnective spectral stakcs. Then the canonical map
$$
\varinjlim_{\,\,\,\,\mathbf{\Delta}^\mathrm{op}}\Spec(B^\bullet)\to \mX^\mathrm{evp}
$$
is an equivalence where $B^\bullet = \sO\big(\Spec(B)^{\times_{\mX}[\bullet]}\big)$. The even periodization $\mX^\mathrm{evp}$ is in this case in particular a geometric spectral stack.
\end{prop}

\begin{proof}
By the definition of periodically eff maps, it follows that $\Spec(B)\vert_{\CAlg^\mathrm{evp}}\to\mX\vert_{\CAlg^\mathrm{evp}}$ is an fpqc cover in $\mathrm{SpStk}^\mathrm{evp}$, as well as that its \v{C}ech nerve is given by
$$
\Spec(B)^{\times_{\mX}[\bullet]}\simeq \Spec(B^\bullet)
$$
for a cosimplicial object $B^\bullet$ in $\CAlg^\mathrm{evp}$. It follows that the map
$$
\varinjlim_{\,\,\,\,\mathbf{\Delta}^\mathrm{op}}\Spec(B^\bullet)\vert_{\CAlg^\mathrm{evp}}\to \mX\vert_{\CAlg^\mathrm{evp}}
$$
is an equivalence in $\mathrm{SpStk}^\mathrm{evp}$. This gives rise to the desired equivalence upon applying the colimit-preserving fully faithful embedding $\varepsilon_! :\mathrm{SpStk}^\mathrm{evp}\to\mathrm{SpStk}$ of Proposition \ref{Even stacks inside stacks}, and recalling that $\mX^\mathrm{evp}\simeq \varepsilon_!(\mX\vert_{\CAlg^\mathrm{evp}})$ by Definition \ref{Def of evp localization}.
\end{proof}

\begin{exun}[Even periodization of even affines over $\mathbf Q$]\label{Rational even localization}
Fix any even $\E$-algebra $R$ over $\Q$. The \textit{$2$-shifted multiplicative group over $R$} is the nonconnective spectral $R$-stack given by
$$
\G_{m, R}[\pm 2] :=\Spec(R[\beta^{\pm 1}])
$$
where $|\beta|=2$. Note here the $R$-module $R[\beta^{\pm 1}]\simeq\bigoplus_{n\in \mathbf Z}\Sigma^{2n}(R)$ carries a canonical (and unique) $\E$-algebra structure over $R$  by virtue of the $\mathbf Q$-algebra assumption.
 By  a standard sum-reindexing argument, we can identify the \v{C}ech nerve of the map $\G_{m, R}[\pm 2]\to\Spec(R)$ with the bar construction for the canonical $\G_{m, R}$-action on $\G_{m, R}[\pm 2]$, corresponding to the evident grading on $R[\beta]$. Since the map $R\to R[\beta^{\pm 1}]$ is clearly eff and $R[\beta^{\pm 1}]$ is even periodic, we may use Proposition \ref{Computing the evening} to identify the even periodization of $\Spec(R)$ with
$$
\Spec(R)^\mathrm{evp}\simeq \G_{m, R}[\pm 2]/\G_{m, R}.
$$
Since the functor of points of $\G_{m, R}[\pm 2]$ amounts to picking out invertible elements in $\pi_2$ of $\E$-$R$-algebras, we obtain the explicit description of the functor of points $\CAlg_R\to\mathrm{Ani}$ of the even periodization of $\Spec(R)$ as
\begin{equation}\label{even periodization over Q}
\Spec(R)^\mathrm{evp}(A) \simeq  \begin{cases} * & \text{if $A$ is weakly $2$-periodic,}\\ \emptyset & \text{otherwise} \end{cases}
\end{equation}
for any $A\in\CAlg_R$.
\end{exun}

\begin{remark}
As we will see in Proposition \ref{Evp for even affines over MU}, the argument from Example \ref{Rational even localization} can be extended more generally to work over the base $\mathrm{MU}$. The one aspect unique to the rational setting, in the sense that it does not hold over more general $\mathrm{MU}$-algebras, is the universal property \eqref{even periodization over Q}. This is roughly due to the the spectral algebro-geometric divergence between $\mathbf G_{m, R} = \Spec(R[\mathbf Z])$ and $\mathrm{GL}_{1, R} = \Spec(R[\Omega^\i(\mathbf S)])$ whenever $R$ is not a $\mathbf Q$-algebra. However, see the
final part of this subsection, in particular Construction \ref{Remark GL_1[2]/GL_1} and after, for a discussion of a universal property such as \eqref{even periodization over Q} over the sphere spectrum.
\end{remark}

The computation from Example \ref{Rational even localization} breaks down
over the
 sphere spectrum $\mathbf S$. The issue is the well-known failure of the spectrum $\bigoplus_{n\in \mathbf Z}\Sigma^{2n}(\mathbf S)$ to support an $\E$-ring structure (or really anything more than an $\mathbb E_2$-structure). Equivalently, in contrast to Construction \ref{Const of shearing}, the shearing autoequivalence of the $\i$-category of graded spectra $(-)^\shear:\mathrm{GrSp}\to\mathrm{GrSp}$ does not support a symmetric monoidal (or really anything more than $\mathbb E_2$-monoidal) structure.

 Nevertheless, a computation of the even periodization of $\mathrm{Spec}(\mathbf S)$ is entirely feasible. It may be given by a similar functor of points description as \eqref{even periodization over Q}, and in fact ends up reproducing a familiar object.

\begin{definition}\label{Definition of mM}
The \textit{chromatic base stack} is given as a functor $\mathcal M :\CAlg\to\Ani$ by
$$
\mathcal M(A) \simeq \begin{cases} * & \text{if $A$ is complex periodic,}\\ \emptyset & \text{otherwise.} \end{cases}
$$
\end{definition}

 This stack was studied extensively in our previous work \cite{ChromaticCartoon}, \cite{ChromaticFiltration}, where we demonstrated its close connection to chromatic homotopy theory, hence the name. We summarize here the relevant results concerning it.
 
 \begin{theorem}[{\cite[Proposition 2.1.9, Corollary 2.3.7]{ChromaticCartoon}}]\label{On M}
The chromatic base stack $\mM$ is equivalent to the moduli stack of oriented formal groups $\M$. It is a geometric nonconnective spectral stack, and  any choice of $\E$-ring structure on $\mathrm{MUP}$ exhibits a simplicial presentation 
$$
\mM\simeq \varinjlim_{\,\,\,\,\mathbf\Delta^\mathrm{op}}\Spec(\MUP^\bullet),
$$
where $\mathrm{MUP}^{\bullet}\simeq \mathrm{MUP}^{\otimes_\mathbf S [\bullet]}$ is the cosimplicial cobar construction for $\mathbf S\to\MUP$.
\end{theorem}

Since any even periodic $\E$-ring is in particular complex periodic, and a map to $\mM$ is essentially unique if it exists, the inclusion $\mathrm{SpStk}^\mathrm{evp}\hookrightarrow\mathrm{SpStk}$ factors canonically through the full subcategory $\mathrm{SpStk}_{/\mM}\subseteq \mathrm{SpStk}$ - see for instance Proposition \ref{complex periodic stacks}. In particular, any even periodic spectral stack $\mX$ admits a canonical and essentially unique map $\mX\to\mM$.

\begin{corollary}\label{Even localization of the sphere}
The canonical map  $\Spec(\mathbf S)^\mathrm{evp}\to \mM$ is an equivalence in $\mathrm{SpStk}$.
\end{corollary}

\begin{proof}
By Example \ref{Example 1 of evp}, the canonical map $\Spec(\MUP)\to\Spec(\mathbf S)$ is periodically eff. The even periodization of the terminal spectral stack $\Spec(S)$ may therefore be expressed as the geometric realization
$$
\Spec(\mathbf S)^\mathrm{evp}\simeq \varinjlim_{\,\,\,\,\mathbf\Delta^\mathrm{op}}\Spec(\MUP^\bullet).
$$
The claim now follows from the analogous simplicial description of the chromatic base stack from Theorem \ref{On M}.
\end{proof}

\begin{remark}
By combining the above identification $\Spec(\mathbf S)^\mathrm{evp}\simeq \mM$ with  Proposition \ref{Voila the even filtration}, we obtain a description of the even filtration for the sphere spectrum as the Postnikov filtration of quasi-coherent sheaves on the chromatic base stack, in the sense that
$$
\mathrm{fil}_\mathrm{ev}^{\ge *}(\mathbf S)\,\simeq\, \Gamma(\mM^\mathrm{cn};\, \tau_{\ge 2*}(\sO_{\mM})).
$$
The is compatible with the observations from \cite{HRW} and \cite{ChromaticCartoon}, which identify each side of this equivalence in turn with the Adams-Novikov filtration for the sphere spectrum. We will give a generalization of this comparison in Proposition \ref{EVPL vs CPL}.
\end{remark}

The spectral stacks $\mM$ and $\Spec(\mathbf S)$ play the role of the terminal presheaf on $\CAlg^{\mathbf C\mathrm p}$ and $\CAlg$ respectively.
We could also consider a middle-ground between them, playing this role for weakly $2$-periodic $\E$-rings.
Said spectral stack admits an explicit quotient description, analogous to the description over $\mathbf Z$ in  Example \ref{Rational even localization}, which we now outline.

\begin{cons}\label{Remark GL_1[2]/GL_1}
The \textit{$2$-shifted smooth affine line} may be given by
$$
\mathbf A^1_\mathrm{sm}[2]\,\simeq\, \Spec(\Sym^*_\mathbf S(\Sigma^{-2}(\mathbf S))),
$$
or as a functor of points by $A\mapsto \Omega^{\infty -2}(A)$. Considering the canonical map of spectra
$$
\Sigma^{-2}(\mathbf S)\,\simeq\, \mathrm{Sym}^1_S(\Sigma^{-2}(\mathbf S))\to \mathrm{Sym}^*_\mathbf S(\Sigma^{-2}(\mathbf S))
$$
as giving an element $\beta\in \pi_{-2}(\mathrm{Sym}_\mathbf S^*(\Sigma^{-2}(\mathbf S)))$, we define
$$
\GL_1[\pm 2]:=\Spec(\mathrm{Sym}^*_\mathbf S(\Sigma^{-2}(\mathbf S))[\beta^{-1}]).
$$
Its functor of points picks out for an $\E$-ring $A$ all those connected components in $\Omega^{\infty -2}(A)$ for which the corresponding element in $\pi_{-2}(A)$ (the image of $\beta$) admits a multiplicative inverse inside $\pi_*(A)$. It is easy to see that the automorphisms of such a choice are given by $\Omega^\i(A)\times_{\pi_0(A)}\pi_0(A)^\times\simeq \GL_1(A)$, exhibiting a $\GL_1$-action on $\GL_1[\pm 2]$.
The quotient stack $\GL_1[\pm 2]/\GL_1$ therefore by construction \textit{classifies weakly $2$-periodic $\E$-rings}, in the sense that
$$
(\GL_1[\pm 2]/\GL_1)(A) \simeq  \begin{cases} * & \text{if $A$ is weakly $2$-periodic,}\\ \emptyset & \text{otherwise} \end{cases}
$$
holds for any $\E$-ring $A$.
\end{cons}

There are essentially unique maps of nonconnective spectral stacks
$$
\mM\to\GL_1[\pm 2]/\GL_1 \to\Spec(\mathbf S).
$$
When restricted along $\CAlg^\mathrm{evp}\subseteq\CAlg$, all three coincide with the terminal presheaf $A\mapsto *$, showing that we have canonical equivalences
\begin{equation}\label{Equation EVP of GL_1[2]/GL_1}
\mM\simeq\mM^\mathrm{evp}\simeq(\GL_1[\pm 2]/\GL_1)^\mathrm{evp} \simeq\Spec(\mathbf S)^\mathrm{evp}.
\end{equation}
upon the passage to even periodization.

Identification \eqref{Equation EVP of GL_1[2]/GL_1} in particular shows that the even periodization functor $\mX\mapsto \mX^\mathrm{evp}$ is not faithful. This may also be seen from the following simple example.

\begin{exun}\label{Example F_2^{tC_2}}
Consider the $\E$-ring $\mathbf F_2^{\mathrm t\mathrm C_2}$, the Tate construction for the trivial action of the cyclic group $\mathrm C_2$ on the Eilenberg-MacLane spectrum $\mathbf F_2$. On the level of homotopy groups, we have an isomorphism of graded rings $\pi_*(\mathbf F_2^{\mathrm t\mathrm C_2}) \simeq \mathbf F_2[s^{\pm 1}]$ where the invertable variable $s$ is in graded degree $-1$. It follows that $\mathbf F_2^{\mathrm t\mathrm C_2}$ admits no ring spectra maps into even ring spectra, much less any $\E$-ring maps into even periodic $\E$-ring. Its even periodization is therefore trivial, i.e.\ the initial spectral stack $\Spec(\mathbf F_2^{\mathrm t\mathrm C_2})=\emptyset$.
\end{exun}

\subsection{Even periodization of affines}

One noteworthy aspect of the chromatic base stack $\mM$ is that the sphere spectrum may be identified as $\mathbf S\simeq \sO(\M)$ with its $\E$-ring of functions \cite[Proposition 2.4.1]{ChromaticCartoon}. This may in light of Corollary \ref{Even localization of the sphere} be viewed as  an instance of a more general result about even periodization.

For this, let the full subcategory of \textit{even-convergent $\E$-rings} $\CAlg^\mathrm{evc}\subseteq\CAlg$ consists of all those $\E$-rings $A$ for which the canonical map $A\to\varinjlim_n \mathrm{fil}_\mathrm{ev}^{\ge n}(A)$ is an equivalence. 

\begin{remark}
Any even $\E$-ring is oviously even-convergent, so that $\CAlg^\mathrm{ev}\subseteq\CAlg^\mathrm{evc}$. By the upcoming work of Achim Krause and Robert Burklund, announced in
\cite[Remark 8.5]{Piotr's even}, any connective $\E$-ring is even-convergent as well, so that $\CAlg^\mathrm{cn}\subseteq\CAlg^\mathrm{evp}$.
\end{remark}

\begin{prop}\label{Prop champs affines}
The functor $(\CAlg^{\mathrm{evc}})^\mathrm{op}\to\mathrm{SpStk}^\mathrm{evp}$, given by $A\mapsto \Spec(A)^\mathrm{evp}$,
is fully faithful. The $\E$-ring of global sections functor $\mX\mapsto \sO(\mX)$ is its left inverse.
\end{prop}

\begin{proof}
According to the definition of the even periodization functor as the right adjoint to the inclusion $\mathrm{SpStk}^\mathrm{evp}\hookrightarrow\mathrm{SpStk}$, there exists for every pair of $\E$-rings $A, B\in\CAlg$ canonical homotopy equivalences
\begin{eqnarray*}
\Map_{\mathrm{SpStk}^\mathrm{evp}}(\Spec(A)^\mathrm{evp}, \Spec(B)^\mathrm{evp})&\simeq &
\Map_{\mathrm{SpStk}}(\Spec(A)^\mathrm{evp}, \Spec(B))\\
&\simeq &
\Map_\mathrm{CAlg}(B, \sO(\Spec(A)^\mathrm{evp}) ).
\end{eqnarray*}
The claim of being fully faithful, as well as the identification of the left inverse functor, now follows directly from Corollary \ref{Ring of functions computation}.
\end{proof}

\begin{remark}[Analogy with the affine stacks of To\"en] 
Proposition \ref{Prop champs affines} asserts that, although the even periodizations $\Spec(A)^\mathrm{evp}$ are seldom even periodic affines themselves, they nonetheless often behave in a  somewhat affine manner. This is  reminiscent of T\"oen's theory of \textit{affine stacks} from \cite{Champs affines}, called \textit{coaffine stacks} in \cite{DAGVIII}. There a class of classical stacks is identified which are not affines, but nonetheless exhibit affine behavior. Over a field $k$ of characteristic zero, this amounts to restricting the nonconnective affines $\Spec(A):\CAlg_k\to \mathrm{Ani}$ of  coconnective $\E$-algebras $A$ over $k$ (outside characteristic zero, it should be coconnective derived rings) to the full subcategory $\CAlg_k^\heart\subseteq\CAlg_k$. That is to say, using Lurie's $\mathrm{cSpec}$ notation for affine stacks from \cite{DAGVIII} (T\"oen's school might denote this by $\Spec^{\Delta}$ instead), we have
$$
\mathrm{cSpec}(A)\simeq \Spec(A)\vert_{\CAlg_k^\heart}.
$$
Analogously in Proposition \ref{Prop champs affines}, the fully faithful image of an even-convergent $\E$-ring $A$ is the even periodic spectral stack whose functor of points may be identified with the restriction
$$
\Spec(A)^\mathrm{evp}\simeq\Spec(A)\vert_{\CAlg^\mathrm{evp}}.
$$
This similarity between even periodization of (even-convergent) affines and To\"en's affine stacks  might perhaps shed some light on the fact that for  the chromatic base stack $\mathrm{Spec}(\mathbf S)^\mathrm{evp}\simeq \mM$, the quasi-coherent sheaves $\QCoh(\mM)$ are not \textit{quite} equivalent to $\Sp$ (though the difference disappears at finite height).
Likewise in the setting of affine stacks,  $\QCoh(\mathrm{cSpec}(A))$ does not generally coincide with $\Mod_A$. Rather by \cite[Remark 4.5.6]{DAGVIII}, the two are related to each other through $t$-structures as the respective right and left completions of each other. Indeed, we can view modules as arising from the affine stack as $\Mod_A\simeq \mathrm{IndCoh}(\mathrm{cSpec}(A))$. This reinforces the author's belief that the usage of the notation $\mathrm{IndCoh}(\mM)$ in \cite{ChromaticCartoon}, though perhaps not \textit{fully} accordant with the usage of $\mathrm{IndCoh}$ in Geometric Langlands and related areas, is nonetheless not fundamentally misguided, and may yet come to be fully justified in future work.
\end{remark}

\begin{remark}
The passage to the $\E$-ring of global sections provides a left adjoint to the functor sending an $\E$-ring to the corresponding affine, i.e.\ there is an adjunction of $\i$-categories
$$
\sO:\mathrm{SpStk}\rightleftarrows \CAlg^\mathrm{op} : \Spec.
$$
Let us compose this adjunction with the one coming from Proposition \ref{Even stacks inside stacks}
$$
\mathrm{SpStk}^\mathrm{evp}\rightleftarrows \mathrm{SpStk} :(-)^\mathrm{evp},
$$
whose left adjoint is the fully faithful embedding through which we identify even periodic spectral stacks as special cases of nonconnective spectral stacks. We obtain an adjunction
$$
\sO:\mathrm{SpStk}^\mathrm{evp}\leftrightarrows \CAlg^\mathrm{op} : \Spec(-)^\mathrm{evp}
$$
It follows by abstract nonsense that a right adjoint is fully faithful on the full subcategory where the unit of the adjunction is an equivalence. The unit is in this case the $\E$-ring map $A\to \sO(\Spec(A)^\mathrm{evp})$, from which
we see that the subcategory of even-convergent $\E$-rings $\CAlg^\mathrm{evc}\subseteq\CAlg$ appears naturally if one wishes a fully faithfulness result of the form of Proposition \ref{Prop champs affines}.
\end{remark}

Since any even $\E$-ring $A$ is in particular even-convergent, the even periodization of the affine $\Spec(A)$ gives a way to faithfully represent it by an even periodic spectral stack. The underlying classical stack of said even periodic spectral stack is therefore an inherently interesting object of classical geometry, which encodes some aspects of the even $\E$-ring $A$. Indeed, in contrast to the affine itself, whose underlying classical stack is $\Spec(A)^\heart\simeq \Spec(\pi_0(A))$, we now prove that  the classical stack $(\Spec(A)^\mathrm{evp})^\heart$ encode (precisely) the information of all of the homotopy groups $\pi_*(A)$ with their grading.

\begin{prop}\label{Heart of evp of affine}
Let $A$ be an even $\E$-ring. There is a canonical and natural equivalence of classical stacks
$$
(\Spec(A)^\mathrm{evp})^\heart \, \simeq\,\Spec\big(\bigoplus_{n\in \mathbf Z}\pi_{2n}(A)\big)/\mathbf G_m,
$$
where the $\mathbf G_m$-action on the right-hand side corresponds to the grading on $\pi_{2*}(A)$.
\end{prop}

\begin{proof} Fix an $\E$-ring structure on $\mathrm{MUP}$.
The map $A\to A\otimes_{\mathbf S}\mathrm{MUP}$ is periodically eff by virtue of $A$ being even. Proposition \ref{Computing the evening}  therefore supplies a simplicial presentation for the even periodization in question as
$$
\Spec(A)^\mathrm{evp}\, \simeq\varinjlim_{\,\,\,\,\mathbf{\Delta}^\mathrm{op}}\Spec(A\o_{\mathbf S}\mathrm{MUP}^\bullet),
$$
where $\mathrm{MUP}^\bullet := \mathrm{MUP}^{\otimes_{\mathbf S}[\bullet]}$ is the cosimplicial cobar construction of $\mathbf S\to\mathrm{MUP}$. Since the underlying classical stack functor $\mX\mapsto \mX^\heart$ commutes with small colimits, we obtain the simplicial presentation 
$$
(\Spec(A)^\mathrm{evp})^\heart\, \simeq\varinjlim_{\,\,\,\,\mathbf{\Delta}^\mathrm{op}}\Spec(\pi_0(A\o_{\mathbf S}\mathrm{MUP}^\bullet)),
$$
for the underlying classical stack.
Next, note that the $\mathbb E_2$-(but not $\mathbb E_\infty$-)ring identification $\mathrm{MUP}\simeq \mathrm{MU}\otimes_{\mathbf S}\big(\bigoplus_{i\in \mathbf Z}\Sigma^{2i}(\mathbf S)\big)$  induces the sequence of $\mathbb E_2$-ring equivalences
\begin{eqnarray*}
A\o_{\mathbf S}\mathrm{MUP}^\bullet&\simeq &A\o_{\mathbf S}\mathrm{MU}^{\bullet}\otimes_{\mathbf S}\bigoplus_{(i_0, \ldots, i_\bullet)\in \mathbf Z^{\bullet+1}}\Sigma^{2(i_0+\cdots + i_\bullet)}(\mathbf S)\\
&\simeq&
A\o_{\mathbf S}\mathrm{MU}^{\bullet}\otimes_{\mathbf S}
\bigoplus_{(i_1, \ldots, i_\bullet)\in \mathbf Z^{\bullet}}\, \bigoplus_{n\in \mathbf Z}\Sigma^{2n}(\mathbf S)\\ 
 &\simeq&
 A\o_{\mathbf S}\mathrm{MU}^{\bullet}\otimes_{\mathbf S}\big(
 \bigoplus_{n\in \mathbf Z}\Sigma^{2n}(\mathbf S)\big)\otimes_{\mathbf S}\mathbf S[\mathbf Z^{\bullet}]
 ,
\end{eqnarray*}
where we have denoted $\mathbf Z^\bullet := \mathbf Z^{\times |\bullet|}$ and $\mathrm{MU}^\bullet := \mathrm{MU}^{\otimes_\mathbf S[\bullet]}$. In the above formula,
 we should   remain 
 cognizant that the cosimplicial structure also involves the middle term, seing how its summation parameter arose as
 $n=i_0+\cdots + i_\bullet$. Noting that we have ring isomorphisms $\pi_0(\mathrm{MU}^\bullet)\simeq \pi_0(\mathrm{MU})^{\otimes_{\mathbf Z[\bullet]}}\simeq \mathbf Z$, we obtain an isomorphism of commutative rings
$$
\pi_0(A\o_{\mathbf S}\mathrm{MUP}^\bullet) \, \simeq \, 
 \bigoplus_{n\in \mathbf Z}\pi_0\big(\mathrm{MU}^{\bullet}\otimes_{\mathbf S}\Sigma^{2n}( A)\otimes_{\mathbf S}\mathbf S[\mathbf Z^{\bullet}]\big)\,\simeq \,\bigoplus_{n\in \mathbf Z}\pi_{2n}(A)[\mathbf Z^\bullet].
$$
and on the level of the underlying classical stack in question
$$
(\Spec(A)^\mathrm{evp})^\heart\, \simeq\varinjlim_{\,\,\,\,\mathbf{\Delta}^\mathrm{op}}\Spec\big(\bigoplus_{n\in \mathbf Z}\pi_{2n}(A)\big) \times \mathbf G_m^{\times \bullet}.
$$
We may recognize the simplicial scheme on the right-hand side to encode the $\mathbf G_m$-action on $\Spec\big(\bigoplus_{n\in \mathbf Z}\pi_{2n}(A)\big)$ which corresponds, via the usual equivalence between gradings and $\E$-actions  -- see for instance \cite{Tasos's Rees} for a discussion in the $\i$-categorical setting -- to the displayed grading. Its colimit is therefore the quotient of the $\mathbf G_m$-action.
\end{proof}

\begin{remark}
Under the standard equivalence between gradings and $\mathbf G_m$-actions, graded $\E$-rings $A_*$ are identified with commutative algebra objects in $\QCoh(\mathrm B\mathbf G_m)$, and hence with affine morphisms $\mX\to \mathrm B\mathbf G_m$ via the relative spectrum $\mX\simeq \Spec_{\mathrm B\mathbf G_m}(A_*)$. Such a relative spectrum may be equivalently written as the quotient stack of the Rees algebra $\Spec_{\mathrm B\mathbf G_m}(A_*)\simeq\Spec\big(\bigoplus_{n\in \mathbf Z}A_n\big)/\mathbf G_m.$
 conclusion of Proposition \ref{Heart of evp of affine} may be rewritten as identifying
$$(\Spec(A)^\mathrm{evp})^\heart\,\simeq\, \Spec_{\mathrm B\mathbf G_m}(\pi_{2*}(A))$$
for any even $\E$-ring $A$.
\end{remark}

\begin{remark}
It should  be noted that the construction $A\mapsto(\Spec(A)^\mathrm{evp})^\heart$ is a form of the
``stacky even filtration" which, while perhaps not appearing explicitly in \cite{HRW}, is strongly alluded to there, and is certainly well-known to the experts. See in particular \cite[Recollection 5.4.7]{Sanath} for related spectral-stack-level constructions.
\end{remark}

\begin{exun}
Let $A$ be a strictly even periodic $\E$-ring.  Proposition \ref{Heart of evp of affine} identifies the univerlying classical stack of the even periodization of the spectrum of the connective cover $\tau_{\ge 0}(A)$ with the quotient stack
$$
(\Spec(\tau_{\ge 0}(A))^\mathrm{evp})^\heart \simeq \, \mathbf A^1_{\pi_0(A)}/\mathbf G_{m, \pi_0(A)}.
$$
In particular, the underlying classical stack of $\Spec(\mathrm{ku})^\mathrm{evp}$ is given by $\mathbf A^1_{\mathbf Z}/\mathbf G_{m, \mathbf Z}$, the classical moduli stack of generalized Cartier divisors.
\end{exun}

\subsection{Shearing and even periodization over $\mathrm{MU}$}\label{Subsection shearing}
To give a version of Example \ref{Rational even localization} that works over the complex bordism spectrum $\mathrm{MU}$,
we will connect it to the notion of \textit{shearing}. This operation on gradings plays a prominent role in some aspects of the Geometric Langlands program, see for instance \cite[Section 2.3]{BZN},
\cite[Section A.2]{Arinkin-Gaitsgory}, \cite[Section 6]{BZSV}, but has also been discussed in homotopical context, e.g.\ \cite{Rotational invariance}, \cite[Section 3.3]{Arpon}. The version of relevance to us -- a symmetric monoidal shearing over $\mathrm{MU}$ -- is due to Devalapurkar in \cite{Sanath}.

\begin{cons}\label{Const of shearing}
The \textit{shearing endofunctor} on graded $\mathrm{MU}$-modules
\begin{eqnarray*}
(-)^{\shear}\, :\mathrm{GrMod}_\mathrm{MU}&\to &\mathrm{GrMod}_\mathrm{MU},\\
M_*&\mapsto& M_*[2*]
\end{eqnarray*}
is shown in \cite[Lemma 2.1.3]{Sanath} (where shearing is denoted by $\widetilde{\mathrm{sh}}$) to
admits a canonical symmetric monoidal structure. It is constructed so that free graded $\E$-algebra over $\mathrm{MU}$ on a generator $t$ in
grading degree $1$ and homotopical degree $0$ is sheared to
$$
\mathrm{MU}[t^{\pm 1}]^{\shear}\,\simeq\,\mathrm{MUP},
$$
the periodic complex bordism spectrum $\mathrm{MUP}$, equipped with the $\E$-ring structure  of the Thom spectrum of an $\E$-map\footnote{This has to be specified since, as shown in \cite{Hahn-Yuan}, the $\E$-ring structure on $\mathrm{MUP}$, unlike the one on $\mathrm{MU}$, is not unique. In particular the Thom $\E$-ring structure disagrees with one coming from the Snaith's Theorem description of $\mathrm{MUP}$.} and graded as $\mathrm{MUP}\simeq \bigoplus_{n\in \mathbf Z}\Sigma^{2n}(\mathrm{MU})$. Since the operation of shearing over $\mathrm{MU}$ is symmetric monoidal, it induces a map on $\E$-algebra objects. Under the classical equivalence between graded modules and quasi-coherent sheaves over $\mathrm B\mathbf G_m$ via the Rees construction -- see \cite{Tasos's Rees} for a rigorous treatment in the $\i$-categorical context\footnote{Let us remark only that the Rees equivalence between $\mathrm{GrSp}\simeq \Fun(\mathbf Z, \Sp)$ and $\QCoh(\mathrm B\mathbf G_m)\simeq \mathrm{cMod}_{\mathbf S[\mathbf Z]}$ may also be seen as a particularly simple form of Koszul duality.} -- together with the relative spectrum identification
$$
\Spec_{\mathrm B\mathbf G_m} : \CAlg(\QCoh(\mathrm B\mathbf G_m)\simeq \mathrm{SpStk}^\mathrm{aff}_{/\mathrm B\mathbf G_m}
$$
between the quasi-coherent sheaves of $\E$-algebras over the classifying stack $\mathrm B\mathbf G_m$ and affine morphisms into it, we may express the shearing functor as an auto-equivalence
$$
(-)^{\shear}:\,\mathrm{SpStk}^\mathrm{aff}_{/\mathrm B\mathbf G_{m, \mathrm{MU}}}\to \mathrm{SpStk}^\mathrm{aff}_{/\mathrm B\mathbf G_{m, \mathrm{MU}}}
$$
of the $\i$-category of affine morphisms into $\mathrm B\mathbf G_{m, \mathrm{MU}}$.
\end{cons}

\begin{exun}\label{Example shear of MU}
Consider the universal $\mathbf G_m$-torsor $\Spec(\mathrm{MU})\to \mathrm B\mathbf G_{m, \mathrm{MU}}$.
In this  case, 
shearing may be identified as
$$
\Spec(\mathrm{MU})^{\shear}\,\simeq\, \Spec(\mathrm{MUP})/\mathbf G_{m}
$$
where the $\mathbf G_m$-action on the right-hand side is non-trivial and corresponds to the grading exhibited through the $\mathrm{MU}$-module equivalence $\mathrm{MUP}\simeq \bigoplus_{n\in \mathbf Z}\Sigma^{2n}(\mathrm{MU})$.
\end{exun}

\begin{remark}
The shearing functor of Construction \ref{Const of shearing} may be described explicitly as follows. Let $\mX\to \mathrm B\mathbf G_{m, \mathrm{MU}}$ be an affine morphism over the classifying stack of the multiplicative group over $\mathrm{MU}$. That means that the stack $\mX$ may be written as
$$
\mX\,\simeq\,\Spec_{\mathrm B\mathbf G_m}(\Gamma(\mX;\, \sO_{\mX}(*)))\,\simeq\, \Spec(\bigoplus_{n\in \mathbf Z}\Gamma(\mX;\, \sO_\mX(n)))/\mathbf G_m,
$$
where $\sO_\mX(n)$ denote the Serre twisting sheaves pulled back from $\mathrm B\mathbf G_m$. The shearing of $\mX$ is then given by
$$
\mX^{\shear}\,\simeq\,
\Spec_{\mathrm B\mathbf G_m}(\Gamma(\mX; \sO_{\mX}(*))^\shear)
\,\simeq\,\Spec\big( 
\bigoplus_{n\in \mathbf Z}\Sigma^{2n}\Gamma(\mX;\, \sO_\mX(n))
\big)/\mathbf G_m,
$$
with the $\E$-ring structure on the direct sum in question coming from symmetric monoidality of shearing over $\mathrm{MU}$.
$$
 (\mX^{\shear}\,)^\heart\,\simeq\, \Spec_{\mathrm B\mathbf G_m}(\mathrm H^*(\mX; \,\sO_{\mX}(*))/\mathbf G_m \,\simeq \, \Spec\big(\bigoplus_{n\in \mathbf Z}\mathrm H^n(\mX; \sO_{\mX}(n)) \big),
$$
where we are using the sheaf cohomology notation for the homotopy groups of global sections $\mathrm H^n(\mX;\, \sF):=\pi_{-n}(\Gamma(\mX; \, \sF))$ of a quasi-coherent sheaf $\sF$ on $\mX$.
\end{remark}

\begin{prop}\label{Evp for even affines over MU}
Let $A$ be an even $\E$-algebra over $\mathrm{MU}$.
 There is a canonical  and natural equivalence of spectral stacks
$$
\Spec(A)^\mathrm{evp}\,\simeq\, \Spec(A)^\shear,
$$
where the structure map to $\mathrm B\mathbf G_{m, \mathrm{MU}}$ on the right is taken to be 
$$
\Spec(A)\to \Spec(\mathrm{MU})\to \mathrm B\mathbf G_{m, \mathrm{MU}}.
$$
\end{prop}

\begin{proof}
Consider the canonical $\E$-ring map $A\to A\o_{\mathrm{MU}}\mathrm{MUP}$, where the periodic complex bordism spectrum $\mathrm{MUP}$ is equipped with the Thom spectrum $\E$-structure. Because the $\E$-ring map $S\to \mathrm{MUP}$ is eff, see the proof of Proposition \ref{Even filtration is even periodic filtration}, and the $\E$-ring $A\otimes_{\mathrm{MU}}\mathrm{MUP}$ is even periodic, the map $\Spec(A\o_{\mathrm{MU}}\mathrm{MUP})\to \Spec(A)$ is periodically eff in the sense of Definition \ref{Definition periodically eff}. By Proposition \ref{Computing the evening}, we therefore obtain the simplicial formula for the even periodization in question
$$
\Spec(A)^\mathrm{evp}\,\simeq\, \varinjlim_{\phantom{{}^\mathrm{op}}\mathbf\Delta^\mathrm{op}}\Spec(B^\bullet),
$$
with respect to the cosimplicial $\E$-ring is the cobar construction
$$
B^\bullet \,\simeq\, (A\o_{\mathrm{MU}}\mathrm{MUP})^{\o_A[\bullet]}\,\simeq\, A\otimes_{\mathrm{MU}}\mathrm{MUP}^\bullet
$$
and where we have set  $\mathrm{MUP}^\bullet :=\mathrm{MUP}^{\o_\mathrm{MU}[\bullet]}$. As explained in \cite[Lemma 2.1.3]{Sanath}, the periodic complex boridsm spectrum $\mathrm{MUP}$, with its Thom $\E$-algebra structure over $\mathrm{MU}$, may be identified with the Thom spectrum of the $\E$-space map $\mathbf Z\mapsto \mathrm{Pic}(\mathrm{MU})$, given by $n\mapsto \Sigma^{2n}(\mathrm{MU})$. As such, one of the basic properties of Thom spectra (the ``torsor triviality") implies that there is a canonical equivalence $\mathrm{MUP}\o_{\mathrm{MU}}\mathrm{MUP}\simeq \mathrm{MUP}[\mathbf Z]$, and by induction
$$\mathrm{MUP}^\bullet\simeq\, \mathrm{MUP}[\mathbf Z^\bullet].$$
Since we are discussing a Thom spectrum of an $\E$-space map, this is an equivalence of $\E$-algebras over $\mathrm{MU}$. By base-changing to $A$ and passing to spectra, this gives rise to an equivalence of simplicial affines
$$
\Spec(A\otimes_{\mathrm{MU}}\MUP^\bullet)\,\simeq \, \Spec(A\o_{\mathrm{MU}}\MUP)\times \mathbf G_m^{\times \bullet},
$$
since the multiplicative group is given by $\mathbf G_m =\Spec(S[\mathbf Z])$. The right-hand side exhibits a $\mathbf G_m$-action on $\Spec(A\otimes_{\mathrm{MU}}\mathrm{MUP})$, and by unwinding Construction \ref{Const of shearing}, we find that this action is base-changed along $\mathrm{MU}\to A$ one from Example \ref{Example shear of MU}, and it is in this sense that passing to the colimit in $\mathrm{SpStk}$ from the above equivalence of simplicial affines exhibits an identification
\begin{equation}\label{Evp of aff wrt MUP}
\Spec(A)^\mathrm{evp}\,\simeq \, \Spec(A\o_{\mathrm{MU}}\MUP)/\mathbf G_m.
\end{equation}
Under the Rees construction equivalence between $\mathbf G_m$-actions and gradings, the grading in question on $A\otimes_{\mathrm{MU}}\mathrm{MUP}$ exhibits it as the $A[t^{\pm 1}]^\shear$ for $t$ in grading degree $1$. It now follows from the definition of the shearing functor $\mathrm{sh}$ that
$$
\Spec(A)^\mathrm{evp}\simeq \Spec(A[t^{\pm 1}]^\shear)/\mathbf G_m\simeq (\Spec(A[t^{\pm 1}])/\mathbf G_m)^\shear\simeq \Spec(A)^\shear
$$
as promised.
\end{proof}

\begin{remark}
The structure map to $\mathrm B\mathbf G_m$ indicated in the statement of Proposition \ref{Evp for even affines over MU} amounts to identifying $\Spec(A)\simeq (\Spec(A)\times \mathbf G_m)/\mathbf G_m\simeq \Spec(A[\mathbf Z])/\mathbf G_m$. The shearing is therefore applied to the \textit{free $\Z$-grading} $A[\mathbf Z]\simeq \bigoplus_{n\in \mathbf Z}A$. This should not be confused with the \textit{trivial $\Z$-grading} $A=A(0)$, which has $A$ in grading-degree $0$ and the zero spectrum in any grading-degree $\ne 0$. Indeed, we should not expect shearing to return even periodic objects for non-free gradings.
\end{remark}

\begin{corollary}\label{Corollary shear on QCoh}
Let $A$ be an even $\E$-algebra over $\mathrm{MU}$.
The even periodization map $\Spec(A)^\mathrm{evp}\to \Spec(A)$ induces an equivalence of symmetric monoidal $\i$-categories  on quasi-coherent sheaves
$$
\QCoh(\Spec(A)^\mathrm{evp})\,\simeq\, \Mod_A.
$$
\end{corollary}

\begin{proof}
Recall that the shearing functor $M_*\mapsto M_*^\shear$ is a symmetric monoidal auto-equivalence of $\mathrm{GrMod}_\mathrm{MU}$, with an inverse functor given by the \textit{unshearing functor} $M_*\mapsto \Sigma^{-2*}(M_*)$. It therefore restricts to a symmetric monoidal equivalence between the $\i$-categories of modules
$$
\Mod_{A_*}(\mathrm{GrMod}_\mathrm{MU})\,\simeq \,\Mod_{A_*^\shear}(\mathrm{GrMod}_\mathrm{MU})
$$
for any graded $\E$-algebra $A_*$ over $\mathrm{MU}$. In terms of a relative affines $\mX\simeq \Spec_{\mathrm B\mathbf G_{m, \mathrm{MU}}}(A_*)$, this amounts to an symmetric monoidal equivalence of the $\i$-categories of quasi-coherent sheaves
$
\QCoh(\mX^\shear) \simeq \QCoh(\mX).
$
When applied to $\mX=\Spec(A)$ with the trivial map to the classifying stack $\mathrm B\mathbf G_{m, \mathrm{MU}}$, this comines with Proposition \ref{Evp for even affines over MU} to give an equivalence of the desired symmetric monoidal $\i$-categories. It remains to identify it with the pullback (or pushforward) along the canonical map $\Spec(A)^\mathrm{evp}\to\Spec(A)$.
Using \eqref{Evp of aff wrt MUP}, this map may be interpreted as
$$
\Spec(A)^\mathrm{evp}\,\simeq\, \Spec(A\o_{\mathrm{MU}}\MUP)/\mathbf G_m\to \Spec(A).
$$
Recall that pushforward along the terminal map of spectral stacks $\mathrm BG\to\Spec(\mathbf S)$ amounts to passing to $G$-fixed points for spectra with a $G$-action for any group $G$. Under the equivalence between $\mathbf G_m$-actions and gradings, we have $M_*^{\mathbf G_m}\simeq M_0$, i.e.\ the fixed-points extract the $0$-the graded piece. The pushforward along the even periodization map, viewed as a functor
$$
\QCoh(\Spec(A)^\mathrm{evp})\,\simeq\,\Mod_{A\o_\mathrm{MU}\MUP}(\mathrm{GrMod}_A)\to\Mod_A,
$$
is therefore the zero-th piece functor $M_*\mapsto M_0$. On the other hand, the equivalence
$$
\Mod_{A[t^{\pm 1}]}(\mathrm{GrMod}_\mathrm{MU})\,\simeq\,
\Mod_{A[t^{\pm 1}]}(\mathrm{GrMod}_A)\simeq\Mod_A
$$
is also given as $M_*\mapsto M_0$. Since shearing does not effect the $0$-th graded piece, the claim follows.
\end{proof}

\begin{remark}
Let $A$ be an even $\E$-ring over $\mathrm{MU}$.
As we saw in the proof of Corollary \ref{Corollary shear on QCoh}, we have canonical equivalences of $\E$-rings
$$
\sO(\Spec(A)^\mathrm{evp})\simeq \sO(\Spec(A\o_{\mathrm{MU}}\mathrm{MUP})/\mathbf G_m)\simeq (A\o_{\mathrm{MU}}\mathrm{MUP})^{\mathbf G_m}\simeq A\o_{\mathrm{MU}}\mathrm{MU}\simeq A.
$$
Since $\CAlg^\mathrm{ev}\subseteq\CAlg^\mathrm{evc}$, this agrees with the result of Proposition \ref{Prop champs affines} that the even priodization functor $A\mapsto \Spec(A)^\mathrm{evp}$ should be in this context undone by the passage to the $\E$-ring of global functions $\mX\mapsto\sO(\mX)$.
\end{remark}

From  Proposition \ref{Evp for even affines over MU} we may also obtain a formula for the even periodization of a not-necessarily-even itself, but only admits an eff cover in the sense of Definition \ref{Def eff} by an even affine.

\begin{corollary}\label{EVP over MU in the presence of an even eff cover}
Let $A\to B$ be an eff $\E$-algebra map over $\mathrm{MU}$ with $B$ even. Then
$$
\Spec(A)^\mathrm{evp}\,\simeq \,\varinjlim_{\phantom{{}^\mathrm{op}}\mathbf{\Delta}^{\mathrm{op}}} \Spec(B^\bullet)^\shear
$$
where $B^\bullet := B^{\otimes_A[\bullet]}$ is the cosimplicial cobar construction of $A\to B$.
\end{corollary}

\begin{proof}
Since $A\to B$ is eff, the map $A\to B\o_{\mathrm{MU}}\mathrm{MUP}$ is an eff cover by an even periodic $\E$-ring. 
The induced map of affines
$\Spec(B\o_{\mathrm{MU}}\mathrm{MUP})\to \Spec(A)$ is therefore periodically eff in the sense of Definition \ref{Definition periodically eff}, so we obtain a simplicial presentation for $\Spec(A)^\mathrm{evp}$ from Proposition \ref{Computing the evening}. The same arguments as in the proof of  Proposition \ref{Evp for even affines over MU} then re-express this presentation in the desired form.
\end{proof}

The discussion of this section \textit{would} admit
 an obvious globalization, applicable to all nonconnective spectral stacks over $\Spec(\mathrm{MU}),$ \textit{if we were able to} extend the shearing endofunctor $\mX\mapsto\mX^\shear$ of Construction \ref{Const of shearing} from relative affines to all of $\mathrm{SpStk}^\mathrm{aff}_{/\mathrm B\mathbf G_{m, \mathrm{MU}}}$. However, this is obstructed by shearing not necessarily preserving fpqc covers.

\begin{exun}
Consider the map $\mathbf Z\to \mathbf Z[t^{\pm 1}]$ of graded commutative $\Z$-algebras, where $t$ is in graded degree $1$. Then $\mathbf Z^\shear\simeq \mathbf Z$ whereas $\mathbf Z[t^{\pm 1}]^\shear\simeq \bigoplus_{n\in \mathbf Z}\Sigma^{2n}(\mathbf Z)$. Since the latter is neither connective nor truncated, it can not be faithfully flat over the classical ring $\mathbf Z$, whereas $\mathbf Z[t^{\pm 1}]$ is certainly a faithfully flat $\mathbf Z$-algebra.
\end{exun}

Nevertheless, the result on quasi-coherent sheaves of Corollary \ref{Corollary shear on QCoh} does admit a global extension in the expected way.

\begin{corollary}\label{Corollary unshear for modules}
Let $\mX\to\Spec(\mathrm{MU})$ be a map on nonconnective spectral stacks. Assume that $\mX$ is even, in the sense that it is may be written as a small colimit of even affines in the $\i$-category $\mathrm{SpStk}_{/\Spec(\mathrm{MU})}$. Then the even periodization map $f:\mX^\mathrm{evp}\to\mX$ induces a symmetric monoidal equivalence on the $\i$-categories of quasi-coherent sheaves
$$
f^*:\QCoh(\mX)\simeq\QCoh(\mX^\mathrm{evp}):f_*
$$
\end{corollary}

\begin{proof}
Since the  quasi-coherent sheaves and even periodization are both compatible with small colimits, this reduces to the affine case, which is precisely Corollary \ref{Corollary shear on QCoh}.
\end{proof}

\subsection{Variant:  even strong periodization}\label{Section: variant strict}
In discussing the even periodization, we have so far always been using a weak form of $2$-periodicity for even periodic $\E$-rings. We could have instead used the more standard stronger notion. We introduced the latter in Definition \ref{List of rings} under the name \textit{even strongly periodic $\E$-rings}, and denoted the $\i$-category they comprise by $\CAlg{}^{\mathrm{ev}\mathfrak{P}}$. Recall here that an even periodic $\E$-ring $A$ is \textit{strongly periodic} if and only if there exists a \textit{non-canonical} $\pi_0(A)$-linear isomorphism $\pi_2(A)\cong \pi_0(A)$.

Let us define even strongly periodic affines in the usual way as the opposite $\i$-category $\mathrm{Aff}{}^{\mathrm{ev}\mathfrak{P}}:=(\CAlg{}^{\mathrm{ev}\mathfrak{P}})^\mathrm{op}$. We next show that the even periodization is unaffected by whether we base it upon $\CAlg^\mathrm{evp}$ or if we were to use $\CAlg{}^{\mathrm{ev}\mathfrak{P}}$ instead.

\begin{lemma}\label{Trivialization lemma}
Let $A$ be an even $\E$-ring and $M$ an even $A$-module so that $\pi_0(M)$ is a projective $\pi_0(A)$-module of rank $r$. There exists a faithfully flat map of even $\E$-rings $A\to B$ so that $M\o_A B\simeq B^{\oplus r}$.
\end{lemma}

\begin{proof}
Due to evenness, $M$ is a flat $A$-module.
Recall from \cite[Proposition 7.2.2.16]{HA} that the base-change along the connective cover $\tau_{\ge 0}(A)\to A$ induces an equivalence between the $\i$-categories of flat modules $\Mod_{\tau_{\ge 0}(A)}^\flat\simeq \Mod_A^\flat.$ Under it, $M$ corresponds to the projective $\tau_{\ge 0}(A)$-module $\tau_{\ge 0}(M)$. There exists by \cite[Proposition 2.9.2.3]{SAG} a  Zariski cover -- which is therefore in particular faithfully flat -- $\tau_{\ge 0}(A)\to B'$ such that there is a $B'$-module equivalence $\tau_{\ge 0}(M)\otimes_{\tau_{\ge 0}(A)} B'\simeq B'^{\oplus r}$.
Set $B := A\o_{\tau_{\ge 0}(A)}B'$, so that the base-change of the module $M$ along the $\E$-ring map $A\to B$ may be determined as
\begin{eqnarray*}
M\o_A B&\simeq &(\tau_{\ge 0}(M)\o_{\tau_{\ge 0}(A)}A)\o_A (A\o_{\tau_{\ge 0}(A)} B')\\
&\simeq & (\tau_{\ge 0}(M)\o_{\tau_{\ge 0}(A)} B')\o_{\tau_{\ge 0}(A)} A\\
&\simeq & B'^{\oplus r}\o_{\tau_{\ge 0}(A)} A\simeq B^{\oplus r}.
\end{eqnarray*}
Since $\tau_{\ge 0}(A)\to B'$ is faithfully flat, and faithfully-flat maps are closed under base-change, so is $A\to B$. In particular, because $A$ is even periodic, $B$ must also be even periodic.
\end{proof}

\begin{prop}\label{Even stacks inside stacks 2}
The  subcategory inclusion $\CAlg{}^{\mathrm{ev}\mathfrak{P}}\subseteq\CAlg^\mathrm{evp}$ is accessible, preserves fpqc covers and has the cover lifting property with respect to it. It induces an equivalence of $\i$-categories upon accessible sheaves
$$
\mathrm{Shv}^\mathrm{acc}_\mathrm{fpqc}(\mathrm{Aff}{}^{\mathrm{ev}\mathfrak{P}})\simeq \mathrm{SpStk}^\mathrm{evp}.
$$
\end{prop}

\begin{proof} The accessibility assertion for the subcategory inclusion
$
\sigma:\CAlg{}^{\mathrm{ev}\mathfrak{P}}\hookrightarrow\CAlg^\mathrm{evp}
$
is clear.
Cover preservation is obvious from the definition of faithfully flat maps. For the cover lifting property, note that if $A\to B$ is a faithfully flat map in $\CAlg^\mathrm{evp}$ with $\mathrm{CAlg}{}^{\mathrm{ev}\mathfrak{P}}$, then $\pi_2(B)$ is the base-change of the free $\pi_0(A)$-module $\pi_2(A)$ under the commutative ring map $\pi_0(A)\to\pi_0(B)$, and is therefore a free $\pi_0(B)$-module, showing that $B$ must be strongly periodic itself.

It follows analogously to Construction \ref{Functoriality of ev} that the map of sites $\sigma$ induces an adjunction between the $\i$-categories of accessible sheaves
$$
\xymatrix{
 \mathrm{Shv}^\mathrm{acc}_\mathrm{fpqc}(\mathrm{Aff}^{\mathrm{ev}\mathfrak{P}}) \ar@<1.2ex>[rr]^{\,\,\quad\sigma_!} \ar@<-1.2ex>[rr]_{\,\,\quad\sigma_*}&  &
 \mathrm{SpStk}^\mathrm{evp}.\ar[ll]|-{\sigma^*\,}
}
$$
By imitating the argument of Proposition \ref{Even stacks inside stacks}, we find that $\sigma_!$ is fully faithful. We wish to show that it is also essentially surjective. Since $\varepsilon_!$ and $\varepsilon^*$ both commute with small colimits, and the $\i$-category $\mathrm{SpStk}$ is generated under colimits by affines, it suffices to show that the essential image of $\sigma_!$ contains $\Spec(A)$ for all even periodic $\E$-rings $A$.
Applying Lemma \ref{Trivialization lemma} to the $A$-module $\Sigma^2(A)$, we find that $A$ admits an fpqc cover $A\to B$ by an even strongly periodic $\E$-ring $B$. The corresponding  cobar construction $B^\bullet$ is a cosimplicial object in $\CAlg^{\mathrm{ev}\mathfrak P}$, and so the geometric realization $\mX:=\varinjlim_{\mathbf{\Delta}^\mathrm{op}}\Spec(B^\bullet)$, interpreted inside the $\i$-category  $\mathrm{Shc}^\mathrm{acc}_\mathrm{fpqc}(\mathrm{Aff}^{\mathrm{ev}\mathfrak{P}})$, satisfies $\sigma_!(\mX)\simeq \Spec(A)$ as desired.
\end{proof}

Our preference for having originally used $\CAlg^\mathrm{evp}$, as opposed to $\CAlg^{\mathrm{ev}\mathfrak P}$, in defining the even periodization stems from the property of being even periodic, unlike even strongly periodic, being fpqc-local among $\E$-rings.
To showcase the advantages that this affords, consider the following result, which distinguishes $\CAlg^\mathrm{evp}\subseteq\CAlg$ as the ``correct" affine site to use in discussing $\mathrm{SpStk}^\mathrm{evp}$.

\begin{prop}\label{EVP affines}
If an even periodic spectral stack $\mX$ is of the form $\mX\simeq\Spec(A)$ for some $\E$-ring $A$, then $A$ must be an even periodic $\E$-ring.
\end{prop}

\begin{proof}
Let $\mX$ be an even periodic spectral stack, which satisfies $\mX\simeq \Spec(A)$ for some $\E$-ring $A$.
Its connective localization is therefore given by $\mX^\mathrm{cn}\simeq \Spec(\tau_{\ge 0}(A))$ and is in particular also affine.
Thanks to affineness, the global sections functor
\begin{equation}\label{affineness win}
 \Gamma(\mX^\mathrm{cn};\, -):\QCoh(\mX^\mathrm{cn})\to \Mod_{\tau_{\ge 0}(A)}.
\end{equation}
is therefore an equivalence of $\i$-categories. From the latter being an equivalence, it follows that the quasi-coherent sheaf $\tau_{\ge n}(\sO_{\mX})$ is fully determined by 
$$
\Gamma(\mX^\mathrm{cn};\,\tau_{\ge n}(\sO_{\mX}))\simeq \varprojlim_{B\in\CAlg_A^\mathrm{evp}}\tau_{\ge n}(B)
$$
where the equivalence is a variant of the discussion in Remark \ref{Postnikov filtration on functions}, together with the fact that $\mX\,\simeq\, \varinjlim_{B\in \CAlg_A^\mathrm{evp}}\Spec(B)$. Since all the $\E$-rings $B$ above are even, the canonical maps $\tau_{\ge {2n+1}}(B)\to\tau_{\ge 2n}(B)$ are equivalences of $\tau_{\ge 0}(A)$-modules for all $n\in \mathbf Z$.  It follows that
$$
\Gamma(\mX^\mathrm{cn};\,\tau_{\ge 2n+1}(\sO_{\mX}))\to \Gamma( \mX^\mathrm{cn};\,\tau_{\ge 2n}(\sO_{\mX}))
$$
is an equivalence as well, and so by the equivalence of $\i$-categories \eqref{affineness win}, the corresponding map of quasi-coherent sheaves $\tau_{\ge 2n+1}(\sO_{\mX})\to \tau_{\ge 2n}(\sO_{\mX})$ is also an equivalence. The fiber of this map is given by $\Sigma^{2n+1}\pi_{2n+1}(\sO_{\mX})$, from which we conclude that $\pi_\mathrm{odd}(\sO_{\mX})=0$. The $\E$-ring $A$ is therefore even.
Since $\mX$ is a spectral stack over $\mM$, the $\E$-ring $A$ is also complex periodic, and hence even periodic.
\end{proof}

\subsection{Variant:  complex periodization}
Instead of replacing the subcategory inclusion $\CAlg^\mathrm{evp}\subseteq\CAlg$ by $\CAlg^{\mathrm{ev}\mathfrak P}\subseteq\CAlg$, as we did in the last section, we might instead consider replacing it by the subcategory $\CAlg^{\mathbf C\mathrm p}\subseteq\CAlg$ 
of complex periodic $\E$-rings in the sense of Definition \ref{List of rings}. That is to say, these are the $\E$-rings which are complex orientable and weakly $2$-perodic. Unlike last section, where replacing $\mathrm{CAlg}^\mathrm{evp}$ with $\mathrm{CAlg}^{\mathrm{ev}\mathfrak P}$ produced the same theory, we will see that using $\CAlg^{\mathbf C\mathrm p}$ instead does lead to different results.

As expected, we begin by setting the $\i$-category of \textit{complex periodic affines} to be the opposite $\i$-category $\mathrm{Aff}^{\mathbf C\mathrm p}:=(\CAlg^{\mathbf C\mathrm p})^\mathrm{op}$. We give in Proposition \ref{complex periodic stacks} a complex periodic analogue of Proposition \ref{Even stacks inside stacks}. 
Due to its central role in the proof, we first highlight  a standard basic property of complex periodic $\E$-rings, e.g.\,{\cite[Remark 4.1.10]{Elliptic 2}},  which we have  already used several times above, for instance in proving Proposition \ref{EVP affines}.

\begin{lemma}\label{Key property}
Let $A$ be a complex periodic $\E$-ring. The inclusion
$\CAlg_A^{\mathbf C\mathrm p}\subseteq\CAlg_A$ is an equivalence of $\i$-categories.
\end{lemma}

\begin{proof}
Let $A\to B$ be a map of $\E$-rings. We must show that $B$ is complex periodic. 
One way to see that $B$ is complex orientable is to recall
 recall that a choice of a complex orientation on $A$ is equivalent to an map $\mathrm{MU}\to A$ in $\mathrm{CAlg}(\mathrm{hSp})$. Composing this map with $A\to B$ thus exhibits a complex orientation on $B$ as well.
For weak $2$-periodicity, recall from Remark \ref{Remark weak vs strong 2-periodicity remark} that it is equivalent to the $A$-module $\Sigma^2(A)$ being invertible. This makes it clear that $\Sigma^2(B)\simeq \Sigma^2(A)\otimes_A B$ is invertible as a $B$-module, and thus that $B$ is weakly $2$-periodic.
\end{proof}

\begin{lemma}\label{Lemma 4.7.2}
The  subcategory inclusion $\CAlg^{\mathbf C\mathrm p}\subseteq\CAlg$ is accessible, preserves fpqc covers and has the cover lifting property with respect to it.
\end{lemma}

\begin{proof}
That the inclusion in question preserves fpqc covers is definitional, since we define the fpqc topology on complex periodic $\E$-rings by restriction from its definition on all $\E$-rings. For the cover lifting property, it suffices to observe that thanks to Lemma \ref{Key property}, and flat $\E$-algebra over a complex periodic $\E$-ring must also be complex periodic itself.

To see that the $\i$-category $\CAlg^{\mathbf C\mathrm p}$ is accessible, and that its inclusion into $\CAlg$ is an accessible functor, observe that the latter features as the left  vertical arrow in the homotopy Cartesian square of $\i$-categories
$$
\begin{tikzcd}
\CAlg^{\mathbf C\mathrm p} \arrow{d}{}\arrow{r} & *\arrow{d}{}\\
\CAlg \arrow{r}{\mM} & \mathrm{Ani}.
\end{tikzcd}
$$
Here $*$ is the terminal $\i$-category, and the right vertical arrow corresponds to the inclusion of the terminal anima $*$ into the $\i$-category of anima (though $\mathrm{Ani}$ may also be replaced in the above pullback square, and hence in this argument, by its full subcategory spanned by the initial and terminal animas $\{\emptyset, *\}$), and we have used the
Definition \ref{Definition of mM} of the chromatic base stack $\mM$. On the other hand, the latter is a spectral stack (as follows for instance from the explicit simplicity description in Theorem \ref{On M}) and so in particular an accessible functor. Accessibility follows by \cite[Proposition 5.4.6.6]{HTT}.
\end{proof}

\begin{prop}\label{complex periodic stacks}
The   inclusion $\CAlg^{\mathbf C\mathrm p}\subseteq\CAlg$  induces an equivalence of $\i$-categories 
$$
\mathrm{Shv}^\mathrm{acc}_\mathrm{fpqc}(\mathrm{Aff}{}^{\mathbf C\mathrm p })\simeq \mathrm{SpStk}_{/\mM}.
$$
Additionally, the following statements hold:
\begin{enumerate}[label = (\alph*)]
\item The forgetful functor $\mathrm{SpStk}_{/\mM}\to \mathrm{SpStk}$ is fully faithful, and its essential image is generated under small colimits by the affine spectral stacks  $\Spec(A)$, where $A$ ranges over complex periodic $\E$-rings.\label{Corollary 100, b}
\item The left adjoint $\mX\mapsto \mX\times\mM$ to the inclusion $\mathrm{SpStk}_{/\mM}\subseteq\mathrm{SpStk}$  preserves all small limits and colimits in $\mathrm{SpStk}$.\label{Corollary 100, a}
\item A nonconnective spectral stack $\mX$ belongs to $\mathrm{SpStk}_{/\mM}$ if and only if the projection map $\mX\times \mM\to\mX$ is an equivalence.\label{Corollary 100, c}
\item The canonical map $(\mX\times \mM)\times \mM\to \mX\times \mM$ is an equivalence for any $\mX\in\mathrm{SpStk}$.\label{Corollary 100, d}
\end{enumerate}
\end{prop}

\begin{proof}
It follows from Lemma \ref{Lemma 4.7.2} that the inclusion $\zeta:\Aff{}^{\mathbf C\mathrm p}\hookrightarrow\Aff$  induces a double adjunction between the $\i$-categories of accessible sheaves
$$
\xymatrix{
 \mathrm{Shv}^\mathrm{acc}_\mathrm{fpqc}(\mathrm{Aff}^{\mathbf C\mathrm p}) \ar@<1.2ex>[rr]^{\,\,\quad\zeta_!} \ar@<-1.2ex>[rr]_{\,\,\quad\zeta_*}&  &
 \mathrm{SpStk}.\ar[ll]|-{\zeta^*\,\,}
}
$$
Consider the leftmost adjoint $\zeta_!$. Its value on a accessible sheaf $\mX$ on $\mathrm{Aff}^{\mathbf C\mathrm p}$ is obtained as
to the sheafification of the functor
$$
\CAlg\,\ni\,A\mapsto \varinjlim_{B\in\CAlg^{\mathbf C\mathrm p}_{/A}}\mX(B)\,\in \,\mathrm{Ani}.
$$
By Lemma \ref{Key property}, the overcategory $\CAlg^{\mathbf C\mathrm p}_{/A}$ is either empty if the $\E$-ring $A$ is not complex periodic, otherwise it has $A$ as a final object. Observing that fpqc descent is automatically satisfied, we find the left adjoint $\zeta_!$ of the above adjunction given for any $\mX\in \mathrm{Shv}^\mathrm{acc}_\mathrm{fpqc}(\mathrm{Aff}^{\mathbf C\mathrm p})$ and $A\in\CAlg$ by the explicit formula
$$
\zeta_!\mX(A)\simeq \begin{cases}
\mX(A) &\text{ if $A$ is complex periodic,}\\ \emptyset & \text{ otherwise}.
\end{cases}
$$
We see immediately that the unit natural transformation $\mX\to \zeta^*\zeta_!\mX$ is an equivalence for all $\mX\in\mathrm{Shv}^{\mathrm{acc}}_{\mathrm{fpqc}}(\mathrm{Aff}^{\mathbf C\mathrm p})$. This implies that the left adjoint $\zeta_!$ is fully faithful. Since limits of spectral stacks are computed object-wise, we also see that $\zeta_!$ commutes with all small limits. Since it already commutes with all small colimits by virtue of being a left adjoint, the Barr-Beck-Lurie Theorem implies that the adjunction $\zeta_!\dashv\zeta^*$ is comonadic. To recognize the comonad, 
observe that it follows directly from the above explicit description of $\zeta_!$, together with the Definition \ref{Definition of mM} of the chromatic base stack $\mM$, that we have for any nonconnective spectral stack $\mX$ a canonical identification
\begin{equation}\label{cp localization}
\zeta_!\zeta^*\mX\,\simeq\,\mX\times \mM.
\end{equation}
Here recall that by abstract nonsense, any object $C\in \mC$ in a Cartesian symmetric monoidal $\i$-category $\mC$ admits a canonical $\E$-coalgebra stucture, and that with respect to it we have $\mathrm{cMod}_C(\mC)\simeq \mC_{/C}$. Applying this to the case $\mC=\mathrm{SpStk}$ and $C=\mM$, together with the obvious observation that $\mM\times \mM\simeq \mM$, finishes the proof.
\end{proof}

Comparing with the Definition \ref{Def of evp localization} of the even periodization $\mX\mapsto \mX^\mathrm{evp}$, we see from \eqref{cp localization} that the functor $\mX\mapsto \mX\times\mM$ gives what we might call \textit{complex periodization}. 

\begin{remark}
The limit preservation claim \ref{Corollary 100, a} of Proposition \ref{complex periodic stacks} highlights a particular way in which complex periodization is formally better behaved than even periodization, see Remark \ref{Evp and limits}.
\end{remark}

The inclusions $\CAlg^\mathrm{evp}\subseteq\CAlg^{\mathbf C\mathrm p}\subseteq\CAlg$ induce for any $\mX\in \mathrm{SpStk}$ canonical maps
$$
\mX^\mathrm{evp}\to \mX\times \mM\to\mX,
$$
where the second map is the projection onto the first factor. The first map in the above factorization furthermore induces an equivalence
\begin{equation}\label{Novikov descent via stacks}
\mX^\mathrm{evp}\simeq (\mX\times \mM)^\mathrm{evp}.
\end{equation}

\begin{remark}
In light of the above discussion, the process of  even periodization can be subdivided into two steps. The first is to pass the complex periodization, i.e.\ taking the product with $\mM$. The second is then to apply even the periodization internally to complex periodic geometry. The upshot is that in the latter setting, it has roughly the effect of killing $\pi_1$ of the structure sheaf.
\end{remark}

We showed
in Proposition \ref{Voila the even filtration} how the Postnikov filtration for the structure sheaf on $\Spec(A)^\mathrm{evp}$ recovers the even filtration. The following is an analogue for the complex periodization $\Spec(A)\times \mM$ and the Adams-Novikov filtration, for more about which see Subsection \ref{ANsection}.

\begin{prop}\label{Voila ANS}
For any $\E$-ring $A$, there is a natural equivalence of filtered $\E$-ring
$$\Gamma\big( (\Spec(A)\times \mM)^\mathrm{cn};\, \tau_{\ge 2*}(\sO_{ \Spec(A)\times \mM})\big)\,\simeq \, \mathrm{fil}^{\ge *}_\mathrm{AN}(A)$$ with  the Adams-Novikov filtration on $A$.
\end{prop}

\begin{proof}
Using simplicial presentation for the chromatic base stack $\mM$ from Theorem \ref{On M}, we can express complex periodization as
$$
\Spec(A)\times \mM\,\simeq\, \varinjlim_{\,\,\,\,\mathbf\Delta^\mathrm{op}}\Spec(A\o_\mathbf S\mathrm{MUP}^\bullet)
$$
As explained in Remark \ref{Postnikov filtration on functions},
this simplicial presentation 
induces a canonical  equivalence of filtered $\E$-rings
$$
\Gamma\big((\Spec(A)\times \mM)^\mathrm{cn};\,\tau_{\ge 2*}(\sO_{ \Spec(A)\times \mM})\big)\,\simeq\, \varprojlim_{\mathbf\Delta} \tau_{\ge 2*}(A\otimes_\mathbf S \mathrm{MUP}^\bullet)\,\simeq\, \mathrm{fil}_\mathrm{AN}^{\ge *}(A),
$$
where the final identification is by Corollary \ref{AN via MUP}
\end{proof}

\begin{remark}\label{Novikov remark}
Let $\mX$ be any nonconnective spectral stack.
Similarly to the proof of Proposition \ref{Voila ANS}, the canonical equivalence \eqref{Novikov descent via stacks} may be rewritten as
\begin{equation}\label{Novikov presentation}
\mX^\mathrm{evp}\simeq \varinjlim_{\,\,\,\,\mathbf\Delta^\mathrm{op}}\big(\mX\times \Spec(\mathrm{MUP}^\bullet)\big)^\mathrm{evp}
\end{equation}
We consider this to be the  spectral stack incarnation of the \textit{Novikov descent for the even filtration} \cite[Theorem 1.1.5, Corollary 2.2.21]{HRW}. 
Indeed, the filtration-level result in question follows by examining the Postnikov filtration $\tau_{\ge 2*}(\sO)$  by Proposition \ref{Voila the even filtration}.
\end{remark}

\begin{prop}\label{EVPL vs CPL}
Let $\mX=\Spec(A)$ be an affine, and assume that the smash product $A\o_\mathbf S\mathrm{MU}$, or equivalently $A\o_\mathbf S\MUP$, is even. Then the canonical map  $\mX^\mathrm{evp}\to\mX\times\mM$ is an equivalence of spectral stacks.
\end{prop}

\begin{proof}
Since the periodic complex bordism can be expressed, on the level of underlying spectra, in terms of its non-periodic variant as $\mathrm{MUP}\simeq \bigoplus_{n\in\mathbf Z}\Sigma^{2n}(\mathrm{MU})$, it is clear that the $\E$-ring $A\o_\mathbf S\MU$ is even if and only if $A\o_\mathbf S\MUP$ is. In this case, formula \eqref{Novikov presentation} for the even periodization simplifies to
$$
\mX^\mathrm{evp}\,\simeq \,\varinjlim_{\,\,\,\,\mathbf\Delta^\mathrm{op}}\mX\times \Spec(\mathrm{MUP}^\bullet)\,\simeq\, \mX\times \mM,
$$
where the second equivalence is the simplicial formula for the chromatic base stack $\mM$ from Theorem \ref{On M}.
\end{proof}

\begin{exun}
Proposition \ref{EVPL vs CPL} allows us to determine many complex periodizations, by reducing to the computation of $\mathrm{MU}$-homology. For instance, we find that the canonical maps garnish equivalences
$$
\Spec(\mathbf Z)^\mathrm{evp}\,\simeq\, \Spec(\mathbf Z)\times \mM, \quad \Spec(\mathbf F_2)^\mathrm{evp}\,\simeq\, \Spec(\mathbf F_2)\times \mM
$$
as consequence of the classical computations of the homology $\mathrm H_*(\mathrm{MU}; \mathbf Z)$ and $\mathrm H_*(\mathrm{MU}; \mathbf F_2)$, and in particular, their evenness. The same kind of argument also works for the $\E$-rings $\mathrm{ku}$, $\mathrm{ko}$ and $\mathrm{tmf}$, based on the well-known computations of their $\mathrm{MU}$-homology.
\end{exun}

\begin{corollary}
Let $\mX\to \Spec(\mathrm{MU})$ be a map of nonconnective spectral stacks with $\mX$ even.
The canonical map $\mX^\mathrm{evp}\to \mX\times \mM$ is an equivalence of nonconnective spectral stacks.
\end{corollary}

The comparison between even periodization and shearing over $\mathrm{MU}$ from Subsection \ref{Subsection shearing} also admits a variant for complex periodization. Unlike the former comparison, which required an evenness assumption, the complex periodization analogue of Proposition
\ref{Evp for even affines over MU} holds for any affine over the complex bordism spectrum.

\begin{prop}\label{xM of aff wrt MUP}
Let $A$ be an $\E$-algebra over $\mathrm{MU}$.
 There is a canonical  and natural equivalence of spectral stacks
$$
\Spec(A)\times \mM\,\simeq\, \Spec(A)^\shear,
$$
where the structure map to $\mathrm B\mathbf G_{m, \mathrm{MU}}$ on the right is taken to be 
$$
\Spec(A)\to \Spec(\mathrm{MU})\to \mathrm B\mathbf G_{m, \mathrm{MU}}.
$$
\end{prop}

\begin{proof}
The proof of Proposition
\ref{Evp for even affines over MU} goes through in this context with no modifications, and no requirements of evenness. Alternatively, we can argue directly: 
\end{proof}

\begin{remark}
By formula \eqref{Evp of aff wrt MUP} from the proof of  Proposition
\ref{Evp for even affines over MU}, we find that the shear of an $\E$-algebra $A$ over $\mathrm{MU}$ may be described in terms of the Thom $\E$-structure on the periodic complex bordism $\mathrm{MUP}$ as the quotient stack
$$
\Spec(A)^\shear\, \simeq \Spec(A\o_{\mathrm{MU}}\mathrm{MUP})/\mathbf G_m.
$$
The $\mathbf G_m$-action on the right-hand-side corresponds to the grading $\mathrm{MUP}\simeq \bigoplus_{n\in \mathbf Z} \Sigma^{2n}(\mathrm{MU})$. Note that this may be rewritten as the fibered product
$$
\Spec(A)\times_{\Spec(\mathrm{MU})}\Spec(\mathrm{MUP})/\mathbf G_m,
$$
and so the equivalence of Proposition \ref{xM of aff wrt MUP} reduces to the fact that
$$
\Spec(\mathrm{MU})\times \mM\,\simeq \,\Spec(\mathrm{MUP})/\mathbf G_m.
$$
\end{remark}

Just as in
Corollary \ref{Corollary unshear for modules}, the existence of unshearing allows us to conclude that complex periodization over $\Spec(\mathrm{MU})$ does not effect quasi-coherent sheaves.

\begin{corollary}\label{Corollary 4.53 qcoh}
Let $\mX\to \Spec(\mathrm{MU})$ be a map of nonconnective spectral stacks. Then the complex periodization map $f:\mX\times \mM\to \mX$ induces a symmetric monoidal adjoint equivalence of $\i$-categories on quasi-coherent sheaves
$$
f^*:\QCoh(\mX)\simeq\QCoh(\mX\times \mM):f_*.
$$
\end{corollary} 

We can also give a precise description of the essential image of the complex periodization functor $A\mapsto \Spec(A)\times \mM$ of affines over $\mathrm{MU}$. In fact, since is is given in this context by shearing, it also admits an inverse in unshearing, and therefore gives rise to an equivalence of $\i$-categories.

\begin{prop}\label{Prop ess im of shear}
The functor $A\mapsto \Spec(A)\times \mM$ induces an equivalence of $\i$-categories
$$
\CAlg_{\mathrm{MU}}^\mathrm{op}\,\simeq \, \mathrm{SpStk}^\mathrm{aff}_{/\mathrm{Spec}(\mathrm{MU})^\shear}.
$$
\end{prop}

\begin{proof}
For the canonical map $f:\Spec(\mathrm{MU})^\shear\simeq \Spec(\mathrm{MU})\times \mM\to \Spec(\mathrm{MU})$, consider the commutative diagram of nonconnective spectral stacks
$$
\begin{tikzcd}
\mathrm{SpStk}^\mathrm{aff}_{/\mathrm{Spec}(\mathrm{MU})} \arrow{d}{\Spec_{\mathrm{MU}}} \arrow{r}{-\times {\mM}} & \mathrm{SpStk}_{/\mathrm{Spec}(\mathrm{MU})^\shear}^\mathrm{aff}\arrow{d}{\Spec_{\Spec(\mathrm{MU})^\shear}} \\
\CAlg(\QCoh(\mathrm{Spec}(\mathrm{MU})))^\mathrm{op} \arrow{r}{f^*} & \CAlg(\QCoh(\mathrm{Spec}(\mathrm{MU})^\shear))^\mathrm{op}.
\end{tikzcd}
$$
The vertical maps are equivalences by the definition of the relative Spec functor, see \eqref{Relative spec def}, and the lower horizontal map is an equivalence  by Corollary \ref{Corollary 4.53 qcoh}. It follows that the upper horizontal map is an equivalence as well.
\end{proof}

\begin{corollary}
The even periodization functor $\mX\mapsto \mX^\mathrm{evp}$ restricts to an equivalence of $\i$-categories
$$
(\CAlg^\mathrm{ev}_\mathrm{MU})^\mathrm{op}\,\simeq \, \mathrm{SpStk}^{\mathrm{evp},\, \mathrm{aff}}_{/\mathrm{Spec}(\mathrm{MU})^\shear}
$$
\end{corollary}

\begin{proof}
Since shearing and even periodization agree in this case, and both it and its inverse preserve evenness -- for instance by Proposition \ref{xM of aff wrt MUP}, through which the inverse functor is identified with unshearing, -- the desired claim follows from Proposition \ref{Prop ess im of shear}.
\end{proof}

There is also a version of the ``affine stacks"-like behavior, by which we mean the fully faithfulness result of Proposition \ref{Prop champs affines}, in the complex periodic setting. Namely, set
$\CAlg^{\mathbf C\mathrm {c}}\subseteq\CAlg$ to be the full subcategory of \textit{complex convergent} $\E$-rings, in the sense that their Adams-Novikov filtration converges, i.e.\ the $\E$-rings $A$ for which the canonical map $A\to\varinjlim_n \mathrm{fil}^{\ge n}_\mathrm{AN}(A)$ is an equivalence. Obviously we have $\CAlg^{\mathbf C\mathrm{or}}\subseteq\CAlg^{\mathbf C\mathrm{c}}$, but in fact, this subcategory is much larger.

\begin{lemma}\label{Lemma connective = convergent}
The subcategory $\CAlg^{\mathbf C\mathrm{c}}\subseteq\CAlg$ contains all those $\E$-ring spectra whose underlying spectrum is almost connective, i.e.\ bounded below.
\end{lemma}

\begin{proof}
By the perspective on the Adams-Novikov filtration presented in Definition \ref{The party line on ANf}, we have
$$
\varinjlim_n \mathrm{fil}^{\ge n}_\mathrm{AN}(A)\, \simeq \, \varinjlim_n \varprojlim_{\mathbf\Delta} \tau_{\ge 2n}(A\o_{\mathbf S}\mathrm{MU}^\bullet).
$$
We can now employ a truncatedness argument as in \cite[Remark 2.3.1]{HRW}: since the map $\tau_{\ge 2n}(A\o_{\mathbf S}\mathrm{MUP}^\bullet)\to A\o_{\mathbf S}\mathrm{MUP}^\bullet$ is $(2n-3)$-truncated, and truncatedness is preserved under limits (indeed, the inclusion $\mathrm{Sp}_{\le n}\subseteq\mathrm{Sp}$ admits a left adjoint in $\tau_{\le n}$, and as such preserves limits), it follows by sending $n\to-\infty$ that the canonical map
$$
\varinjlim_n\varprojlim_{\mathbf\Delta}\tau_{\ge 2n}(A\o_{\mathbf S}\mathrm{MU}^\bullet)\to \varprojlim_{\mathbf\Delta}A\o_{\mathbf S}\mathrm{MU}^\bullet
$$
 is  an equivalence. What we are trying to show is therefore that the map of $A$-modules
 \begin{equation}\label{Map to MU-nilpotent completion}
A\to \varprojlim_{\mathbf\Delta}A\o_{\mathbf S}\mathrm{MU}^\bullet
 \end{equation}
 is an equivalence for any almost connective $\E$-ring $A$. Since both sides commute with suspension, we may assume without loss of generality that $A$ is connective.
 In the language of \cite{Bousfield}, the right-hand side of \eqref{Map to MU-nilpotent completion} is the \textit{$\mathrm{MU}$-nilpotent completion of $A$}, which is identified by \footnote{Some equivalent versions of this result from different perspectives appear  in \cite[Corollary 5]{Torii} and \cite[Theorem 6.1]{Carelsson}.} \cite[Theorem 6.5]{Bousfield} with the Bousfield localization $L_{\mathrm{MU}}(A)\simeq L_{\mathbf S}(A)\simeq A$, thanks to the connectiveness assumption.\end{proof}

\begin{prop}\label{Prop champs affines over M}
The functor $(\CAlg^{\mathbf C\mathrm {pc}})^\mathrm{op}\to\mathrm{SpStk}_{\/\mM}$, given by $A\mapsto \Spec(A)\times \mM$,
is fully faithful. The $\E$-ring of global sections functor $\mX\mapsto \sO(\mX)$ is its left inverse.
\end{prop}

\begin{proof}
By the simplicial presentation for the chromatic base stack $\mM$ from
Theorem \ref{On M}, we have
$$
\Spec(A)\times \mM \,\simeq \,\varinjlim_{\phantom{{}^\mathrm{op}}\mathbf\Delta^\mathrm{op}}\Spec(A\o_{\mathbf S}\mathrm{MUP}^\bullet)
$$
and so we have for any pair of $\E$-rings $A$ and $B$ natural homotopy equivalences
\begin{eqnarray*}
\Map_{\mathrm{SpStk}_{/\mM}}(\Spec(A)\times \mM, \,\Spec(B)\times \mM) &\simeq&\Map_\mathrm{SpStk}(\Spec(A)\times \mM,\, \Spec(B))\\
&\simeq & \Map_{\CAlg}(B,\, \mathcal O(\Spec(A)\times \mM)) \\
&\simeq & \Map_{\CAlg}(B, \, \varprojlim_{\mathbf\Delta}A\otimes_{\mathbf S}\mathrm{MUP}^\bullet).
\end{eqnarray*}
To prove the result, we must therefore show that the canonical map
$$
A\to \varprojlim_{\mathbf\Delta}A\otimes_{\mathbf S}\mathrm{MUP}^\bullet
$$
is an equivalence for every complex convergent $\E$-ring $A$. The same connectivity argument as we have given above in
\ref{Lemma connective = convergent}, but replacing $\mathrm{MU}$ with $\mathrm{MUP}$, combined with Corollary \ref{AN via MUP}, shows that
$$
\varprojlim_{\mathbf\Delta}A\otimes_{\mathbf S}\mathrm{MUP}^\bullet\, \simeq \, \varinjlim_n\mathrm{fil}_{\mathrm{AN}}^{\ge n}(A),
$$
from which the desired claim follows.
\end{proof}

The argument given above to prove Lemma \ref{Lemma connective = convergent} did not actually make any use of the ring structure. Together with the kind of argument we used in the proof of Proposition \ref{Prop champs affines over M}, this can be used to show:

\begin{prop}\label{On the unfortunately-named IndCoh}
Let $p:\mM\to \Spec(\mathbf S)$ be
the terminal map of spectral stacks. The quasi-coherent pullback functor $p^* :\mathrm{Sp}\to\QCoh(\mM)$ is fully faithful when restricted to the full subcategory of almost connective spectra $\mathrm{Sp}^\mathrm{acn}\subseteq\mathrm{Sp}$, with the left inverse given by the global sections functor $p_*\simeq \Gamma(\mM; -)$.
\end{prop}

\begin{remark}
In \cite{ChromaticCartoon}, we used $\mathrm{IndCoh}(\mM)$ to denote the $\i$-category obtained by ind-completion from the thick subcategory spanned in $\QCoh(\mM)$ by the structure sheaf $\mO_{\mM}$. With this definition, it can be shown that $\mathrm{IndCoh}(\mM)\simeq \mathrm{Sp}$, and Proposition \ref{On the unfortunately-named IndCoh} show that the canonical functor
$$
\mathrm{IndCoh}(\mM)\to\QCoh(\mM)
$$
agrees on the bounded-above part of some canonical $t$-structures on both sides, and exhibits the left-hand side as the right completion of $\QCoh(\mM)$. Admittedly, this is not the \textit{quite same as} the relationship between $\QCoh$ and $\mathrm{IndCoh}$ as those are understood in and around the geometric Langlands program, e.g.\ \cite{Gaitsgory}, \cite{GaRo}, \cite{Harold},  where the latter is the left anticompletion of the former (at least for connective geometric stacks). But it is at least similar in that it is a kind of renormalization of the $t$-structure, affecting only the ``behavior at $-\infty$". Let us remark also that this sort of $t$-structure renormalization is also analogous to the one occurring in the functor $\QCoh^!$, which appears on the module-level Koszul duality in \cite[Part IV]{SAG}.
\end{remark}

\begin{remark}
To conclude this subsection, we suggest a perspective on the relationship between even periodic and complex periodic spectral stacks. A spectral stack $\mX$ possesses not only an underlying classical stack $\mX^\heart$, but in fact an \textit{underlying Dirac stack} $\mX^\mathrm{Dir}$ in the sense of \cite{Dirac II}. \textit{Dirac geometry} in this sense is roughly a $\mathbf Z$-graded analogue of the mathematical physicist's \textit{supergeometry}, where the grading is always by $\mathbf Z/2$. From this perspective, a spectral stack being complex periodic implies that its underlying Dirac stack is actually a superstack, in that the $\mathbf Z$-gradings in question are 2-periodic\footnote{We are purposefully blurring the line between strong and weak 2-periodicity here. But as we saw in the last subsection, if done correctly, that tends to not affect the resulting stacks.}, and are as such fully determined by $\mathbf Z/2$-graded variants. A complex periodic spectral stack $\mX$ is then \textit{even periodic} if the underlying superstack is in fact even, i.e.\ equal to $\mX^\heart$. The author has once heard Lars Hesselholt assert that ``evenness seems to mean coming from geometry." By the preceding discussion, this is literally true if interpreted in the context of complex periodic stacks and in the sense of ``being underlied by geometry, \textit{as opposed to} supergeomerty." From this perspective,  perhaps Proposition \ref{Prop champs affines}, which means that many object, naturally underlied by Dirac stacks, can be fully faithfully encoded by objects which are naturally underlied by even superstacks,  hints at some kind of supersymmetry phenomena.
\end{remark}

\subsection{Even periodization of topological modular forms}\label{Section TMF}
The spectrum of \textit{topological modular forms} is defined as the $\E$-ring of global functions 
$$
\mathrm{TMF}:=\sO(\mM{}_\mathrm{Ell}^\mathrm{or})
$$
on an even periodic spectral stack $\mM{}_\mathrm{Ell}^\mathrm{or}$, the \textit{moduli stack of oriented elliptic curves}; see \cite{Elliptic 2} for a modern account. In this section, we show that the process of even periodization reverses this construction, in the following sense:

\begin{theorem}\label{Theorem Evp for TMF}
There is a canonical equivalence of nonconnective spectral stacks
$$\Spec(\mathrm{TMF})^\mathrm{evp}\,\simeq\, \mM{}_\mathrm{Ell}^\mathrm{or}.$$
\end{theorem}

The proof will require some preliminary work.

\begin{lemma}\label{Lemma of the heart for TMF}
The canonical map $\mM{}_\mathrm{Ell}^\mathrm{or}\to \mM$ is an affine morphism.
\end{lemma}

\begin{proof}
According to Theorem \ref{On M},
the chromatic base stack $\mM$ is geometric, with an fpqc cover $\Spec(\mathrm{MUP})\to \mM$ obtained by fixing an $\E$-ring structure on $\mathrm{MUP}$. Since affineness is an fpqc-local property, it suffices to show that the pullback $\mM{}_\mathrm{Ell}^\mathrm{or}\times_{\mM}\Spec(\MUP)$ is an affine nonconnective spectral scheme. 
 Since maps to the chromatic base stack are essentially unique if they exist, we have for any $\mX, \mY\in \mathrm{SpStk}_{/\mM}$ that the canonical map
 $$
\mX\times_{\mM}\mY\to\mX\times \mY
 $$
 is an equivalence of nonconnective spectral stacks, for instance by Proposition \ref{complex periodic stacks}. In particular, we have
 $$
 \mM{}_\mathrm{Ell}^\mathrm{or}\times_{\mM}\Spec(\MUP)\,\simeq\, \mM_\mathrm{Ell}^\mathrm{or}\times \Spec(\MUP).
 $$
Nonconnective spectral Deligne-Mumford stacks are closed under products in the presheaf $\i$-category $\mP(\mathrm{Aff})$   via the functor of points embedding. Since the inclusion $\mathrm{SpStk}\subseteq\mP(\mathrm{Aff})$ preserves products, it follows that nonconnective spectral Deligne-Mumford stacks are also closed under products in $\mathrm{SpStk}$. The pullback under consideration $\mM{}_\mathrm{Ell}^\mathrm{or}\times_{\mM}\Spec(\MUP)$ may consequently be identified with a nonconnective spectral Deligne-Mumford stack. By \cite[Corollary 1.4.7.3]{SAG}, a nonconnective spectral Deligne-Mumford stack $\mX$ is affine if and only if its underlying classical Deligne-Mumford stack $\mX^\heart$ is. We are therefore reduced to showing that the underlying classical stack $(\mathcal M_\mathrm{Ell}^\mathrm{or}\times_{\mM}\Spec(\MUP))^\heart$ is affine.

To that end, we first determine its connective localization.
Since $\Spec(\mathrm{MUP})\to \mM$ is an fpqc cover, and the connective localization may also be computed fpqc locally as discussed in Example \ref{connective localization for geometric stacks}, we obtain a canonical equivalence
$$
(\mM{}_\mathrm{Ell}^\mathrm{or}\times_{\mM}\Spec(\MUP))^\mathrm{cn} \,\simeq \,(\mM{}_\mathrm{Ell}^\mathrm{or})^\mathrm{cn}\times_{\mM^\mathrm{cn}}\Spec(\tau_{\ge 0}(\MUP)).
$$
These are all \textit{connective} spectral stacks, so the passage to the underlying classical stacks $\mX\mapsto \mX^\heart$ preserves limits by Remark \ref{Remark issues of the heart}. We thus obtain the underlying classical stack as the pullback
$$
(\mM{}_\mathrm{Ell}^\mathrm{or}\times_{\mM}\Spec(\MUP))^\heart\, \simeq \,(\mM{}_\mathrm{Ell}^\mathrm{or})^\heart\times_{\mM^\heart}\Spec(\pi_{ 0}(\MUP)),
$$
and we must identified the underlying classical stacks of the factors. For the affine factor, we have $\Spec(\pi_0(\MUP))\simeq \Spec(L)$ for the Lazard ring $L$ directly by Quillen's famous computation,  e.g.\ \cite[Theorem 5.3.10]{Elliptic 2}. For $\mM$, its underlying moduli stack is the classical moduli stack of formal groups $\mM_\mathrm{FG}^\heart$ by Theorem \ref{On M} (though the proof in \cite{ChromaticCartoon} is also a direct application of Quillen's Theorem).
Finally for $\mM{}_\mathrm{Ell}^\mathrm{or}$,  its underlying classical stack is the classical stack of elliptic curves $\mM{}^\heart_\mathrm{Ell}$. This follows either directly construction in the original Goerss-Hopkins-Miller approach, or is once again derived from Quillen's Theorem in Lurie's approach \cite[Theorem 7.0.1, Remark 7.3.2]{Elliptic 2}. For the pullback in question, we therefore find its underlying classical stack given by
$$
(\mM{}_\mathrm{Ell}^\mathrm{or}\times_{\mM}\Spec(\MUP))^\heart \,\simeq\, \mathcal M^\heart_\mathrm{Ell}\times_{\mathcal M_\mathrm{FG}^\heart}\Spec(L).
$$
That this is an affine scheme now follows from the classical observation that the canonical map of stacks $\mM^\heart_\mathrm{Ell}\to \mM^\heart_\mathrm{FG}$, sending an elliptic curve to its corresponding formal group, is an affine morphism, see for instance \cite[proof of Theorem 7.2.(1)]{Affineness in chromatic}
\end{proof}

As a consequence of Lemma \ref{Lemma of the heart for TMF}, we can reprove the celebrated $0$-affineness result of \cite[Theorem 7.2]{Affineness in chromatic}. There are no key new insights - we   carry out the idea of the original proof in \textit{loc.\,cit.} in the setting of nonconnective spectral stacks. We therefore only highlight that part of the argument where
the ability to talk about objects such as the chromatic base stack $\mM$  might aid  in  conceptual clarity.

For this purpose, let $p$ be a prime and $n$ a non-negative integer, and recall from \cite{ChromaticFiltration} the open substack of its $p$-localization $\mM{}^{\le n}\subseteq \mM_{(p)}$, given as a functor $\mM{}^{\le n} :\CAlg{}_{(p)}\to \mathrm{Ani}$ explicitly by
$$
\mathcal M{}^{\le n}(A) \simeq \begin{cases} * & \text{if $A$ is complex periodic and $L_n$-local,}\\ \emptyset & \text{otherwise.} \end{cases}
$$
The $L_n$-locality can also be equivalently rephrased as requiring that the classical Quillen formal group $\w{\G}{}^{\mathcal Q}_A$ be of height $\le n$. This spectral stack was denoted $\mathcal M_\mathrm{FG}^{\mathrm{or}, \le n}$ in \cite{ChromaticFiltration} and called the \textit{moduli stack of oriented formal groups of height $\le n$}. It is also shown there that as $n$ varies, the resulting open stratification of the chromatic base stack $\mM$ induces on $(-)^\heart$ the height stratification on the classical stack of formal groups $\mM{}^\heart_\mathrm{FG}$, and induces on functions (or more generaly, on sections of a quasi-coherent sheaf) the chromatic filtration.

\begin{theorem}[Mathew-Meier]\label{Theorem affineness MM}
The global sections $\Gamma:\QCoh(\mM{}_\mathrm{Ell}^{\mathrm{or}})\to \Mod_\mathrm{TMF}$ are a symmetric monoidal equivalence of $\i$-categories.
\end{theorem}

\begin{proof}
By repeating the arguments in \cite[Subsection 4.3]{Affineness in chromatic}, we can reduce to showing the claim after $p$-localization at each prime $p$.
The canonical map $\mM{}_\mathrm{Ell}^\mathrm{or}\to \mM$ is an affine morphism by Lemma \ref{Lemma of the heart for TMF}, and hence so is its $p$-localization.
But in fact, by the classical fact that the height of the formal group of an elliptic curve must always be $\le 2$, we see that the map $\mM{}_{\mathrm{Ell}, (p)}^\mathrm{or}\to \mM{}_{(p)}$ actually factors through the open embedding $\mM{}^{\le 2}\subseteq\mM{}_{(p)}$. We thus obtain an affine morphism 
$$
f:\mM{}_{\mathrm{Ell}, (p)}^\mathrm{or}\to\mM{}^{\le 2},
$$
and so the pushforward $f_*$ induces a symmetric monoidal equivalence of $\i$-categories
$$
\QCoh(\mM{}_{\mathrm{Ell}, (p)}^\mathrm{or})\to \Mod_{f_*(\sO_{\mM{}_{\mathrm{Ell}, (p)}^\mathrm{or}})}(\QCoh(\mM^{\le 2})).
$$
According to \cite[Theorem 3.3.1]{ChromaticFiltration},  the global sections functor
$$
\Gamma:\QCoh(\mM^{\le 2})\to\Mod_{L_2\mathbf S}\,\simeq\, L_2\Sp,
$$
where $L_2\mathbf S\simeq \sO(\mM^{\le 2})$, is a symmetric monoidal equivalence
-- this is nothing but the spectral-algebro-geometric rephrasing of the Hopkins-Ravenel Smash Product Theorem. 
By passing to global sections, we obtain the equivalences of $\E$-rings
$$
\Gamma(\mM^{\le 2}; f_*(\sO_{\mM{}_{\mathrm{Ell}, (p)}^\mathrm{or}}))\,\simeq\, \Gamma(\mM{}_{\mathrm{Ell}, (p)}^\mathrm{cn};  \sO_{\mM{}_{\mathrm{Ell}, (p)}^\mathrm{or}})\,\simeq \,\mathrm{TMF}_{(p)}
$$
and
the desired result follows.
\end{proof}

\begin{proof}[Proof of Theorem \ref{Theorem Evp for TMF}]
The desired identification of the even periodization will follow immediately from Proposition \ref{EVP as a colimit} if we prove the following claim:
\begin{enumerate}[label = ($\star$)]
\item\label{Assertion star} \textit{Pullback along the affinization map   $\mM{}_\mathrm{Ell}^\mathrm{or}\to \Spec(\mathrm{TMF})$ induces an equivalence of $\i$-categories
$$
\mathrm{Aff}^\mathrm{evp}_{/\mM{}_\mathrm{Ell}^\mathrm{or}}\,\simeq \,\mathrm{Aff}^\mathrm{evp}_{/\Spec(\mathrm{TMF})}
$$
compatible with the forgetful functor to $\mathrm{Aff}$.}
\end{enumerate}
The remainder of this proof is devoted to proving this assertion.

As a relatively affine stack over the geometric spectral stack $\mM$, the moduli stack of oriented formal groups $\mM{}_\mathrm{Ell}^\mathrm{or}$ is geometric itself. This implies that any map from an affine $f:\Spec(A)\to \mM{}_\mathrm{Ell}^\mathrm{or}$ is also an affine morphism. By pursuing the chain of fully faithful functors
$$
\mathrm{Aff}_{/\mM{}_\mathrm{Ell}^\mathrm{or}}\hookrightarrow \CAlg(\QCoh(\mM{}_\mathrm{Ell}^\mathrm{or}))^\mathrm{op}\simeq \CAlg(\Mod_{\mathrm{TMF}})^\mathrm{op}\simeq\mathrm{Aff}_{/\Spec(\mathrm{TMF})}, 
$$
the last of which is the equivalence of Theorem \ref{Theorem affineness MM}, we find that sending  a map from an affine $f:\Spec(A)\to \mM{}_\mathrm{Ell}^\mathrm{or}$ to its affinization $f^\mathrm{aff}:\Spec(A)\to (\mM{}_\mathrm{Ell}^\mathrm{or})^\mathrm{aff}\simeq \Spec(\mathrm{TMF})$ induces a fully faithful embedding
$$
\mathrm{Aff}_{/\mM{}_\mathrm{Ell}^\mathrm{or}}\hookrightarrow \mathrm{Aff}_{/\Spec(\mathrm{TMF})}.
$$
This embedding is compatible with the forgetful functor to $\mathrm{Aff}$, so it restricts to a fully faithful embedding of the subcategories of even periodic affines
\begin{equation}\label{Equation ff map of evens over TMF}
\mathrm{Aff}^\mathrm{evp}_{/\mM{}_\mathrm{Ell}^\mathrm{or}}\hookrightarrow \mathrm{Aff}^\mathrm{evp}_{/\Spec(\mathrm{TMF})}.
\end{equation}
To establish assertion \ref{Assertion star}, it remains to show that this functor is essentially surjective.

Before tackling that, let us equip ourselves with another useful observation.
Consider the commutative diagram
$$
\begin{tikzcd}
\mM{}^\mathrm{or}_\mathrm{Ell}\arrow{rr} \arrow{dr} & & \mM\times \Spec(\mathrm{TMF})\arrow{dl}{}\\
& \mM, &\end{tikzcd}
$$
where the horizontal map is induced by the universal property of a product from the canonical map into $\mM$ and the affinization map $\mM{}^\mathrm{or}_\mathrm{Ell}\to \Spec(\mathrm{TMF})$. The left diagonal map is an affine morphism by Lemma \ref{Lemma of the heart for TMF}. The right (anti)diagonal map is also an affine morphism, since it is obtained by pullback from the affine map $\Spec(\mathrm{TMF})\to\Spec(\mathbf S)$. We may now invoke Lemma \ref{Lemma affine 3-1} to conclude that the horizontal map, which is to say $\mM{}^\mathrm{or}_\mathrm{Ell}\to\Spec(\mathrm{TMF})\times \mM$, is also an affine morphism.

Now let $\Spec(A)\to \Spec(\mathrm{TMF})$ be any morphism from an even periodic affine. The $\E$-ring $A$ is  even periodic and in particular complex orientable, so the projection map $\Spec(A)\times \mM\to \Spec(A)$ is an equivalence of nonconnective spectral stacks; this follows directly from the definition of $\mM$, but is also a special case of Proposition \ref{complex periodic stacks}.
Consider the following diagram of pullback squares
$$
\begin{tikzcd}
\Spec(A)\times_{\Spec(\mathrm{TMF})}\mM{}_\mathrm{Ell}^\mathrm{or}\arrow{r}{}\arrow{d}&\Spec(A)\times \mM\arrow{r}{\simeq} \arrow{d}{} & \Spec(A)\arrow{d}{}\\
\mM{}^\mathrm{or}_\mathrm{Ell} \arrow{r}{} &\Spec(\mathrm{TMF})\times \mM\arrow{r}{} & \Spec(\mathrm{TMF}).
\end{tikzcd}
$$
As we have observed above, the bottom left horizontal arrow of this diagram is an affine morphism. This means that the upper left horizontal arrow, which is its pullback, is an affine morphism too.  Its codomain is the affine  $\Spec(A)\times \mM\simeq \Spec(A)$, and since an affine morphism into an affine must itself be affine, we can conclude that
$$
\Spec(A)\times_{\Spec(\mathrm{TMF})}\mM{}_\mathrm{Ell}^\mathrm{or}\,\simeq\, \Spec(B).
$$
Eliminating the redundancy from the above diagram, we can rewrite it as
$$
\begin{tikzcd}
\Spec(B)\arrow{r}{f'} \arrow{d}{g'} & \Spec(A)\arrow{d}{g}\\
\mM{}^\mathrm{or}_\mathrm{Ell}\arrow{r}{f} & \Spec(\mathrm{TMF}).
\end{tikzcd}
$$
It is still a pullback square. Note that both $\mathcal M{}^{\mathrm{or}}_\mathrm{Ell}$ and $\Spec(\mathrm{TMF})$ are geometric spectral stacks, and it follows from Theorem \ref{Theorem affineness MM} that the unit $\sO_{\mM{}^{\mathrm{or}}_\mathrm{Ell}}$ is a compact object in $\QCoh(\mM{}^{\mathrm{or}}_\mathrm{Ell})$. We may therefore apply the base-change formula of Proposition \ref{Prop push-pull} to see that the canonical map of $A$-modules
$$
A\simeq g^*(\mathrm{TMF})\,\simeq\, g^*f_*(\sO_{\mM{}^\mathrm{or}_\mathrm{Ell}})\to f'_*g'^*(\sO_{\mM{}^\mathrm{or}_\mathrm{Ell}})\,\simeq\, f'_*(B)\,\simeq\, B
$$
is an equivalence. In particular, since it is identified with $A$, the $\E$-ring $B$ is even periodic. The thus-obtained object $\Spec(B)\to \mM{}^\mathrm{or}_\mathrm{Ell}$ of $\mathrm{Aff}^\mathrm{evp}_{/\mM{}^\mathrm{or}_\mathrm{Ell}}$ therefore maps
to the chosen object $\Spec(A)\to \Spec(\mathrm{TMF})$ in $\mathrm{Aff}^\mathrm{evp}_{/\Spec(\mathrm{TMF})}$ under the functor \eqref{Equation ff map of evens over TMF}. Since the latter object was arbitrary, It follows that the fully faithful embedding \eqref{Equation ff map of evens over TMF} is also essentially surjective and hence an equivalence of $\i$-categories, verifying assertion \ref{Assertion star} and concluding the proof.
\end{proof}

As a corollary of the proof of Theorem \ref{Theorem Evp for TMF}, we obtain a stronger version of the result.

\begin{corollary}\label{Cor Cp for TMF}
The affinization map $\mM{}^{\mathrm{or}}_\mathrm{Ell}\to \Spec(\mathrm{TMF})$ induces an equivalence of nonconnective spectral stacks on the complex periodization
$$
\Spec(\mathrm{TMF})\times \mM\,\simeq\,
\mM{}^{\mathrm{or}}_\mathrm{Ell}.
$$
\end{corollary}

\begin{proof}
We can replace all occurrences of ``even periodic" in the proof of Theorem \ref{Theorem Evp for TMF} with ``complex periodic" without affecting any change in any of the arguments. The analogue of assertion \ref{Assertion star} is then a canonical equivalence
$$
\mathrm{Aff}^{\mathbf C\mathrm p}_{/\mM{}_\mathrm{Ell}^\mathrm{or}}\simeq \mathrm{Aff}^{\mathbf C\mathrm p}_{/\Spec(\mathrm{TMF})},
$$
compatible with the forgetful functors to $\mathrm{Aff}$, from which the desired identification follows by passing to the colimits.
\end{proof}

\begin{remark}
By passing to the underlying classical stacks from the equivalence of spectral stacks of Corollary \ref{Cor Cp for TMF}, we can extract a presentation of the classical moduli stack of elliptic curves as a groupoid quotient stack
$$
\mM_\mathrm{Ell}^\heart \,\simeq\, \varinjlim_{\phantom{{}^\mathrm{op}}\mathbf{\Delta}^\mathrm{op}}\Spec(\pi_0(\mathrm{MUP}^\bullet\otimes\mathrm{TMF})),
$$
where we have set $\mathrm{MUP}^\bullet := \mathrm{MUP}^{\otimes_{\mathbf S}[\bullet]}$. Such a presentation is well-known to the experts; it can be inferred from \cite[Section 20]{Rezk notes} or \cite[Section 5.1]{Akhil on tmf}, and is explicitly stated in \cite[Example 6.6.1]{Behrens survey}.
\end{remark}

Another form of topological modular forms is obtained as the global function $\E$-ring
$$
\mathrm{Tmf}:=\sO(\overline{\mM}{}^\mathrm{or}_\mathrm{Ell})
$$
on an even periodic spectral enhancement $\overline{\mM}{}^\mathrm{or}_\mathrm{Ell}$ of the Deligne-Mumford compactification $\overline{\mM}{}^\heart_\mathrm{Ell}$ of the classical stack of elliptic curves $\mM{}^\heart_\mathrm{Ell}$. The map $\overline{\mM}{}^\heart_\mathrm{Ell}\to \mM_\mathrm{FG}^\heart$ is no longer affine, hence neither can $\overline{\mM}{}^\mathrm{or}_\mathrm{Ell}\to \mM$ be, so the above arguments do not apply to this case directly. However it is still \textit{quasi-affine}, which is sufficient. A straightforward modification of the above proofs of Theorem \ref{Theorem Evp for TMF} and Corollary \ref{Cor Cp for TMF} gives rise to the following result,  compatible with the $0$-affineness of $\overline{\mM}{}^\mathrm{or}_\mathrm{Ell}$ by \cite[Theorem 7.2]{Affineness in chromatic}.

\begin{corollary}
The affinization map
$$
\overline{\mM}{}^\mathrm{or}_\mathrm{Ell}\to \Spec(\mathrm{Tmf})
$$
exhibits both the even and complex periodization 
$$
 \Spec(\mathrm{Tmf})^\mathrm{evp}\,\simeq\, \Spec(\mathrm{Tmf})\times \mM
\,\simeq\,
\overline{\mM}{}^\mathrm{or}_\mathrm{Ell}
.
$$
\end{corollary}

The proof of Theorem \ref{Theorem Evp for TMF} given above would also apply to the even periodic spectral stack $\Spec(\mathrm{KU})/\mathrm C_2$, where  the quotient is taken with respect to complex conjugation. Indeed, the underlying map on classical stacks of the structure map $\Spec(\mathrm{KU})/\mathrm C_2\to \mM$ is the map $\mathrm B\mathrm C_2\to \mathcal M_{\mathrm{FG}}^\heart$, exhibiting the inversion isogeny on the multiplicative formal group $\w{\mathbf G}_{m, \mathbf Z}$, is an open immersion, \cite[Example 6.1]{Affineness in chromatic}. Since the $\E$-ring of global sections is given by real topological K-theory $\mathrm{KU}^{\mathrm h\mathrm C_2}\simeq \mathrm{KO}$, we obtain the following identification of another even periodization:

\begin{prop}
The vector space complexification map $\mathrm{KO}\to\mathrm{KU}$ induces an equivalence of nonconnective spectral stacks
$$
\mathrm{Spec}(\mathrm{KO})^\mathrm{evp}\,\simeq \,\mathrm{Spec}(\mathrm{KO})\times \mM\, \simeq \, \Spec(\mathrm{KU})/\mathrm C_2,
$$
where the $\mathrm C_2$-action on $\mathrm{KU}$ is given by complex conjugation.
\end{prop}

The same kinds of arguments as used in the proof of Theorem \ref{Theorem Evp for TMF} also apply -- and in fact substantially simplify -- to identify  the even and complex periodizations of the spectrum of the $L_n$-local sphere spectrum. They recover the open substack $\mM{}^{\le n}\subseteq\mM{}_{(p)}$ of the $p$-localized chromatic base stack, corresponding to all $L_n$-local even periodic $\E$-rings. We had studied this stack in \cite{ChromaticFiltration} under the name $\mM{}^{\mathrm{or}, \le n}_\mathrm{FG}$.

\begin{prop}\label{Variant EVP of L_nS}
There are canonical equivalence of spectral stacks
$$
\Spec(L_n\mathbf S)^\mathrm{evp}\,\simeq\, \Spec(L_n\mathbf S)\times \mM\,\simeq\, \mM{}^{\le n}.
$$
\end{prop}

\begin{remark}
In \cite{ChromaticFiltration} we showed that the spectral stacks $\mM^{\le n}$ fit into an open stratification, the spectral enhancement of the height stratification of formal groups,
$$
\mM_{\mathbf Q}\,\simeq\, \mM{}^{\le 0}\, \subseteq \,
\mM{}^{\le 1}\subseteq\mM{}^{\le 2}\subseteq \cdots \subseteq \mM{}^{\le \infty}\subseteq\mM{}_{(p)},
$$
which recovers on the $\E$-rings of global functions the chromatic filtration
$$
\mathbf S_{(p)}\,\simeq\, \varprojlim_n L_n \mathbf S\to \cdots \to L_2\mathbf S\to L_1S\to L_0\mathbf S\,\simeq\, \mathbf Q.
$$
It is now clear that we can pass  back to the spectral height stratification from the chromatic filtration by applying either of the functors $A\mapsto \Spec(A)^\mathrm{evp}$ or $A\mapsto \Spec(A)\times \mM$.
\end{remark}

Thanks to the celebrated solution of Ravenel's Telescopic Conjecture in \cite{Telescope Disproved}, we now know that the canonical map $L_n^f\mathbf S\to L_n\mathbf S$ is not an equivalence for all $n$. Nevertheless, he failure of the telescope conjecture is not visible on the level of the even and complex periodizations.

\begin{corollary}\label{Corollary how L_n from L_n^f}
There are canonical equivalences of spectral stacks
$$
\Spec(L_n^f\mathbf S)^\mathrm{evp}\,\simeq \, \Spec(L_n^f\mathbf S)\times \mM\, \simeq \, \mM^{\le n}.
$$
\end{corollary}

\begin{proof}
In light of Proposition \ref{Variant EVP of L_nS}, it suffices to show that the canonical map $L_n^fA\to L_nA$ is an equivalence for any even periodic $\E$-ring $A$. Instead, this holds for any even $\E$-ring. Indeed, by \cite[Proposition 7.42]{MNN}, an even $\E$-ring $A$ admits an $\mathbb E_1$-complex orientation, i.e.\ an $\mathbb E_1$-ring map $\mathrm{MU}\to A$. It therefore suffices to show that $L_n^f\mathrm{MU}\to L_n\mathrm{MU}$ is an equivalence, which amounts to a well-known computation.
\end{proof}

\begin{remark}
Since quasi-coherent pullback along the terminal map $\mM^{\le n}\to\Spec(\mathbf S)$ amounts precisely to $L_n$-localization, Corollary \ref{Corollary how L_n from L_n^f} may be interpreted as showing how the telescopic localizations $L_n^f$ give rise to their chromatic counterparts $L_n$ by passing through even periodization.
\end{remark}

\newpage

\section{Formal spectral algebraic geometry via the functor of points}\label{Section FAG}
In order to connect it to $p$-adic geometry in general, and prismatization in particular, it is necessary to give a $p$-complete version of even periodization. This requires first describing the setting of $p$-adic formal spectral algebraic geometry; in fact, we work in slightly greater generality.
As such, the majority of this section is dedicated to laying out 
 the foundations to a functor of points approach to nonconnective spectral algebraic geometry, in analogy with the non-formal setup presented in Section \ref{Section great 2}.

\subsection{The setup of nonconnective formal spectral stacks}
In our approach to formal spectral algebraic geometry, we broadly follow the foundations  laid out in \cite[Chapter 8]{SAG},
in that we base it on certain adic\footnote{It seems likely that a more general and robust variant of the theory should be built on the foundations of condensed mathematics. But since the scope of formal algebraic geometry built on adic $\E$-rings largely suffices for our purposes, we do not pursue this endeavor here.} $\E$-rings as formal affines. 

\begin{definition}
An \textit{adic $\E$-ring} is an $\E$-ring $A$ together with the choice of an adic topology on the commutative ring $\pi_0(A)$, i.e.\ an $I$-adic topology for some finitely generated ideal $I\subseteq\pi_0(A)$, such that $A$ is $I$-complete.
\end{definition}

\begin{remark}
Note that the ideal $I\subseteq\pi_0(A)$ itself, sometimes called an \textit{ideal of definition} of the adic $\E$-ring $A$, is not part of the data. Indeed, different ideals may give rise to the same topology, for instance if their radicals agree.
\end{remark}

Let $\mathrm{CAlg}_\mathrm{ad}$ denote the $\i$-category of adic $\E$-rings.
It has adic $\E$-rings as its objects, and its morphisms are given by all such $\E$-ring maps $f:A\to B$ for which the underlying ring homomorphism $\pi_0(f):\pi_0(A)\to\pi_0(B)$ is continuous with respect to the respective adic topologies. Equivalently, they satisfy $\pi_0(f)(I^n)\subseteq J$ for $n>\!>0$ where $I\subseteq\pi_0(A)$ and $J\subseteq\pi_0(B)$ are some (or indeed any) ideals of definition.

For any subcategory $\CAlg^?\subseteq\CAlg$, we take its adic variant
$$
\CAlg{}^?_\mathrm{ad}:=\CAlg{}^?\times_{\CAlg}\CAlg_\mathrm{ad}
$$
to be the pullback along the forgetful functor $\CAlg_\mathrm{ad}\to\CAlg$. This allows us to make sense of $\CAlg^\mathrm{cn}_\mathrm{ad}$, $\CAlg^{\mathbf C\mathrm p}_\mathrm{ad}$, $\CAlg_\mathrm{ad}^\heart$ etc. Note that this is particularly appropriate, since the definition of adic $\E$-rings itself boils down to the pullback of $\i$-categories
$$
\CAlg_\mathrm{ad}\simeq \CAlg\times_{\CAlg^\heart}\CAlg_\mathrm{ad}^\heart
$$
along the functor $\pi_0 : \CAlg\to\CAlg^\heart$.

\begin{definition}\label{Definition 5.1.3 of formal spectrum}
Let $A$ be an adic $\E$-ring. Its \textit{formal spectrum} $\Spf(A):\CAlg\to \mathrm{Ani}$ 
 is the functor given by $R\mapsto \Map_{\CAlg_\mathrm{ad}}(A, R)$ or, as is equivalent, by
\begin{equation}\label{Equation adic Hom pullback formula}
R\mapsto \Map_{\CAlg}(A, R)\times_{\Hom_{\CAlg^\heart}(\pi_0(A), \pi_0(R)) }\Hom_{\CAlg^\heart}^\mathrm{cont}(\pi_0(A), \pi_0(R)).
\end{equation}
\end{definition}

\noindent More explicitly, the forgetful functor $\CAlg_\mathrm{ad}\to\CAlg$ induces a natural transformation $\Spf(A)\to \Spec(A)$, which picks out those connected components of $\Map_\CAlg(A, R)$ which correspond to $\E$-ring maps $f:A\to R$ satisfying $\pi_0(f)(I^n) = 0$ for $n>\!>0$.

\begin{remark}\label{Pullback square}
Defined as above, the formal spectrum $\Spf(A)$ of an adic $\E$-ring $A$ is completely determined by the spectral affine $\Spec(A)$ and the classical formal affine $\Spf(\pi_0(A))$. One way to make this precise is to expressed in terms of the embedding $\mX\mapsto \mX\circ \pi_0$ of classical into spectral stacks from Remark \ref{Remark compression}, as the pullback square
$$
\begin{tikzcd}
\Spf(A) \ar{r} \ar{d}& \Spec(A)\ar{d}\\
\Spf(\pi_0(A))\circ\pi_0\ar{r} & \Spec(\pi_0(A))\circ\pi_0
\end{tikzcd}
$$
in the $\i$-category of spectral prestacks $\Fun(\CAlg, \mathrm{Ani})$, where the vertical maps are the decompression maps of Remark \ref{Remark compression}.
\end{remark}

\begin{prop}\label{Theorem Spf is ff}
The functor $A\mapsto \Spf(A)$ defines a fully faithful embedding
$$
(\CAlg_\mathrm{ad}^\mathrm{cn})^\mathrm{op}\hookrightarrow \mathrm{SpStk}^\mathrm{cn}.
$$
\end{prop}

\begin{proof}
A commination of \cite[Theorem 8.1.5.1]{SAG} and \cite[Corollary 8.1.5.4]{SAG} shows that the
functor $A\mapsto \Spf(A)$ gives a fully faithful embedding
$$
(\CAlg_\mathrm{ad}^\mathrm{cn})^\mathrm{op}\hookrightarrow \Fun^\mathrm{acc}(\CAlg^\mathrm{cn}, \mathrm{Ani}).
$$
Both of the cited results rely on \cite[Lemma 8.1.2.2]{SAG}, by which we may express the formal spectrum of any connective $\E$-ring $A$ through an $\E$-ring version of the $I$-adic filtration as
$$\Spf(A)\,\simeq\, \varinjlim_n \Spec(A_n)
$$
with the filtered colimit computed either in presheaves or in \'etale sheaves. It  follows from this latter description that $\Spf(A):\CAlg^\mathrm{cn}\to\mathrm{Ani}$ is an accessible presheaf for any adic connective $\E$-ring $A$. It remains to show that $\Spf(A)$ is also an fpqc sheaf. From its definition  (or equivalently, the pullback square of Remark \ref{Pullback square}), and due to $\Spec(A)$ satisfying fpqc descent \cite[Theorem D.6.3.5]{SAG}, this reduces to proving flat descent for $\Spf(\pi_0(A)):\CAlg^\heart\to \mathrm{Set}$. Since it is a subfunctor of the fpqc sheaf $\Spec(\pi_0(A))$, that in turn follows from the observation  that if
$$
\pi_0(A)\xrightarrow{f}R\xrightarrow{g} S
$$
is a composable pair of commutative ring maps, and $g$ is fully faithful, then since a fully faithful ring map is in particular injective, the equality  $(g\circ f)(I^n)=0$ implies that $f(I^n)=0$ holds as well. 
\end{proof}

\begin{remark}\label{Remark adic topology}
As we saw in the proof of Proposition \ref{Theorem Spf is ff}, the formal spectrum $\Spf(A)$ satisfies fpqc descent for any connective adic $\E$-ring $A$. Consequently, the sequential colimit formula
$$
\Spf(A)\,\simeq\, \varinjlim_n \Spec(A_n)
$$
for $\Spf(A)$ remains valid when interpreted in the $\i$-category $\mathrm{SpStk}^\mathrm{cn}$. To see that this is indeed an $\E$-form of the adic filtration, let us recall from \cite[Lemma 8.1.2.2]{SAG} how the $\E$-rings $A_n$ are constructed in the case where the adict topology on $A$ is with respect to a principal ideal $I=(x)$ for some $x\in \pi_0(A)$. In that case, we have for each $n\ge 1$ a pullback square of $\E$-rings
$$
\begin{tikzcd}
A\{t\} \arrow{r}{t\mapsto x^n} \arrow{d}{t\mapsto 0} & A\arrow{d}{} & \\
A \arrow{r} & A_n,
\end{tikzcd}
$$
where $A\{t\}\simeq \Sym_A^*(A)$ is the free $\E$-$A$-algebra on one generator. That is to say, the $A$-algebra $A_n$ is the $\E$-ring quotient $A/\!/_{\E}x^n$.
\end{remark}

Proposition \ref{Theorem Spf is ff} admits an obvious analogue in the classical, i.e.\ $0$-truncated, case, and a weakened variant in the nonconnective case.

\begin{corollary}\label{Theorem Spf is ff classic}
The formal spectrum functor $A\mapsto\Spf(A)$ restricts to a fully faithful embedding
$
\mathrm{FAff}^\heart\hookrightarrow \mathrm{SpStk}^\heart.
$
\end{corollary}

\begin{corollary}\label{Spf factors through SpStk}
The formal spectrum functor $A\mapsto\Spf(A)$ takes values inside the full subcategory of nonconnective spectral stacks $\mathrm{SpStk}\subseteq\Fun(\CAlg, \mathrm{Ani})$.
\end{corollary}

\begin{proof}
Let $A$ be an $\E$-ring. We need to show that $\Spf(A)$ is a spectral stack.
It follows from the definition of the formal spectrum, e.g.\ from the pullback square of Remark \ref{Pullback square}, that the obvious functors constitute a Cartesian square
$$
\begin{tikzcd}
\Spf(A) \arrow{r} \arrow{d} & \Spec(A)\arrow{d} \\
\Spf(\tau_{\ge 0}(A)) \arrow{r} &\Spec(\tau_{\ge 0}(A))
\end{tikzcd}
$$
in $\Fun(\CAlg, \mathrm{Ani})$. The right vertical arrow is an affine morphism, hence so is the left vertical arrow. In particular, since all the vertices of this diagram but the initial limit vertex are nonconnective spectral stacks, it follows that the formal spectrum $\Spf(A)$ lies in $\mathrm{SpStk}$ as well.
\end{proof}

We can use the formal spectrum
to define convenient a notion of faithful flatness for maps of adic $\E$-rings by restriction from faithfully flat maps of spectral stacks.

\begin{definition}\label{Def of adic fpqc}
A map of adic $\E$-rings $A\to B$ is \textit{adically faithfully flat} if the morphism of spectral stacks $\Spf(B)\to\Spf(A)$ is affine and faithfully flat.
\end{definition}

\begin{remark}
For a map of adic $\E$-rings $A\to B$, the property of being adically faithfully flat is equivalent to $\Spf(B)\to\Spf(A)$ being affine and the $\E$-ring maps $B\to B\o_A A_n$ being faithfully flat for all $n$, where $\Spf(A)\simeq \varinjlim_n \Spf(A_n)$ is the $\E$-version of the adic filtration on $A$ from Remark \ref{Remark adic topology}. The latter requirement is equivalent to the base-change $B\to B\o_A M$ being faithfully flat for every $I$-nilpotent $A$-module $M$ (indeed, each $A_n$ is itself $I$-nilpotent as an $A$-module). For a connective $\E$-ring $A$, this is closely connected with the
notion of $I$-adic faithful flatness from \cite[Definition 7.1]{Akhil Almost}. See also Proposition \ref{Lemma ff comparison} for a simpler description of adically faihtful flatness and Corollary \ref{Adic = complete ff} for a comparison with $p$-adic faithful flatness of \cite{HRW}.
\end{remark}

The notion of adically faithfully flat maps equips the $\i$-categories of \textit{formal spectral affines} $\mathrm{FAff} :=\CAlg_\mathrm{ad}^\mathrm{op}$ and \textit{formal spectral affines} $\mathrm{FAff} := \CAlg_\mathrm{ad}^\mathrm{op}$ with a Grothendieck topology in the usual finitary way. By an abuse of notation, we will refer to it simply as the \textit{fpqc topology}. We define the corresponding $\i$-categories of connective and discrete variants of formal spectral affines
analogously as $\mathrm{FAff}^\mathrm{cn}:=(\CAlg^\mathrm{cn}_\mathrm{ad})^\mathrm{op}$ and $\mathrm{FAff}^\heart:=(\CAlg^\heart_\mathrm{ad})^\mathrm{op}$, and they inherit the fpqc topology as well.

\begin{definition}\label{Def of FSpStk}
A functor $\mathfrak X:\CAlg_\mathrm{ad}\to \mathrm{Ani}$ is a \textit{nonconnective formal spectral stack} if it is  accessible and satisfies fpqc descent. We let
$$
\mathrm{FSpStk}\, :=\,\mathrm{Shv}^\mathrm{acc}_\mathrm{fpqc}(\mathrm{FAff})
$$
denote the $\i$-category of nonconnective formal spectral stacks. The analogous connective and classical variants are defined as
$$
\mathrm{FSpStk}^\mathrm{cn}\, :=\,\mathrm{Shv}^\mathrm{acc}_\mathrm{fpqc}(\mathrm{FAff}^\mathrm{cn}),
\qquad \mathrm{FStk}^\heart\, :=\,\mathrm{Shv}^\mathrm{acc}_\mathrm{fpqc}(\mathrm{FAff}^\heart).
$$
\end{definition}

\begin{remark}\label{What _is_ Shv_acc here?}
In Definition \ref{Def of FSpStk}, we run into the obstacle that the $\i$-category of adic $\E$-rings $\CAlg_\mathrm{ad}$ is not accessible, on account of not being closed under filtered colimits\footnote{This is the consequence of having required that an ideals of definition of an adic $\E$-ring must be finitely generated. A flexible theory of completions with respect to not-necessarily-finitely-generated ideals should circumvent this issue altogether, which is why we are treating it somewhat lightly.}. As such, it is not \textit{\`a priori} clear what we mean by ``accessible presheaf" on $\mathrm{FAff}$. Motivated by the discussion in Remarks \ref{Remark access why} and \ref{sheafification issues}, we interpret it as the completion under small colimits, i.e.\ we take $\mathcal P^\mathrm{acc}(\mathrm{FAff})\subseteq \Fun(\CAlg_\mathrm{ad}, \mathrm{Ani})$ to be the full subcategory generated under small colimits by representable the functors $\Spf(A)$ for adic $\E$-rings $A$. The cautious reader is encouraged to consider Remark \ref{Reason for cavaliers}, and to place any of the proceeding discussion that makes them uneasy into the $I$-adic context that we lay out in Subsection \ref{Subsection I-adic context}, and where these issues do not arise.
\end{remark}

\begin{remark}
Given an adic $\E$-ring $A$, we will denote the corresponding nonconnective formal affine by $\Spf(A)$. This in principle allows for some potential confusion with $\Spf(A)$, viewed as a spectral stack as in Definition \ref{Equation adic Hom pullback formula}. But since the latter is the restriction of the former to the full subcategory of discretely-topologized $\E$-rings $\CAlg\subseteq\CAlg_\mathrm{ad}$, and since the two are very close anyway by Theorem \ref{Theorem Spf is ff}, we trust no confusion will arise.
\end{remark}

As in the non-formal context, we may identify $\mathrm{FSpStk}^\mathrm{cn}$ and $\mathrm{FStk}^\heart$ with the full subcategories of $\mathrm{FSpStk}$, spanned under small colimits by connective and classical formal affines respectively.

\begin{variant}
The theory of connective localizations and underlying classical stack as laid out in Sections \ref{Section connective localization} and \ref{Section underlying classical stack} carries over to the formal setting without a change. In particular, 
we can make sense of the functors $\mathfrak X\mapsto \mathfrak X^\mathrm{cn}$ and $\mathfrak X\mapsto \mathfrak X^\heart$ for nonconnective formal spectral stacks, so that they are given on formal affines by $\Spf(A)^\mathrm{cn}\simeq \Spf(\tau_{\ge 0}(A))$ and $\Spf(A)^\heart\simeq \Spf(\pi_0(A))$ respectively.
\end{variant}

Though we have defined formal spectral stacks above in their own right, it turns out that at least in the connective case, this just reduces to the non-formal theory of connective spectral stacks.

\begin{theorem}\label{Theorem formal stacks = stacks}
The fully faithful embedding $\mathrm{FAff}^\mathrm{cn}\hookrightarrow \mathrm{SpStk}^\mathrm{cn}$ of Proposition \ref{Theorem Spf is ff} admits a canonical extension to an equivalence of $\i$-categories $\mathrm{FSpStk}^\mathrm{cn}\simeq\, \mathrm{SpStk}^\mathrm{cn}$.
\end{theorem}

\begin{proof}
Let $\delta :\CAlg^\mathrm{cn}\to \CAlg^\mathrm{cn}_\mathrm{ad}$ denote the functor which equips a connective $\E$-ring with the discrete, i.e.\ $0$-adic, topology. It is immediate from Definition \ref{Def of adic fpqc} that $\delta$ preserves fpqc covers. Since $\delta$ is left adjoint to the forgetful functor $\CAlg_\mathrm{ad}^\mathrm{cn}\to\CAlg^\mathrm{cn}$, it also commutes with colimits. It is therefore accessible and preserves pushouts, hence giving rise by passage to the $\i$-categories of accessible sheaves to the adjunction
$$
\delta_! : \mathrm{SpStk}^\mathrm{cn} \rightleftarrows \mathrm{FSpStk}^\mathrm{cn}:\delta^*.
$$
The left adjoint is given by left Kan extension along $\delta$, followed up by sheafification. But since $\delta_!(\Spec(A))\simeq \Spf(A)$, the formal spectrum with respect to the $0$-adic topology on a connective $\E$-ring $A$,  already satisfies fpqc descent by Proposition \ref{Theorem Spf is ff}, the functor $\delta_!$ can be fully characterized by its value on affines and commuting with all small colimits.

On the other hand, the fully faithful embedding
$\mathrm{FAff}^\mathrm{cn}\hookrightarrow \mathrm{SpStk}^\mathrm{cn}$ from Proposition \ref{Theorem Spf is ff} induces, by the universal property of the accessible presheaf $\i$-category \cite[Remark A.4]{Dirac II} (more precisely, see Remark \ref{What _is_ Shv_acc here?}) and the fact that the $\i$-category of connective spectral stacks has all small colimits, an essentially-unique extension
$$
\mP^\mathrm{acc}(\mathrm{FAff}^\mathrm{cn})\to \mathrm{SpStk}^\mathrm{cn}
$$
which commutes with all small colimits. Explicitly, it is uniquely determined by preserving colimits and sending $\Spf(A)\mapsto \varinjlim_n\Spec(A_n)$ in the sense of Remark \ref{Remark adic topology}
It follows from Definition \ref{Def of adic fpqc} that the augmented  \v{C}ech nerves of fpqc covers are taken to colimit diagrams  in connective spectral stacks, implying a colimit-preserving factorization through the localization that is the fpqc sheafification $L:\mP^\mathrm{acc}(\mathrm{FAff}^\mathrm{cn})\to\mathrm{FSpStk}^\mathrm{cn}.$ That is to say, we obtain a functor
$$
\rho:\mathrm{FSpStk}^\mathrm{cn}\to \mathrm{SpStk}^\mathrm{cn}
$$
which commutes with small colimits and sends $\Spf(A)\mapsto \varinjlim_n\Spec(A_n)$ in the sense of Remark \ref{Remark adic topology}. In particular, this is the promised extension of the fully faithful embedding of Proposition \ref{Theorem Spf is ff} promised in the Theorem statement.

We claim that this functor $\rho$ is an inverse functor to $\delta_!$.
It suffices to verify that both composites $\delta_!\circ\rho$ and $\rho\circ \delta_!$ are homotopic to the respective identity functors. Since both commute with colimits, and the $\i$-categories in question are generated under colimits by objects of the form $\Spf(A)$ and $\Spec(R)$ for $A\in \mathrm{CAlg}^\mathrm{cn}_\mathrm{ad}$ and $R\in\CAlg^\mathrm{cn}$ respectively, it suffices to observe that there are natural equivalences
$$
\delta_!\rho(\Spf(A))\simeq \varinjlim_n\Spec(A_n)\simeq \Spf(A), \qquad \rho(\delta_!(\Spec(R))\simeq \rho(\Spf(R))\simeq \Spec(R)
$$
directly by how everything is defined, together with Proposition \ref{Theorem Spf is ff}.
\end{proof}

\begin{remark}\label{Remark ff of SpStk into FSpStk}
In the proof of Theorem \ref{Theorem formal stack = stacks classic}, the construction of the functors $\delta_!, \delta^*,$ and $\rho$ would proceed without change in the nonconnective setting as well. So would the verification of a natural equivalence $\rho\delta_!\simeq \mathrm{id}$, showing that the functor $\delta_!$ provides a canonical fully faithful embedding $\mathrm{SpStk}\hookrightarrow\mathrm{FSpStk}$ of nonconnective spectral stacks into the analogous formal $\i$-category. Issues would instead arise with the other composite $\delta_!\rho$.
Somewhat confusingly, it would still follows from Corollary \ref{Spf factors through SpStk} that
$$
\delta_!\rho(\Spf(A))\simeq \Spf(A)
$$
holds for any adic $\E$-ring $A$. But since, unlike in the connective case, we do not know whether or not the functor $\rho$ is fully faithful on nonconnective formal affines, we can not conclude that $\delta_!\rho\simeq \mathrm{id}$ holds. Rather, all we get is the much weaker claim that
 said composite is essentially surjective. In short, while \textit{all formal spectral stacks} can be found among spectral stacks, it is unclear whther \textit{all their morphisms} can be as well.
\end{remark}

\begin{corollary}\label{Theorem formal stack = stacks classic}
The fully faithful embedding $\mathrm{FAff}^\heart\hookrightarrow \mathrm{SpStk}^\heart$ of Corollary \ref{Theorem Spf is ff}
admits a canonical extension to an equivalence of $\i$-categories $\mathrm{FStk}^\heart\simeq\, \mathrm{Stk}^\heart$.
\end{corollary}

Whenever possible, we will use Theorem \ref{Theorem formal stacks = stacks} and Corollary \ref{Theorem formal stack = stacks classic} to view connective and classical formal spectral stacks as  connective and classical spectral stacks respectively. In particular, they both fall under the umbrella of nonconnective spectral stacks. We likewise use the functor $\rho :\mathrm{FSpStk}\to \mathrm{SpStk}$, defined by respecting small colimits and sending $\Spf(A)\to\Spf(A)$ for any adic $\E$-ring $A$, to associate to any nonconnective formal spectral stack $\mathfrak X$ an corresponding nonconnective spectral stack. 

\begin{remark} 
 Theorem \ref{Theorem formal stacks = stacks} is an extension to arbitrary connective formal spectral stacks of the corresponding result  for Deligne-Mumford stacks \cite[Theorem 8.1.5.1]{SAG}, as built upon the foundation of $\E$-ringed $\i$-topoi. Indeed, according to said result,  the $\i$-category of connective formal spectral Deligne-Mumford stacks embeds fully faithfully via the functor of points into  the $\i$-category $\mathrm{SpStk}^\mathrm{cn}$.
\end{remark}

\subsection{Quasi-coherent sheaves on formal spectral stacks}

When attempting to set up a theory of quasi-coherent sheaves in the formal setting however, we encounter a difficulty. We wish to replace modules over an $\E$-ring with the full subcategory $\Mod_A^\mathrm{cplt}\subseteq\Mod_A$, spanned by all  the $I$-complete $A$-modules, and equipped with the completed relative tensor product $M\widehat{\otimes}_AN \simeq (M\otimes_A N)_I^\wedge$,
where $I\subseteq\pi_0(A)$ is an ideal of definition for the adic $\E$-ring $A$.
By repeating the proof of Proposition \ref{Def of QCoh}, we can get as far as follows:

\begin{cons}\label{Cons of formal QCoh}
We can use Kan extension, in particular the specific approach to it from \cite[Section 5.2.1]{SAG} via the coCartesian fibtration $\Mod^\mathrm{cplt}\to\CAlg^\mathrm{ad}$, to define an essentially unique functor
$$
\mathfrak {QC}\mathrm{oh}:\mathcal P^\mathrm{acc}(\mathrm{FAff})^\mathrm{op}\to \CAlg(\mathrm{Pr^L})
$$
specified by the following two properties:
\begin{enumerate}[label = (\roman*)]
\item The functor $\mathfrak {QC}\mathrm{oh}$ commutes with all small limits,
\item The composite
$$\CAlg_\mathrm{ad}\overset{\Spec}\simeq \mathrm{FAff}^\mathrm{op}\hookrightarrow \text{$\mP$}^\mathrm{acc}(\mathrm{FAff})^\mathrm{op}
\xrightarrow{\mathfrak {QC}\mathrm{oh}}\CAlg(\mathrm{Pr^L})
$$
recovers the functor $A\mapsto\Mod^\mathrm{cplt}_A$, sending any adic $\E$-ring to the $\i$-category of complete modules over it.
\end{enumerate}
For any $\mathfrak X\in \mP(\mathrm{FAff}^\mathrm{cn})$, the stable $\i$-category $\mathfrak{QC}\mathrm{oh}(\mathfrak X)$ also carries a canonical $t$-structure, recovering on formal affines the standard $t$-structure on $\mathfrak{QC}\mathrm{oh}(\Spf(A))\simeq \Mod_A^\mathrm{cplt}$.
\end{cons}

The setting enables us to the define the analogues in formal geometry of the standard quasi-coherent technology from Section
\ref{Section 3.4}.
Indeed, any map $f:\mathfrak X\to \mathfrak Y$ in $\mP^\mathrm{acc}(\mathrm{FAff})$ induces an adjunction of $\i$-categories
$$
f^*:\mathfrak{QC}\mathrm{oh}(\mathfrak Y) \rightleftarrows \mathfrak{QC}\mathrm{oh}(\mathfrak{X}) :f_*
$$
with the left adjoint $f^*$ symmetric monoidal and the right adjoint $f_*$ lax symmetric monoidal. In the special case of the terminal map $p : \mathfrak X\to \mathrm{Spec}(\mathbf S)$, and under the canonical identification $\mathfrak{QC}\mathrm{oh}(\Spec(\mathbf S))\simeq \Sp$, we introduce the special names
$
{\mO}_{\mathfrak X}:=p^*(\mathbf S),
$
called the \textit{structure sheaf} of $\mathfrak X$ (this is also the symmetric monoidal unit in $\mathfrak{QCoh}(\mathfrak X)$), and for any $\sF\in\mathfrak{QC}\mathrm{oh}(\mathfrak X)$ set
$$
\Gamma(\mathfrak X; \sF) := p_*(\sF)
$$
to be the spectrum of its \textit{global sections}. Combining these together, we obtain the \textit{$\E$-ring of global sections} given by
$$
\mathfrak O(\mathfrak X):=\Gamma(\mathfrak X; {\mathcal O}_{\mathfrak X}).
$$

\begin{remark}
The functor $\mathfrak O : \mP^\mathrm{acc}(\mathrm{FAff})\to \CAlg$ admits an alternative description as the decategorification of $\mathfrak{QC}\mathrm{oh}$. That is to say, it may be recovered by repeating Construction \ref{Cons of formal QCoh}, but replacing all instances of $\CAlg(\mathrm{Pr^L})$ with $\CAlg$ and the functor $\CAlg_\mathrm{ad}\to \CAlg(\mathrm{Pr^L})$, given by $A\mapsto\Mod_A^\mathrm{cplt}$, with the forgetful functor $\CAlg_\mathrm{ad}\to\CAlg$, given by $A\mapsto A$. In particular, note that
$$
\mathfrak O(\Spf(A))\simeq A
$$
holds for any adic $\E$-ring $A$. In contrast, it is not clear without a connectivity assumption whether or not the canonical $\E$-ring map $A\to \sO(\Spf(A))$ is an equivalence.
\end{remark}

The issue we encounter is that it is not clear whether the restriction of the functor $\mathfrak{QC}\mathrm{oh}$ to the full subcategory $\mathrm{FSpStk}\subseteq\mathcal P^\mathrm{acc}(\mathrm{FAff})$ still commutes with small limits. As is equivalent, it is unclear whether or not the functor $A\mapsto\Mod_A^\mathrm{cplt}$ satisfies fpqc descent.
Unble to resolve this issue in general, we will instead provide two distinct conditional affirmative answers, which will largely suffice for the purpose of paper.

The first of said affirmative result is  found in
 the context of connective formal spectral stacks. We already know that these coincide with connective spectral stacks by Theorem \ref{Theorem formal stacks = stacks}, so it should come as little surprise that the corresponding notions of quasi-coherent sheaves are closely related also. 

\begin{prop}\label{Formal and informal QCoh}
Let $\mathfrak X$ be a connective formal spectral stack. There is a canonical equivalence of symmetric monoidal  $\i$-categories
$$
\mathfrak{QC}\mathrm{oh}(\mathfrak X)^\mathrm{cn}\simeq \QCoh(\mathfrak X)^\mathrm{cn}.
$$
\end{prop}

\begin{proof}
By how both $\i$-categories were defined via Kan extension,
there is a canonical $t$-exact symmetric monoidal comparison functor $\mathfrak{QC}\mathrm{oh}(\mathfrak X)\to \QCoh(\mathfrak X)$ that we wish to show restricts to an equivalence on the connective subcategories. In the case of a connective formal affine $\mathfrak{X}=\Spf(A)$, this is the equivalence
\begin{equation}\label{Formal affineness formula 3.42}
(\Mod_A^\mathrm{cplt})^\mathrm{cn}\simeq\QCoh(\Spf(A))^\mathrm{cn},
\end{equation}
which is a special case of \cite[Theorem 8.3.4.4]{SAG}, in particular following quite readily 
 from the ``$\E$ adic filtration" description $\Spf(A)\simeq \varinjlim_n\Spec(A_n)$ of Remark \ref{Remark adic topology}. Indeed, it ostensibly suffices to verify whether the canonical map $M^\wedge_I\to\varprojlim_n M\otimes_A A_n$ is an equivalence of $A$-modules, which is the case for any connective $A$-module $M$. It follows from this equivalence of $\i$-categories on connective formal affines  \eqref{Formal affineness formula 3.42} that the composite functor
 $$
(\mathrm{FAff}^\mathrm{cn})^\mathrm{op}\hookrightarrow \text{$\mP$}^\mathrm{acc}(\mathrm{FAff}^\mathrm{cn})^\mathrm{op}\xrightarrow{\mathfrak{QC}\mathrm{oh}^\mathrm{cn}}\CAlg(\mathrm{Pr^L})
 $$
 satisfies fpqc descent. Therefore, the functor  $\mathfrak{QC}\mathrm{oh}:\text{$\mP$}^\mathrm{acc}(\mathrm{FAff}^\mathrm{cn})^\mathrm{op}\xrightarrow{\mathfrak{QC}\mathrm{oh}^\mathrm{cn}}\CAlg(\mathrm{Pr^L})$ admits a small limit-preserving factorization through the fpqc-sheafification localization $L:\mP^\mathrm{acc}(\mathrm{FAff}^\mathrm{cn})\to \mathrm{FSpStk}^\mathrm{cn}$. The so-obtained functor, composed with the equivalence of $\i$-categories of Theorem \ref{Theorem formal stacks = stacks}, is now of the form
$$(\mathrm{SpStk}^\mathrm{cn})^\mathrm{op}\simeq (\mathrm{FSpStk}^\mathrm{cn})^\mathrm{op}\xrightarrow{\mathfrak{QC}\mathrm{oh}^\mathrm{cn}}\CAlg(\mathrm{Pr^L}),
 $$
 commutes with all small limits, and its value on nonconnective affines is $A\mapsto \Mod_A^\mathrm{cn}$, since any module is complete with respect to the $0$-adic topology. These properties characterize the functor $\QCoh^\mathrm{cn}$ by Proposition \ref{Def of QCoh}, and so the two connective quasi-coherent sheaf functors $\mathfrak{QC}\mathrm{oh}^\mathrm{cn}$ and $\QCoh^\mathrm{cn}$ coincide.
\end{proof}

As noted in \cite[Remark 8.3.4.5]{SAG}, it is unclear whether the equivalence of $\i$-categories of Proposition  \ref{Formal and informal QCoh} remains valid if we drop the connectivity assumption. We can certainly weaken it to that of \textit{almost connectivity} -- also known as eventual connectivity, or being  bounded below -- which is to say, to the subcategory $\mC^\mathrm{acn} :=\bigcup_n\mC_{\ge n}\subseteq\mC$ where $\mC$ denotes a stable $\i$-category with a $t$-structure.

\begin{corollary}\label{Acn theorem QCoh}
Let $\mathfrak X$ be a connective formal spectral stack. There is a canonical equivalence of symmetric monoidal  $\i$-categories
$$
\mathfrak{QC}\mathrm{oh}(\mathfrak X)^\mathrm{acn}\simeq \QCoh(\mathfrak X)^\mathrm{acn}.
$$
\end{corollary}

\begin{proof}
The canonical comparison functor between the two types of quasi-coherent sheaves $\mathfrak{QC}\mathrm{oh}(\mathfrak X)\to \QCoh(\mathfrak X)$ commutes with suspension , so the equivalence of $\i$-categories of Proposition \ref{Formal and informal QCoh} automatically extends to $\mathfrak{QC}\mathrm{oh}(\mathfrak X)_{\ge n}\simeq \QCoh(\mX)_{\ge n}$ for all $n\in \mathbf Z$. The desired claim follows by passing to the unions over $n$ on both sides.
\end{proof}

We know how to do away with the connectivity assumptions in these quasi-coherent comparison results in only one case: that of classical formal stacks. The following may be inferred from \cite[Remark 8.3.4.5]{SAG}, and a proof is sketched (in the greater generality of a derived formal spectral stack) in \cite[Footnote 7]{BS22}:

\begin{prop}\label{Prop 5.2.5}
Let $\mathfrak X$ be a classical formal stack. There is a canonical equivalence of symmetric monoidal  $\i$-categories
$$
\mathfrak{QC}\mathrm{oh}(\mathfrak X)\simeq \QCoh(\mathfrak X).
$$
\end{prop}

Though the connectivity restrictions encountered above are a real issue in developing a flexible theory\footnote{It strikes the author that these difficulties are of the same general flavor as the difficulties with nonconnective analytic rings. Therefore, it does not seem clear to the author whether this particular set of difficulties would be substantially eased through a systematic adoption of the condensed technology.} of nonconnective formal spectral algebraic geometry via the functor of points, it does not prevent us from imitating Construction \ref{Cons 3.22 Postnikov filtration on QCoh}. In particular, we can make sense of Postnikov filtrations for quasi-coherent sheaves on any $\mathfrak X\in \mathrm{FSpStk}$
\begin{eqnarray*}
\mathfrak{QC}\mathrm{oh}(\mathfrak X)&\xrightarrow{\tau_{\ge *}}&\mathrm{Fil}(\mathfrak{QC}\mathrm{oh}(\mathfrak X^\mathrm{cn}))\\
\mathscr F &\mapsto &\tau_{\ge *}(\mathscr F),
\end{eqnarray*}
recovering the usual functors $M\mapsto \tau_{\ge *}(M)$ for complete modules $M\in \Mod_A^\mathrm{cplt}$ for any adic $\E$-ring $A$ when applied to the formal affine $\mathfrak X=\Spf(A)$.

\begin{remark}
Note that by Corollary \ref{Acn theorem QCoh}, the value of the Postnikov filtration functor $\tau_{\ge *}$ can equivalently be taken to be $\mathrm{Fil}(\QCoh(\mathfrak X^\mathrm{cn}))$. In particular, there are canonical maps
$$
\varinjlim_n\tau_{\ge n}(\mathscr F)\to \mathscr F
$$
regardless of whether you interpret this as taking place in $\QCoh(\mathfrak X^\mathrm{cn})$ or $\mathfrak{QC}\mathrm{oh}(\mathfrak X^\mathrm{cn})$. Indeed, the difference between the two $\i$-categories occurs by Corollary \ref{Acn theorem QCoh} only at the subcategory $\mathfrak{QC}\mathrm{oh}(\mathfrak X^\mathrm{cn})_{\le -\infty}$, measuring the failure of the $t$-structure to be right separated.
\end{remark}

\subsection{Quasi-coherent sheaves on complex periodic formal spectral stacks}
Our next goal is to establish flat descent  for the complete modules over aidc $\E$-ring which are also complex periodic. This is the second of the two affirmative answers, alluded to in the discussion preceeding Proposition \ref{Formal and informal QCoh}.

\begin{theorem}\label{Descent for complex periodic formal QCoh}
The functor $\CAlg{}_\mathrm{ad}^{\mathbf C\mathrm p}\to \CAlg(\mathrm{Pr^L})$ given by $A\mapsto\Mod_A^\mathrm{cplt}$ satisfies fpqc descent.
\end{theorem}

This allows us to dispense with  the issues concerning the quasi-coherent sheaf functor $\mathfrak{QC}\mathrm{oh}$ of Construction \ref{Cons of formal QCoh}, without needing to adopt any connectivity hypotheses. In particular, when restricted to the complex periodic setting, the functor $\mathfrak X\mapsto \mathfrak{QC}\mathrm{oh}(\mathfrak X)$ takes all small colimits of formal spectral stacks to the corresponding limits of $\i$-categories.

We will give a proof of Theorem \ref{Descent for complex periodic formal QCoh} at the end of this section. First, we will need to undertake a closer study of modules over a complex periodic $\E$-ring. The following discussion was inspired by \cite[Section 2]{PvK}.

\begin{cons}\label{Const of Bott map}
Recall that a complex periodic $\E$-ring $A$ always carries a line bundle $\omega{}_{\w{\G}{}^{\mathcal Q}_A}\in \Mod_A$ and a Bott isomorphism $\omega{}_{\w{\G}{}^{\mathcal Q}_A}\simeq \Sigma^{-2}(A)$. It can be seen to arise as the dualizing line of the Quillen formal group of $A$, but it being a line bundle also follows
directly from the assumption of $2$-periodicity. On the connective cover, $\tau_{\ge 0}(A)$, this manifests as a canonical $\tau_{\ge 0}(A)$-linear \textit{Bott map}
$$
\beta :\Sigma^2(\omega)\to \tau_{\ge 0}(A)
$$
where $\omega :=\tau_{\ge 0}(\omega{}_{\w{\G}{}^{\mathcal Q}_A})$ is a projective $\tau_{\ge 0}(A)$-module of rank $1$. Note that we can also express this line bundle as $\omega\simeq\Sigma^{-2}\tau_{\ge 2}(A)$.
\end{cons}

For $A$ complex periodic,
the connective cover map $\tau_{\ge 0}(A)\to A$ exhibits an equivalence of $\E$-rings $A\simeq \tau_{\ge 0}(A)[\beta^{-1}]$ with the localization along the Bott map $\beta$.
More generally, we have a canonical identification
\begin{equation}\label{Periodic modules equation}
M\,\simeq\, \tau_{\ge 0}(M)[\beta^{-1}]\,\simeq \, \tau_{\ge 0}(M)\otimes_{\tau_{\ge 0}(A)}A
\end{equation}
for any $A$-module $M$. The passage to the connective cover therefore looses no information for a module over a complex periodic $\E$-ring. To systematize this observation, we isolate the particular class of connective $\tau_{\ge 0}(A)$-modules, which we will show in Proposition \ref{Lemma essential image of connective cover noncomplete} arises from $A$-modules via connective covers.

\begin{definition}\label{Def of periodic modules over connective cover}
Let $A$ be a complex periodic $\E$-ring. A connective $\tau_{\ge 0}(A)$-module $N$ is \textit{periodic} if the Bott map
$$
\beta:\Sigma^2(\omega)\otimes_{\tau_{\ge 0}(A)} N\to  N
$$
exhibits an equivalence of $\tau_{\ge 0}(A)$-modules
$\Sigma^2(\omega)\otimes_{\tau_{\ge 0}(A)}N\simeq \tau_{\ge 2}(N)$. Let us denote by
$
\mathrm{Mod}{}_{\tau_{\ge 0}(A)}^{\mathrm{per}}\subseteq\Mod_{\tau_{\ge 0}(A)}^\mathrm{cn},
$
the full subcategory spanned by all connective $\tau_{\ge 0}(A)$-modules.
\end{definition}

\begin{prop}\label{Lemma essential image of connective cover noncomplete}
Let $A$ be a complex periodic $\E$-ring. In that case, the connective cover functor $\tau_{\ge 0}:\Mod_A\to \Mod_{\tau_{\ge 0}(A)}^\mathrm{cn}$ is fully faithful and its essential image consists of the full subcategory of periodic $\tau_{\ge 0}(A)$-modules.
\end{prop}

\begin{proof}
For ease of writing, let us set $R:=\tau_{\ge 0}(A)$.
Noting that we have $M\simeq \tau_{\ge 0}(M)\otimes_{R}A$ for any $A$-module $M$, we get via the universal properties of base change and the connective cover canonical equivalences
\begin{eqnarray*}
\Map_{\Mod_A}(M, M') &\simeq& \Map_{\Mod_A}(\tau_{\ge 0}(M)\otimes_{R}A, M')\\
&\simeq& \Map_{\Mod_{R}}(\tau_{\ge 0}(M), M')\\
&\simeq& \Map_{\Mod^\mathrm{cn}_{R}}(\tau_{\ge 0}(M), \tau_{\ge 0}(M'))
\end{eqnarray*}
for all $A$-modules $M, M'$. The connective cover functor is therefore indeed fully faithful on $A$-modules. To identify its essential image, note that
$$
\tau_{\ge 2}(M)\,\simeq\, \Sigma^2(\tau_{\ge 0}(\Sigma^{-2}(M)))\,\simeq \,\Sigma^2(\tau_{\ge 0}(\omega{}_{\w{\G}{}^{\mathcal Q}_A}\otimes_A M))\,\simeq\, \Sigma^2(\omega)\otimes_{R}\tau_{\ge 0}(M),
$$
holds for any $M\in\Mod_A$, where the final equivalence is a consequence of $\omega{}_{\w{\G}{}^{\mathcal Q}_A}$ being a flat $A$-module. Since said equivalence of $R$-modules is indeed exhibited by $\beta$, we conclude that $\tau_{\ge 0}(M)\in \mathrm{Mod}{}_{R}^{\mathrm{per}}$ as desired.

Conversely, if $N\in \mathrm{Mod}{}_{R}^{\mathrm{per}}$, then by induction, the map $\beta^n$ exhibits an equivalence of $R$-modules $\Sigma^{2n}(\omega^{\otimes n})\o_{R}N\simeq \tau_{\ge 2n}(N)$ for all $n\ge 0$.
We thus obtain from the telescopic formula for $\beta$-localization that the connective cover of the $\beta$-localization of $N$ may be expressed as
\begin{eqnarray*}
\tau_{\ge 0}(N[\beta^{-1}])& \simeq& \varinjlim\big(\tau_{\ge 0}(N)\xrightarrow{\beta}\tau_{\ge 0}(\Sigma^{-2}(\omega^{\o-1}\o_{R}N))\xrightarrow{\beta}\tau_{\ge 0}(\Sigma^{-4}(\omega^{\o {-2}}\o_{R}N))\xrightarrow{\beta}\cdots\big)\\
&\simeq& \varinjlim\big(\tau_{\ge 0}(N)\xrightarrow{\beta}\Sigma^{-2}(\omega^{\o -1}\o_{R}\tau_{\ge 2}(N))\xrightarrow{\beta}\Sigma^{-4}(\omega^{\o -2}\o_{R}\tau_{\ge 4}(N))\xrightarrow{\beta}\cdots\big)\\
&\simeq&
\varinjlim(N\xrightarrow{1} N\xrightarrow{1} N\xrightarrow{1}\cdots)\,\,\simeq\,\,N.
\end{eqnarray*}
It follows that if $N$ is in the connective cover of the $A$-module $N[\beta^{-1}]\simeq N\otimes_{R}A$, which verifies  the identification of the essential image of $\tau_{\ge 0}:\Mod_A\to\Mod_{R}^\mathrm{cn}$ with the full subcategory $\mathrm{Mod}{}_{R}^{\mathrm{per}}\subseteq \Mod_{R}^\mathrm{cn}$.
\end{proof}

\begin{remark}
The inclusion $\mathrm{Mod}{}_{\tau_{\ge 0}(A)}^{\mathrm{per}}\subseteq\Mod_{\tau_{\ge 0}(A)}$ map between two stable $\i$-categories, but it is no exact. Inside the latter $\i$-category, we have $\Sigma^{2n}(N)\simeq \omega^{\otimes - n}\otimes_{\tau_{\ge 0}(A)}N$ for any $n\in\mathbf Z$. In particular, although the mapping anima of objects $N, N'\in \mathrm{Mod}{}_{\tau_{\ge 0}(A)}^{\mathrm{per}}$ coincide whether taken inside $\mathrm{Mod}{}_{\tau_{\ge 0}(A)}^{\mathrm{per}}$ or in $\Mod_{\tau_{\ge 0}(A)}$, this is not true of the mapping spectra. In the former stable $\i$-category, these are rather computed by the formula
$$
\underline{\Map}_{\mathrm{Mod}{}_{\tau_{\ge 0}(A)}^{\mathrm{per}}}(N, N')
\,\simeq\, \varinjlim_n \Sigma^{\infty -2n}(\Map_{\tau_{{\ge 0}(A)}}(N, N' \otimes_{\tau_{\ge 0}(A)}
\omega^{\otimes - n})),
$$
with the transition maps in the sequential colimit on the right-hand side given by the Bott map $\beta$.
\end{remark}

In the presence of an adic topology of $A$, we let $\Mod_{\tau_{\ge 0}(A)}^{\mathrm{cplt}, \mathrm{per}}\subseteq\Mod_A$ denote the full subcategory of those $\tau_{\ge 0}(A)$-modules which are both complete and periodic in the sense of Definition \ref{Def of periodic modules over connective cover}.

\begin{corollary}\label{Lemma essential image of connective cover complete}
Let $A$ be a complex periodic adic $\E$-ring $A$. The connective cover functor restricts to an equivalence of $\i$-categories
$$
\Mod_A^\mathrm{cplt}\simeq \Mod_{\tau_{\ge 0}(A)}^{\mathrm{cplt}, \mathrm{per}}
$$
with the inverse given by $\beta$-localization, i.e.\ base-change along $\tau_{\ge 0}(A)\to A$.
\end{corollary}

\begin{proof}
Recall that by \cite[Theorem 7.3.4.1]{SAG}, a module over an adic $\E$-ring is complete if and only if all of its homotopy groups are. In light of the periodicity assumption on $A$, it follows that an $A$-module $M$ is complete if and only if its connective cover $\tau_{\ge 0}(M)$ is a complete $\tau_{\ge 0}(A)$-module. The equivalence of $\i$-categories of Proposition \ref{Lemma essential image of connective cover noncomplete}
$$
\tau_{\ge 0}:\Mod_A\simeq \Mod_{\tau_{\ge 0}(A)}^{\mathrm{per}}
$$
therefore restricts to an equivalence between the subcategories of complete objects on both sides.
\end{proof}

\begin{remark}
It follows from the proof of Corollary \ref{Lemma essential image of connective cover complete} that for any $A\in \CAlg_\mathrm{ad}^{\mathbf C\mathrm p}$ and $M\in\Mod_A^\mathrm{cplt}$, we may identify the $A$-modules
$
M\simeq \tau_{\ge 0}(M)\o_{\tau_{\ge 0}(A)}A\simeq \tau_{\ge 0}(M)\,\widehat{\o}_{\tau_{\ge 0}(A)}\, A
.$
That is to say, the colimit in the telescopic formula for $\beta$-localization may in this case be taken in either of the $\i$-categories $\Mod_A$ or $\Mod_A^\mathrm{cplt}$.
\end{remark}

In the next two lemmas, we establish the compatibility of completed base-change along adically faithfully flat maps, in the sense of Definition \ref{Def of adic fpqc} and at with some additional assumptions, with connective covers of complete modules. This is analogous to the usual property of  flat maps of $\E$-rings in the non-complete case.

\begin{lemma}\label{Base-change connective 3.54}
Let $R\to S$ be a morphism of connective adic $\E$-rings which is adically faithfully flat. For any almost connective complete $S$-module $N$, the canonical map
$$\tau_{\ge 0}(N)\,\widehat{\o}_R S\to \tau_{\ge 0}(N\,\widehat{\o}_R S)$$
is an equivalence of $S$-modules.
\end{lemma}

\begin{proof}[Proof 1 (Appeal to SAG)]
It is a general fact that if $f:\mX\to\mY$ is a flat map in $\mathrm{SpStk}^\mathrm{cn}$, then the pullback functor $f^*:\QCoh(\mY)\to\QCoh(\mX)$ is compatible with connective covers, in the sense that $f^*(\tau_{\ge 0}(\sF))\simeq \tau_{\ge 0}(f^*(\sF))$ holds for any quasi-coherent sheaf $\sF$ on $\mY$. This may be proved by reducing to the case of affines via descent. where it is a standard property of flat maps.
The result in question is a direct application of this fact to the case $\mX=\Spf(S)$, $\mY=\Spf(R)$, and $\sF$ correspondign to $N$ in light of Corollary \ref{Acn theorem QCoh}.
\end{proof}

\begin{proof}[Proof 2 (Direct computation)]
Since $N$ is almost connective, so is the base-change $N\o_R S$. Its completion may therefore be expressed as the limit
$$
N\,\widehat{\o}_R\,S\,\simeq\, (N\otimes_R S)^\wedge\,\simeq\,\varprojlim_m (N\o_R S)\o_S S_m
\,\simeq\, \varprojlim_m (N\o_R R_m)\o_{R_m} S_m,
$$
where $\Spf(R)\simeq \varinjlim_m \Spf(R_m)$ and $\Spf(S)\simeq \varinjlim_m \Spf(S_m)$ are the respective sequential presentations of formal affines from Remark \ref{Remark adic topology}.
The connective cover functor $\tau_{\ge 0}$ is a right adjoint and as such commutes with limits, giving rise to the first equivalence in
\begin{equation}\label{Equation 3.39}
\tau_{\ge 0}(N\,\widehat{\o}_R\,S)\,\simeq\, \varprojlim_m \tau_{\ge 0}((N\o_R R_m)\o_{R_m} S_m)\, \simeq \,\varprojlim_m \tau_{\ge 0}(N\o_R R_m)\o _{R_m}S_m,
\end{equation}
with the second coming from the fact that the $\E$-ring map $R_m\to S_m$ is, by the definition of adically faithful flatness of $R\to S$, faithfully flat. 
Since $N$ being complete implies that
$$
\varprojlim_m \tau_{\ge 0}(N\o_R R_m)\,\simeq\, \tau_{\ge 0}(\varprojlim_m N\o_R R_m)\,\simeq\, \tau_{\ge 0}(N^\wedge)\,\simeq\, \tau_{\ge 0}(N)
$$
and so the right-most term in equation \eqref{Equation 3.39} may be identified as the completed tensor product $\tau_{\ge 0}(N)\,\widehat{\o}_R\, S$.
\end{proof}

\begin{lemma}\label{Lemma completed basehcange complete}
Let $A\to B$ be a map of adically faithfully flat map of complex periodic adic $\E$-rings and $M\in\Mod_A$. The canonical map $\tau_{\ge 0}(M)\otimes_{\tau_{\ge 0}(A)}\tau_{\ge 0}(B)\to \tau_{\ge 0}(M\o_A B)$ is an equivalence of $\tau_{\ge 0}(B)$-modules.
\end{lemma}

\begin{proof}
For ease of writing, let us denote $R:=\tau_{\ge 0}(A)$ and $S:=\tau_{\ge 0}(B)$. 

First, we claim that the map of $\E$-rings $A\to B$ being adically faithfully flat implies the same thing for the induced map on the connective covers $R\to S$. Indeed, these two adically faithfully flatness statements amount to demanding that the respectively spectral stack maps $\Spf(B)\to\Spf(B)$ and $\Spf(S)\to\Spf(R)$ are affine and faithfully flat. It follows from the pullback square of Remark \ref{Pullback square} that $\Spf(A)^\mathrm{cn}\simeq \Spf(R)$ and $\Spf(B)^\mathrm{cn}\simeq \Spf(S)$. The claim is now a special case of the fact that if an affine map of nonconnective spectral stacks $\mX\to \mY$ is faithfully flat, then the map of connective spectral stacks $\mX^\mathrm{cn}\to\mY^\mathrm{cn}$ is too. This finally follows from the class of faithfully flat $\E$-ring maps evidently being closed under the passage to connective covers.

Next, note that we have by \eqref{Periodic modules equation} canonical $S$-module equivalences
$$
\tau_{\ge 0}(M\o_A B)\,\simeq\, \tau(
(\tau_{\ge 0}(M)\otimes_R S)[\beta^{-1}])\,\simeq \,
\varinjlim_n\tau_{\ge 0}(\tau_{\ge 0}(M){\o}_R S\o_R \Sigma^{-2n}(\omega^{\o-n})),
$$
where the final step used that the $t$-structure on $S$-modules is compatible with filtered colimits.
By passing to completions, which connective covers are compatible with and which takes colimits in $\Mod_S$ to colimits in $\Mod_S^\mathrm{cplt}$, we may replace $\o$ everywhere with $\widehat{\o}$. Further using Lemma \ref{Base-change connective 3.54}, applied to the adically faithfully flat map $R\to S$, we obtain the first equivalence in the chain
\begin{eqnarray*}
\tau_{\ge 0}(M\,\w{\o}_A\, B) &\simeq &
\varinjlim_n\tau_{\ge 0}\big(\tau_{\ge 0}(M)\,\widehat{\o}_R \,  \Sigma^{-2n}(\omega^{\o -n})\big )\,\widehat{\o}_R\, S\\
&\simeq & \tau_{\ge 0}\big(\varinjlim_n\tau_{\ge 0}(M)\,\widehat{\o}_R \,  \Sigma^{-2n}(\omega^{\o -n})\big )\,\widehat{\o}_R\, S \\
&\simeq & \tau_{\ge 0}\big(\tau_{\ge 0}(M)[\beta^{-1}]\big )\,\widehat{\o}_R\, S \\
&\simeq & \tau_{\ge 0}(M)\,\widehat{\o}_R\, S 
\end{eqnarray*}
in which we also used that the completed tensor product commutes with small colimits in each variable, and 
finished up with another application of \eqref{Periodic modules equation}.
\end{proof}

\begin{proof}[Proof of Theorem \ref{Descent for complex periodic formal QCoh}]
In light of the definition of the functor $\mathfrak{QC}\mathrm{oh}$ in Construction \ref{Cons of formal QCoh}, we must show that the functor $\CAlg^{\mathbf C\mathrm p}_\mathrm{ad}\to \mathrm{Cat}_\i$, given by $A\mapsto \Mod^\mathrm{cplt}_A$, satisfies fpqc descent.
Let $\mC$ be the non-full wide subcategory of $\CAlg^{\mathbf C\mathrm p}_\mathrm{ad}$, consisting of all the same objects, where the anima of morphisms are taken to be the unions of all connected components which correspond to adically faithfully flat maps. Since faithfully flat descent only concerns adically faithfully flat maps, it suffices to verify descent for the restriction
$$
\mC\to\CAlg_\mathrm{ad}^{\mathbf C\mathrm p}\xrightarrow{\Mod^\mathrm{cplt}}\mathrm{Cat}_\infty.
$$
Corollary \ref{Lemma essential image of connective cover complete} and Lemma \ref{Lemma completed basehcange complete} together give a natural equivalence between this functor and the one given by $A\mapsto \Mod_{\tau_{\ge 0}(A)}^{\mathrm{cplt}, \mathrm{per}}$. The subcategory inclusion $\Mod_{\tau_{\ge 0}(A)}^{\mathrm{cplt}, \mathrm{per}}\subseteq(\Mod_A^\mathrm{cplt})^\mathrm{cn}$
extends to an object-wise fully faithful natural transformation from this functor in question to the composite
$$
\mC\xrightarrow{\tau_{\ge 0}}\CAlg_\mathrm{ad}^\mathrm{cn}\xrightarrow{(\Mod^\mathrm{cplt})^\mathrm{cn}}\mathrm{Cat}_\infty,
$$
i.e.\ to restriction along the connective cover of the functor $\mathfrak{QC}\mathrm{oh}^\mathrm{cn}$. This latter functor is by Theorem \ref{Formal and informal QCoh} naturally equivalent to $\QCoh^\mathrm{cn}$, and so does indeed satisfy fpqc descent. From this, the desired fpqc descent result  follows from the observation that periodicity, in the sense of Definition \ref{Def of periodic modules over connective cover}, is an fpqc-local property for modules over a connective covers of an adic $\E$-ring.
\end{proof}

\subsection{$I$-adic formal spectral stacks}\label{Subsection I-adic context}

Let $R$ be an $\E$-ring with a finitely generated ideal $I\subseteq\pi_0(R)$. We may  consider any $\E$-algebra $A$ over $R$ as an adic $\E$-ring by equipping $\pi_0(A)$ with the $I$-adic topology. Equivalently, we have a pullback square
pullback square of nonconnective spectral stacks
\begin{equation}\label{Pullback square 3.9}
\begin{tikzcd}
\Spf(A) \ar[r] \ar[d] & \Spec(A)\ar[d] \\
\Spf(R) \ar[r] &\Spec(R).
\end{tikzcd}
\end{equation}
Let $\mathrm{FAff}_R^{I-\mathrm{ad}}\subseteq \mathrm{FAff}_{/\Spf(R)}$ denote the full subcategory spanned by formal affines of this sort, which we will call \textit{$I$-adic formal affines}. It is by defintion the opposite category of the $\i$-category $\CAlg^{\mathrm{cplt}}_R$ of $I$-complete $\E$-algebras over $R$.
The fpqc topology from $\mathrm{FAff}$ restricts to $\mathrm{FAff}_R^{I-\mathrm{ad}}$, giving it a Grothendieck topology.

\begin{definition}\label{Def of IFSpStk}
A functor $\mathfrak X:\CAlg^{\mathrm{cplt}}_R\to \mathrm{Ani}$ is a \textit{nonconnective $I$-adic formal spectral stack} if it is  accessible and satisfies fpqc descent. We let
$$
\mathrm{FSpStk}_I\, :=\,\mathrm{Shv}^\mathrm{acc}_\mathrm{fpqc}(\mathrm{FAff}_R^{I-\mathrm{ad}})
$$
denote the $\i$-category of nonconnective formal spectral stacks. 
The analogous connective and classical variants are defined as
$$
\mathrm{FSpStk}_I^\mathrm{cn}\, :=\,\mathrm{Shv}^\mathrm{acc}_\mathrm{fpqc}((\mathrm{FAff}^{I-\mathrm{ad}}_{\tau_{\ge 0}(R)})^\mathrm{cn}),
\qquad \mathrm{FStk}_I^\heart\, :=\,\mathrm{Shv}^\mathrm{acc}_\mathrm{fpqc}((\mathrm{FAff}^{I-\mathrm{ad}}_{\tau_{\ge 0}(R)})^\heart).
$$
\end{definition}

\begin{remark}\label{Reason for cavaliers}
The difficulties of the adic context, alluded to in Remark \ref{What _is_ Shv_acc here?}, disappear in the $I$-adic context. That is to say, unlike $\CAlg_\mathrm{ad}$, the $\i$-category $\CAlg_R^{\mathrm{cplt}}$ is accessible. Indeed, for any $\kappa$, it is generated under $\kappa$-filtered colimits by the completed free $R$-algebras $\mathrm{Sym}^*_R(R[K])^\wedge_I$ for $\kappa$-small anima $K$.  There are consequently no issues in making sense of accessible presheaves on $I$-adic formal affines.
\end{remark}

There is an obvious fully faithful embedding
\begin{equation}\label{From adic to FAG}
\mathrm{FSpStk}_I\hookrightarrow \mathrm{FSpStk}_{/\Spf(R)}.
\end{equation}
As such, we can apply all the constructions defined in the latter context to $I$-adic formal spectral stacks too, for instance the functors $\mathfrak{QC}\mathrm{oh} :\mathrm{FSpStk}_I^\mathrm{op}\to\CAlg(\mathrm{Pr}^\mathrm L)$ and $\mathfrak O :\mathrm{FSpStk}_I^\mathrm{op}\to \CAlg$ (though they do not necessarily commute with limits unless we either adopt a connectivity assumption and invoke Proposition \ref{Formal and informal QCoh}, or impose complex periodicity and use Theorem \ref{Descent for complex periodic formal QCoh}). We also have the connective covers and underlying classical stacks $\mathfrak X\mapsto \tau_{\ge 0}(\mathfrak X)$ and $\mathfrak X \mapsto\mathfrak X^\heart$, respectively defining functors
$$\mathrm{FSpStk}_I\xrightarrow{\tau_{\ge 0}}\mathrm{FSpStk}_I^\mathrm{cn}, \qquad\quad\mathrm{FSpStk}_I\xrightarrow{(-)^\heart}\mathrm{FStk}_I^\heart,
$$
and which also participate in the expected adjunctions when $R$ is assumed to be connective.

\begin{remark}\label{Remark T-adic prep}
Through the fully faithful embedding \eqref{From adic to FAG}, an $I$-adic formal affine $\Spf(A)$ identifies with the full subcategory of $\mathrm{FSpStk}_{/\Spf(R)}$, spanned by all the morphisms into $\Spf(R)$ which are formally affine (defined in an obvious formal-setting analogue of Section \ref{Section affine morphisms}). Indeed, a relative version of the formal spectrum gives rise to the equivalences of $\i$-categories
$$
\mathrm{FSpStk}^\mathrm{faff}_{/\Spf(R)} \simeq
\CAlg(\mathfrak{QCoh}(\Spf(R)))^\mathrm{op}\simeq \CAlg(\Mod_R^\mathrm{cplt})^\mathrm{op}\simeq (\CAlg_R^\mathrm{cplt})^\mathrm{op}
\simeq
\mathrm{FAff}_{/\Spf(R)}^{I-\mathrm{ad}}.
$$
This suggests  an obvious variant of the theory of $I$-adic formal stacks, where we replace the affine base $\Spec(R)$ -- or more precisely, its completion $\Spf(R^\wedge_I)$ -- with 
an arbitrary base formal spectral stack $\mathfrak X$. Namely, we take $\mathrm{FAff}^\mathrm{faff}_{/\mathfrak X}\subseteq\mathrm{FAff}_{/\mathfrak X}$ to be the full subcategory spanned by those maps from formal affines into $\mathfrak X$ which are formally affine morphisms, equip it with the inherited fpqc topology, and set the $\i$-category of \textit{adic formal spectral stacks over $\mathfrak X$} to be
$$
\mathrm{FSpStk}^\mathrm{ad}_{/\mathfrak X}:=\mathrm{Shv}^\mathrm{acc}_\mathrm{fpqc}(\mathrm{FAff}_{/\mathfrak X}^\mathrm{faff}).
$$
The theory over a formally affine base, as we are discussing it in the rest of this section, may be recovered as $\mathrm{FSpStk}_I\simeq \mathrm{FSpStk}_{/\Spf(R^\wedge_I)}^\mathrm{ad}$.
\end{remark}

There is a canonical way to pass from spectral stacks over $R$ to $I$-adic formal spectral stacks by the process of $I$-adic formal completion. This  operation, which we now describe, globalizes the $I$-completion $A\mapsto A^\wedge_I$, i.e.\ extends it to not-necessarily-affine spectral stacks.

\begin{lemma}\label{Fpqc preservation of completion}
The $I$-completion functor $\CAlg_R\to\CAlg^\mathrm{cplt}_R$, given by $A\mapsto A^\wedge_I$, preserves fpqc covers and is accessible.
\end{lemma}

\begin{proof}
For any $\E$-algebra $A$ over $R$, the pullback square \eqref{Pullback square 3.9} provides
the identification of the corresponding $I$-adic formal affine as $\Spf(A^\wedge_I)\simeq \Spec(A)\times_{\Spec(R)}\Spf(R)$. It follows that for any map of $\E$-algebras $A\to B$ over $R$, the corresponding map of spectral stacks $\Spf(B^\wedge_I)\to\Spf(A^\wedge_I)$ is therefore base-changed along the map $\Spf(R)\to\Spec(R)$ from the map of affines $\Spec(B)\to\Spec(A)$. Since faithfully flat are closed under pullbacks, we see that $A^\wedge_I\to B^\wedge_I$ is $I$-adically faithfully flat if $A\to B$ is faithfully flat. This proves that the $I$-completion functor preserves fpqc covers.
The accessibility claim is clear from completion being a left adjoint and as such commuting with not only filtered but in fact all small colimits.
\end{proof}

\begin{cons}[$I$-adic formal completion]\label{Const of I-completion of stacks}
Recall there is a canonical equivalence of $\i$-categories
$$
\mathrm{SpStk}_{/\Spec(R)}\simeq \mathrm{Shv}^\mathrm{acc}_\mathrm{fpqc}(\CAlg_R^\mathrm{op}),
$$
given by sending a map of spectral stacks $\mX\to \Spec(R)$ to its relative functor of points $\mX:\CAlg_R\to\mathrm{Ani}$. Using this identification in tandem with Lemma \ref{Fpqc preservation of completion},
we find that the $I$-completion functor $\mu:\CAlg_R\to\CAlg^\mathrm{cplt}_R$ gives rise to an adjunction between the $\i$-categories of accessible fpqc sheaves which we may view as
$$
\mu_!:\mathrm{SpStk}_{/\Spec(R)}\rightleftarrows \mathrm{FSpStk}_I :\mu^*.
$$
For a map of sectral stacks $\mX\to\Spec(R)$, we view the resulting $I$-adic formal spectral stack
$
\mX^\wedge_I:= \mu_!\mX
$
as the \textit{$I$-adic formal completion of $\mX$}.
\end{cons}

\begin{remark}
The $I$-adic formal completion of a spectral stack $\mX$ over $R$ may be explicitly expressed using the colimit formula
$$
\mX^\wedge_I \,\simeq\, \varinjlim_{\Spec(A)\in \mathrm{Aff}_{/\mX}}\Spf(A^\wedge_I).
$$
In particular, we have $\Spec(A)^\wedge_I\simeq \Spf(A^\wedge_I)$ for affines.
\end{remark}

\begin{remark}
If $R$ is assumed to be connective, we may obtain as a special case of Theorem \ref{Theorem formal stacks = stacks} a canonical identification $\mathrm{FSpStk}^\mathrm{cn}_I\simeq \mathrm{SpStk}^\mathrm{cn}_{/\Spf(R)}$. In terms of it, $I$-adic formal completion is given by the base-change
$$
\mX^\wedge_I\simeq \mX\,\times_{\Spec(R)}\,\Spf(R)
$$
for any connective spectral stack map $\mX\to \Spec(R)$.
\end{remark}

We conclude this section with providing a simpler description of adically faithfully flat maps in the $I$-adic context, at least with the additional assumption of even periodicity. In this case, we show that they coincide with $I$-completely eff maps from Definition \ref{Def of I-complete evf}

\begin{prop}\label{Lemma ff comparison}
Let $I\subseteq \pi_0(R)$ be a regular locally principal ideal.
Then a map $A\to B$ of even periodic $I$-complete $\E$-rings of bounded $I$-torsion is adically faithfully flat if and only if the $\E$-ring map
$$
\pi_0(A)\o_{\pi_0(R)} \pi_0(R)/I\to \pi_0(B)\o_{\pi_0(R)} \pi_0(R)/I
$$
is faithfully flat.
\end{prop}

\begin{proof}
The $\E$-rings $A$ and $B$ are even periodic and hence also complex periodic, so they can be recovered from their connective covers as localizations with respect to the Bott maps $A\simeq \tau_{\ge 0}(A)[\beta^{-1}]$ and $B\simeq \tau_{\ge 0}(B)[\beta^{-1}]$.The canonical $\E$-ring map $\tau_{\ge 0}(B)\o_{\tau_{\ge 0}(A)}A\to B$ is therefore an equivalence, and since $B$ is complete, so is $\tau_{\ge 0}(B)\,\widehat{\o}_{\tau_{\ge 0}(A)}\,A\to B$. The commutative square
$$
\begin{tikzcd}
\Spf(B) \ar[r] \ar[d] & \Spf(\tau_{\ge 0}(B))\ar[d] \\
\Spf(A) \ar[r] & \Spf(\tau_{\ge 0}(A))
\end{tikzcd}
$$
is therefore a pullback in the $\i$-category $\mathrm{SpStk}$. Since faithfully flat maps are preserved under pullback, as well as  by passage to connective localizations, it  suffices to replace $A$ and $B$ with their connective covers.

Observing the diagram of spectral stacks
$$
\begin{tikzcd}
\Spf(B) \ar[r] \ar[d] & \Spec(B)\ar[d] \\
\Spf(A) \ar[r] \ar[d] & \Spec(A)\ar[d] \\
\Spf(R) \ar[r] &\Spec(R)
\end{tikzcd}
$$
in which the bottom square and the outter square are both pullback squares, it follows from the ``pasting law for pullbacks" that the top square is a pullback square as well. That is to say, $\mathrm{Spf}(B)\to\Spf(A)$ is always an affine morphism, which is one part of a requirement for an adically faithfully flat map.

Since $f:\Spf(B)\to\Spf(A)$ is an affine morphism, it is faithfully flat if and only if the quasi-coherent sheaf of $\sO_{\Spf(A)}$-algebras $f_*(\sO_{\Spf(B)})$ is faithfully flat, i.e.\ if we have
$$
\pi_*(f_*(\sO_{\Spf(B)}))\simeq \pi_*(\sO_{\Spf(A)})\otimes_{\pi_0(\sO_{\Spf(A)})}\pi_0(\sO_{\Spf(B)})
$$
and the quasi-coherent sheaf $\pi_0(f_*(\sO_{\Spf(B)}))$ is a faithfully flat over $\Spf(A)^\heart$.
This may be rephrased under the equivalence of $\i$-categories $\QCoh(\Spf(A))^\mathrm{cn}\simeq \Mod_A^{\mathrm{cplt}, \mathrm{cn}}$ from \eqref{Formal affineness formula 3.42} (which we may invoke due to having passed to connective covers, and knowing that $f$ is affine meaning that $f_*$ is $t$-exact) as asking that
$$
\pi_*(B)\simeq \pi_*(A)\,\widehat{\o}_{\pi_0(A)}\,\pi_0(B)
$$
and that $\pi_0(B)$ is faithfully flat over $\Spf(\pi_0(A))$. Because $A$ and $B$ are both even periodic, the first requirement is automatically satisfied. That is to say, the map of spectral stacks $\Spf(B)\to\Spf(A)$ is faithfully flat if and only if $\Spf(\pi_0(B))\to\Spf(\pi_0(A))$ is.

It remains to identify what it means for $\pi_0(A)\to \pi_0(B)$ to be adically complete, i.e.\ for the map of classical stacks $\Spf(\pi_0(B))\to\Spf(\pi_0(A))$ to be faithfully flat. Since both $\E$-rings are discrete (i.e.\ static in Newspeak), we may replace $R$ by $\pi_0(R)$ with no loss of generality. Using the $\E$-version of the adic filtration as in Remark \ref{Remark adic topology}, let us write $\Spf(R)\simeq \varinjlim_n \Spec(R_n)$, from which we get
$$
\Spf(\pi_0(A))\simeq \varinjlim_n \Spec(\pi_0(A)\o_R R_n).
$$
It follows that $\Spf(\pi_0(B))\to\Spf(\pi_0(A))$ is faithfully flat if and only if  the $\E$-ring map $\pi_0(A)\otimes_R R_n\to \pi_0(B)\otimes_R R_n$ is for all $n$. We assumed that $I$ is a regular locally principal ideal, so (after possibly Zariski localizing $R$) we can set $I=(x)$ for some non-zero-divisor $x\in R$. Then it follows from the proof of \cite[Lemma 8.1.2.2]{SAG} that $R_n\simeq R/x^n$, i.e.\ the $\E$-version of the $x$-adic filtration agrees with the classical one. For $n = 1$, we obtain the desired condition that $\pi_0(A)\o_R R/x\to \pi_0(B)\o_R R/x$ is an equivalence, and thanks to $A$ -- hence equivalently $\pi_0(A)$ -- having bounded $I$-torsion, it holding for $n=1$ implies  that it holds for all $n\ge 2$ as well by induction.
\end{proof}

\begin{remark}\label{Remark to L or not to L, 2}
Keep in mind that our tensor products are always formed $\i$-categorically in spectra. For classical rings, that means these are the derived tensor products, so that the condition in the statement of Prop \ref{Lemma ff comparison} might be more classically written as the faithful flatness of $\pi_0(A)\o^\mathrm L_{\pi_0(R)}\pi_0(R)/I \to \pi_0(B)\o^\mathrm L_{\pi_0(R)}\pi_0(R)/I$.
\end{remark}

\begin{corollary}\label{adic fpqc = I-complete fpqc}
A map of even periodic $I$-complete $\E$-algebras $A\to B$ over $R$ is adically faithfully flat if and only if it is $I$-completely faithfully flat in the sense of Definition \ref{Def p-eff}.
\end{corollary}

\begin{proof}
The map $A\to B$ being $I$-completely faithfully flat is equivalent to the map of commutative $\pi_0(R)$-algebras $\bigoplus_{n\in \mathbf Z}\pi_n(A)\to \bigoplus_{n\in \mathbf Z}\pi_n(B)$ inducing a faithfully flat map upon base-change along the quotient map $\pi_0(R)\to \pi_0(R)/I$. By the evenness assumption, this may be rewritten (after choosing any $\E$-ring structure on $\mathrm{MUP}$) as asking that $\pi_0(A\o_{\mathbf S}\mathrm{MUP})\to \pi_0(B\o_{\mathbf S}\mathrm{MUP})$ is fully faithful after tensoring with $\pi_0(R)/I$ over $\pi_0(R)$. This is by Proposition \ref{Lemma ff comparison} the same as asking $(A\o_{\mathbf S}\mathrm{MUP})^\wedge_I\to (B\o_{\mathbf S}\mathrm{MUP})^\wedge_I$ being adically faithfully flat, or in terms of Definition \ref{Def of adic fpqc} for the map of spectral stacks
$$
\Spf((B\o_{\mathbf S}\mathrm{MUP})^\wedge_I)\to \Spf((A\o_{\mathbf S}\mathrm{MUP})^\wedge_I)
$$
to be fully faithful. Note that this is the pullback of the map $\Spf(B)\to\Spf(A)$ along the fpqc cover $\Spec(\mathrm{MUP})\to \mM$, where the canonical (and essentially unique) maps of the formal spectra in question to the chromatic base stack $\mM$ come from the composite
$$
\Spf(A)\to\Spec(A)\to \mM,
$$
whose right-hand map exists due to the assumption that $A$ is even periodic and hence also complex periodic. Recalling finally that the property of being a faithfully flat map is itself fpqc local, we see that faithful flatness after base-change along the fpqc cover $\Spec(\mathrm{MUP})\to \mM$ is equivalent to a map of spectral stacks to be faithfully flat itself. Thus we have translated $I$-compact faithful flatness of $A\to B$ into faithful flatness of the map of spectral stacks $\Spf(B)\to \Spf(A)$. But by Definition \ref{Def of adic fpqc}, that is precisely what it means for $A\to B$ to be adically faithfully flat.
\end{proof}

\begin{corollary}\label{Adic = complete ff}
Let $R=\mathbf S^\wedge_p$, equipped with the $p$-adic topology. A map of $p$-complete even periodic $\E$-rings of bounded $p$-torsion $A\to B$ is adically faithfully flat if and only if it is $p$-completely faithfully flat in the sense of \cite[Definition 2.2.6]{HRW}.
\end{corollary}

Most of the proof of Lemma \ref{Lemma ff comparison} given above did not make use of the full strength of the even periodicity assumption. In the complex periodic, the same argument gives:

\begin{corollary}\label{Lemma ff comparison: cplx per}
Let $I\subseteq \pi_0(R)$ be a regular locally principal ideal.
Then a map $A\to B$ of complex periodic $I$-complete $\E$-rings of bounded $I$-torsion is adically faithfully flat if and only if the $\E$-ring map
$$
\pi_0(A)\o_{\pi_0(R)} \pi_0(R)/I\to \pi_0(B)\o_{\pi_0(R)} \pi_0(R)/I
$$
is faithfully flat, and the canonical  $\pi_0(B)$-module map
$$
 \pi_1(A)\,\widehat{\o}_{\pi_0(A)}\,\pi_0(B)\to \pi_1(B)
$$
is an isomorphism.

\end{corollary}

\begin{remark}\label{Remark adic ff homotopy formula}Going even further,
part of the proof of Proposition \ref{Lemma ff comparison} above made use of neither the even periodicity nor the bounded $I$-torsion assumption. That is, we saw that if $A\to B$ is an adically faithfully flat map of $I$-complete $\E$-algebras over $R$, then the multiplication map induces a canonical isomorphism 
of complete $\pi_0(B)$-modules
$$
\pi_n(B)\simeq \pi_n(A)\,\widehat{\o}_{\pi_0(A)}\,\pi_0(B)
$$ for all $n\in \mathbf Z$.
\end{remark}

\subsection{$I$-adic even periodization}
After having expended some effort so far in setting up formal spectral algebraic geometry, we can now make sense of even periodization in this context.

We continue to fix a base $\E$-ring $R$ and a regular locally principal ideal $I\subseteq\pi_0(R)$, i.e.\ an effective Cartier divisor $\Spec(\pi_0(R)/I)\subseteq\Spec(\pi_0(R))$.  
For any $I$-complete $\E$-algebra $A$ over $R$, we take its formal spectrum $\Spf(A)$  with respect to the $I$-adic topology. Unless stated otherwise, formal spectra $\Spf(A)$ for $A\in\CAlg_R^\mathrm{cplt}$ are to be interpreted with respect to the $I$-adic topology.

To imitate the formal development of even periodization from Section \ref{Section EVP}, we begin with an $I$-adic analogue of Lemma \ref{Lemma cover preservation and lifting}. In the present setting, it concerns the $\i$-category $\CAlg_I^\mathrm{evp}$ from Definition  \ref{Def of I-adic ev and evp rings}, which let us recall consists of all even periodic $I$-complete $\E$-algebras over $R$ with bounded $I$-torsion.

\begin{lemma}\label{Lemma prep I-adic for evp inclusion}
The inclusion $\CAlg^\mathrm{evp}_I\subseteq\CAlg^\mathrm{cplt}_R$ is accessible, and both preserves fpqc covers, as well as satisfying the cover lifting property for fpqc covers.
\end{lemma}

\begin{proof}
For the accessibility claim, see Proposition \ref{I-adic accessibility Lemma}.
 The fpqc cover preservation is definitional, and for the cover lifting property, it suffices to show that if $A\to B$ is an adically faithfully flat map in $\CAlg_R^\mathrm{cplt}$ and $A$ is even periodic, then so is $B$. This follows from Remark \ref{Remark adic ff homotopy formula}, by which we have $\pi_n(A)\,\widehat{\o}_{\pi_0(A)}\, \pi_0(B)\simeq \pi_n(B)$ for all $n\in \mathbf Z$.
\end{proof}

It follows that $\mathrm{FAff}^\mathrm{evp}_I\,:=\, (\CAlg_I^\mathrm{evp})^\mathrm{op}$ inherits a Grothendieck topology from the fpqc topology on $\mathrm{FAff}^{I-\mathrm{ad}}_R$.
At this point, we can formally imitate much of the development of the even periodization from
Section \ref{Section EVP} in the setting of $I$-adic formal spectral stacks.

\begin{definition}\label{Def of I-adic even stacks}
An \textit{even periodic $I$-adic formal spectral stack} is an accessible functor $\mathfrak X:\CAlg_I^\mathrm{evp}\to \mathrm{Ani}$ which satisfies fpqc descent. We let
$$\mathrm{FSpStk}_I^\mathrm{evp} :=\mathrm{Shv}^\mathrm{acc}_\mathrm{fpqc}(\mathrm{FAff}_I^\mathrm{evp})$$
denote the $\i$-category of even periodic $I$-adic formal spectral stacks.
\end{definition}

This is the even periodic version of the $\i$-category of $I$-adic formal spectral stacks $\mathrm{FSpStk}_I$ from
Definition \ref{Def of IFSpStk}.

\begin{cons}\label{Functoriality of ev I-version}
Let us denote by $\varepsilon_I:\CAlg_I^\mathrm{evp}\to\CAlg^\mathrm{cplt}_R$  the subcategory inclusion. Analogously to Construction \ref{Functoriality of ev} and thanks to Lemma \ref{Lemma prep I-adic for evp inclusion}, the functor $\varepsilon_I$ induces a double-adjunction on the $\i$-categories of accessible fpqc sheaves
$$
\xymatrix{
 \mathrm{FSpStk}_I^\mathrm{evp} \ar@<1.2ex>[r]^{\,\,\varepsilon_{I!}} \ar@<-1.2ex>[r]_{\,\,\varepsilon_{I*}}&  
 \mathrm{FSpStk}_I.\ar[l]|-{\varepsilon_I^*}
}
$$
The sheaf-level functor $\varepsilon_I^*$ is given by the restriction $\mathfrak X\mapsto\mathfrak X\vert_{\CAlg_I^\mathrm{evp}}$. The sheaf-level left adjoint $\varepsilon_{I!}$ is obtained as the left Kan extension along $\varepsilon_I$, followed by fpqc sheafification, which is in turn well-defined by virtue of being applied to an accessible presheaf.
\end{cons}

\begin{definition}\label{Def of evp localization, I-version}
The \textit{even periodization} of $\mathfrak X\in\mathrm{FSpStk}_I$ is given by
$
\mathfrak X^\mathrm{evp}:= \varepsilon_{I!}\varepsilon_I^*\mathfrak X.
$
\end{definition}

It is immediate from the definition and Construction \ref{Functoriality of ev} that even periodization is functorial, and that there is a canonical map $\mathfrak X^\mathrm{evp}\to\mathfrak X$ for all $\mathfrak X\in \mathrm{FSpStk}_I$.

\begin{prop}\label{EVP as a colimit, I-version}
The even periodization of a nonconnective $I$-adic formal spectral stack $\mX$ may be expressed as
$$
\mathfrak X^\mathrm{evp} \,\,\simeq \varinjlim_{\Spf(A)\in (\mathrm{FAff}_I^\mathrm{evp})_{/\mathfrak X}}\Spf(A),
$$
with the colimit taken in $\mathrm{FSpStk}_I$, and where  $(\mathrm{FAff}_I^\mathrm{evp})_{/\mathfrak X}:= \mathrm{FAff}_I^\mathrm{evp}\times_{\mathrm{FSpStk}_I}(\mathrm{FSpStk}_I)_{/\mathfrak X}.$
\end{prop}

\begin{proof}
This follows from the $\i$-category $\mathrm{SpStk}_I^\mathrm{evp}$ being generated under colimits by its subcategory $\mathrm{FAff}_I^\mathrm{evp}$, so that the canonical map
$$
\varinjlim_{\Spf(A)\in (\mathrm{FAff}_I^\mathrm{evp})_{/\mathfrak X}}\Spf(A)\vert_{\CAlg_I^\mathrm{evp}}\to\mathfrak X\vert_{\CAlg_I^\mathrm{evp}}
$$
is an equivalence, and the functor $\varepsilon_{I!}:\mathrm{FSpStk}_I^\mathrm{evp}\to\mathrm{FSpStk}_I$ preserving colimits.
\end{proof}

\begin{exun}\label{Example for proof I-version}
Let $A\in\CAlg_I^\mathrm{evp}$, i.e.\
$A$ is an even periodic $I$-complete $\E$-algebra over $R$ with bounded $I$-torsion. Consider the corresponding formal affine $\Spf(A)$, which is to say,the corepresentable functor $\CAlg_R^\mathrm{cplt}\ni B\mapsto \Map_{\CAlg_R}(A, B)\in\mathrm{Ani}$. In that case, the $\i$-category $(\mathrm{FAff}_I^\mathrm{evp})_{/\Spf(A)}\simeq (\CAlg_I^\mathrm{evp})_{A/}^\mathrm{op}$ has $\Spf(A)$ as its final object, and so  its even periodization is $\Spf(A)^\mathrm{evp}\simeq \Spf(A)$.
\end{exun}

\begin{prop}\label{Even stacks inside stacks, I-adic version}
The following statements hold:
\begin{enumerate}[label = (\alph*)]
\item The even periodization functor $\mathfrak X\mapsto \mathfrak X^\mathrm{evp}$ preserves all small colimits in $\mathrm{FSpStk}_I$.\label{Corollary 10, a, I-ad}
\item The functor $\varepsilon_{I!} : \mathrm{FSpStk}_I^\mathrm{evp}\to \mathrm{FSpStk}_I$ is fully faithful. Its essential image is generated under small colimits by the full subcategory $\mathrm{FAff}_I^\mathrm{evp}\subseteq\mathrm{FSpStk}_I$.\label{Corollary 10, b, I-ad}
\item A nonconnective spectral stack $\mathfrak X$ belongs to the essential image of $\mathrm{SpStk}^\mathrm{evp}$ inside $\mathrm{SpStk}$ as explained in
\ref{Corollary 10, b} if and only if $\mathfrak X^\mathrm{evp}\to\mathfrak X$ is an equivalence.\label{Corollary 10, c, I-ad}
\item The canonical map $(\mathfrak X^\mathrm{evp})^\mathrm{evp}\to \mathfrak X^\mathrm{evp}$ is an equivalence for any $\mathfrak X\in\mathrm{SpStk}$.\label{Corollary 10, d, I-ad}
\end{enumerate}
\end{prop}

\begin{proof}
For \ref{Corollary 10, a, I-ad}, observe that even periodization may be written via Construction \ref{Functoriality of ev I-version} as the composite of two left adjoint functors $\varepsilon_{I!}$and $\varepsilon_I^*$.
Claims \ref{Corollary 10, c, I-ad} and \ref{Corollary 10, d, I-ad} follow easily from \ref{Corollary 10, b, I-ad},
To establish the latter,
we must show that the unit $\mathfrak X\to\varepsilon_I^*\varepsilon_{I!}\mathfrak X$ is an equivalence for any $\mathfrak X\in \mathrm{FSpStk}^\mathrm{evp}_I$. Since both sides commute with colimits in $\mathfrak X$, we may reduce to the corepresentable case, i.e.\ $\mathfrak X =\Spf(A)\vert_{\CAlg^\mathrm{evp}}$ for an arbitrary even periodic $\E$-ring $A$. But the claim is clear case since we have $\varepsilon_{I!}(\Spf(A)\vert_{\CAlg_I^\mathrm{evp}})\simeq \Spf(A)$ in that case, as explained in Example \ref{Example for proof I-version}.
\end{proof}

In light of Proposition \ref{Even stacks inside stacks, I-adic version}, we can identify the $\i$-category $\mathrm{FSpStk}_I^\mathrm{evp}$ with its essential image inside $\mathrm{FSpStk}_I$. That is to say, we view even periodic $I$-adic formal spectral stacks as a special class of nonconnective $I$-adic formal spectral stacks.

Just as the even periodization of affines gives rise to the even filtration by Proposition \ref{Voila the even filtration}, so does the even periodization of $I$-adic formal affines encode the $I$-complete even filtration of Definition \ref{Def of I-complete evf}.

\begin{prop}\label{Voila the even filtration, I-adic}
Let $A$ be an $I$-complete $\E$-algebra over $R$. There is a canonical equivalence of filtered $\E$-rings
$$
\Gamma\big((\Spf(A))^\mathrm{evp})^\mathrm{cn};\, \tau_{\ge 2*}(\sO_{\Spf(A)^\mathrm{evp}})\big)\,\simeq\, \mathrm{fil}_{\mathrm{ev}, I}^{\ge *}(A).
$$
\end{prop}

\begin{proof}
Specializing Proposition \ref{EVP as a colimit, I-version} to the formally affine case, we get the even periodization of $\Spf(A)$ expressed as the colimit
$$
\Spf(A)^\mathrm{evp}\,\,\simeq \varinjlim_{B\in(\CAlg_I^\mathrm{ev})_{A/}}\Spf(B).
$$
Using the identification \eqref{Postnikov filtration on functions formula} from Remark \ref{Postnikov filtration on functions} -- or really the corresponding $I$-adic formal versions -- we find the filtered $\E$-ring in question to be
$$
\Gamma\big((\Spf(A)^\mathrm{evp})^\mathrm{cn};\,   \tau_{\ge 2*}(\sO_{\Spf(A)^\mathrm{evp}}))\,\,\simeq \varprojlim_{B\in(\CAlg_I^\mathrm{evp})_{A/}}\tau_{\ge 2*}(B).
$$
The right-hand side is by definition equivalent to the $I$-complete even periodic filtration $\mathrm{fil}^{\ge *}_{\mathrm{evp}, I}(A)$ of Definition \ref{Def of I-complete evf}. The claim now follows since we have
 $\mathrm{fil}^{\ge *}_{\mathrm{evp}, I}(A)\simeq\mathrm{fil}^{\ge *}_{\mathrm{ev}, I}(A)$
by Proposition \ref{Complete evp = ev filtr}.
\end{proof}

\begin{remark}
From the perspective of the $I$-adic formal stacks, the $I$-complete eff descent of $\mathrm{fil}^{\ge *}_{\mathrm{ev}, I}$ in the sense of Lemma \ref{Eff descent citation result, p-adic} is a reflection of the fpqc descent of $\mathfrak{QC}\mathrm{oh}$ in the even periodic (and more generally, complex periodic) setting by Theorem \ref{Descent for complex periodic formal QCoh}, and the agreement between adic faithfully flat and $I$-completely faithfully flat maps of Lemma \ref{adic fpqc = I-complete fpqc}
\end{remark}

As we have seen in this section so far, most of the aspects of even periodization transfer without change to the $I$-adic context. As an example, let us spell out the $I$-adic analogue of the cosimplicial formula 
of Proposition \ref{Computing the evening}
for even periodization in the presence of a periodically eff cover. Of course, we must first specify an $I$-adic version of the latter.

\begin{definition}\label{Definition periodically eff, I-adic}
A map of $I$-adic formal spectral stacks $\Spf(B)\to \mathfrak X$ is \textit{$I$-completely periodically eff} if for every $\mathfrak Y\in (\mathrm{FAff}_I^\mathrm{evp})_{/\mathfrak X}$, the map $\mathfrak Y\times_{\mathfrak X}\Spf(B)\to\Spf(B)$ is an fpqc cover in $\mathrm{FAff}^\mathrm{evp}_I$.
\end{definition}

\begin{prop}\label{Computing the evening, I-adic}
Let $\Spf(B)\to \mathfrak X$ be an $I$-completely periodically eff map of $I$-adic formal spectral stakcs. Then the canonical map
$$
\varinjlim_{\,\,\,\,\mathbf{\Delta}^\mathrm{op}}\Spf(B^\bullet)\to \mathfrak X^\mathrm{evp}
$$
is an equivalence, where we have denoted $\Spf(B^\bullet):= \Spec(B)^{\times_{\mathfrak X}[\bullet]}$, since this product is an $I$-adic formal affine by Definition \ref{Definition periodically eff, I-adic}.
\end{prop}

\begin{proof}
By Definition \ref{Computing the evening, I-adic} it follows that $\Spf(B)\vert_{\CAlg_I^\mathrm{evp}}\to\mathfrak X\vert_{\CAlg_I^\mathrm{evp}}$ is an fpqc cover in $\mathrm{FSpStk}_I^\mathrm{evp}$, and its \v{C}ech nerve is goven by $\Spf(B^\bullet)$ as in the statement of the Proposition. Consequently, the map
$$
\varinjlim_{\,\,\,\,\mathbf{\Delta}^\mathrm{op}}\Spf(B^\bullet)\vert_{\CAlg_I^\mathrm{evp}}\to \mathfrak X\vert_{\CAlg_I^\mathrm{evp}}
$$
is an equivalence in $\mathrm{FSpStk}_I^\mathrm{evp}$. This gives rise to the desired equivalence upon applying the colimit-preserving fully faithful embedding $\varepsilon_{I!} :\mathrm{FSpStk}_I^\mathrm{evp}\to\mathrm{SpStk}_I$ of Proposition \ref{Even stacks inside stacks, I-adic version}, and recalling that $\mathfrak X^\mathrm{evp}\simeq \varepsilon_{I!}(\mathfrak X\vert_{\CAlg_I^\mathrm{evp}})$ by Definition \ref{Def of evp localization, I-version}.
\end{proof}

As consequence, we may derive a comparison between the $I$-adic even periodization of an $I$-adic completion and the $I$-completion of an even priodization, a geometric version of the corresponding result for the even filtration \cite[Proposition 2.4.1]{HRW}.

\begin{corollary}\label{evp & completion}
Let $\mX\to \Spec(R)$ be a map of spectral stacks for which there exists an periodically eff map $\Spec(B)\to \mX$ such that the the $I$-completion $B^\wedge_I$ is even periodic and has bounded $I$-torsion. The canonical comparison map
$$
(\mX^\wedge_I)^\mathrm{evp}\to (\mX^\mathrm{evp})^\wedge_I
$$
is an equivalence of even periodic $I$-adic formal spectral stacks.
\end{corollary}

\begin{proof}
Fixing a periodically eff map $\Spec(B)\to \mX$ as in the statement of the corollary, the non-formal even periodization may be expressed by Proposition \ref{Computing the evening} as
$$
\mX^\mathrm{evp}\,\simeq
\varinjlim_{\,\,\,\,\mathbf{\Delta}^\mathrm{op}}\Spec(B^\bullet)
$$
where $\Spec(B^\bullet)\simeq \Spec(B)^{\times_{\mX}[\bullet]}$.
Since the $I$-completion functor $\mX\mapsto \mX^\wedge_I$ of Construction \ref{Const of I-completion of stacks} commutes with small colimits as a functor $\mathrm{SpStk}_{/\Spec(R)}\to \mathrm{FSpStk}_I$, this allows us to identify
$$
(\mX^\mathrm{evp})^\wedge_I\,\simeq
\varinjlim_{\,\,\,\,\mathbf{\Delta}^\mathrm{op}}\Spf((B^\bullet)^\wedge_I).
$$
The $I$-adic completion is given in Construction \ref{Const of I-completion of stacks} as the left adjoint $\mu_!$ induced from a morphism of sites $\mu$. It is therefore obtained by the restriction to accessible sheaves of a left adjoint constituent of a morphism of (big) $\i$-topoi, and as such commutes with finite limits.  That means that the canonical comparison morphism
$$
\Spf((B^\bullet)^\wedge_I)\to \Spf(B^\wedge_I)^{\times_{\mX^\wedge_I}[\bullet]}
$$
is an equivalence, where the fibered product on the right is taken in the $\i$-category $\mathrm{FSpStk}_I$. But the right-hand side in question is nothing but the \v{C}ech nerve of the map $\Spf(B^\wedge_I)\to \mX^\wedge_I$, which is  $I$-completely periodically eff by assumption. The desired claim now follows from the presentation for $(\mX^\wedge_I)^\mathrm{evp}$ garnished by Proposition \ref{Computing the evening, I-adic}.
\end{proof}

\begin{remark}
The hypotheses of Corollary \ref{evp & completion} are satisfied for instance for any affine $\mX=\Spec(A)$ which admits an eff map $A\to B$ into an even $\E$-algebra over $R$ for which $B^\wedge_I$ is also even and has bounded $I$-torsion. Indeed, in that case $B\to B\otimes_{\mathbf S}\mathrm{MUP}$ is periodically eff and satisfied the desired assumptions. 
\end{remark}

\newpage

\section{Equivariant even periodization}
Let $\T$ denote the circle group.
For an $\E$-ring with a $\mathbf T$-action $A$,  Hahn-Raksit-Wilson define in \cite[Construction 2.1.14]{HRW}  a  $\T$-equivariant variant of the even filtration
$$
\mathrm{fil}^{\ge *}_{\mathrm{ev}, \mathrm h\mathbf T}(A) \,:=\varprojlim_{B\in (\CAlg^\mathrm{ev})_{A/}^{\mathrm B\mathbf T}} \tau_{\ge 2*}(B^{\mathrm h\mathbf T}),
$$
for the homotopy fixed points  with respect to the $\mathbf T$-action. This is used to define the motivic filtrations on $\mathrm{TC}^-(A)$, and is in particular quintessential to the \cite{BMS2} approach to prismatic cohomology.
Our goal in this and the next section will be to give a natural natural spectral algebro-geometric extension of this construction,  compatibly with how the even periodization extends the even filtration. 

This section will discuss how to intertwine group actions with even periodization. We will be able to work in greater generality  than just with circle actions here, mostly that of an action by a compact $\mathbb E_1$-group anima $G$. At some point however, we will need specialize to the case of  $G=\T$, in which case we will be able to obtain an even periodic geometrization of the filtration $\mathrm{fil}^{\ge *}_{\mathrm{ev}, \mathrm h\mathbf T}(A)$, as well as a natural even periodic enhancement of filtered prismatization.

\subsection{Spectral stacks with actions of group anima}
There are at least two different ways one might approach the notion of group actions on nonconnective spectral stacks. The first is internal to algebraic geometry.

\begin{definition}\label{Def of staquivariance}
Let $\mG$ be an $\mathbb E_1$-group object in $\mathrm{SpStk}$. The $\i$-category of \textit{$\mG$-equivariant\footnote{Since we will not consider the more sophisticated notion of \textit{genuine equivariance} -- which should perhaps rather be called \textit{animated equivariance} or perhaps \textit{orbivariance} -- in this paper anyway, we expect this unqualified use of the adjective ``equivariant" to cause little confusion.} nonconnective spectral stacks} is
$$
\mathrm{SpStk}_\mG\,:=\,\LMod_\mG(\mathrm{SpStk}),
$$
where the module $\i$-category is taken
in the sense of \cite{HA}.
\end{definition}

\begin{remark}\label{equivariance via quotient stacks}
The $\i$-category of $\mG$-equivariant nonconnective spectral stacks may also be described as the overcategory inside $\mathrm{SpStk}$ over the classifying stack $\mathrm B\mG\simeq \varinjlim_{\mathbf{\Delta}^\mathrm{op}}\mG^{\times[\bullet]}$. The equivalence $\mathrm{SpStk}_\mG\simeq \mathrm{SpStk}_{/\mathrm B\mG}$ sends a $\mG$-equivariant spectral stack $\mX$ to the quotient stack $\mX/\mG$, and a relative stack $\mY\to\mathrm B\mG$ to the fiber $\mY\times_{\mathrm B\mG}\Spec(\mathbf S)$ with the naturally induced $\Spec(\mathbf S)\times_{\mathrm B\mG}\Spec(\mathbf S)\simeq \mG$-equivariant structure.
\end{remark}

On the other hand, there is an inherent notion of the action of  an $\mathbb E_1$-group \textit{anima} on not only spectral stacks, but in fact on objects of any $\i$-category with enough colimits.

\begin{definition}\label{Def of action}
Let $G$ be an $\mathbb E_1$-group anima and $\mC$ be an $\i$-category which is closed under small colimits. A \textit{$G$-action on an object of $\mC$} is a functor $\mathrm BG\to \mC$, and the \textit{$\i$-category of objects in $\mC$ with a $G$-action} is the functor $\i$-category
$$
\mC^{\mathrm BG}\,:=\,\Fun(\mathrm BG, \mC).
$$
The \textit{underlying object} functor $\mC^{\mathrm BG}\to \mC$ is induced by taking functors into $\mC$ from the quotient map of anima $*\to */G\simeq \mathrm BG$.
\end{definition}

The two notions of equivariance exist \textit{\`a priori} in different settings: $G$ is once a spectral stack and once an anima. But of course, they can be connected as follows.

\begin{cons}\label{Cons constant group}
Let $G$ be  an $\mathbb E_1$-group anima. Let $\underline G_\mathbf S$ denote the constant stack with value $G$, i.e.\ the tensor $\underline{G}_\mathbf S:=G\otimes_{\mathrm{SpStk}}\Spec(\mathbf S)\simeq\varinjlim_G \Spec(\mathbf S)$ of the final spectral stack $\Spec(\mathbf S)$ with the anima $G$ in the $\i$-category $\mathrm{SpStk}$. Since the functor $\mathrm{Ani}\to \mathrm{SpStk}$ of tensoring with $\Spec(\mathbf S)$ preserves products, the constant stack $\underline G_\mathbf S$ is equipped with a canonical $\mathbb E_1$-group structure in $\mathrm{SpStk}$.
\end{cons}

\begin{prop}\label{Action = equivariance}
Let $G$ be  an $\mathbb E_1$-group anima. There exists a canonical equivalence of $\i$-categories
$$
\mathrm{SpStk}^{\mathrm BG}\,\simeq\, \mathrm{SpStk}_{\underline G_\mathbf S},
$$
compatible with the forgetful functors to $\mathrm{SpStk}$.
\end{prop}

\begin{proof}
This follows directly from applying of the Barr-Beck-Lurie Monadicity Theorem to the forgetful functor $\mathrm{SpStk}^{\mathrm BG}\to\mathrm{SpStk}$.
\end{proof}

\begin{remark}
Given a nonconnective spectral stack $\mX$, a $G$-action on it is  equivalent to it being $\underline G_\mathbf S$-equivariant. In particular, the abstract $\i$-categorical quotient $\mX/G$, i.e.\ the colimit of the corresponding functor $\mathrm BG\to \mathrm{SpStk}$, encoding the $G$-action on $\mX$, can be identified with the quotient stack $\mX/\underline G_\mathbf S\simeq \varinjlim_{\mathbf{\Delta}^{\mathrm{op}}}$ $\underline G_\mathbf S^{\times [\bullet]}\times \mX$.
\end{remark}

Let us therefore focus for now on the perspective of spectral stacks with a $G$-action. Our next objective is to establish a ``functor of points" approach to those, i.e.\ to identify them with the appropriate functors from $G$-equivariant $\E$-rings to anima.

\begin{exun}
The $\i$-category $\CAlg^{{\mathrm B}G}$ of $\E$-rings with a $G$-action may equivalently be identified with the  $\E$-algebra objects in the symmetric monoidal $\i$-category $\Sp^{\mathrm{B}G}$ of spectra with a $G$-action (with respect to the point-wise symmetric monoidal structure, as opposed to for instance Day convolution). In terms of it, the $\i$-category of nonconnective spectral affines with a $G$-action may be expressed as
$$
\mathrm{Aff}^{\mathrm BG}\simeq (\CAlg^\mathrm{op})^{\mathrm BG}\simeq \Fun((\mathrm BG)^\mathrm{op}, \CAlg)^\mathrm{op}\simeq (\CAlg^{\mathrm BG^\mathrm{op}})^\mathrm{op},
$$
using the identification of classifying spaces $(\mathrm BG)^\mathrm{op}\simeq \mathrm B G^\mathrm{op}$, where $G^\mathrm{op}$ is the opposite group of $G$. 
To avoid notational clutter, we will mostly not distinguish between $G$ and its opposite group $G^\mathrm{op}$ in notation -- after all, their underlying anima is the same.
\end{exun}

\begin{cons}\label{Equivariant fpqc topology}
 Set an \textit{fpqc cover in} $\mathrm{Aff}^{\mathrm BG}$ to be such a sieve whose image under the forgetful functor $\mathrm{Aff}^{\mathrm BG}\to\mathrm{Aff}$ is a covering sieve for the fpqc topology in the usual sense. This equips the $\i$-category $\mathrm{Aff}^{\mathrm BG}$ with a Grothendieck topology. 
 \end{cons}

We  show in Proposition \ref{Why affinization representation} below how sheaves on the fpqc site $\mathrm{Aff}^{\mathrm BG}$ are related to nonconnective spectral stacks with a $G$-action $\mathrm{SpStk}^{\mathrm BG}$ in the sense of  Definition \ref{Def of action}. This turns out to involve the affinization $\underline G_\mathbf S^\mathrm{aff}$, so we first collect
a basic observation concerning it, which may for instance also be found in\footnote{One reason why we are repeating these arguments is that, \textit{technically speaking}, the affinization used in \cite{BZN} does not coincide with ours. They are instead using affinization in the sense of To\"en's affine stacks \cite{Champs affines}; that is to say, using the functor $\Spec^{\Delta} \simeq (\Spec(-))\vert_{\CAlg^\heart}$. However, this has no bearing on the validity of any results abount affinization that we wish to invoke.} \cite[Proposition 3.7, Corollary 3.8]{BZN}.

 \begin{definition}
The \textit{affinization} of a nonconnective spectral stack $\mX$ is given by $$
\mX^\mathrm{aff}\,:=\,\Spec(\sO(\mX)).
$$
\end{definition}

For any $\mX\in \mathrm{SpStk}$, there is a canonical \textit{affinization map} $\mX\to \mX^\mathrm{aff}$, the composition with which induces a homotopy equivalence
$$
\Map(\mX, \Spec(A))\simeq \Map(\mX^\mathrm{aff}, \Spec(A))
$$
for all $\E$-rings $A$. That is to say, the affinization functor $\mathrm{SpStk}\to \mathrm{Aff}$, given by $\mX\mapsto \mX^\mathrm{aff}$, is left adjoint to the inclusion $\mathrm{Aff}\hookrightarrow\mathrm{SpStk}$.

\begin{lemma}
Let $G$ be an $\mathbb E_1$-group in finite anima. The affinization $\underline G_\mathbf S^\mathrm{aff}$ of the costant stack $\underline G_\mathbf S$ carries a natural structure of an $\mathbb E_1$-group in the $\i$-category $\mathrm{SpStk}$.
\end{lemma}

\begin{proof}
While the affinization
functor $\mX\mapsto \mX^\mathrm{aff}$ does not preserve products in general, it does so as an application of the projection formula for quasi-coherent sheaves when restricted to those nonconnective spectral stacks $\mX$ for which $\sO_{\mX}$ is a compact object of $\QCoh(\mX)$. Since constant stacks valued in finite anima, as finite colimits of $\Spec(\mathbf S)$, evidently satisfy this condition, the result follows.
\end{proof}
 
\begin{prop}\label{Why affinization representation}
Let $G$ be an $\mathbb E_1$-group in finite anima.
There is a canonical equivalence of $\i$-categories
$$
\mathrm{Shv}_\mathrm{fpqc}^\mathrm{acc}(\mathrm{Aff}^{\mathrm BG})\,\simeq\, \mathrm{SpStk}_{\underline G^\mathrm{aff}_\mathbf S}
$$
compatible with the forgetful functors to $\mathrm{SpStk}$.
\end{prop}

\begin{proof}
Since the classifying space $\mathrm BG$ is a small $\i$-category, the accessibility of the functor $\i$-category $\CAlg^{\mathrm BG}$ follows by \cite[Proposition 5.4.4.3]{HTT} from that of $\CAlg$. Consider the adjunction
$$
\omega:\mathrm{CAlg}^{\mathrm BG}\rightleftarrows \mathrm{CAlg} : \varphi
$$
between the forgetful and the cofree $G$-action functors. The latter is given explicitly by the cotensor $\varphi(A)\simeq  A^G$ of an $\E$-ring $A$ with the anima $G$. Since $G$ is compact, the canonical $G$-equivairant $\E$-ring map $A\otimes_\mathbf S \mathbf S^G\to A^G$ is an equivalence. Because fpqc covers are closed under base-change, it follows that the functor $\varphi$ preserves fpqc covers. The same is true for the forgetful functor $\omega$ by definition of the fpqc topology on $\mathrm{Aff}^{\mathrm BG}$. Since both functors $\omega$ and $\varphi$ are also clearly accessible, they induce an adjunction between the $\i$-categories of accessible sheaves
$$
\xymatrix{
 \mathrm{SpStk} \ar@<1.2ex>[rr]^{\varphi_!\qquad} \ar@<-1.2ex>[rr]_{\omega^*\qquad}&  &\mathrm{Shv}_\mathrm{fpqc}^\mathrm{acc}(\mathrm{Aff}^{\mathrm BG}).\ar[ll]|-{\varphi^*\simeq\omega_!}
}
$$
The left-most adjoint $\varphi_!$ is given by sheafification of the left Kan extension along $\varphi$. Since its right adjoint $\varphi^*$ is identified with the left adjoint functor $\omega_!$, it commutes with colimits. The Barr-Beck-Lurie monadicity theorem therefore gives rise to an identification
$$
\mathrm{Shv}_\mathrm{fpqc}^\mathrm{acc}(\mathrm{Aff}^{\mathrm BG})\,\simeq\, \mathrm{LMod}_{\varphi^*\varphi_!}(\mathrm{SpStk})
$$
with the $\i$-category of modules over the monad $\varphi^*\varphi_!$ on nonconnective spectral stacks. On affines, this monad is given explicitly by
$$
\varphi^*\varphi_!(\Spec(A))\,\simeq\, \Spec(A^G)\,\simeq\, \Spec(A)\times\Spec(\mathbf S^G),
$$
and since $\varphi^*\varphi_!$ commutes with colimits, it follows that $\varphi^*\varphi_!(\mX)\simeq \mX\times \Spec(\mathbf S^G)$ holds for all nonconnective spectral stacks $\mX$. The monad structure on $\varphi^*\varphi_!$ consequently determines and is determined by a  monoid structure on $\Spec(\mathbf S^G)$.

Note that there are canonical equivalences of $\E$-rings
$$
\sO(\underline G_\mathbf S)\,\simeq\, \sO\big(\varinjlim_G\Spec(\mathbf S)\big)\,\simeq\, \varprojlim_G\sO(\Spec(\mathbf S))\,\simeq\, \mathbf S^G,
$$
allowing us to identify $\Spec(\mathbf S^G)$ with the affinization $\underline G_\mathbf S^\mathrm{aff}$ of the constant spectral stack $\underline G_\mathbf S$.
Unwinding the above discussion, we see that the monoid structure obtained on $\underline G_\mathbf S^\mathrm{aff}$ agrees with the one that it inherits from the one on $G$ by Construction \ref{Cons constant group}, and the proof is complete.
\end{proof}

\begin{remark}\label{affines}
Pulling back the equivalence of $\i$-categories of Proposition \ref{Why affinization representation} via the forgetful functor to nonconnective spectral stacks along the inclusion $\mathrm{Aff}\subseteq\mathrm{SpStk}$, we recover the equivalence between the $\i$-categories of equivariant affines
$$
\mathrm{Aff}^{\mathrm BG}\simeq \mathrm{Aff}_{\underline G_\mathbf S^\mathrm{aff}}
$$
from \cite[Proposition 3.12]{BZN}. Furthermore as explained in  \cite[Corollary 3.11]{BZN}, the quasi-coherent pullback and pushforward along the delooping $\mathrm B\underline G_\mathbf S\to \mathrm B\underline G_\mathbf S^\mathrm{aff}$ of the affinization map induces an equivalence of symmetric monoidal $\i$-categories 
$$
\Rep(\underline G_\mathbf S^\mathrm{aff})\,\simeq\, \QCoh(\mathrm B\underline G_\mathbf S^\mathrm{aff})\,\simeq\, \QCoh(\mathrm B\underline G_\mathbf S)\,\simeq\, \Sp^{\mathrm BG},
$$
by identifying both sides with the $\i$-category of $\sO(\underline G_\mathbf S)\simeq \sO(\underline G_\mathbf S^\mathrm{aff})\simeq \mathbf S^G$-comodules in spectra. As a consequence, we also obtain
 an equivalence of fixed-points $M^{\underline G_\mathbf S^\mathrm{aff}}\simeq M^{\mathrm hG}$ for any spectrum with a $G$-action $M$.
\end{remark}

\begin{remark}
Let us denote by $\mG:=\underline G_\mathbf S^\mathrm{aff}$ the affinization of the constant spectral group scheme $G_{\mathbf S}$.
It follows from the proof of Proposition \ref{Why affinization representation} that the forgetful functor that the forgetful functor $\mathrm{SpStk}_{\mG}\to \mathrm{SpStk}$, sending a nonconnective $\mG$-equivariant spectral stack $\mX$ to its underlying spectral stack, can be identified in terms of the functor of points $\mX:\mathrm{Aff}{}_{\mG}^\mathrm{op}\to\mathrm{Ani}$ to the composite
$\mathrm{Aff}^\mathrm{op}\xrightarrow{\mathrm{Free}_{\mG}}\mathrm{Aff}{}_{\mG}^\mathrm{op}\xrightarrow{\mX}\mathrm{Ani}.$ Here $\mathrm{Free}_{\mG}$ is the free $\mG$-action functor, i.e.\ given by $\mX\mapsto \mG\times \mX$.
\end{remark}

Recall that our aim is to discuss $G$-actions on  spectral stacks, or as is by Proposition \ref{Action = equivariance} equivalent, $\underline G_\mathbf S$-equivariant  spectral stacks. Since Proposition \ref{Why affinization representation} shows that the naive notion of equivariant functor of points gives rise to $\underline G_\mathbf S^\mathrm{aff}$-equivariant spectral stacks instead, we must now express the latter in terms of $G$-actions, i.e.\ $\underline G_\mathbf S$-equivariance.

\begin{prop}\label{Compare G and G^aff}
Let $G$ be an $\mathbb E_1$-group in finite anima.
The affinization morphism $\underline G_\mathbf S\to \underline G_\mathbf S^\mathrm{aff}$, viewed as a map of $\mathbb E_1$-group spectral stacks, induces
a
double adjunction of $\i$-categories
$$
\xymatrix{
 \mathrm{SpStk}_{\underline G_\mathbf S} \ar@<1.2ex>[rr]^{\mathrm{Ind}_\mathrm{aff}\,\,\,} \ar@<-1.2ex>[rr]_{\mathrm{Coind}_\mathrm{aff}}&  &
 \mathrm{SpStk}_{\underline G_\mathbf S^\mathrm{aff}}.\ar[ll]|-{\mathrm{Res}_\mathrm{aff}}
}
$$
The middle adjoint $\mathrm{Res}_{\mathrm{aff}}$ compatible with the forgetful functors to $\mathrm{SpStk}$ and fully faithful. Its essential image consists of those $\underline G_\mathbf S$-equivariant nonconnective spectral stacks which are colimits of $\underline G_\mathbf S$-equivariant affines.
\end{prop}

\begin{proof}
Recall that we defined equivariant stacks in Definition \ref{Def of staquivariance} as modules in the sense of \cite{HA}. The fact that any map of $\mathbb E_1$-group objects induces a free-forgetful adjunction of the described sort therefore follows by abstract nonsense \cite[Corollary 4.2.3.7]{HA}. Let us nonetheless describe it explicitly below.

If $\mX$ is a $\underline G_\mathbf S^\mathrm{aff}$-equivariant nonconnective spectral stack, then the $\mathbb E_1$-group map $\underline G_\mathbf S\to \underline G_\mathbf S^\mathrm{aff}$ allows us to view it as a $\underline G_\mathbf S$-equivariant nonconnective spectral stack as well, which we set to be $\mathrm{Res}_\mathrm{aff}(\mX)$. In particular, it follows that the functor $\mathrm{Res}_\mathrm{aff}$ does not affect the underlying nonconnective spectral stack.  
Since  limits in module $\i$-categories are preserved and created by the underlying object functor \cite[Proposition 4.2.3.1]{HA}, it follows that $\mathrm{Res}_\mathrm{aff}$ commutes with all small limits.  Since the functors $\mX\mapsto \mX\times \mY$ commute with colimits in $\mathrm{SpStk}$, it also follows from \cite[Corollary 4.2.3.5]{HA} that $\mathrm{Res}_\mathrm{aff}$ also commutes with all small colimits. In particular, it is accessible.
It therefore admits both a left and a right adjoint $\mathrm{Ind}_\mathrm{aff}$ and $\mathrm{Coind}_\mathrm{aff}$ by the Adjoint Functor Theorem.

As already observed in Remark \ref{affines}, the equivalence of $\i$-categories from Proposition \ref{Why affinization representation} restricts to an equivalence between the subcategories of affines. Unwinding the proof 
of  Proposition \ref{Why affinization representation}, we see that this is in fact given by a restriction of the functor $\mathrm{Res}_\mathrm{aff}$. In particular, the restriction-induction adjunction in question restricts on affines to an adjoint equivalence
$$
\mathrm{Ind}_\mathrm{aff}:\mathrm{Aff}_{\underline G_\mathbf S} \simeq \mathrm{Aff}{}_{\underline G_\mathbf S^\mathrm{aff}}:\mathrm{Res}_\mathrm{aff}.
$$
The counit of the adjunction $\mathrm{Ind}_\mathrm{aff}\mathrm{Res}_\mathrm{aff}(\mX)\to \mX$ is therefore an equivalence for any affine $\underline G_\mathbf S^\mathrm{aff}$-equivariant spectral stack $\mX$.
But since $\mathrm{SpStk}_{\underline G_\mathbf S^\mathrm{aff}}$ is generated under colimits by $\mathrm{Aff}_{\underline G_\mathbf S^\mathrm{aff}}\simeq \mathrm{Aff}^{\mathrm BG}$ by
Proposition \ref{Why affinization representation}, and we saw above that both $\mathrm{Ind}_\mathrm{aff}$ and $\mathrm{Res}_\mathrm{aff}$ preserve colimit, it follows that the counit $\mathrm{Ind}_\mathrm{aff}\mathrm{Res}_\mathrm{aff}\to\mathrm{id}$ is an equivalence. Consequently, the right adjoint $\mathrm{Res}_\mathrm{aff}$ is fully faithful. The description of its essential image follows from combining that it restricts to the canonical equivalence on affines, and that it commutes with colimits.
\end{proof}

\begin{remark}
The names of the change-of-group functors  $\mathrm{Ind}_\mathrm{aff}$, $\mathrm{Res}_\mathrm{aff}, \mathrm{Coind}_\mathrm{aff}$ in Proposition \ref{Compare G and G^aff} are intended to evoke the precise but somewhat cumbersome notation
  $\mathrm{Ind}_{\underline G_\mathbf S/\underline G_\mathbf S^\mathrm{aff}}$,
  $\mathrm{Res}_{\underline G_\mathbf S/\underline G_\mathbf S^\mathrm{aff}}$,
    $\mathrm{Coind}_{\underline G_\mathbf S/\underline G_\mathbf S^\mathrm{aff}}$
    inspired by representation theory.
\end{remark}

\begin{remark}
As noted in Remark \ref{equivariance via quotient stacks}, we may identify $\mathrm{SpStk}_{\underline G_\mathbf S}\simeq \mathrm{SpStk}_{\mathrm B\underline G_\mathbf S}$ and $\mathrm{SpStk}_{\underline G^\mathrm{aff}_\mathbf S}\simeq \mathrm{SpStk}_{\mathrm B\underline G^\mathrm{aff}_\mathbf S}$. Under these equivalences of $\i$-categories, the functors $\mathrm{Ind}_\mathrm{aff}$, $\mathrm{Res}_\mathrm{aff}$, and $\mathrm{Coind}_\mathrm{aff}$ can be described as follows. The functor $\mathrm{Ind}_\mathrm{aff}$
sends a morphism of stacks $\mY\to \mathrm B\underline G_\mathbf S$ to the composite $\mY\to\mathrm B\underline G_\mathbf S\to \mathrm B\underline G_\mathbf S^\mathrm{aff}$. Its right adjoint $\mathrm{Res}_\mathrm{aff}$ sends $\mY\to \mathrm B\underline G_\mathbf S^\mathrm{aff}$ to the pullback $\mY\times_{\mathrm B\underline G_S^\mathrm{aff}}\mathrm B\underline G_S\to \mathrm B\underline G_\mathbf S$. The further right adjoint $\mathrm{Coind}_\mathrm{aff}$ is finally given by Weil restriction along $\mathrm B\underline G_\mathbf S\to \mathrm B\underline G_\mathbf S^\mathrm{aff}$, in the sense of \cite[Section 19.1]{SAG}.
\end{remark}

We have found the above-espoused perspective, using the affinization of the constant group $\underline G_\mathbf S^\mathrm{aff}$, to be very useful in organizing the proofs and lending conceptual clarity. Nonetheless, we will now mostly abandon it, since we now have the ability to phrase the comparison between $G$-equivariance and the ``$G$-equivariant functor of points" directly.

\begin{corollary}\label{Equivariance fin}
There is a canonical double adjunction of $\i$-categories
$$
\xymatrix{
 \mathrm{SpStk}^{\mathrm BG} \ar@<1.2ex>[rr]^{\mathrm{Ind}_\mathrm{aff}\,\,\quad} \ar@<-1.2ex>[rr]_{\mathrm{Coind}_\mathrm{aff}\quad}&  &
 \mathrm{Shv}_\mathrm{fpqc}^\mathrm{acc}(\mathrm{Aff}^{\mathrm BG}),\ar[ll]|-{\mathrm{Res}_\mathrm{aff}}
}
$$
the middle adjoint $\mathrm{Res}_\mathrm{aff}$ of which is fully faithful. The right-most adjoint $\mathrm{Coind}_\mathrm{aff}$ sends a nonconnective spectral stack $\mX$ with a $G$-action to its functor of points
$$
\CAlg^{\mathrm BG}\ni A\mapsto \Map_{\mathrm{SpStk}^{\mathrm BG}}(\Spec(A), \mX)\in \mathrm{Ani},
$$
and the counit of the lower adjunction is given by
$$
\mathrm{Res}_\mathrm{aff}\,\mathrm{Coind}_\mathrm{aff}(\mX)\,\,\simeq \varinjlim_{\Spec(A)\in \mathrm{Aff}^{\mathrm BG}_{/\mX}}\Spec(A).
$$
\end{corollary}

\begin{proof}
For the existence of the adjunction and the fully faithfulness of the middle adjoint, combine Propositions \ref{Action = equivariance}, \ref{Why affinization representation}, and \ref{Compare G and G^aff}.
Now fix $\mX\in \mathrm{SpStk}^{\mathrm BG}$ and $A\in \CAlg^{\mathrm BG}$. By the already-established adjunction, we have
$$
\Map_{\mathrm{Shv}_\mathrm{fpqc}^\mathrm{acc}(\mathrm{Aff}^{\mathrm BG})}(\Spec(A), \mathrm{Coind}_\mathrm{aff}(\mX))\,\simeq\, \Map_{\mathrm{SpStk}^{\mathrm BG}}(\mathrm{Res}_\mathrm{aff}(\Spec(A)), \mX).
$$
But as remarked in the proof of Proposition \ref{Compare G and G^aff} above, the functor $\mathrm{Res}_\mathrm{aff}$ restricts on $G$-equivariant affines to an equivalence; in fact, under all the identifications we are making here, it is the identity functor $\Spec(A)\to\Spec(A)$. Since the fpqc sheaf $\mathrm{Coind}_\mathrm{aff}(\mX)$ is fully determined by its values on all representables, this verifies the asserted description of it. For the description of the counit of the adjunction between $\mathrm{Res}_\mathrm{aff}$ and $\mathrm{Coind}_\mathrm{aff}$, note that the canonical map
$$
\varinjlim_{\Spec(A)\in \mathrm{Aff}_{/\mY}^{\mathrm BG}}\Spec(A)\to\mY
$$
is an equivalence for any $\mY\in \mathrm{Shv}_\mathrm{fpqc}^\mathrm{acc}(\mathrm{Aff}^{\mathrm BG})$. Since $\mathrm{Res}_\mathrm{aff}$ preserves colimits, and we have $\mathrm{Aff}_{/\mX}^{\mathrm BG}\simeq \mathrm{Aff}_{/\mathrm{Coind}_\mathrm{aff}(\mX)}$ by the preceding discussion, we obtain the desired formula for the counit functor $\mathrm{Res}_\mathrm{aff}\,\mathrm{Coind}_\mathrm{aff}(\mX)$.
\end{proof}

\subsection{Equivariant even periodization}\label{Section equivariant evp}

We can set up the equivairant version of even periodization by a direct imitation of Section \ref{Section EVP}.

To begin with, note that the full subcategory $(\CAlg^{\mathrm{evp}})^{\mathrm BG}\subseteq\CAlg^{\mathrm BG}$ inherits the fpqc topology of Construction \ref{Equivariant fpqc topology} enitrely like in the non-equivariant setting of Lemma \ref{Lemma cover preservation and lifting}. This makes $(\mathrm{Aff}^\mathrm{evp})^{\mathrm BG}$ into a Grothendieck site.

\begin{definition}\label{Def of even G-stacks}
Let $G$ be an $\mathbb E_1$-group object in finite anima. 
A \textit{$G$-equivariant even periodic spectral stack} is an accessible functor $\mX:(\CAlg^\mathrm{evp})^{\mathrm BG}\to \mathrm{Ani}$ which satisfies fpqc descent. We denote by
$$
\mathrm{SpStk}_G^\mathrm{evp} \,:=\,\mathrm{Shv}^\mathrm{acc}_\mathrm{fpqc}\big((\mathrm{Aff}^\mathrm{evp})^{\mathrm BG}\big)
$$ the $\i$-category of $G$-equivariant even periodic spectral stacks.
\end{definition}

\begin{prop}\label{Add G adjunctions}
Let $G$ be an $\mathbb E_1$-group in finite anima.
The subcategory inclusion $\varepsilon_G : (\CAlg^\mathrm{evp})^{\mathrm BG}\subseteq\CAlg^{\mathrm BG}$ induces a double adjunction
$$
\xymatrix{
 \mathrm{SpStk}^\mathrm{evp}_{G} \ar@<1.2ex>[rr]^{\varepsilon_{G!}\quad} \ar@<-1.2ex>[rr]_{\varepsilon_{G*}\quad}&  &
 \mathrm{Shv}_\mathrm{fpqc}^\mathrm{acc}(\mathrm{Aff}^{\mathrm BG}).\ar[ll]|-{\,\,\varepsilon_G^*}
}
$$
the left-most adjoint  $\varepsilon_{G!}$ of which is fully faithful. Its middle adjoint $\varepsilon_{G^*}$ is given by the restriction $\mX\mapsto \mX\vert_{(\CAlg^\mathrm{evp})^{\mathrm BG}}$.
\end{prop}

\begin{proof}
Repeat the proofs of Lemma \ref {Functoriality of ev}, Construction \ref{Lemma cover preservation and lifting}, and
Proposition \ref{Even stacks inside stacks}. The extra presence of $G$-actions on $\E$-rings in sight changes nothing. 
\end{proof}

\begin{remark}
 Using Proposition \ref{Why affinization representation}, we can interpret the above-appearing ``functors of points with a $G$-action" as $\underline G_\mathbf S^\mathrm{aff}$-equivariant spectral stacks. From this perspective, we see that we may identify $\mathrm{SpStk}^\mathrm{evp}_G$ with a full subcategory of $\mathrm{SpStk}_{\underline G_\mathbf S^\mathrm{aff}}$. Indeed, when compared not against $G$-actions but against $\underline G_\mathbf S^\mathrm{aff}$-equivariance, $G$-equivariant even periodic stacks behave completely analogously to the non-equivariant story.
\end{remark}

\begin{corollary}\label{Gevp into evp}
Let $G$ be an $\mathbb E_1$-group in finite anima.
There is a canonical composable pair of adjunctions of $\i$-categories
$$
\xymatrix{
\mathrm{SpStk}^\mathrm{evp}_{G} \ar@<0.7ex>[rr]^{\varepsilon_{G!}\quad} &  &
 \mathrm{Shv}^\mathrm{acc}_\mathrm{fpqc}(\mathrm{Aff}^{\mathrm BG})\ar@<0.7ex>[ll]^{\varepsilon_G^*\quad}
 \ar@<0.7ex>[rr]^{\,\,\,\quad\mathrm{Res}_\mathrm{aff}} &  &
 \mathrm{SpStk}^{\mathrm BG}.\ar@<0.7ex>[ll]^{\qquad\mathrm{Coind}_\mathrm{aff}}
}
$$
The left adjoint $\mathrm{Res}_\mathrm{aff}\,\varepsilon_{G!}$ of the composite adjunction is fully faithful.
\end{corollary}

\begin{proof}
Combine
Propositions \ref{Equivariance fin} and \ref{Add G adjunctions}.
\end{proof}

\begin{remark} By unpacking all the definitions, we find that the fully faithful left adjoint of Corollary \ref{Gevp into evp} participates in the commutative diagram
$$
\begin{tikzcd}
\mathrm{SpStk}^\mathrm{evp}_G \arrow{r}{\varepsilon_{G!}} \arrow{d}{} & \mathrm{Shv}_\mathrm{fpqc}^\mathrm{acc} \arrow{d}{}
\arrow{r}{\mathrm{Res}_\mathrm{aff}} \arrow{d}{} & \mathrm{SpStk}^{\mathrm BG} \arrow{d}{}
\\
\mathrm{SpStk} \arrow{r}{=} & \mathrm{SpStk} \arrow{r}{=} & \mathrm{SpStk},
\end{tikzcd}
$$
in which the vertical maps are the respective forgetful functors to the $\i$-category of nonconnective spectral stacks.
\end{remark}

In light of Corollary \ref{Gevp into evp}, we may identify $G$-equivariant even periodic spectral stacks with a full subcategory of nonconnective spectral stacks, namely its essential image under the fully faithful functor $\mathrm{Res}_\mathrm{aff}\, \varepsilon_{G!}$. It  follows from the same result that this is a colocalizing subcategory, suggesting the following version of Definition \ref{Def of even stacks} in the equivariant setting.

\begin{definition}\label{Def of G-evp localization}
Let $G$ be an $\mathbb E_1$-group in finite anima and $\mX$ a nonconnective spectral stack with a $G$-action.
The \textit{$G$-equivariant even periodization of $\mX$} is given by
$$
\mX^{\mathrm{evp}_G}\,:=\, \mathrm{Res}_\mathrm{aff}\,\varepsilon_{G!}\,\varepsilon_G^*\,\mathrm{Coind}_\mathrm{aff}(\mX).
$$
\end{definition}

It is immediate from the definition that the $G$-equivariant even periodization $\mX\mapsto \mX^{\mathrm{evp}_G}$ is functorial and that there exists a canonical map $\mX^\mathrm{evp}\to\mX$ for all $\mX\in \mathrm{SpStk}^{\mathrm BG}$.

\begin{prop}\label{GEVP as a colimit}
Let $G$ be an $\mathbb E_1$-group in finite anima.
 Then the equivariant even periodization of a nonconnective spectral stack with a $G$-action $\mX$ can be expressed as
$$
\mX^{\mathrm{evp}_G}\,\, \simeq \varinjlim_{\Spec(A)\in (\mathrm{Aff}^\mathrm{evp})^{\mathrm BG}_{/\mX}}\Spec(A),
$$
with the colimit taken in $\mathrm{SpStk}^{\mathrm BG}$, and where  $(\mathrm{Aff}^\mathrm{evp})^{\mathrm BG}_{/\mX}:= (\mathrm{Aff}^\mathrm{evp})^{\mathrm BG}\times_{\mathrm{SpStk}^{\mathrm BG}}\mathrm{SpStk}^{\mathrm BG}_{/\mX}.$
\end{prop}

\begin{proof}
The $\i$-category $\mathrm{SpStk}^\mathrm{evp}_G$ is generated under colimits by its subcategory of affines $(\mathrm{Aff}^\mathrm{evp})^{\mathrm BG}$, so that the canonical map
$$
\varinjlim_{\Spec(A)\in (\mathrm{Aff}^\mathrm{evp})^{\mathrm BG}_{/\mX}}\Spec(A)\vert_{(\CAlg^\mathrm{evp})^{\mathrm BG}}\to\mathrm{Coind}_\mathrm{aff}(\mX)\vert_{\CAlg^\mathrm{evp}}
$$
is an equivalence in $\mathrm{SpStk}_G^\mathrm{evp}$, where recall from Corollary \ref{Equivariance fin} that $\mathrm{Coind}_\mathrm{aff}(\mX)$ is the $G$-equivariant functor of points of $\mX\in \mathrm{SpStk}^{\mathrm BG}$. The desired formula now follows because $\mathrm{Res}_\mathrm{aff}\,\varepsilon_!$ is a left adjoint functor and hence commutes with colimits, and takes 
$$
\Spec(A)\vert_{(\CAlg^\mathrm{evp})^{\mathrm BG}}\mapsto \Spec(A)
$$
for any even periodic $\E$-ring with a $G$-action $A$.
\end{proof}

Let $\mX$ be a nonconnective spectral stack with a $G$-action. The underlying spectral stack functor $\mathrm{SpStk}^{\mathrm BG}\to\mathrm{SpStk}$ commutes with colimits, therefore the colimit formula for $\mX^{\mathrm{evp}_G}$ of Proposition \ref{GEVP as a colimit} remains valid
on the level of underlying spectral stacks. Since all the $\E$-rings $A$ involved in it are even, we have on the level of underlying spectral stacks a canonical map
\begin{equation}\label{Comparison map with and withoug G}
\mX^{\mathrm{evp}_G}\,\,\simeq \varinjlim_{\Spec(A)\in (\mathrm{Aff}^\mathrm{evp})^{\mathrm BG}_{/\mX}}\Spec(A)\to \varinjlim_{\Spec(A)\in \mathrm{Aff}^\mathrm{evp}_{/\mX}}\Spec(A)\simeq \,\mX^\mathrm{evp}
\end{equation}
in $\mathrm{SpStk}^\mathrm{evp}$, but there is in general no reason to expect it to be an equivalence. It does happen in particularly good cases though, such as in the following direct equivariant analogue of Proposition \ref{Computing the evening}.

\begin{prop}\label{Computing the Gevening}
Let $G$ be an $\mathbb E_1$-group in finite anima.
Let $\Spec(B)\to \mX$ be a map in $\mathrm{SpStk}^{\mathrm BG}$ 
which is periodically eff on the underlying nonconnective spectral stacks, in the sense of Definition \ref{Definition periodically eff}.
Then the canonical $G$-equivariant map
$$
\varinjlim_{\,\,\,\,\mathbf{\Delta}^\mathrm{op}}\Spec(B^\bullet)\to \mX^{\mathrm{evp}_G}
$$
is an equivalence where $B^\bullet = \sO\big(\Spec(B)^{\times_{\mX}[\bullet]}\big)$.
\end{prop}

\begin{proof}
By the definition of periodically eff maps, it follows that the restriction
$$
\Spec(B)\vert_{(\CAlg^\mathrm{evp})^{\mathrm BG}}\to\mathrm{Coind}_\mathrm{aff}(\mX)\vert_{(\CAlg^\mathrm{evp})^{\mathrm BG}}
$$
is an fpqc cover in $\mathrm{SpStk}_G^\mathrm{evp}$, as well as that its \v{C}ech nerve is given by
$$
\Spec(B)^{\times_{\mX}[\bullet]}\,\simeq\, \Spec(B^\bullet)
$$
for a cosimplicial object $B^\bullet$ in $(\CAlg^\mathrm{evp})^{\mathrm B\T}$. It follows that the induced map
$$
\varinjlim_{\,\,\,\,\mathbf{\Delta}^\mathrm{op}}\Spec(B^\bullet)\vert_{(\CAlg^\mathrm{evp})^{\mathrm BG}}\to \mathrm{Coind}_\mathrm{aff}(\mX)\vert_{(\CAlg^\mathrm{evp})^{\mathrm BG}}
$$
is an equivalence in $(\mathrm{SpStk}^\mathrm{evp})^{\mathrm BG}$. The desired equivalence is now obtained  by applying the colimit-preserving fully faithful embedding $\mathrm{Res}_\mathrm{aff}\,\varepsilon_{G!} :\mathrm{SpStk}^\mathrm{evp}_G\to\mathrm{SpStk}^{\mathrm BG}$ of Corollary \ref{Gevp into evp}, and recalling from Definition \ref{Def of G-evp localization} that the $G$-equivairant even periodization is given by the composite $\mX^{\mathrm{evp}_G}\simeq \mathrm{Res}_\mathrm{aff}\,\varepsilon_{G!}\,\varepsilon_!(\mathrm{Coind}_\mathrm{aff}(\mX)\vert_{\CAlg^\mathrm{evp}})$.
\end{proof}

\begin{corollary}
Let $G$ be an $\mathbb E_1$-group in finite anima.
Let $\mX$ be a nonconnective spectral stack with a $G$-action which admits a $G$-equivariant periodically eff map from an affine. Then the  comparison map \eqref{Comparison map with and withoug G} induces a $G$-equivariant equivalence
$$
\mX^{\mathrm{evp}_G}\,\simeq\, \mX^\mathrm{evp}.
$$
In particular, the underlying nonconnective spectral stack of $\mX^{\mathrm{evp}_G}$ is in this case given by the non-equivariant even periodization $\mX^\mathrm{evp}$.
\end{corollary}

\begin{proof}
The $G$-action on $\mX$ is encoded as a functor $\mathrm{B}G\to \mathrm{SpStk}$, from which the underlying object $\mX$ is obtained as the composite with the quotient map of anima $*\to */G\simeq \mathrm{B}G$. The composite functor
$$
\mathrm BG\to \mathrm{SpStk}\xrightarrow{(-)^\mathrm{evp}}\mathrm{SpStk}^\mathrm{evp}\subseteq\mathrm{SpStk}
$$
therefore exhibits a $G$-action on the even periodization $\mX^\mathrm{evp}$. The  comparison map \eqref{Comparison map with and withoug G} thus naturally extends to a $G$-equivariant map $\mX^{\mathrm{evp}_G}\to \mX^{\mathrm{evp}}$.
Let $\Spec(B)\to \mX$ be a $G$-equivariant periodically eff map, as in the statement of Proposition \ref{Computing the Gevening}. By forgetting the $G$-action, the map of underlying nonconnective spectral stacks $\Spec(B)\to \mX$ is still periodically eff. It follows by Proposition \ref{Computing the evening} that the induced map
$$
\varinjlim_{\,\,\,\,\mathbf\Delta^\mathrm{op}}\Spec(B^\bullet)\to \mX^\mathrm{evp}$$
is an equivalence of nonconnective spectral stacks.
But this map is by construction also $G$-equivariant, and since  both limits and colimits are preserved under the forgetful functor $\mathrm{SpStk}^{\mathrm BG}\to \mathrm{SpStk}$, its left-hand side may be fully interpreted as taking place in $\mathrm{SpStk}^{\mathrm BG}$. The desired identification with the $G$-equivariant even periodization now follows directly from Proposition \ref{Computing the Gevening}.
\end{proof}

\begin{variant}\label{Variant: G works with I}
Everything in this section admits direct analogues in the $I$-adic formal setting of Section \ref{Section FAG}. In particular, we have for any $\E$-group object in finite anima $G$ an $\i$-category of \textit{$G$-equivariant even periodic $I$-adic formal stacks} defined in our usual paradigm as accessible sheaves 
$$
\mathrm{FSpStk}^\mathrm{evp}_{I, G}\,:=\, \mathrm{Shv}_\mathrm{fpqc}^\mathrm{acc}\big((\mathrm{FAff}_I^\mathrm{evp})^{\mathrm BG} \big).
$$
Analogously to the discussion in Corollary \ref{Gevp into evp} above, there is a canonical adjunction
$$
\mathrm{SpStk}^{\mathrm{evp}}_{I, G}\rightleftarrows\mathrm{FSpStk}_I^\mathrm{BG}
$$
whose corresponding monad is the endofunctor we take to be the $I$-adic formal version of \textit{$G$-equivariant even periodization} $\mathfrak X\mapsto \mathfrak X^{\mathrm{evp}_G}$. With the obvious definition of $I$-completely $G$-equivariantly periodically eff maps, intertwining Definitions \ref{Def of G-evp localization} and \ref{Definition periodically eff, I-adic}, there are analogous simplicial presentations for $\mathfrak X^{\mathrm{evp}_G}$ to Proposition \ref{GEVP as a colimit}, as well as comparisons with $I$-completions as in Corollary \ref{evp & completion}.
\end{variant}

\subsection{The coarse quotient of a $\mathbf T$-action}
Given a nice enough  stack $\mX$ with a $G$-action, we might try to define the \textit{coarse quotient} $\mX/\!/G$ by setting it on the level of affines to be $\Spec(A)/\!/G:=\Spec(A^{\mathrm hG})$, and then attempt to extend along colimits. The issue is that this procedure is generally not compatible with descent, and so is unlikely to produce anything more structured than a prestack.

\begin{remark}[Comparison with GIT quotients]
From the perspective of the coarse quotient, Mumford's geometric invariant theory
essentially amounts to one solution of the same issue in
 in the quasi-projective setting. Indeed, the input data for it includes, other than the expected scheme $X$ with a $G$-action, also a $G$-equivariant ample line bundle $\mathcal L$ on $X$, traditionally called a \textit{linearized line bundle}. The latter is used to obtain a distinguished collection of affine opens $X_s\subseteq X$, taken to be the non-vanishing loci $X_s=\{s\ne 0\}$ of the invariant sections $s\in \Gamma(X; \mathcal L^{\otimes N})$ for some $N>\!>0$. The \textit{GIT quotient} is then defined in terms of them, see \cite[Theorem 1.10]{GIT}, as the colimit
 $$
 X/\!/_\mathcal LG\,:=\,\varinjlim_s \Spec(\sO(X_s)^G),
 $$
 with the indexing poset induced by the inclusions $X_s\subseteq X$.
In the special case of a projective scheme $X\simeq \Proj(R)$ and the ample line bundle $\mathcal L=\mathcal O_X(1)$, this reproduces the  more familiar formula $X/\!/_{\mathcal O_X(1)}G\simeq \Proj(R^G).$ But the more general quasi-projective setting illustrates much better how the choice of the linearized line bundle $\mL$ is there to mitigate the issue of the incompatibility of the naive coarse quotient with colimits. This workaround of course comes at a cost: the GIT quotient $X/\!/_{\mathcal L} G$ is only the quotient of the \textit{semi-stable locus} $X^\mathrm{ss}(\mathcal L)\subseteq X$ by the $G$-action, rather than the quotient of $X$ itself.
 \end{remark}

Nevertheless, in the present context, it is possible to define a variant of the coarse quotient for $\T$-actions on even periodic spectral stacks which outputs  \textit{formal} spectral stacks. For this purpose, and unless explicitly stated otherwise throughout this section, the homotopy fixed-points $A^{\mathrm h\T}$
of an $\E$-ring $A$ with a $\T$-action will be equipped with the $I_\mathrm{aug}$-adic topology, where $I_\mathrm{aug}=\ker(\pi_0(A^{\mathrm h\mathbf T})\to \pi_0(A))$ is the augmentation ideal. Here is our main result:

\begin{theorem}\label{Existence of coarse quotients in SAG}
The functor $(\Aff^\mathrm{evp})^{\mathrm B\T}\to \mathrm{FAff}$, given by $\Spec(A)\mapsto \Spf(A^{\mathrm B\T})$, admit an essentially unique small-colimit-preserving extension
$
\mathrm{SpStk}^\mathrm{evp}_{\mathbf T}\to \mathrm{FSpStk}.
$
\end{theorem}

\begin{definition}
Let $\mX$ be a $\T$-equivariant even periodic spectral stack. The \textit{coarse quotient of $\mX$ by its $\T$-action} is the formal spectral stack $\mX/\!/\T$ obtained as the image of $\mX$ under the functor of Theorem \ref{Existence of coarse quotients in SAG}.
\end{definition}

\begin{remark}
The coarse quotient may be written more explicitly as the colimit
$$
\mX/\!/\T\, \simeq \, \varinjlim_{\Spec(A)\in (\mathrm{Aff}^{\mathrm {evp}})_{/\mX}^{\mathrm B\T}} \Spf(A^{\mathrm h\T}),
$$
computed in the $\i$-category of formal spectral stacks. 
\end{remark}

As noted above, the main difficulty involved in Theorem \ref{Existence of coarse quotients in SAG} is showing that the prescribed functor on affines preserves fpqc covers, i.e.\ takes $\T$-equivariant fpqc covers in the sense of spectral algebraic geometry to fpqc covers of formal spectral algebraic geometry. We will prove this assertion below in Proposition \ref{Prop end of proof of Thm on coarse quotients existing}, but it will require establishing some groundwork first.

The key result we make use of in establishing Theorem \ref{Existence of coarse quotients in SAG} is the following unipotence result, which its authors suggest may also be viewed as a form of Koszul duality.

\begin{theorem}[{\cite[Theorem 1.3]{MNN}}]\label{Thanks, Akhil!}
Let $A$ be an even  periodic $\E$-ring. Then the homotopy fixed-point functor $M\mapsto M^{\mathrm h\mathbf T}$ gives rise to an equivalence of symmetric monoidal $\i$-categories
$$
\Mod_A^{\mathrm B\mathbf T}\simeq \Mod_{A^{\mathrm{B\T}}}^\mathrm{cplt}.
$$
Its inverse functor is given by $N\mapsto N\o_{A^{\mathrm B\mathbf T}}A$, equipped with a canonical $\T$-action.
\end{theorem}

\begin{remark}
The conclusion of Theorem \ref{Thanks, Akhil!} remains  valid if the even periodicity assumption is weakened to complex periodicity. This can be seen as a corollary of our version of the result over $\mM$ in Proposition \ref{T-equiv vs QFG on QCoh}, or can be argued directly, see  \cite[Lemma 4.7]{Meyer-Wagner}. But since we will not explicitly require anything more than the even periodic case, we have chosen to cite the original version due to Mathew-Naumann-Noel.
\end{remark}

\begin{remark}\label{Remark inequivalent Postnikov towers}
Though the equivalence of $\i$-categories of Theorem \ref{Thanks, Akhil!} is symmetric monoidal, it is ``not $t$-exact". Of course, this needs to be appropriately interpreted, since all the $\E$-rings involved are periodic, hence can not be connective and so their module $\i$-categories do not carry canonical $t$-structures. But such $t$-structures do exist on the level of connective covers $\tau_{\ge 0}(A)$ and $\tau_{\ge 0}(A^{\mathrm h\T})$ respectively. Indeed, the Postnikov towers on the two sides of the equivalence of Theorem \ref{Thanks, Akhil!} are respectively given by the filtered spectrum with a $\T$-action $\tau_{\ge *}(M)$, and the filtered spectrum $\tau_{\ge *}(M^{\mathrm h\T})$ for a $\T$-equivariant $A$-module $M$. Since the canonical map $\tau_{\ge 0}(M^{\mathrm h\T})\to \tau_{\ge 0}(M)^{\mathrm h\T}$ is seldom an equivalence, \textit{the two sides of the equivalence  of Theorem \ref{Thanks, Akhil!} give rise to different Postnikov filtrations}. This will ultimately prove to be one of its key utilities, such as for Theorem \ref{Thanks, Akhil!}
\end{remark}

Our first step towards towards Theorem \ref{Existence of coarse quotients in SAG} is to globalize the equivalence of Theorem \ref{Thanks, Akhil!}, from an even periodic affine $\Spec(A)$ to the chromatic base stack $\mM\simeq \Spec(\mathbf S)^\mathrm{evp}$. This will involve a canonical formal spectral stack over $\mM$, the \textit{universal Quillen formal group}. This is the universal object with respect to the universal property of $\mM$ as the moduli stack of oriented formal group, the primary perspective from which we considered it in \cite{ChromaticCartoon}. For our purposes here, it will be convenient to express it in terms of a rather explict construction.

\begin{definition}[{\cite[Section 4.1]{Elliptic 2}}]\label{Def of QFG}
The \textit{Quillen formal group} of a complex periodic $\E$-ring $A$ is the formal spectral stack $\w{\G}{}^{\mathcal Q}_A:=\Spf(A^{\mathrm B\mathbf T})$.
\end{definition}

\begin{cons}
The chromatic base stack $\mM$ is by Corollary \ref{Even localization of the sphere} equivalent to the even periodization $\Spec(\mathbf S)^\mathrm{evp}$. It may therefore
 be expresses in terms of even periodic $\E$-rings as the colimit
$$
\mM\,\,\simeq \varinjlim_{A\in\CAlg^{\mathrm{evp}}}\Spec(A)
$$
in the $\i$-category of spectral stacks. It supports the \textit{universal Quillen formal group}, given in terms of this colimit presentation as
\begin{equation}\label{Equation for universal QFG}
\w{\G}{}^{\mathcal Q}_{\mM} \,\,\simeq \varinjlim_{A\in \CAlg^\mathrm{evp}}\w{\G}{}^{\mathcal Q}_A.
\end{equation}
It is universal in the sense that for any complex periodic $\E$-ring $A$, its Quillen formal group is obtained by pullback of spectral stacks
\begin{equation}\label{Equation univ prop of univ QFG}
\begin{tikzcd}
\w{\G}{}^{\mathcal Q}_A \ar{r}\ar{d} & \w{\G}{}^{\mathcal Q}_{\mM}\ar{d}\\
\Spec(A) \ar{r} & \mM
\end{tikzcd}
\end{equation}
along the essentially unique map $\Spec(A)\to \mM$. The evident zero-sections $\Spec(A)\to \w{\G}{}^{\mathcal Q}_A$, induced from the quotient map $*\to */\mathbf T\simeq \mathrm B \mathbf T$, gives rise by passage to the colimit to the \textit{universal zero-section} $\mM\to \w{\G}{}^{\mathcal Q}_\mM$.
\end{cons}

We may now give the desired extension of  Theorem \ref{Thanks, Akhil!} to the chromatic base stack.

\begin{prop}\label{T-equiv vs QFG on QCoh}
There is a canonical symmetric monoidal equivalence of $\i$-categories
$$
\QCoh(\mM)^{\mathrm B\mathbf T}\,\simeq \, \mathfrak{QC}\mathrm{oh}(\w{\G}{}_{\mM}^{\mathcal Q}).
$$
\end{prop}

\begin{proof}
Note that the $\E$-ring $A^{\mathrm B\T}$ is even periodic, and so in particular complex periodic, whenever $A$ is. It follows that the colimit formula \eqref{Equation for universal QFG} for the universal Quillen formal group takes place in the full subcategory $\mathrm{FSpStk}^{\mathbf C\mathrm p}\subseteq\mathrm{FSpStk}$ spanned by formal affines whose underlying $\E$-rings are complex periodic. The composite functor
$$
(\mathrm{FSpStk}^{\mathbf C\mathrm p})^\mathrm{op}\hookrightarrow\mathrm{FSpStk}^\mathrm{op}\xrightarrow{\mathfrak{QC}\mathrm{oh}} \CAlg(\mathrm{Pr^L})
$$
preserves small limits as consequence of the fpqc descent for complete modules over adic $\E$-rings of Theorem \ref{Descent for complex periodic formal QCoh}. The right-hand side $\i$-category in question may therefore be written as\footnote{Technically the limit in question is indexed by \textit{accessible} and not \textit{small} indexing $\i$-category. But this is only an aesthetic choice, as we could have rewritten it as a colimit over an actual small index category, such as using the functor $\mathrm{MUP}^\bullet:\mathbf{\Delta}\to\CAlg^\mathrm{evp}$, the cosimplicial cobar construction $\mathrm{MUP}^\bullet :=\mathrm{MUP}^{\o_{\mathbf S}[\bullet]}$ of the $\E$-ring map $\mathbf S\to\mathrm{MUP}$, for any choice of $\E$-ring structure on periodic complex bordism. Such a replacement is possible everywhere throughout this proof.}
$$
\mathfrak{QC}\mathrm{oh}(\w{\G}{}_{\mM}^{\mathcal Q})\, \,\simeq  \varprojlim_{A\in\CAlg^\mathrm{evp}}\Mod_{A^{\mathrm B\T}}^\mathrm{cplt},
$$
with completeness for $A^{\mathrm B\T}$-modules on the right being as always understood with respect to the augmentation ideal. Since the functors $\mX\mapsto \QCoh(\mX)$ and $\mC\mapsto \mC^{\mathrm B\T}\simeq \Fun(\mathrm B\T, \mC)$ both commute with all small limits, we can likewise write the left-hand side as a limit
$$
\QCoh(\mM)^{\mathrm B\T}\,\simeq \, \QCoh\big(\varinjlim_{A\in\CAlg^\mathrm{evp}}\Spec(A)\big)^{\mathrm B\T}\, \simeq \, \varprojlim_{A\in\CAlg^\mathrm{evp}}\Mod_A^{\mathrm B\T}.
$$
Since the two sides can therefore be viewed as limits of the functors $\CAlg^\mathrm{evp}\to\CAlg(\mathrm{Pr^L})$, given respectively by $A\mapsto \Mod_A^{\mathrm B\T}$ and $A\mapsto \Mod_{A^{\mathrm B\T}}^\mathrm{cplt}$. it suffices to exhibit a natural equivalence between them. We claim that this is
accomplished by Theorem \ref{Thanks, Akhil!}. To show this, it suffices to verify that the homotopy fixed-points $M\mapsto M^{\mathrm h\T}$, viewed as functors $\Mod_A^{\mathrm B\T}\to \Mod_{A^{\mathrm B\T}}^\mathrm{cplt}$, are natural in the even periodic $\E$-ring $A$. But this in turn also follows from Theorem \ref{Thanks, Akhil!}. Indeed, the homotopy-fixed point functors are thanks to it symmetric monoidal, and so we have for any map of even periodic $\E$-rings $A\to B$ and $A$-module $M$ canonical $B$-module equivalences
$$
(M\o_A B)^{\mathrm h\T}\,\simeq\, M^{\mathrm h\mathbf T}\,\widehat{\o}_{A^{\mathrm h\T}}\,B^{\mathrm h\T}\,\simeq\, M^{\mathrm h\mathbf T}\,\widehat{\o}_{A^{\mathrm B\T}}\,B^{\mathrm B\T}
$$
and showing the desired naturality.
Here we have used that, since the $\T$-actions on $A$ and $B$ are trivial, their homotopy fixed-points are given by cotensoring with $\mathrm B\T$.
\end{proof}

\begin{remark}\label{Underlying objects as zero-section pullback}
In terms of the equivalence of Proposition \ref{T-equiv vs QFG on QCoh}, the forgetful functor $\QCoh(\mM)^{\mathrm B\T}\to \QCoh(\mM)$ is identified with the quasi-coherent pullback
$$
s^*:\mathfrak{QC}\mathrm{oh}(\w{\G}{}^{\mathcal Q}_{\mM})\to \mathfrak{QC}\mathrm{oh}(\mM)\simeq \QCoh(\mM)
$$
along the universal zero-section $s:\mM\to \w{\G}{}^{\mathcal Q}_{\mM}$. Here we have identified the chromatic base stack $\mM$ with a formal spectral stack using the canonical fully faithful embedding $\mathrm{SpStk}\hookrightarrow\mathrm{FSpStk}$
discussed in Remark \ref{Remark ff of SpStk into FSpStk}. Recall that this is induced by passing to accessible sheaves from the inclusion $\CAlg\hookrightarrow\CAlg_\mathrm{ad}$ of $\E$-rings into adic $\E$-rings by equipping them with the $0$-adic -- i.e.\ trivial -- topology.
\end{remark}

The next step towards proving Theorem \ref{Existence of coarse quotients in SAG} is to pass from Proposition \ref{T-equiv vs QFG on QCoh}, a result about quasi-coherent sheaves, to one about respective affines and formal affines. 

\begin{prop}\label{FF of hT}
The functor $\Spec(A)\mapsto \Spf(A^{\mathrm h\T})$ defines a fully faithful embedding
$$
(\mathrm{Aff}^{\mathbf C\mathrm p})^{\mathrm B\T}\hookrightarrow \mathrm{FAff}_{/\w{\G}{}^{\mathcal Q}_{\mM}},
$$
the essential image consists of those maps from formal affines $\Spf(B)\to \w{\G}{}^{\mathcal Q}_{\mM}$ which are formally affine morphisms.
\end{prop}

\begin{proof}
The equivalence of Proposition \ref{T-equiv vs QFG on QCoh} induces an equivalence between the $\i$-categories of commutative algebra objects
$$
\CAlg(\QCoh(\mM))^{\mathrm B\T}\,\simeq \,\CAlg(\QCoh(\mM)^{\mathrm B\T})\,\simeq \,\CAlg(\mathfrak{QC}\mathrm{oh}(\w{\G}{}^{\mathcal Q}_{\mM})).
$$
By passing to the opposite $\i$-categories, this may be viewed in terms of relatively affine spectral stacks and relatively formally affine formal spectral stacks respectively as
\begin{equation}\label{Equiv on relative (f)affines}
(\mathrm{SpStk}^\mathrm{aff}_{/\mM})^{\mathrm B\T}\, \simeq \, \mathrm{FSpStk}^\mathrm{faff}_{/\w{\G}{}^{\mathcal Q}_{\mM}}.
\end{equation}
Because the chromatic base stack $\mM$ is geometric, we have $\mathrm{Aff}_{/\mM}\subseteq \mathrm{SpStk}_{/\mM}^\mathrm{aff}$ -- but of course, that is nothing but $\mathrm{Aff}^{\mathbf C\mathrm p}$, since the existence of a map from an affine into $\mM$ is equivalent to its underlying $\E$-ring being complex periodic. By unwinding the proof of Proposition \ref{T-equiv vs QFG on QCoh}, we find that a map of spectral stacks $\Spec(A)\to \mM$ with a $\T$-action is sent under \eqref{Equiv on relative (f)affines} to the map of formal spectral stacks $\Spf(A^{\mathrm h\T})\to \w{\G}{}^{\mathcal Q}_{\mM}$. 
\end{proof}

We will need two more auxiliary lemmas concerning homotopy fixed-points of circle actions on even periodic $\E$-rings.

\begin{lemma}\label{Lemma hT is even periodic}
Let $A$ be an even periodic $\E$-ring with a $\T$-action. Then the homotopy fixed-point $\E$-ring $A^{\mathrm h\T}$ is also even periodic.
\end{lemma}

\begin{proof}
Since even periodicity may be checked Zariski-locally, and since any even periodic $\E$-ring is Zariski-locally strongly even periodic (since the line bundle $\pi_2(A)$ trivializes Zariski-locally over $\pi_0(A)$), we may assume without loss of generality that $A$ is strongly even periodic. But since $A$ is even, its  homotopy fixed-point spectral
sequence  collapses and we have $\pi_*(A^{\mathrm h\T})\simeq \pi_*(A)[\![t]\!]$, from which it is clear that $A^{\mathrm h\T}$ is also strongly even periodic.
\end{proof}

\begin{lemma}\label{Lemma on pullbacks of twisted Quillen FGs}
Let $A$ be an even periodic $\E$-ring with a $\T$-action. There is a canonical map of formal spectral stacks $\Spf(A^{\mathrm h\T})\to \w{\G}{}^{\mathcal Q}_{\mM}$ which fits into a pullbacks square of formal spectral stacks
$$
\begin{tikzcd}
\Spec(A)\arrow{r}{} \arrow{d}{} &\Spf(A^{\mathrm h\T}) \arrow{d}{} \\
\mM \arrow{r}{} & \w{\G}{}^{\mathcal Q}_{\mM},
\end{tikzcd}
$$
\end{lemma}

\begin{proof}
If we denote $R=A^{\mathrm h\T}$ -- which is also an even periodic $\E$-ring by Lemma \ref{Lemma hT is even periodic} -- and equip it with the trivial $\T$-action, the universal property of homotopy-fixed points gives rise to a canonical $\T$-equivariant map $R\to A$. From it, we obtain the canonical map to the universal Quillen formal group as
$$
\Spf(A^{\mathrm h\T})\to \Spf(R^{\mathrm h\T})\simeq \Spf(R^{\mathrm B\T})\simeq \w{\G}{}^{\mathcal Q}_R\to \w{\G}{}^{\mathcal Q}_\mM.
$$
It follows from the construction of this canonical map that it appears as the right horizontal edge of the outer square in the commutative diagram
\begin{equation}\label{Pullback diagram, first of many}
\begin{tikzcd}
\Spec(A)\arrow{r}{} \arrow{d}{} &\Spf(A^{\mathrm h\T}) \arrow{d}{} \\
\Spec(R)\arrow{r}{} \arrow{d}{} &\Spf(R^{\mathrm B\T}) \arrow{d}{} \\
\mM \arrow{r}{}& \w{\G}{}^{\mathcal Q}_{\mM},
\end{tikzcd}
\end{equation}
in the $\i$-category $\mathrm{FSpStk}$. We must show that the outer square in this diagram is a pullback square.

Let us show first that the upper square is. For that, consider the commutative diagram
$$
\begin{tikzcd}
\Spec(A)\arrow{r}{} \arrow{d}{} &\Spf(A^{\mathrm h\T}) \arrow{d}{} \arrow{r}{} & \Spec(A^{\mathrm h\T})\arrow{d}{} \\
\Spec(R)\arrow{r}{} &\Spf(R^{\mathrm B\T})  \arrow{r}{} & \Spec(R^{\mathrm B\T})
\end{tikzcd}
$$
in which the right square is Cartesian since both of the formal spectra are equipped with the same adic topology (that is, given by the augmentation ideal). Thus the left square of this second diagram is Cartesian if and only if its outer square is, i.e.\ if the canonical map of $A^{\mathrm h\T}\o_A R\to A$ is an equivalence. That is indeed the case by Theorem  \ref{Thanks, Akhil!}, proving that the upper square of the diagram \eqref{Pullback diagram, first of many} is Cartesian.

It follows by the pasting law of pullbacks that outer square of \eqref{Pullback diagram, first of many} is Cartesian if and only if the lower square is. To verify the latter, we make use of yet another commutative diagram
$$
\begin{tikzcd}
\Spec(R)\arrow{r}{} \arrow{d}{} &\Spf(R^{\mathrm B\T}) \arrow{d}{} \arrow{r} & \Spec(R)\arrow{d}{}\\
\mM \arrow{r}{}& \w{\G}{}^{\mathcal Q}_{\mM}\arrow{r}{} & \mM.
\end{tikzcd}
$$
Its outer square is evidently Cartesian, so the left square being a pullback is equivalent to the right one being. But that follows, after noting that $\Spf(R^{\mathrm B\T})\simeq \w{\G}{}^{\mathcal Q}_R$, from the universal property of the universal Quillen formal group $\w{\G}{}^{\mathcal Q}_\mM$, or more specifically, from the pullback square \eqref{Equation univ prop of univ QFG}.
\end{proof}

At this point we can tackle the heart of the proof of Theorem \ref{Existence of coarse quotients in SAG}: showing that the prescribed functor of affines preserves fpqc covers. In fact, we prove slightly more.

\begin{prop}\label{Prop end of proof of Thm on coarse quotients existing}
The fully faithful embedding $\Spec(A)\mapsto \Spf(A^{\mathrm h\T})$ of Proposition \ref{FF of hT} reflects fpqc covers. Its restriction to the subcategory $(\mathrm{Aff}^\mathrm{evp})^{\mathrm B\T}\subseteq (\mathrm{Aff}^{\mathbf C\mathrm p})^{\mathrm B\T}$ also preserves them.
\end{prop}

\begin{proof}
Let $A\to B$ be a $\T$-equivariant map of complex periodic $\E$-rings with $\T$-actions.
We must show that it is faithfully flat if and only if the corresponding  map  of homotopy fixed-points $A^{\mathrm h\T}\to B^{\mathrm h\T}$ is adically faithfully flat with respect to the augmentation ideal.

If $A^{\mathrm h\T}\to B^{\mathrm h\T}$ is adically faithfully flat, that means that the map of formal spectra $\Spf(B^{\mathrm h\T})\to \Spf(A^{\mathrm h\T})$ is faithfully flat.
We have seen from Lemma \ref{Lemma on pullbacks of twisted Quillen FGs} that the map of affines $\Spec(B)\to\Spec(A)$ may be obtained from it by pullback along the universal zero-section $\mM\to \w{\G}{}^{\mathcal Q}_{\mM}$. Since faithfully flat maps are closed under pullbacks, we see that $A\to B$ is faithfully flat.

Conversely, suppose that $A\to B$ is faithfully flat, and additionally assume that both $A$ and $B$ are even periodic. 
If we let $R=A^{\mathrm h\T}$, which is even periodic by Lemma \ref{Lemma hT is even periodic}, then $A\to B$ may be viewed as a map of $\E$-algebra objects in $\Mod_R^{\mathrm B\T}$, and $A^{\mathrm h\T}\to B^{\mathrm h\T}$ is a morphism in $\CAlg_{R^{\mathrm B\T}}^\mathrm{cplt}$.
Let $t\in \pi_{-2}(R^{\mathrm B\T})$ be a choice of complex orientation, i.e.\ a choice of generator for the augmentation ideal. Then the equivalence
$
A\simeq A^{\mathrm h\T}\o_{R^{\mathrm B\T}}R$ from Theorem \ref{Thanks, Akhil!} induces (since $t$ is a unit in $\pi_*(R^{\mathrm B\T})$, allowing us to identify the $\i$-categorical and $1$-categorical quotients with respect to it) an equivalence of $\pi_0(R)$-modules
\begin{equation}\label{Quotient by t formula}
\pi_*(A^{\mathrm h\T})\o_{\pi_0(R^{\mathrm B\T})}\pi_0(R) \,\simeq\,\pi_*(A^{\mathrm h\T})/t\,\simeq\, \pi_*(A),
\end{equation}
and likewise for $B$. In particular, the map $\pi_0(A^{\mathrm h\T})\to \pi_0(B^{\mathrm h\T})$ induces a faithfully flat map after base-change along $\pi_0(R^{\mathrm B\T})\to \pi_0(R^{\mathrm B\T})/t\simeq \pi_0(R)$. To verify that $A^{\mathrm h\T}\to B^{\mathrm h\T}$ is adically faithfully flat, it therefore suffices by  Proposition \ref{Lemma ff comparison} to show that the canonical $\pi_0(B^{\mathrm h\T})$-module map
$$\pi_1(A^{\mathrm h\T})\o_{\pi_0(A^{\mathrm h\T})}\pi_0(B^{\mathrm h\T})\to \pi_1(B^{\mathrm h\T})$$
induces an isomorphism upon $t$-completion. Of course, that is equivalent to it inducing an isomorphism upon quotienting out by $t$, which follows from the formula \eqref{Quotient by t formula} and its $B$-variant.
\end{proof}

\begin{proof}[Proof of Theorem \ref{Existence of coarse quotients in SAG}] By Proposition \ref{Prop end of proof of Thm on coarse quotients existing},
the fully faithful functor
$$
\upsilon:(\mathrm{Aff}^\mathrm{evp})^{\mathrm B\T}\to \mathrm{FAff}^\mathrm{faff}_{/\w{\G}{}^{\mathcal Q}_{\mM}}\subseteq \mathrm{FAff}_{/\w{\G}{}^{\mathcal Q}_{\mM}},
$$
given by   $\Spec(A)\mapsto \Spf(A^{\mathrm h\T})$, induces -- by the Definition \ref{Def of even G-stacks} of the $\i$-category of $\T$-equivariant even periodic spectral stacks -- a double adjunction upon the $\i$-categories of accessible fpqc sheaves
$$
\xymatrix{
 \mathrm{SpStk}^\mathrm{evp}_\T \ar@<1.2ex>[rr]^{\upsilon_!\quad} \ar@<-1.2ex>[rr]_{\upsilon_*\quad}&  &
 \mathrm{FSpStk}_{/\w{\G}{}^{\mathcal Q}_{\mM}}.\ar[ll]|-{\upsilon^*}
}
$$
The composite of the left-most adjoint $\upsilon_!$ with the forgetful functor $\mathrm{FSpStk}_{/\w{\G}{}^{\mathcal Q}_{\mM}}\to \mathrm{FSpStk}$ is the desired colimit-preserving extension. It is essentially unique because its domain $\i$-category $\mathrm{SpStk}_\T^\mathrm{evp}$ is generated under colimits by the full subcategory $(\mathrm{Aff}^\mathrm{evp})^{\mathrm B\T}$.
\end{proof}

As a corollary of the proof, we obtain an indication that passage to the coarse quotient loses little information.

\begin{corollary}\label{Coarse quotient is ff}
The coarse quotient functor
$
 \mathrm{SpStk}^\mathrm{evp}_\T \to
 \mathrm{FSpStk}_{/\w{\G}{}^{\mathcal Q}_{\mM}}
 $ is fully faithful. 
\end{corollary}

\begin{proof}
In the notation from the proof of Theorem \ref{Existence of coarse quotients in SAG}] above, since $\mX/\!/\T\,\simeq\, \upsilon_!\mX$, we must show that the the unit transformation $\mathrm{id}\to\upsilon^*\upsilon_!$ of the upper adjunction is a natural equivalence. Since both functors $\upsilon^*$ and $\upsilon_!$ are left adjoints and therefore commute with small colimits, and the $\i$-category $\mathrm{SpStk}^\mathrm{evp}_\T$ is by defintion generated under small colimits by $\T$-equivariant even periodic affines, this reduces to showing that the functor $\upsilon_!$ is fully faithful on such affines, which in turn holds by Proposition \ref{FF of hT}.
\end{proof}

\begin{remark}
 Recall that  no information is lost in the passage $\mX\mapsto\mX/\underline{\T}_{\mathbf S}$ from a given $\T$-equivariant spectral stacks $\mX$ to the corresponding quotient stacks $\mX/\underline{\T}_{\mathbf S}$ if the structure map $\mX/\underline{\T}_{\mathbf S}\to \mathrm B\underline{\T}_{\mathbf S}$ is kept track of. Corollary \ref{Coarse quotient is ff} show that the same is true for the passage to the coarse quotient $\mX\mapsto \mX/\!/\T$, so long as we keep track of the structure map $\mX/\!/\T\to \w{\G}{}^{\mathcal Q}_{\mM}$. The  universal Quillen formal group is appearing here in the role of the coarse quotient $\Spec(\mathbf S)^\mathrm{evp}/\!/\T$, see Example \ref{Coarse quotient examples QFG}, in analogy with the classifying stack $\mathrm B\underline{\T}_{\mathbf S}\simeq \Spec(\mathbf S)/\underline{\T}_{\mathbf S}$.
\end{remark}

\begin{exun}\label{Coarse quotient examples QFG}
Let $\Spec(A)$ be an even periodic affine. Its coarse quotient with respect to the trivial $\T$-action is by definition given by
$$
\Spec(A)/\!/\T\,\simeq \, \Spf(A^{\mathrm B\T})\,\simeq\,\w{\G}{}^{\mathcal Q}_A,
$$
coinciding with the Quillen formal group from Definition \ref{Def of QFG}. More generally, the coarse quotient of the trivial $\T$-action on the chromatic base stack recovers the universal Quillen formal group
$
\mM/\!/\T\, \simeq \,\w{\G}{}^{\mathcal Q}_\mM.
$
\end{exun}

In light of the identification of the universal Quillen formal group as the coarse quotient of $\mM$ along the trivial $\T$-action, the globalization of Theorem \ref{Thanks, Akhil!} to the chromatic base stack $\mM$ in
Proposition \ref{T-equiv vs QFG on QCoh} admits an obvious generalization to arbitrary $\T$-equivariant even periodic spectral stacks. It provides another piece of evidence that the passage to the coarse quotient is not wasteful, and that it encodes much of the same information as the $\i$-categorical quotient stack.

\begin{prop}\label{QCoh doesn't care}
Let $\mX$ be a $\T$-equivariant even periodic spectral stack. There is a canonical symmetric monodial equivalence of $\i$-categories
$$
\QCoh(\mX/\underline{\T}_{\mathbf S})\, \simeq \, \mathfrak{QC}\mathrm{oh}(\mX/\!/\T).
$$
\end{prop}

\begin{proof}
By the definition of the $\i$-category $\mathrm{SpStk}_{\T}^\mathrm{evp}$, we may write $\mX$ as a small colimit of $\T$-equivariant even periodic affines $\mX\simeq \varinjlim_i \Spec(A_i)$. The left and right hand sides from the equivalence in the statement of the Proposition may respectively be identified as the limits of presentably symmetric monoidal $\i$-categories
$$
\QCoh(\mX/\underline{\T}_{\mathbf S})\, \simeq \, \varprojlim_i \Mod_{A_i}^{\mathrm B\T}, \qquad\quad \mathfrak{QC}\mathrm{oh}(\mX/\!/\T)\,\simeq \,\varprojlim_i \Mod_{A^{\mathrm h\T}_i}^\mathrm{cplt}.
$$
Since we have shown  in the proof of Proposition \ref{T-equiv vs QFG on QCoh}  that the two relevant functors
$\CAlg^\mathrm{evp}\to\CAlg(\mathrm{Pr^L})$, given respectively by $A\mapsto \Mod_A^{\mathrm B\T}$ and $A\mapsto \Mod_{A^{\mathrm B\T}}^\mathrm{cplt}$, are naturally equivalent, the claim follows.
\end{proof}

\begin{remark}\label{If only special fiber}
As consequence of the description of the essential image in Proposition \ref{FF of hT}, we find that the coarse quotient more precisely takes values in the $\i$-category $\mathrm{FSpStk}^\mathrm{ad}_{/\w{\G}{}^{\mathcal Q}_{\mM}}$ of adic formal spectral stacks over the universal Quillen formal group, in the sense of Remark \ref{Remark T-adic prep}. In fact, we conjecture that Theorem \ref{Existence of coarse quotients in SAG}, which is to say the construction of the coarse quotient $\mX\mapsto \mX/\!/\T$, extends to $\T$-equivariant \textit{complex periodic} spectral stacks as well, and provides an equivalence of $\i$-categories
$$
(\mathrm{SpStk}^{\mathbf C\mathrm p})^{\mathrm B\T}\, \simeq \, \mathrm{FSpStk}^\mathrm{ad}_{/\w{\G}{}^{\mathcal Q}_{\mM}}.
$$
Under a formal analogy between the $\i$-category $\mathrm{FSpStk}^\mathrm{ad}_{/\w{\G}{}^{\mathcal Q}_{\mM}}$, which we might refer to as the \textit{$\T$-adic formal setting}, and the more familiar setting of $p$-adic formal geometry, the universal zero-section $\mM\to \w{\G}{}^{\mathcal Q}_{\mM}$ corresponds to the special point $\Spec(\mathbf F_p)\to \Spf(\mathbf Z_p)$. Recall that a $p$-adic formal scheme $X$ may be viewed as a mixed-characteristic thickening of its \textit{special fiber} $X_0$, a scheme over $\mathbf F_p$ defined by a pullback square
$$\begin{tikzcd}
X_0\arrow{r}{}\arrow{d}{} & X\arrow{d}{}\\
\Spec(\mathbf F_p) \arrow{r}{} & \Spf(\Z_p).
\end{tikzcd}$$
Likewise in the $\T$-adic setting, we have for any $\T$-equivariant even periodic spectral stack $\mX$ a pullback of formal spectral stacks
$$
\begin{tikzcd}
\mX \arrow{r}{} \arrow{d}{} & \mX/\!/\T \arrow{d}{} \\
\mM \arrow{r}{} & \w{\G}{}^{\mathcal Q}_\mM.
\end{tikzcd}
$$
Extracting the underlying even periodic stack $\mX$ from its coarse quotient $\mX/\!/\T$ may thus be seen as the $\T$-adic version of the special fiber. Perhaps in some world of analytic spectral stacks, we expect there should also exist a ``generic point" $\{\eta\}\to \w{\G}{}^{\mathcal Q}_{\mM}$, for which any even periodic $\E$-ring $A$ with a $\T$-action would fit into a diagram of two pullback squares
$$
\begin{tikzcd}
\Spec(A) \arrow{r}{} \arrow{d}{} & \Spf(A^{\mathrm h\T}) \arrow{d}{} &  \mathrm{Spa}(A^{\mathrm t\T})\arrow{l}{} \arrow{d}{} \\
\mM \arrow{r}{} & \w{\G}{}^{\mathcal Q}_\mM & \{\eta\}\arrow{l}{},
\end{tikzcd}
$$
where $A^{\mathrm t\T}$ denotes the Tate construction.
But it is not quite clear how to set up such a generic point, nor even what precise kind of analytic geometry to adopt, and
we do not pursue this any further in this paper.
\end{remark}

Equipped with the coarse quotient, as constructed in the last subsection, we can finally pin-point a spectral-algebro-geometric origin to the homotopy fixed-point variant of the even filtration from \cite{HRW}, akin to Proposition \ref{Voila the even filtration}.

\begin{prop}\label{Voila the even filtration on hT}
Let $A$ be an $\E$-ring with a $\T$-action. Let $\mathfrak X:=\Spec(A)^{\mathrm{evp}_\T}/\!/\T$ denote the coarse quotient of the $\T$-equivariant even periodization. There is a canonical equivalence of filtered $\E$-rings
$$
\Gamma(\mathfrak X^\mathrm{cn};\, \tau_{\ge 2*}(\mathfrak O_{\mathfrak X}))\, \simeq \, \mathrm{fil}^{\ge *}_{\mathrm{ev}, \mathrm h\T}(A).
$$
\end{prop}

\begin{proof}
Specializing Proposition \ref{GEVP as a colimit} for $G=\T$ to the affine case, we get the $\T$-equivariant even periodization $\Spec(A)^{\mathrm{evp}_\T}$ expressed as the colimit
$$
\Spec(A)^{\mathrm{evp}_\T}\,\, \simeq \varinjlim_{B\in (\CAlg^\mathrm{evp})^{\mathrm B\T}_{A/}}\Spec(B).
$$
By Theorem \ref{Existence of coarse quotients in SAG}, the coarse quotient functor $\mathfrak Y\mapsto \mathfrak Y/\!/\T$ commutes with small colimits. It follows (to be precise, after replacing the relevant colimits by equivalent ones with small indexing categories by use of accessibility), we obtain the expression
$$
\mathfrak X\,\, \simeq \varinjlim_{B\in (\CAlg^\mathrm{evp})^{\mathrm B\T}_{A/}}\Spf(B^{\mathrm h\T}).
$$
By the formal spectral algebraic geometry variant of the identification \eqref{Postnikov filtration on functions formula} from Remark \ref{Postnikov filtration on functions}, we find the Postnikov tower of the structure sheaf $\mathfrak O_{\mathfrak X}$ in $\mathfrak{QC}\mathrm{oh}(\mathfrak X^\mathrm{cn})$ given by
\begin{equation}\label{Limit formula for Postnikov tower on //T}
\Gamma(\mathfrak X^\mathrm{cn};\, \tau_{\ge 2*}(\mathfrak O_{\mathfrak X}))\,\simeq \varprojlim_{B\in (\CAlg^\mathrm{evp})^{\mathrm B\T}_{A/}}\tau_{\ge 2*}(B^{\mathrm h\T}).
\end{equation}
Contrast this with the $\mathrm \T$-equivariant homotopy fixed-point version of the even filtration
from \cite[Construction 2.1.14]{HRW}, which may be described as the limit
\begin{equation}\label{Even filtration on hT}
\mathrm{fil}^{\ge *}_{\mathrm{ev}, \mathrm h\mathbf T}(A) \,:=\varprojlim_{B\in (\CAlg^\mathrm{ev})_{A/}^{\mathrm B\mathbf T}} \tau_{\ge 2*}(B^{\mathrm h\mathbf T}).
\end{equation}
By \cite[Corollary 2.2.17, (2)]{HRW}, the functor $A\mapsto \mathrm{fil}^{\ge *}_{\mathrm{ev}, \mathrm h\T}(A)$ satisfies a $\T$-equivariant version of eff descent, completely analogous to the non-equivariant version recalled above as Proposition \ref{Eff descent citation result}.
Noting that $A\to A\otimes_\mathbf S \mathrm{MUP}$ is an eff cover, as well as $\T$-equivariant for any $A\in \CAlg^{\mathrm B\T}$, we can now repeat the same argument as we used in Proposition  \ref{Even filtration is even periodic filtration} to show that the $\T$-equivariant even filtration \eqref{Even filtration on hT} remains unchanged if pass in the index of the limit to the subcategory $(\CAlg^\mathrm{evp})^{\mathrm B\T}_{A/}\subseteq(\CAlg^{\mathrm{ev}})^{\mathrm B\T}_{/A}$.
This results precisely in the right-hand side of the Postnikov tower \eqref{Limit formula for Postnikov tower on //T}, and since the comparison map between the two is clearly a map of filtered $\E$-rings, this completes the proof.
\end{proof}

\begin{remark}
A version of the construction $A\mapsto \Spec(A)^\mathrm{evp}_{\T}/\!/\T$ appeared in \cite{Arpon's talk}, reporting on joint work with Devalapurkar, Hahn, and Yuan. Defined there in terms of even $\E$-rings, as opposed to even periodic ones, their $\Pi(A)^{\mathbf T}$ is precisely the underlying classical formal stack 
$(\mathrm{Spec}(A)^{\mathrm{evp}_\T}/\!/\T)^\heart$.
\end{remark}

\begin{remark}
What makes
Proposition \ref{Voila the even filtration on hT} work is that the symmetric monoidal equivalence of $\i$-categories of Proposition \ref{QCoh doesn't care} fails to be $t$-exact --  see Remark \ref{Remark inequivalent Postnikov towers}. Indeed, if we tried using the quotient stack $\mX:=\Spec(A)^\mathrm{evp}/\underline{\T}_{\mathbf S}$ instead of the coarse quotient, the Postnikov filtration in $\QCoh(\mX^\mathrm{cn})$ of the structure sheaf would give rise to the filtered $\E$-ring
$$
\Gamma(\mX^\mathrm{cn};\, \tau_{\ge 2*}(\mO_{\mX}))\,\varprojlim_{B\in (\CAlg^\mathrm{evp})^{\mathrm B\T}_{A/}}\tau_{\ge 2*}(B)^{\mathrm h\T}.
$$
In comparison with \eqref{Limit formula for Postnikov tower on //T}, we see that $\tau_{2*}$ and the homotopy fixed-points have switched places, and we in particular do not obtain the desired filtration.
\end{remark}

\begin{exun}
Let $k\to A$ be a chromatically quasi-lci map of connective $\E$-rings in the sense\footnote{This notion is actually very natural from the perspective of this paper as well.
A map of $\E$-rings $A\to B$ is chromatically quasi-lci precisely when the induced map $\Spec(A)\times\mM\to \Spec(B)\times \mM$ is a map of even spectral stacks, and the map of underlying classical stacks 
$(\Spec(A)\times \mM)^\heart\to (\Spec(B)\times \mM)^\heart$ is quasi-lci, i.e.\ its algebraic cotangent complex has Tor-amplitude in $[0, 1]$. It in particular follows from \eqref{EVPL vs CPL} that in this case $\Spec(A)^\mathrm{evp}\simeq \Spec(A)\times \mM$ and similarly for $B$.} of \cite[Definition 4.1.11]{HRW}. Then if $\mX =\Spec(\mathrm{THH}(A/k)^{\mathrm{evp}_\T})$, the formal spectral stack $\mX/\!/\T$ satisfies $\mathfrak O(\mX/\!/\T)\simeq \mathrm{TC}^-(A/k)$ and there is an quivalence of filtered $\E$-rings
$$
\Gamma((\mX/\!/\T)^\mathrm{cn};\, \tau_{\ge 2*}(\mathfrak  O_{\mX/\!/\T}))\,\simeq \, \mathrm{fil}_{\mathrm{ev}, \mathrm h\T}^{\ge *}(\mathrm{THH}(A/k))\, \simeq \,\mathrm{fil}_\mathrm{mot}^{\ge *}(\mathrm{TC}^-(A/k))
$$
with the motivic filtration on negative topological cyclic homology of $A$ over $k$. Indeed, the latter is defined in \cite[Definition 4.2.1]{HRW} to be this special case of the even filtration. Similarly the motivic filtration on topological Hochschild homology itself is
$$
\Gamma(\mX^\mathrm{cn};\, \tau_{\ge 2*}(\mathfrak  O_{\mX}))\, \simeq \, \mathrm{fil}^{\ge *}_\mathrm{ev}(\mathrm{THH}(A/k))\, \simeq \,\mathrm{fil}_\mathrm{mot}^{\ge *}(\mathrm{THH}(A/k)).
$$
\end{exun}

\begin{remark}\label{Remark TP}
If we had a theory of the special fiber $\eta$ as indicated in Remark \ref{If only special fiber}, we could likewise obtain the motivic filtration on $\mathrm{TP}(A/k)$ by taking the Postnikov filtration in quasi-coherent sheaves on the generic fiber $\Spf(\mathrm{TC}^-(A/k))_\eta\simeq \mathrm{Spa}(\mathrm{TP}(A/k))$.
\end{remark}

\subsection{Connection to prismatization}
Prismatic cohomology is some sense precedes the introduction of the algebraic notion of  a \textit{prism} and the corresponding prismatic site in \cite{BS22}. Though it did not yet bear the prismatic name, the notation $\widehat{\Prism}_R$ already appears in \cite{BMS2}. There, Bhatt-Morrow-Scholze obtain it from topological Hochschild homology, showcasing a striking example of how working over the sphere spectrum $\mathbf S$ can lead to interesting arithmetic results on the classical level. In retrospect, their fundamental result may be summarized as follows, where here and throughout in this section, quasi-regular semiperfectoid rings are always taken to be classical and in the $p$-complete setting.

\begin{theorem}[{\cite[Theorem 1.12]{BMS2}}]\label{BMS result}
Let $R$ be a  quasi-regular semiperfectoid ring. Then $\mathrm{TC}(R)^\wedge_p$ is an even $\E$-ring, and
there is a canonical and natural isomorphism of graded rings
$$
\pi_{2*}(\mathrm{TC}^-(R)^\wedge_p)\, \simeq \,\mathrm{Fil}^*_{\mathcal N}\,\widehat{\Prism}_R\{*\},
$$
on the right-hand side of which is found the Breuil-Kisin-twisted  Nygaard filtration on the Nygaard-complete prismatic cohomology of $R$.
\end{theorem}

Our goal for the rest of this subsection is to globalize this result, away from both the quasi-regular semiperfectoid and affineness assumptions. We will show how the geometry of the free loop space in spectral algebraic geometry, together with even periodization, gives rise to a natural even periodic enhancement of a prismatization stack, which geometrically encodes prismatic cohomology.

Such prismatization stacks are the mixed-characteristic analogues, relating to prismatic cohomology and $p$-adic Hodge theory, of the better-known \textit{de Rham stack} $X^\mathrm{dR}$ of Simpson, which in characteristic zero possesses the same relationship with de Rham cohomology and Hodge theory.
They were defined, in a number of variants $X^{\mathbbl{\Delta}}, X^{\mathcal N}, X^\mathrm{Syn}$,  by Bhatt-Lurie and Drinfeld, see \cite{Drinfeld}, \cite{BLb}, \cite{Bhatt F-gauges}. In this paper, we are mostly concerned with the middle one. 

The \textit{filtered prismatization} is a functor $X\mapsto X^{\mathcal N}$ from $p$-adic formal schemes to formal stacks, defined in \cite{Bhatt F-gauges} via transmutation. For a quasi-syntomic $p$-adic formal scheme $X$, it may by \cite[Theorem 5.5.10, Remarks 5.5.16 \& 5.5.18]{Bhatt F-gauges} also be expressed in terms of the prismatic cohomology of quasi-regular semiperfectoid rings as the colimit (of classical formal stacks)
\begin{equation}\label{Bhatt's formula}
X{}^{\mathcal N} \simeq \varinjlim_{R\in X_\mathrm{qrsp}}\Spf\big(\bigoplus_{n\in \mathbf Z}\operatorname{Fil}^n_{\mathcal N}\Prism_R\{n\}\big)/\mathbf G_m,
\end{equation}
over the category $X_\mathrm{qrsp}$ of (classical, i.e.\ static) quasi-regular semiperfectoid rings $R$  with a map $\Spec(R)\to X$.
We need to consider a slight variant of this construction:

\begin{definition}\label{Def of X^hat N}
The \textit{Nygaard-complete prismatization} of a quasi-syntomic $p$-adic formal scheme $X$ is the classical formal stack given by
$$
X{}^{\widehat{\mathcal N}} \simeq \varinjlim_{R\in X_\mathrm{qrsp}}\Spf\big(\bigoplus_{n\in \mathbf Z}\operatorname{Fil}^n_{\mathcal N}\,\widehat{\Prism}_R\{n\}\big)/\mathbf G_m.
$$
\end{definition}

\begin{remark}\label{Remark simpler presentation for Nygaardization}
According to \cite[Definition 8.12]{BLb}, a classical $p$-adic formal scheme $X$ is quasi-syntomic if and only precisely when each  affine open subscheme   of $X$ admits a cover by a quasi-regular semiperfectoid affine in the quasi-syntomic topology.
As consequence, a choice of Zariski cover by affines on $X$, and of quasi-regular semiperfectoid covers of said affines, allows us to replace the colimit appearing in the formula \ref{Bhatt's formula} for $X^{\mathcal N}$ with a much smaller and more convenient colimit. The same holds for the Nygaard complete prismatization $X^{\widehat{\mathcal N}}$, which ultimately follows from the corresponding statement for $X^{\mathcal N}$ and the fact that the canonical map of filtered rings $\mathrm{Fil}_{\mathcal N}^*\Prism_R\to\mathrm{Fil}_{\mathcal N}^* \widehat{\Prism}_R$ induces an isomorphism on the associated graded rings, see \cite[Definition 5.8.5]{BLa}.
\end{remark}

\begin{remark}\label{Thanks, Deven!}
The completion  map of filtered rings $\mathrm{Fil}_{\mathcal N}^*\Prism_R\to\mathrm{Fil}_{\mathcal N}^* \widehat{\Prism}_R$, natural in the quasi-regular semiperfectoid ring $R$, induces a canonical comparison map of formal stacks $X^{\widehat {\mathcal N}}\to X^{\mathcal  N}.$ According to 
\cite[Remark 2.5 \& Lemma 2.8]{Deven}, it induces an equivalence upon the $\i$-categories of quasi-coherent sheaves. In particular, we may identify
$$
\QCoh(X^{\widehat{\mathcal N}})\, \simeq \, \QCoh(X^\mathcal N)\, \simeq \, \mathrm{Gauge}_{\mathbbl{\Delta}}(X)
$$
with the $\i$-category of \textit{prismatic gauges over $X$} in the sense of \cite{Bhatt F-gauges}.
\end{remark}

In what follows, we let $\mathscr L\mX$ denote the free loop space on $\mX$. This is equivalently the cotensor $\mathscr L\mX\simeq \mX^{\mathbf T}$ of $\mX$ with the anima $\T$, or equivalently still, the self-intersection of the diagonal $\mathscr L\mX\simeq\mX\times_{\mX\times \mX}\mX$.
Loop-rotation equips $\mathscr L\mX$ with a canonical $\T$-action. Note that if $\mX=\Spec(A)$ is an affine, so are the free loop space $\mathscr LX\simeq \Spec(\mathrm{THH}(A))$ and its coarse quotient $\mathscr LX/\!/\T\simeq \Spf(\mathrm{TC}^-(A))$. The formal version works analogously, and in the next theorem statement, we are considering $\mathscr L$ as an endofunctor on formal spectral stacks.

\begin{theorem}\label{Main theorem on filtered prismatization}
Let $X$ be a quasi-syntomic $p$-adic formal scheme. Let
$$
\mathfrak X\, :=\, (\mathscr LX)^\mathrm{evp}/\!/\T 
$$
be the coarse quotient under the loop-roation action of the even periodization of the free loop space over $X$, formed in the $\i$-category $\mathrm{FSpStk}_p$ of $p$-adic formal spectral stacks. Its underlying classical formal stack may be identified as
$$
\mathfrak X^\heart \, \simeq \, X^{\widehat{\mathcal N}}
$$
with the Nygaard-complete filtered prismatization of $X$.
\end{theorem}

\begin{proof}
Since topological Hochschild homology preserves Zariski localizations -- particularly clear from its presentation $\mathrm{THH}(A)\simeq \mathbf T\otimes_{\CAlg} A$ as the tensor with the anima $\T$ in the $\i$-category of $\E$-rings, but see also \cite[Lemma 4.2]{BZN} -- the free loop space functor preserves Zariski covers. As consequence, its $p$-complete version preserves Zariski covers of $p$-adic formal spectral schemes, allowing us by Remark \ref{Remark simpler presentation for Nygaardization} to reduce to the case when $X$ is a $p$-adic formal affine.

We therefore assume that $X\simeq \Spf(R)$ for a quasi-syntomic classical ring $R$. Following \cite[Proof of Theorem 5.0.3]{HRW}, we select a particular quasi-syntomic cover of $R$ by a quasi-regular semiperfectoid $S$. Namely, start with the polynomial ring in $R$-many vairables $\mathbf Z[\mathbf N]^{\o_{\mathbf Z} R}$ and set
$$
S = (\mathbf Z[\mathbf N[\tfrac 1p]]^{\otimes_{\mathbf Z} R}\otimes_{\mathbf Z[\mathbf N]^{\otimes_{\mathbf Z}R}}R)^\wedge_p
$$ to be the $p$-completion of the base-change along the canonical surjection $\mathbf Z[\mathbf N]^{\otimes_{\mathbf Z} R}\to R$, in which the $x$-th map $\mathbf Z[\mathbf N]\to R$ picks out the element $x\in R$. As shown in  \cite[Proof of Theorem 5.0.3]{HRW} (by way of considering the analogous construction where $\mathbf Z$ is replaced by the sphere spectrum $\mathbf S$), the induced map $\mathrm{THH}(R)^\wedge_p\to \mathrm{THH}(S)^\wedge_p$ is $p$-adically eff. Note that the canonical map
$$
\mathrm{MU}\to \pi_0(\mathrm{MU})\simeq\mathbf{Z}\to R\to \mathrm{THH}(R)
$$
allows us to view this as a map of $p$-adic $\E$-algebras over $\mathrm{MU}$. It is $p$-adically eff, so by (the $p$-adic formal variant of) Corollary \ref{EVP over MU in the presence of an even eff cover} (and using the formula \eqref{Evp of aff wrt MUP} to express sharing), we obtain a simplicial formula for the even periodization
\begin{equation}\label{evp of loops via qrsp}
\left(\mathscr L\Spf(R)\right)^\mathrm{evp}\,\, \simeq \varinjlim_{\phantom{{}^\mathrm{op}}\mathbf\Delta^\mathrm{op}} \Spf((\mathrm{THH}(S^\bullet)\o_{\mathrm{MU}}\mathrm{MUP})^\wedge_p)/\mathbf G_m
\end{equation}
where $S^\bullet =R^{\otimes_R[\bullet]}$ is the cosimplicial cobar construction for the map $R\to S$.
Furthermore by (the $p$-analogue of) Proposition \ref{Computing the Gevening}, the map  $\mathrm{THH}(R)^\wedge_p\to \mathrm{THH}(S)^\wedge_p$ is $\T$-equivariant, and so \eqref{evp of loops via qrsp} exhibits the $\T$-equivariant even periodization of $\mathscr L\Spf(R)$. We may therefore pass to the coarse quotient. Since this operation respects small colimits by Theorem \ref{Existence of coarse quotients in SAG}, we obtain the stack $\mathfrak X$ from the statement of the Theorem expressed as
$$
\mathfrak X\,:=\,\left(\mathscr L\Spf(R)\right)^\mathrm{evp}/\!/\T\, \simeq \,\varinjlim_{\phantom{{}^\mathrm{op}}\mathbf\Delta^\mathrm{op}} \Spf\big(((\mathrm{THH}(S^\bullet)\o_{\mathrm{MU}}\mathrm{MUP})^\wedge_p)^{\mathrm h\T}\big)/\mathbf G_m
$$
We wish to show that the canonical map
\begin{equation}\label{TC Thh comparison}
\mathrm{TC}^-(S^\bullet)\o_\mathrm{MU}\MUP\to (\mathrm{THH}(S^\bullet)\o_\mathrm{MU}\mathrm{MUP})^{\mathrm h\T}
\end{equation}
induces an equivalence upon $(p, I)$-adic completion, where (by the evenness of $\mathrm{TC}^-(S^\bullet)$)
$$
I=I_\mathrm{aug}\subseteq\pi_*(\mathrm{TC}^-(S^\bullet)) \simeq \pi_0(\mathrm{TC}^-(S)\o_{\mathrm{MU}}\MUP)
$$
is the augmentation ideal, with respect to which the formal spectra in the defintion of the coarse quotient are always taken. For this, recall  that $I$-adic completion coincides with the Bousfield localization along $\mathrm{TC}^-(S^\bullet)\to \mathrm{THH}(S^\bullet)$. It therefore suffices to show that the map \eqref{TC Thh comparison} induces an equivalence upon base-change along $\mathrm{TC}^-(S^\bullet)\to \mathrm{THH}(S^\bullet)$. Said base-change identifies the domain of the map \eqref{TC Thh comparison} with $\mathrm{THH}(S^\bullet)\o_\mathrm{MU}\mathrm{MUP}$. The base-change of the codomain is also identified with
$$
(\mathrm{THH}(S^\bullet)\o_\mathrm{MU}\mathrm{MUP})^{\mathrm h\T}\o_{\mathrm{TC}^-(S^\bullet)}\mathrm{THH}(S^\bullet)\,\simeq \,\mathrm{THH}(S^\bullet)\o_\mathrm{MU}\mathrm{MUP}
$$
as a special case of \cite[Theorem 7.43]{MNN} (a generalization of Theorem \ref{Thanks, Akhil!} which the periodicity assumption at the cost of having to complete at the augmentation ideal in $\pi_*$ as opposed to in $\pi_0$), applied to the $\T$-equivariant $\mathrm{THH}(S^\bullet)$-module $\mathrm{THH}(S^\bullet)\o_\mathrm{MU}\mathrm{MUP}$. By chasing through all the identifications, we find that the base-change in question of the map \eqref{TC Thh comparison} gives the identity on $\mathrm{THH}(S^\bullet)\o_{\mathrm{MU}}\MUP$, and hence the map \eqref{TC Thh comparison} induces an equivalence upon $I$-completion. By further $p$-completing to arrive at the $(p, I)$-completion, we obtain an equivalence
$$
(((\mathrm{THH}(S^\bullet)\o_{\mathrm{MU}}\mathrm{MUP})^\wedge_p)^{\mathrm h\T})^\wedge_I\, \simeq \, (\mathrm{TC}^-(S^\bullet)\o_{\mathrm{MU}}\mathrm{MUP})^\wedge_{(p, I)},
$$
which is furthermore compatible with the grading that comes from $\mathrm{MUP}\simeq \bigoplus_{n\in \mathbf Z}\Sigma^{2n}(\mathrm{MU})$. The corresponding map of fromal affines is therefore $\G_m$-equivariant, and so we may rewrite \eqref{evp of loops via qrsp} in the form
$$
\mathfrak X
\,\,\simeq 
\varinjlim_{\phantom{{}^\mathrm{op}}\mathbf\Delta^\mathrm{op}} \Spf((\mathrm{TC}^-(S^\bullet)\o_{\mathrm{MU}}\mathrm{MUP})^\wedge_{(p, I)})/\mathbf G_m.
$$
Since the underlying classical formal stack functor $\mathfrak X\mapsto \mathfrak X^\heart$ preserves all small colimits, we find that it is given by
$$
\mathfrak X^\heart
\,\,\simeq 
\varinjlim_{\phantom{{}^\mathrm{op}}\mathbf\Delta^\mathrm{op}} \Spf(\pi_0((\mathrm{TC}^-(S^\bullet)\o_{\mathrm{MU}}\mathrm{MUP})^\wedge_{(p, I)}))/\mathbf G_m.
$$
Note that $\mathrm{TC}^-(S^\bullet)\o_{\mathrm{MU}}\mathrm{MUP}$ is a smash product of three even $\E$-rings, and the map $\mathrm{MU}\to\mathrm{MUP}$ is eff. It follows by \cite[Proposition 2.2.8]{HRW} (stated there in the $p$-adic case, but which extends to the present $(p, I)$-adic case just as well) that there is a natural isomorphism of graded rings
\begin{eqnarray*}
\pi_*\big((\mathrm{TC}^-(S^\bullet)\o_{\mathrm{MU}}\mathrm{MUP})^\wedge_{(p, I)}\big) &\simeq &
\pi_*(\mathrm{TC}^-(S^\bullet))^\wedge_{(p, I)}\widehat{\o}_{\pi_*(\mathrm{MU})^\wedge_{(p, I)}}\pi_*(\mathrm{MUP})^\wedge_{(p, I)}\\
 &\simeq &
\big(\pi_*(\mathrm{TC}^-(S^\bullet)){\o}_{\pi_*(\mathrm{MU})}\pi_*(\mathrm{MUP})\big)^\wedge_{(p, I)}
\end{eqnarray*}
which specializes, in light of the identification $\pi_*(\mathrm{MUP})\simeq \pi_*(\mathrm{MU})[t^{\pm 1}]$ with $t$ in graded degree $2$, on $\pi_0$ to
$$
\pi_0\big((\mathrm{TC}^-(S^\bullet)\o_{\mathrm{MU}}\mathrm{MUP})^\wedge_{(p, I)}\big)\,\simeq \,\bigoplus_{n\in \mathbf Z}\pi_{2n}(\mathrm{TC}^-(S^\bullet)^\wedge_{(I, p)})\, \simeq \, \bigoplus_{n\in \mathbf Z}\pi_{2n}(\mathrm{TC}^-(S^\bullet){}^\wedge_p),
$$
where the coproducts are taken inside the $\i$-category of $(p, I)$-complete modules, and where the second isomorphism results from observing that $\mathrm{TC}^-(S^\bullet)$ is already $I$-complete. We can therefore return to the underlying classical formal stack of $\mathfrak X$, and rewrite it as
$$
\mathfrak X^\heart
\,\,\simeq 
\varinjlim_{\phantom{{}^\mathrm{op}}\mathbf\Delta^\mathrm{op}} \Spf\big(\bigoplus_{n\in \mathbf Z}\pi_{2n}(\mathrm{TC}^-(S^\bullet){}^\wedge_p)\big)/\mathbf G_m \simeq 
\varinjlim_{\phantom{{}^\mathrm{op}}\mathbf{\Delta}^\mathrm{op}}\Spf(S^\bullet)^{\widehat{\mathcal N}},
$$
having used Bhatt-Morrow-Scholze result, recalled above as Theorem \ref{BMS result} for the second \linebreak equivalence.
Since $\Spf(R)\simeq \varinjlim_{\mathbf{\Delta}^\mathrm{op}}\Spf(S^\bullet)$, we may identify the above colimit with the Nygaard-complete filtered prismatization of $\Spf(R)$ by Remark \ref{Remark simpler presentation for Nygaardization}.
\end{proof}

\begin{remark}\label{Previous reamrk 6.4.6}
The equivalence of
Theorem \ref{Main theorem on filtered prismatization} sheds some light on why we needed to take the coarse quotient in the setting of formal geometry. It has to do with how, in the non-filtered setting, the prismatization of a quasi-regular semiperfectoid $p$-adic affine $\Spf(R)$ is given by the formal affine $\Spf(\Prism_R)$, but taken with respect to the $(p, I)$-adic topology as opposed to merely $p$-adic, where $I\subseteq\Prism_R$ is the prismatic ideal. It is this second completion parameter that is built into our spectral enhancement at the level of the passage to the coarse quotient.
\end{remark}

\begin{remark}
The issue of $(p, I)$-adic completeness, highlighted in Remark \ref{Previous reamrk 6.4.6}, should not be confused with that of \textit{Nygaard-completeness}. The latter concerns the difference between $X^{\widehat{\mathcal N}}$ as given in Definition \ref{Def of X^hat N} and $X^{\mathcal N}$ of \cite{Bhatt F-gauges}. For instance, the initial prism $\Prism_R$ over a quasi-regular semiperfectorid ring $R$ is \textit{always} $(I, p)$-complete, but it is typically \textit{not} complete with respect to its Nygaard filtration $\mathrm{Fil}_{\mathcal N}^*\,\Prism_R$ until we pass to its Nygaard completion $\widehat{\Prism}_R$. See also Remark \ref{Accounts of failure}.
\end{remark}

\begin{corollary}\label{QCoh deformation corollary}
Let $X$ be as in Theorem \ref{Main theorem on filtered prismatization}.
Then the $\i$-category $\QCoh(\mathscr L X)^{\mathrm h\T}$ of $\T$-equivariant quasi-coherent sheaves on the spectral free loop space $\mathscr LX$ is the generic fiber of a $1$-parameter deformation, with the $\i$-category of prismatic gauges $\mathrm{Gauge}_{\mathbbl{\Delta}}(X)$ as its special fiber.
\end{corollary}

\begin{proof}
If we denote $\mathfrak X\, :=\, (\mathscr LX)^\mathrm{evp}/\!/\T$ as in Theorem \ref{Main theorem on filtered prismatization}, we claim that $\mathfrak{QC}\mathrm{oh}(\mathfrak X^\mathrm{cn})$ is the total space of the deformation. Indeed, the spectral stack $\mathfrak X$ is even periodic, so on its connective cover we have by a global version of Construction \ref{Const of Bott map} a canonical Bott map $\beta :\Sigma^2(\omega)\to \mathfrak O_{\mathfrak X^\mathrm{cn}}$  in $\mathfrak{QC}\mathrm{oh}(\mathfrak X^\mathrm{cn})$. Since $\omega$ is a line bundle, this gives up to shearing (which we can make symmetric monoidal by noting that everything is happening over $\mathrm{MU}$ and using \cite{Sanath}) a $\QCoh(\mathbf A^1/\mathbf G_m)$-module structure on the $\i$-category $\mathfrak{QC}\mathrm{oh}(\mathfrak X^\mathrm{cn})$, which is precisely the structure of a 1-parameter deformation.

Its generic fiber is given by the relative Lurie tensor product
$$
\mathfrak{QC}\mathrm{oh}(\mathfrak X^\mathrm{cn})\otimes_{\QCoh(\mathbf A^1/\mathbf G_m)}\QCoh(\mathbf G_m/\mathbf G_m)\, \simeq \, \Mod_{\mathfrak O_{\mathfrak X^\mathrm{cn}}[\beta^{-1}]}(\mathfrak{QC}\mathrm{oh}(\mathfrak X^\mathrm{cn}))\, \simeq \, \mathfrak{QC}\mathrm{oh}(\mathfrak X),
$$
where we have used that the canonical map $f:\mathfrak X\to \mathfrak X^\mathrm{cn}$ is a formally affine morphism, and $f_*(\mathfrak O_{\mathfrak X})\simeq \mathfrak O_{\mathfrak X}[\beta^{-1}]$. From Proposition \ref{QCoh doesn't care} we obtain further equivalences\footnote{We are writing $\QCoh$, but to be precise, we should still be writing $\mathfrak{QC}\mathrm{oh}$, though only taken with respect to $p$-completions. That is to say, this is the usual $\QCoh$ of $p$-adic formal spectral geometry.}
$$
\mathfrak{QC}\mathrm{oh}(\mathfrak X)\,\simeq \, \QCoh((\mathscr LX)^\mathrm{evp}/\underline{\mathbf T}_S)\, \simeq \, \QCoh((\mathscr LX)^\mathrm{evp})^{\mathrm h\T}
$$
Since everything is taking place over $\Spec(\mathbf Z)\simeq \Spec(\pi_0(\mathrm{MU}))\to\Spec(\mathrm{MU})$, we see by Theorem \ref{Corollary unshear for modules} that shearing and unshearing give a symmetric monoidal identification
$$
\QCoh((\mathscr LX)^\mathrm{evp})\, \simeq \, \QCoh(\mathscr LX),
$$
and since the $\T$-action on $(\mathscr LX)^\mathrm{evp}$ is induced by passage to even periodization from that on $\mathscr LX$, this equivalence is also $\T$-equivariant. The generic fiber of the $\i$-categorical deformation in question is therefore indeed given by $\QCoh(\mathscr LX)^{\mathrm h\T}$.

The special fiber is given similarly as
$$
\mathfrak{QC}\mathrm{oh}(\mathfrak X^\mathrm{cn})\otimes_{\QCoh(\mathbf A^1/\mathbf G_m)}\QCoh(\{0\}/\mathbf G_m)\, \simeq \, \Mod_{\mathfrak O_{\mathfrak X^\mathrm{cn}}/\beta}(\mathfrak{QC}\mathrm{oh}(\mathfrak X^\mathrm{cn}))\, \simeq \, \mathfrak{QC}\mathrm{oh}(\mathfrak X^\heart),
$$
where we have once again used that $\mathfrak X^\heart\to\mathfrak X^\mathrm{cn}$ is affine, and that  the complex periodicity of $\mathfrak X$, we have a canonical cofiber sequence
$$
\Sigma^2(\omega)\xrightarrow{\beta}\mathfrak O_{\mathfrak X^\mathrm{cn}}\to \pi_0(\mathfrak O_{\mathfrak X^\mathrm{cn}})\, \simeq \, \mathfrak O_{\mathfrak X^\heart}
$$
in the $\i$-category $\mathfrak{QC}\mathrm{oh}(\mathfrak X^\mathrm{cn})$. Since $\mathfrak X^\heart$ is a classical formal stack, we can invoke Theorem \ref{Prop 5.2.5} to identify the special fiber in question with
$$
\mathfrak{QC}\mathrm{oh}(\mathfrak X^\heart)\,\simeq \,\QCoh(\mathfrak X^\heart),
$$
and as such by the conclusion of Theorem \ref{Main theorem on filtered prismatization} and Remark \ref{Thanks, Deven!} (which is to say, the results of \cite{Deven}, with the $\i$-category of prismatic gauges $\mathrm{Gauge}_{\mathbbl{\Delta}}(X)$.
\end{proof}

\begin{remark}
Since the formal spectral stack $\mathfrak X$ in Theorem \ref{Main theorem on filtered prismatization} is even periodic by construction, it is in particular complex periodic. As  such, it admits an essentially unique map to the chromatic base stack $\mM$. This map $\mathfrak X\to \mM$ induces upon passage to the underlying classical formal stacks a map
$$
X^{\widehat{\mathcal N}}\simeq \mathfrak X^\heart\to \mM^\heart\simeq \mM^\heart_\mathrm{FG}
$$
into the classical moduli space of formal groups. That is to say, it exhibits a canonical formal group on the Nygaard-complete filtered prismatization of $X$. It follows by the work of Manam  \cite[Theorem 0.1]{Deven} that this formal group coincides with (the pullback along the canonical map $X^{\widehat{\mathcal N}}\to X^{\mathcal N}$ of) the one exhibited by Drinfeld in \cite{Drinfeld FG}. We should point out how non-trivial this is; the Drinfeld formal group encodes the multiplication law for the ``prismatic first Chern classes", i.e.\ it concerns $\mathrm H^2_{\mathbbl \Delta}(\mathrm B\mathbf G_m)$. In contrast, the Quillen formal group encodes the multiplication law for the topological first Chern class of the even cohomology theory $\mathrm{TC}^-(R)^\wedge_p$, i.e.\ it concerns $(\mathrm{TC}^-(R)^\wedge_p)^2(\mathrm{BU}(1))$.
Them agreeing is therefore a surprising phenomenon of a motivic nature.
\end{remark}

The conclusion of Theorem \ref{Main theorem on filtered prismatization} may be summarized as asserting that  $(\mathscr LX)^\mathrm{evp}/\!/\T$ provides an \textit{even periodic enhancement} of the Nygaard-complete filtered prismatization $X^{\widehat{\mathcal N}}$ for any quasi-syntomic $p$-adic formal scheme $X$. Observe however that unlike the definition of prismatization itself, the building blocks of this enhancement: the free loop space, even periodization, and the coarse quotient, are in no way special to the $p$-complete setting. They require no knowledge of anything about quasi-regular semiperfectoid rings, nor prisms, nor Cartier-Witt divisors, for instance. As such, we are tempted to make the following provisional definition:

\begin{definition}\label{Def of Ncfp}
Let $\mathfrak X$ be any formal spectral stack. Its \textit{Nygaard-complete filtered prismatization} is the classical formal stack
$$
\mathfrak X{}^{\widehat{\mathcal N}} \, \simeq \, (\left(\mathscr L\mathfrak X\right)^{\mathrm{evp}_\T}/\!/\T)^\heart.
$$
\end{definition}

\begin{remark}
This should likely be accompanied with some kind of quasi-syntomicity requirement on $\mathfrak X$, e.g.\ that it is chromatically quasi-lci in the sense of \cite{BMS2}. Under such an assumption, we would for instance have that the $\E$-ring of functions $\sO(\mathfrak X^{\widehat{\mathcal N}})$ agrees with the negative topological cyclic homology $\mathrm{TC}^-(\mathfrak X)$, which is to say, with the $\E$-ring of functions $\sO(\mathscr L\mathfrak X/\T)$.
\end{remark}

\begin{remark}\label{Agreement with R^N}
In forthcoming work of Devalapurkar-Hahn-Raksit-Yuan, as announced for instance in \cite{Arpon's talk}, the authors define a stack $R^{\widehat{\mathcal N}}$ for a chromatically quasi-syntomic $\E$-ring $R$. Though their construction does not use even periodic $\E$-rings but rather even $\E$-rings, arguments along the lines of Proposition \ref{Even filtration is even periodic filtration} show that this agrees with Definition \ref{Def of Ncfp}.
That is, there is a canonical equivalence of classical formal stacks
 $\Spf(R)^{\widehat{\mathcal N}}\simeq \widehat{R^{\mathcal N}}$, and so our notion of Nygaard-complete filtered prismatization agrees with that of Devalapurkar-Hahn-Raksit-Yuan. Also announced in \cite{Arpon's talk} was the computation of the stack $\widehat{\mathbf S^{\mathcal N}}$, which agrees with our Example \ref{Exun1end} below.
\end{remark}

Definition \ref{Def of Ncfp} should perhaps be taken as a \textit{preliminary} approach to prismatization outside the setting of $p$-adic geometry, analogously to how the Bhatt-Morrow-Scholze result Theorem \ref{BMS result} offered a narrow window into prismatic cohomology prior to the actual development of prisms. In particular, we offer no suggestion as to what ``even periodic prisms" might be. Nonetheless, this approach offers the following tantalizing computations:

\begin{exun}\label{Exun1end}
Noting that $\mathscr L\Spec(\mathbf S)\simeq \Spec(\mathbf S)$, we find that the Nygaard-complete filtered prismatization of the sphere spectrum 
$$
\Spec(\mathbf S)^{\widehat{\mathcal N}}\,\simeq\, (\Spec(\mathbf S)^{\mathrm{evp}}/\!/\T)^\heart \, \simeq \, (\mM/\!/\T)^\heart\, \simeq \, (\w{\G}{}^{\mathcal Q}_{\mM})^\heart\, \simeq \, \widehat{\G}{}^{\mathrm{uni}}_{\mM^\heart}
$$
is the universal classical formal group (always assumed to be smooth and $1$-dimensional) over the classical stack of formal groups $\mM^\heart\simeq \mM^\heart_\mathrm{FG}$. This of course also coincides with the Nygaard-complete filtered prismatization of the chromatic base stack $\mM$.
\end{exun}

\begin{exun}
The same kind of computation as in Example \ref{Exun1end} (when implicitly working $(p)$-locally for any choice of prime $p$) identifies the Nygaard-complete filtered prismatization of the $L_n$-local sphere 
$$
\Spec(L_n\mathbf S)^{\widehat{\mathcal N}}\,\simeq \,\widehat{\G}{}^{\mathrm{uni}}_{\mM^{\heart, \le n}},
$$
with the universal classical formal group of height $\le n$.
\end{exun}

\begin{exun}\label{Exun1endend}
For Morava E-theory $E_n=E(\kappa, \w{\G}_0)$, i.e.\ the Lubin-Tate spectrum of  formal group $\w{\G}_0$ of height $n$ over a perfect field $\kappa$, we consider the formal affine $\Spf(E_n)$ with respect to the adic topology given by the Landweber ideal. Then since $L_{K(n)}\mathrm{THH}(E_n)\simeq E_n$, the Nygaard-complete filtered prismatization
$$
\Spf(E_n){}^{\widehat{\mathcal N}}\, \simeq \, \widehat{\G}{}^\mathrm{uni.def.}
$$
is the universal classical deformation of $\w{\G}_0$.
\end{exun}

\begin{remark}\label{Accounts of failure}
Though we are able to construct  \textit{ad hoc} even periodic enhancements in some specific cases, we have been generally unable to find a way of using even periodization to give rise to a natural
spectral enhancement of the prismatization $X^{\mathbbl\Delta}$, or  some kind of Nygaard-complete variant $X^{\widehat{\mathbbl\Delta}}$ thereof, analogously to what  Theorem \ref{Main theorem on filtered prismatization} does for $X^{\widehat{\mathcal N}}$. The fundamental issue is roughly that we would like to do something like restricting to the open point $*\simeq \mathbf G_m/\mathbf G_m\subseteq\mathbf A^1/\mathbf G_m$, but our objects only live over the formal locus $\widehat{\mathbf A}^1/\mathbf G_m$, from which the open point is inaccessible. The traditional way to get around such an issue in the analogous $p$-adic formal setting is to tap into rigid analytic geometry for access to a generic fiber construction. It is therefore unsurprising that, with the conjectural theory of the geometric fiber $\eta$, as hypothesized in Remark \ref{If only special fiber}, at our disposal, the issue would disappear.
The desired even periodic enhancement of an analytic version of $X^{\widehat{\mathbbl \Delta}}$ would then be given by
the generic fiber of the coarse quotient $((\mathscr LX)^\mathrm{evp}/\!/\T)_\eta$; see also   Remark \ref{Remark TP} and \cite[Remark 2.13]{Deven} for related discussion.
Alternatively, instead of (or perhaps concurrently to) tapping into analytic geometry, the fundamental issue could also be sidestepped by ``decompleting" -- see \cite[Section 1]{Deven} for some work in this direction. 

A (perhaps different) way this could be achieved would be to systematically build in \textit{genuine $\T$-equivariance} of $\mathscr LX$, replacing the role played by the homotopy fixed-points $A^{\mathrm h\T}$ in the preceding sections with the categorical fixed-points $A^{\T}$. Computational results of Atiyah-Segal type suggest that this tends to have the effect of ``decompleting" rings in question, i.e.\ of replacing formal power series rings with polynomial ones. This should also connect to equivariant formal groups in the sense of \cite{Greenlees}, and such a role of genuine $\mathbf T$-equivariant structures on complex oriented ring spectra is discussed in \cite[Section 1.4]{Sanath}. A closely-related construction (without a periodicity assumption) is carried out in the upcoming work of Devalapurkar, Hahn, Raksit, and Yuan, announced in \cite{Jeremy's talk}.

\end{remark}

\subsection{Prismatization over the rationals}
According to
Definition
\ref{Def of Ncfp}
of
the previous subsection, we can make sense of the Nygaard-complete filtered prismatization of various algebro-geometric objects that fall outside the scope of $p$-adic formal geometry. We have already considered some ramifications of this for the sphere spectrum and related objects in Examples \ref{Exun1end} - \ref{Exun1end}. As something of a sanity-check, let us verify that this recovers a familiar construction when working over the rationals.

Here we let $X^{\mathrm{dR}, +}$ denote the \textit{filtered de Rham stack of $X$}, following the terminology of \cite[Definition 2.3.5]{Bhatt F-gauges}. This object was introduced by Simpson in \cite{Simpson} under the name of the \textit{Hodge stack of $X$} and denoted $X_\mathrm{Hod}$, as it is  known in most of the literature. It comes with a canonical map $X^{ \mathrm{dR}, +}\to\mathbf A^1/\mathbf G_m$, which encodes the Hodge filtration on de Rham cohomology of $X$; see \cite[Theorem 2.3.6]{Bhatt F-gauges} for a very clear exposition.

\begin{remark}
Though we will not need it in what follows, let us describe the relative functor of points of the filtered de Rham stack $X^{\mathrm{dR}, +} : (\mathrm{Aff}_{/(\mathbf A^1/\mathbf G_m)})^\mathrm{op}\to \mathrm{Ani}$. If we identify a map $\Spec(R)\to \mathbf A^1/\mathbf G_m$ with the pair of an invertible $R$-module $L$ and an $R$-linear map $s:R\to L$, then we have
$$
X^{\mathrm{dR}, +}(\Spec(R)\to \mathbf A^1/\mathbf G_m)\, :=\, X(\mathrm{cofib}(\mathrm{Nil}(R)\otimes_R L\xrightarrow{s} R)).
$$
It is clear from this that the pullback of $X^{\mathrm{dR}, +}$ along the open point $*\simeq \mathbf G_m/\mathbf G_m\to \mathbf A^1/\mathbf G_m$, which amounts to restricting to the case where $s:R\to L$ is an equivalence of $R$-modules,
 recovers the perhaps more familiar de Rham space $X^\mathrm{dR}$, given by
 $$
X^{\mathrm{dR}}(R)\, :=\, X(R/\mathrm{Nil}(R))\, \simeq \, X(R^\mathrm{red}).
 $$
 This is the ``generic fiber" of $X^{\mathrm{dR}, +}$. Its ``special fiber" -- the pullback along the closed point $*\to \mathrm B\mathbf G_m\simeq \{0\}/\mathbf G_m\to \mathbf A^1/\mathbf G_m$ -- recovers the $-1$-shifted derived tangent bundle $\widehat { T}[-1]X$, completed at the zero section, which Bhatt in \cite{Bhatt F-gauges} denotes by $X^\mathrm{Hod}$ and Simpson in \cite{Simpson} by $X_\mathrm{Dol}$.
\end{remark}

\begin{theorem}\label{Char 0 prism}
Let  $X\to \Spec(\mathbf Q)$ be a quasi-lci and locally of finite type map of classical schemes. There is a canonical equivalence of classical formal stacks $X^{\widehat{\mathcal N}}\,\simeq\, X^{\mathrm{dR}, +}$.
\end{theorem}

For the proof of Theorem \ref{Char 0 prism}, we require a prerequisite result. We denote by
\footnote{
Though the notation ${\mathrm L\Omega}_{A/k}$ for the relative derived de Rham complex is rather standard, used for instance in \cite{BMS2} and \cite{Antieau B-filtration}, another commonly-used notation is
$\mathrm{dR}_{A/k}$, as used for instance in \cite{Bhatt completions}, \cite{BLa}, \cite{Arpon}. It should be noted that this \textit{is not} merely the complex of differential forms. That is to say, for a smooth $k$-algebra, it is not merely the graded complex of differential forms $\Omega^*_{A/k} = \Lambda^*_A(\Omega^1_{A/k})$, but rather (the $A$-module spectrum presented by) the chain complex obtained from it via the de Rham differential $\mathrm d :\Omega^*_{A/k}\to \Omega^{*+1}_{A/k}$.} $\widehat{\mathrm L\Omega}_{A/k}$ the \textit{Hodge-completed derived de Rham complex} of a commutative $k$-algebra $A$, with the respective Hodge-filtered pieces denoted $\widehat{\mathrm L\Omega}_{A/k}^{\ge *}$. See \cite[Construction 4.1]{Bhatt completions}, \cite[Construction E.2]{BLa}, or \cite[Recollection 5.3.8]{Arpon} for a careful discussion.
Recall also from  \cite[Definition 4.1.1]{HRW} that a map of commutative rings $S\to R$ is called a \textit{quasi-regular quotient} if it is surjective and its algebraic cotangent complex $L_{S/R}^\mathrm{alg}$ is concentrated in homotopical degrees $[0,1]$ (or equivalently, in degree $1$, see \cite[Proposition 4.1.3]{HRW}).

\begin{lemma}\label{Lemma 6.5.2}
    Let $S\to R$ be a quasi-regular quotient of commutative rings. There is a canonical isomorphism $\pi_{2n}(\mathrm{HC}^-(R/S)) \, \simeq \, \widehat{\mathrm L\Omega}{}^{\ge 2n}_{R/S}$ for all $n$.
\end{lemma}

\begin{proof}
It was shown in \cite[Theorem 5.0.2]{HRW}
that the $\mathrm h\T$-version of the even filtration on $\mathrm{HH}(A/k)$ agrees with Antieau's motivic filtration on $\mathrm{HC}^-(A/k)$ from
\cite{Antieau B-filtration} for any quasi-lci map of commutative rings $k\to A$. The latter filtration is designed by \cite[Theorem 1.1]{Antieau B-filtration} to have its associated graded given by the pieces of the Hodge filtration on $\widehat{\mathrm{L}\Omega}_{A/k}$. Consequently, we have a canonical identification
\begin{equation}\label{gr of B-filtr}
\mathrm{gr}^*_{\mathrm{ev}, \mathrm h\T}(\mathrm{HH}(A/k))\,\simeq\, \Sigma^{2*}(\widehat{\mathrm L\Omega}^{\ge *}_{A/k}).
\end{equation}
Since a quasi-regular quotient is quasi-lci by definition, this applies to the map $S\to R$. 
Thanks to it being the quasi-regular quotient, the relative Hochschild homology $\mathrm{HH}(R/S)$ is even by
\cite[Proposition 4.2.5]{HRW}, and so the even filtration by definition reduces to the double-speed Postnikov filtration. That is to say, we have
$$
\mathrm{fil}_{\mathrm{ev}, \mathrm h\T}^{\ge *}(\mathrm{HH}(R/S))\, \simeq\, \tau_{\ge 2*}(\mathrm{HC}^-(R/S)).
$$
Since $\mathrm{HC}^-(R/S)$, as homotopy $\T$-fixed-points of the even $\E$-ring $\mathrm{HH}(R/S)$, is even itself, the associated graded of its double-speed Postnikov filtration is canonically identified with $\Sigma^{2*}\pi_{2*}(\mathrm{HC}^-(R/S))$. The desired result follows by comparison with \eqref{gr of B-filtr}.
\end{proof}

\begin{proof}[Proof of Thoerem \ref{Char 0 prism}]
As in the proof of Theorem \ref{Main theorem on filtered prismatization}, we may use Zariski descent to reduce the affine case $X=\Spec(R)$ for a classical quasi-lci finite-type $\mathbf Q$-algebra $R$. That means that there exists a polynomial $\mathbf Q$-algebra (on finitely-many variables) $S$ and a $\mathbf Q$-algebra surjection $S\to R$ which is a quasi-regular quotient.
As shown in
\cite[Proof of Theorem 5.0.2]{HRW} (or more precisely, \cite[Corollary 4.2.6]{HRW}), the map
\begin{equation}\label{THH = HH over Q}
\mathrm{THH}(R)\,\simeq\, \mathrm{HH}(R/\mathbf Q)\to \mathrm{HH}(R/S)
\end{equation}
is an even eff cover, whose \v{C}ech nerve is given by the cosimplicial $R$-algebra $\mathrm{HH}(R/S^\bullet)$ where $S^\bullet =S^{\otimes_{\mathbf Q}[\bullet]}$.
Then we may argue as in the proof of Theorem \ref{Main theorem on filtered prismatization} to arrive at the simplicial formula for the Nygaard-complete filtered prismatization
$$
\Spf(R){}^{\widehat{\mathcal N}}\,:=\,\,(\left(\mathscr L\Spf(R)\right)^\mathrm{evp}/\!/\T)^\heart
\,\,\simeq 
\varinjlim_{\phantom{{}^\mathrm{op}}\mathbf\Delta^\mathrm{op}} \Spf\big(\bigoplus_{n\in \mathbf Z}\pi_{2n}(\mathrm{HC}^-(R/S^\bullet))\big)/\mathbf G_m.
$$
We may rewrite the summands inside the formal spectrum on the right-hand side through the equivalence $\pi_{2n}(\mathrm{HC}^-(R/S^\bullet)) \, \simeq \, \widehat{\mathrm L\Omega}{}^{\ge 2n}_{R/S^\bullet}$ of Lemma \ref{Lemma 6.5.2}. That is to say, the formal spectrum is taken with respect to the Rees algebra of the Hodge filtration on the completed derived de Rham complex $\widehat{\mathrm L\Omega}_{R/S^\bullet}$.
 The map of $\mathbf Q$-algebras $S^\bullet\to R$ is surjective, so \cite[Proposition 4.16]{Bhatt completions} applies to it to identify the Hodge filtration on $\widehat{\mathrm{L}\Omega}_{R/S^\bullet}$ with the adic filtration on the derived completion $\widehat S^\bullet$ of $S^\bullet$ along the kernel of the map $S^\bullet\to S\to R$. If we let $\mathrm{Rees}(\widehat S^\bullet)$ denote the Rees construction of the adic filtration on $\widehat{S}^\bullet$, we thus obtain an identification
$$
\Spf(R){}^{\widehat{\mathcal N}} \, \simeq \,
\varinjlim_{\phantom{{}^\mathrm{op}}\mathbf\Delta^\mathrm{op}} \Spf(\mathrm{Rees}(\widehat S^\bullet))/\mathbf G_m.
$$
For any finitely generated ideal $I\subseteq A$ in a classical ring $A$, there is, ostensibly by the construction of the filtered de Rham stack, a canonical pullback square
\begin{equation}\label{Pullback square dR +}
\begin{tikzcd}
\Spf(\mathrm{Rees}(A^\wedge_I))/\mathbf G_m \arrow[d] \arrow[r] & \Spec(A/I)^{\mathrm{dR}, +} \arrow[d] \\
\Spec(A)\times \mathbf A^1/\mathbf G_m\arrow[r]  & \Spec(A)^{\mathrm{dR}, +},
\end{tikzcd}
\end{equation}
 of stacks over $\mathbf A^1/\mathbf G_m$,
where the Rees construction on the upper left corner is taken with respect to the $I$-adic filtration of the derived completion  $A^\wedge_I$. Let us apply this to the ring $S$ and the kernel of the surjection $S\to R$, to  obtain a Cartesian diagram
$$
\begin{tikzcd}
\mathrm{Spf}(\mathrm{Rees}(\widehat S))/\mathbf G_m \arrow[d] \arrow[r] & \Spec(R)^{\mathrm{dR}, +} \arrow[d] \\
\Spec(S)\times {\mathbf A}^1/\mathbf G_m\arrow[r]  & \Spec(S)^{\mathrm{dR}, +}.
\end{tikzcd}
$$
Thanks to the quasi-smoothness of $S$, the lower horizontal map is an \'etale cover, and thus in particular an fpqc cover. the upper horizontal arrow is likewise, and so the claim follows from noting that $\Spf(\mathrm{Rees}(\widehat S^\bullet))/\mathbf G_m$ is its \v{C}ech nerve. \end{proof}

\begin{remark}
By base-change along the open point $*\simeq \mathbf G_m/\mathbf G_m\to \mathbf A^1/\mathbf G_m$, the pullback square  \eqref{Pullback square dR +}
assumes a likely more familiar form
$$
\begin{tikzcd}
\mathrm{Spf}(A^\wedge_I) \arrow[d] \arrow[r] & \Spec(A/I)^{\mathrm{dR}} \arrow[d] \\
\Spec(A)\arrow[r]  & \Spec(A)^{\mathrm{dR}},
\end{tikzcd}
$$
the usual way in which de Rham spaces encode completions (if $A\to A/I$ is quasi-regular and $I$ finitely generated -- which is the context we use this above -- this is a special case of the much more general \cite[Theorem 18.2.3.1]{SAG}). The content of \eqref{Pullback square dR +} is that the filtered de Rham completion also keeps track of the adic filtration on said formal completions.
\end{remark}

Relatively little of the proof of Theorem \ref{Char 0 prism} made use of the assumption of working over $\mathbf Q$. The key result we used, the identification of Hodge filtration on the relative derived de Rham complex and the adic filtration on formal completions of \cite[Proposition 4.16]{Bhatt completions} does require working with \textit{$\mathbf Q$-algebras}. But the only place in the proof where we used that $X$ is quasi-lci over $\mathbf Q$ itself was in \eqref{THH = HH over Q}, that is to say, to identify $\mathrm{THH}(R)\simeq \mathrm{HH}(R/\mathbf Q)$.
If we replace it with the relative Hochschild homology over $k$, which is to say if we consider the free loop space over $\Spec(k)$ (as opposed to over the absolute base $\Spec(\mathbf S)$)
$$
\mathscr L_{/k}X\,\simeq \, \underline{\mathrm{Map}}_{\mathrm{Stk}_{/\Spec(k)}^\heart}(\underline{\mathbf T}_k, X)  \,\simeq \, X\times_{X\times_{\Spec(k)}X}X,
$$
we may repeat the same proof to obtain the following generalization of Thoerem \ref{Char 0 prism}:

\begin{corollary}\label{Char 0 prism over k}
Let  $X\to \Spec(k)$ be a quasi-lci and locally of finite type presentation of classical schemes over $\mathbf Q$. There is a canonical equivalence of classical stacks
$$
((\mathscr L_{/k}X)^\mathrm{evp}/\!/\T)^\heart\,\simeq\, X^{\mathrm{dR}, +}.
$$
\end{corollary}

We obtain a characteristic zero analogue of the deformation result Corollary \ref{QCoh deformation corollary}. Here let us denote
$$
\mathrm{DMod}^\mathrm{fil}(X)\,:=\, \QCoh(X^{\mathrm{dR}, +}),
$$
motivated by the fact that this indeed recovers the usual notion of filtered D-modules when $X$ is smooth.

\begin{corollary}
Let $X$ and $k$ be as in Corollary \ref{Char 0 prism over k}.
The $\i$-category $\QCoh(\mathscr L_{/k} X)^{\mathrm h\T}$ of $\T$-equivariant quasi-coherent sheaves on the free loop space $\mathscr L_{/k}X$ is the generic fiber of a $1$-parameter deformation, with the $\i$-category of filtered $D$-modules $\mathrm{DMod}^\mathrm{fil}(X)$ as its special fiber.
\end{corollary}

\begin{proof}
The total space of this deformation is given by $\QCoh((X^{\mathrm{dR}, +})^\mathrm{cn})$. The statement in question follows by repeating the proof of Corollary \ref{QCoh deformation corollary}.
\end{proof}

This is closely related to the main results of \cite{BZN}, as well as to the filtered loop spaces and the universal HKR theorem of \cite{Filtered circle}, specifically to the synthetic version discussed in \cite{Alice and Tasos} and the explicit connection to the filtered de Rham stack in \cite{Tasos Hodge}.

\end{document}